\documentclass[11pt,a4paper]{article}

\usepackage{amsfonts}
\usepackage{amsthm}
\usepackage{amsmath}
\usepackage{amssymb}
\usepackage{amscd}
\usepackage{mathrsfs}
\usepackage[mathscr]{eucal}
\usepackage{graphicx}
\usepackage{epstopdf} 
\usepackage{pict2e}
\usepackage{epic}
\usepackage{xcolor}
\usepackage{float}
\usepackage{mathtools}
\usepackage{centernot}

\usepackage{cite}
\usepackage{hyperref}
\hypersetup{colorlinks=true, urlcolor= black, linkcolor=black, citecolor=black}

\usepackage[utf8]{inputenc}
\usepackage[T1]{fontenc}
\usepackage{lmodern}
\usepackage{dsfont}
\usepackage{indentfirst}

\numberwithin{equation}{section}

\theoremstyle{plain}
\newtheorem{Thm}{Theorem}[section]
\newtheorem{Lem}[Thm]{Lemma}
\newtheorem{Coro}[Thm]{Corollary}
\newtheorem{Prop}[Thm]{Proposition}

\theoremstyle{definition}
\newtheorem{Def}[Thm]{Definition}

\newtheorem{Rem}[Thm]{Remark}

\usepackage[margin=2.9cm]{geometry}

\newcommand\n{\mathbf{n}}
\newcommand\m{\mathbf{m}}
\newcommand\sn{\partial\mathbf{n}}
\newcommand\sm{\partial\mathbf{m}}

\newcommand{\connect}{\xleftrightarrow}
\newcommand\urs{U_4^\beta}
\newcommand\xb{x_{\bot}}
\newcommand\yb{y_{\bot}}

\title{Triviality of the scaling limits of critical Ising and $\varphi^4$ models with effective dimension at least four}
\begin{document}

\author{Romain Panis\footnotemark[1]\footnote{Institut Camille Jordan, \url{panis@math.univ-lyon1.fr}}}
\date{}
\maketitle
\begin{abstract}
We prove that any scaling limit of a critical reflection positive Ising or $\varphi^4$ model of effective dimension $d_{\textup{eff}}$ at least four is Gaussian. This extends the recent breakthrough work of Aizenman and Duminil-Copin \cite{AizenmanDuminilTriviality2021}--- which demonstrates the corresponding result in the setup of nearest-neighbour interactions in dimension four--- to the case of long-range reflection positive interactions satisfying $d_{\textup{eff}}=4$. The proof relies on the random current representation which provides a geometric interpretation of the deviation of the models' correlation functions from Wick's law. When $d=4$, long-range interactions are handled with the derivation of a criterion that relates the speed of decay of the interaction to two different mechanisms that entail Gaussianity: interactions with a sufficiently slow decay induce a faster decay at the level of the model's two-point function, while sufficiently fast decaying interactions force a simpler geometry on the currents which allows to extend nearest-neighbour arguments. When $1\leq d\leq 3$ and $d_{\textup{eff}}=4$, the phenomenology is different as long-range effects play a prominent role.
\end{abstract}

\section{Introduction}

\subsection{Motivation}
We are interested in  ferromagnetic real-valued spin models on $\mathbb Z^d$ that arise in statistical mechanics. Mathematically, these models can be seen as probability measures on spin configurations $\tau : \mathbb Z^d\rightarrow \mathbb R$ formally given by
\begin{equation}\label{eq: model we study}
    \langle F(\tau)\rangle_\beta=\frac{1}{Z}\int F(\tau) \exp\bigg(\beta \sum_{x,y\in \mathbb Z^d}J_{x,y}\tau_x\tau_y\bigg)\prod_{x\in \mathbb Z^d}\textup{d}\rho(\tau_x),
\end{equation}
where $\beta>0$ is the inverse temperature, $J_{x,y}\geq 0$ are (possibly long-range) interactions, $Z$ is a normalisation constant, and $\textup{d}\rho$ is a single-site probability measure. Of particular interest to us are the Ising model which corresponds to choosing $\textup{d}\rho$ to be the uniform measure on $\lbrace-1,+1\rbrace$, and the $\varphi^4$ model which corresponds to confining the spins in a quartic potential given by
\begin{equation}\label{eq: phi4 measure}
 \textup{d}\rho(\varphi)=\frac{1}{z_{g,a}}e^{-g\varphi^4-a\varphi^2}\textup{d}\varphi,
\end{equation}
with $g>0$, $a\in \mathbb R$ and $z_{g,a}$ a normalisation constant, and where $\textup{d}\varphi$ is the Lebesgue measure on $\mathbb R$. 

The study of these models plays a key role in two distinct, yet interacting, research areas: \textit{constructive Euclidean field theory} and \textit{statistical mechanics}.

Constructive Euclidean field theory aims at constructing random distributions on $\mathbb R^d$, with a particular focus on interacting, or non-\textit{trivial}, field theories. This contrasts with the study of Gaussian fields, which have a trivial correlation function structure in the sense that it is entirely determined by their two-point function via Wick's law. A natural attempt to build non-Gaussian field theories is to try to define a measure on the set of functions $\mathbb R^d\rightarrow \mathbb R$ whose averages are given by
\begin{equation}
    \langle F(\Phi)\rangle=\frac{1}{Z}\int F(\Phi) \exp\left(-H(\Phi)\right)\prod_{x\in \mathbb R^d}\mathrm{d}\Phi_x,
\end{equation}
with
\begin{equation}
   H(\Phi):=\int_{\mathbb R^d}\left[A|\nabla \Phi(x)|^2+B|\Phi(x)|^2+ P(\Phi(x))\right]\mathrm{d}x,
\end{equation}
where $A,B>0$ and $P$ is an even polynomial of degree $4$ with a strictly positive leading coefficient. This choice corresponds to what would be the definition of the $\varphi^4$ field theory on $\mathbb R^d$. Due to the lack of a natural Lebesgue measure on infinite dimensional spaces, the above quantity is ill-defined. However, it is still possible to make sense of it using a pair of \textit{ultraviolet} (short distance) and \textit{infrared} (long distance) cutoffs. Highlights of this approach include the rigourous construction of the $\varphi^4$ measure, with infrared cutoff, in dimension two by Nelson \cite{Nelson1966PHI42d}, and in dimension three by Glimm and Jaffe \cite{GlimmJaffe1973PHI43d}. These works were later extended to the infinite volume limit \cite{MagnenSeneor1976infinitePHI4,GlimmJaffeQuantumBOOK}. A few years after these first results, Aizenman \cite{AizenmanGeometricAnalysis1982} and Fröhlich \cite{FrohlichTriviality1982} showed that $\varphi^4$ is not a good candidate to construct interacting field theories when $d\geq 5$. In their works, they proved that any field obtained as a scaling limit of critical Ising or $\varphi^4$ models in dimension $d\geq 5$ is Gaussian. These papers, and other subsequent works \cite{Sokal1982Destructive,AizenmanGrahamRenormalizedCouplingSusceptibility4d1983,GawedzkiKupiainen1985massless,FeldmanMagnenRivasseau1987construction,Hara1987rigorous,BauerschmidtBrydgesSlade2014Phi4fourdim}, provided strong heuristics that the same result should hold in dimension $d=4$. It was not until very recently that these heuristics were confirmed by the work of Aizenman and Duminil-Copin \cite{AizenmanDuminilTriviality2021}. 

Constructive Euclidean field theory is also closely related to \textit{constructive quantum field theory} (CQFT). Indeed, the Osterwalder--Schrader Theorem \cite{OsterwalderSchrader1973axioms,OsterwalderSchrader1975axioms} provides a way to build quantum field theories in the sense proposed by Wightman \cite{Wightman1956QFTAxioms} from Euclidean field theories. We refer to \cite{GlimmJaffeQuantumBOOK,AizenmanDuminilTriviality2021,AizenmanReviewTriviality2021} for a more complete description of the CQFT point of view.

From the perspective of statistical mechanics, the Ising model and the $\varphi^4$ model are among the simplest examples which exhibit a phase transition\footnote{The conditions $J$ has to satisfy for a phase transition to occur are recalled below.} at a critical parameter $\beta_c\in (0,\infty)$. As proved in \cite{AizenmanDuminilSidoraviciusContinuityIsing2015,GunaratnamPanagiotisPanisSeveroPhi42022}, this phase transition is \textit{continuous} for reflection positive interactions\footnote{This result holds for all reflection positive interactions if $d\geq 3$, and is restricted to some interactions when $d\in \lbrace 1,2\rbrace$.}, and one of the main challenges of the field is to understand the nature of their scaling limits at criticality.

The connection between the Ising model and the $\varphi^4$ model is predicted to be very rich: they are believed to belong to the same universality class. Renormalisation group heuristics (see \cite{Griffiths1970RG,Kadanoff1993RG} or the recent book \cite{BauerschmidtBrydgesSladeBOOKRG2019}) predict that at their respective critical points many of their properties (e.g. critical exponents) coincide exactly. Hints of these deep links where established by Griffiths and Simon in \cite{GriffithsSimon}, where they show that the $\varphi^4$ model emerges as a certain near-critical scaling limit of a collection of mean-field Ising models. This permits to transfer rigourously many useful properties of the Ising model, such as correlation inequalities, to the $\varphi^4$ model. In the other direction, the Ising model can be obtained as a limit of $\varphi^4$ using the following limit
\begin{equation}
    \frac{\delta_{-1}+\delta_1}{2}=\lim_{g\rightarrow \infty}\frac{1}{z_{g,-2g}}e^{-g(\varphi^2-1)^2+g}\textup{d}\varphi.
\end{equation}

The high dimension triviality results mentioned above are related to the simplicity of the critical exponents of these models, suggesting that for $d\geq 4$, they must take their \textit{mean-field} values. Rigourous results in that direction have been obtained in \cite{AizenmanGeometricAnalysis1982,AizenmanGrahamRenormalizedCouplingSusceptibility4d1983,AizenmanFernandezCriticalBehaviorMagnetization1986,AizenmanFernandezLongrange,BauerschmidtBrydgesSladeBOOKRG2019,MichtaParkSladeBoundaryconditionsFiniteSizeScaling2023,DuminilPanis2024newLB}.
What appeared to be a negative result from the perspective of constructive Euclidean field theory is positive in the framework of statistical mechanics as it provides information at criticality for a wide class of non-integrable models.

The main step in the proof of Aizenman and Fröhlich in dimension $d\geq 5$ is the derivation of the so-called \textit{tree diagram bound} through  geometric representations, and the use of \textit{reflection positivity} \cite{FrohlichSimonSpencerIRBounds1976,FrohlichIsraelLiebSimon1978} to argue that Ursell's four point function's scaling limit always vanishes at criticality. Although initially presented in the case of nearest-neighbour interactions $J_{x,y}=\mathds{1}_{|x-y|_1=1}$, these methods are robust and extend to more general (in particular long-range) reflection positive interactions. However, this is no longer true in dimension four, where only the case of nearest-neighbour interactions is treated \cite{AizenmanDuminilTriviality2021}. 

The interest in the study of long-range interactions comes from the fact that the rate of decay of the interactions may change the effective dimension of the model by increasing it, meaning that one can recover high-dimensional features in some well-chosen one, two or three dimensional systems. An observation of this phenomenon was made for algebraically decaying long-range interactions of the form $1/r^{d+\alpha}$ by Fischer, Ma and Nickel \cite{FisherMaNickel1972critical} using renormalisation group heuristics. They noted that the parameter $\alpha$ had the effect of changing the value of the upper critical dimension\footnote{That is, the dimension above which the system simplifies drastically, adopting Gaussian features. This is also the dimension above which the \emph{bubble diagram} converges, which is an indicator of mean-field behaviour as recalled below.} into $d_c(\alpha)=\min(2\alpha,4)$, suggesting that the effective dimension of the model should be given by $d_{\textup{eff}}(\alpha)=d/(1\wedge (\alpha/2))$ (see Figure \ref{fig: effective dimension}). This was later studied by Aizenman and Fernández \cite{AizenmanFernandezLongrange}, and lead to the observation that some Ising models in dimension $1\leq d \leq 3$  present trivial scaling limits at criticality, which is not expected in the case of nearest-neighbour interactions\footnote{Indeed, non-triviality of the nearest-neighbour Ising model has been proven for $d=2$ in \cite{AizenmanGeometricAnalysis1982}, while the case $d=3$ remains open. Let us mention that the recent conformal bootstrap approach to the study of the critical $3d$ Ising model strongly supports this conjecture \cite{RychkovNonGaussianity2017}.}. Other rigourous results were obtained through lace expansion methods \cite{HeydenreichvdHofstadSakai2008mean,ChenSakaiLongRange2015,ChenSakaiLongRange2019}. Conversely, if the interaction decays fast enough, the upper critical dimension of the model is unchanged (this corresponds to $\alpha\geq 2$ in the example above). The prediction is that only two situations may occur: either the interaction decays very fast and we expect to fall into the universality class of the nearest-neighbour models, or the decay is \textit{exactly} fast enough for additional logarithmic corrections to appear. The latter scenario was shown to occur \cite{ChenSakaiLongRange2019} in dimension $d\geq 4$, for (sufficiently spread out) interactions decaying such as $1/r^{d+2}$.

Finally, let us briefly mention that long-range interactions of the above type have been used to conduct rigourously the so-called ``$(d_c-\varepsilon)$-expansions''--- motivated by Wilson and Fisher through renormalisation group heuristics \cite{WilsonFisher1972critical}--- which give a precise understanding of the critical exponents of these models below the upper critical dimension, see \cite{FisherMaNickel1972critical,SuzukiYamazaki1972wilson,BrydgesDimock1998non,Slade2018criticaldcexpan}.

\begin{figure}
    \centering
    \includegraphics{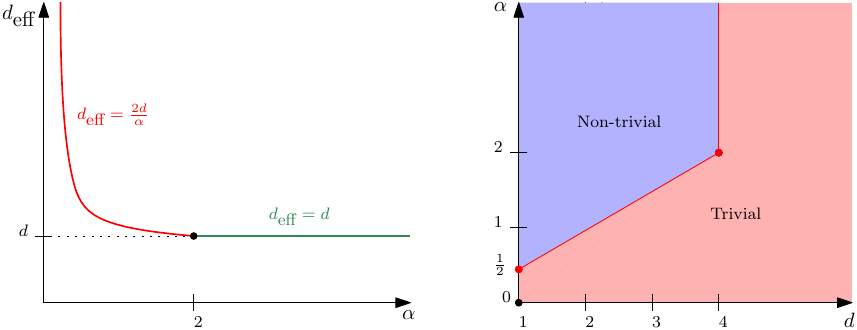}
    \caption{\textsc{Left}: The graph of $\alpha\mapsto d_{\textup{eff}}(\alpha)$ for the interaction $J$ given by $J_{x,y}=C|x-y|_1^{-d-\alpha}$ (for $d\geq 2$). The transition between the regime $d_{\textup{eff}}(\alpha)>d$ and $d_{\textup{eff}}=d$ occurs at $\alpha=2$. At this point, we expect logarithmic corrections at the level of the decay of the critical two-point function. \textsc{Right}: A summary of the expected behaviour of the critical scaling limits for the same interaction $J$. The red region (including segments and points), correspond to interactions which are expected to yield trivial scaling limits. The results of this paper concern the study of the ``marginal'' cases that separate the two phases.}
    \label{fig: effective dimension}
\end{figure}

The goal of this paper is threefold. First, we prove that reflection positive Ising or $\varphi^4$ models with effective dimension strictly above four are trivial. This revisits some of the results of \cite{AizenmanGeometricAnalysis1982,AizenmanFernandezLongrange} together with the notion of effective dimension, and provides explicit examples of trivial models in dimensions one, two, and three. Second, we extend the results of \cite{AizenmanDuminilTriviality2021} to near-critical and critical reflection positive Ising and $\varphi^4$ models in dimension $d=4$ beyond the nearest-neighbour case. In particular, this case contains algebraically decaying interactions (as above) with $\alpha>0$. The result was already known \cite{AizenmanFernandezLongrange} for $\alpha\in (0,2)$ but is new in the case $\alpha\geq 2$. Third, we prove triviality of the scaling limits of one, two, and three dimensional reflection positive Ising models with effective dimension four. This is the main novelty of the paper. Such examples of models can be obtained choosing $\alpha=d/2$ above for $1\leq d \leq 3$. Our results apply to a wide class of single-site measures called the \textit{Griffiths-Simon class} of measures (see Section \ref{section: gs class def}), which in particular contains the examples mentioned below \eqref{eq: model we study}, and which can be recovered from weak limits of Ising-type single site measures.

As in \cite{AizenmanGeometricAnalysis1982,AizenmanDuminilTriviality2021}, we use the random current representation of the Ising model which enables, by means of the \textit{switching lemma}, to express the correlation functions' deviation from Wick's law in terms of intersection probabilities of two independent random currents with distanced sources. 

When $d=4$, two situations may occur. First, the interaction's decay may be ``slow'', in which case we observe a decay of the model's two-point function which is slightly better than the one obtained for nearest-neighbour interactions (see Corollary \ref{cor: infinite second moment improvement ir bound}). We can then conclude using the \emph{tree diagram bound} obtained in \cite{AizenmanGeometricAnalysis1982}. In particular, this first case contains reflection positive interactions of algebraic decay with $\alpha\in (0,2]$. Second, the decay of the interaction may be ``too fast'', in which case we observe no improvement at the level of the decay of the two-point function. As explained above, this case corresponds to the situation where we expect the model to behave like a nearest-neighbour one, this corresponds to choosing $\alpha>2$ above. We follow the strategy of \cite{AizenmanDuminilTriviality2021} and improve the tree diagram bound. The proof goes by arguing that in dimension four, just like random walks, if two independent random currents intersect at least once, they must re-intersect a large number of times. By the mean of a multi-scale analysis, the authors of \cite{AizenmanDuminilTriviality2021} showed that intersections occur with large probability in a density of (well-chosen) scales. This essentially required three tools: regularity properties for the model's two-point function, a proof of the fact that intersections happen with (uniform) positive probability on many scales, and a mixing statement which allows to argue that intersections at different scales are roughly independent events.
However, in the case of long-range interactions, these steps fail. Indeed, the extension of the proof to the general setup requires an adaptation of the reflection positivity arguments to the case of arbitrary interactions which builds on a different viewpoint on the spectral analysis of these models (see Section \ref{section: reflection positivity} and Appendix \ref{appendix spectral representation}). This viewpoint was already introduced in \cite{BorgsChayesCovarianceMatrixPotts1996,Ott2019OZThesis}. Then, long-range interactions may have the effect of making intersections less likely as it becomes possible to ``jump'' scales. Finally, long-range interactions may create more dependency between pieces of the current at different scales. We solve these problems by arguing that the currents do not jump above a $(4-\varepsilon)$-dimensional annulus with very high probability (see Section \ref{section: properties of the current trivia}). As it turns out, this is enough to (essentially) recover the same geometric properties of currents as in the nearest-neighbour case.

When $1\leq d \leq 3$ and $d_{\textup{eff}}=4$, the above improvements are not sufficient. The main reason is that in this precise regime, the decay of the interaction is too slow to exclude jumps above $(d-\varepsilon)$-dimensional annuli (see Remark \ref{rem: zigzag difficult for deff=4}). This additional difficulty is treated by going one step further in the analysis of the currents (see Section \ref{section: prop currents deff nice}). As a byproduct of our methods, we obtain a (quantitative) mixing statement that is valid for all models of effective dimension at least four (see Section \ref{section: mixing for deff nice}).

To study the near-critical regime, it is important to introduce a typical length below which the model essentially behaves like a critical one. It is tempting to use the correlation length $\xi(\beta)$ defined for $\beta<\beta_c$, by
\begin{equation}
    \xi(\beta):=-\left(\lim_{n\rightarrow \infty}\frac{\log \langle \tau_0\tau_{n\mathbf{e}_1}\rangle_\beta}{n}\right)^{-1}.
\end{equation}
However, this quantity is not relevant in the case of long-range interactions since one may have $\xi(\beta)=\infty$ for all $\beta<\beta_c$ (see \cite{NewmanSpohn1998shiba,Aoun2021SharpAsymp,AounOttVelenik2023twopt}). Another contribution of this paper is the introduction of the \textit{sharp length} $L(\beta)$ (defined in Section \ref{section lower bounds}), whose definition is inspired by \cite{DuminilTassionNewProofSharpness2016}.

We first prove an improved tree diagram bound for the Ising model, and then extend it to the $\varphi^4$ model (and more generally every model in the Griffiths--Simon class) using its viewpoint as a generalised Ising model.

Let us mention that we also expect a \textit{direct} analysis, meaning at the level of the $\varphi^4$ model and without any mention of the Ising model, to be possible with the use of the \textit{random tangled current representation} of $\varphi^4$ recently introduced in \cite{GunaratnamPanagiotisPanisSeveroPhi42022}.

\subsection{Definitions and statement of the results} 
We start by stating the results for the case of the Ising model.
\subsubsection{Results for the Ising model}

In what follows, $\Lambda$ is a finite subset of $\mathbb Z^d$. Let $J=(J_{x,y})_{\lbrace x,y\rbrace\subset \mathbb Z^d}$ be an interaction (or a collection of coupling constants) and $h \in \mathbb R$. For $\sigma=(\sigma_x)_{x\in \Lambda}\in \lbrace \pm 1\rbrace^\Lambda$, introduce the \textit{Hamiltonian}
\begin{equation}
    H_{\Lambda,J,h}(\sigma):=-\sum_{\lbrace x,y\rbrace\subset \Lambda}J_{x,y}\sigma_x\sigma_y-h\sum_{x\in \Lambda}\sigma_x,
\end{equation}
and define the associated finite volume Gibbs equilibrium measure $\langle \cdot \rangle_{\Lambda,J,h,\beta}$ at inverse temperature $\beta\geq 0$ to be the probability measure under which, for each $F:\lbrace \pm 1\rbrace^\Lambda\rightarrow \mathbb R$, 
\begin{equation}
    \langle F\rangle_{\Lambda,J,h,\beta}:=\dfrac{1}{Z(\Lambda,J,h,\beta)}\sum_{\sigma\in \lbrace \pm 1\rbrace^\Lambda}F(\sigma)\exp\left(-\beta H_{\Lambda,J,h}(\sigma)\right),
\end{equation}
where
\begin{equation}
    Z(\Lambda,J,h,\beta):=\sum_{\sigma\in \lbrace \pm 1\rbrace^\Lambda}\exp\left(-\beta H_{\Lambda,J,h}(\sigma)\right),
\end{equation}
is the \textit{partition function} of the model. We make the following assumptions on the interaction $J$:
\begin{enumerate}
    \item[\textbf{(A1)}]  Ferromagnetic: For all $x,y\in \mathbb Z^d$, $J_{x,y}\geq 0$,
    \item[\textbf{(A2)}]  Locally finite: For any $x\in \mathbb Z^d$, 
    \begin{equation}
        |J|:=\sup_{x\in \mathbb Z^d}\sum_{y\in \mathbb Z^d}J_{x,y}<\infty,
    \end{equation}
    \item[\textbf{(A3)}]  Translation invariant: For all $x,y\in \mathbb Z^d$, $J_{x,y}=J_{0,y-x}$,
    \item[\textbf{(A4)}] Irreducible: For all $x,y\in \mathbb Z^d$, there exist $x_1,\ldots,x_k\in \mathbb Z^d$ such that
    \begin{equation}
        J_{x,x_1}J_{x_1,x_2}\ldots J_{x_{k-1},x_k}J_{x_k,y} >0,
    \end{equation}
    \item[\textbf{(A5)}] Reflection positive: see Section \ref{section: reflection positivity}.
\end{enumerate}

We postpone the definition of reflection positivity to Section \ref{section: definition ref pos}, but to fix the ideas the reader might keep in mind the following examples\footnote{Here, $|.|_1$ refers to the $\ell^1$ norm on $\mathbb R^d$.} of interactions which satisfy $\mathbf{(A1)}$--$\mathbf{(A5)}$:
\begin{enumerate}
        \item (nearest-neighbour interactions) $J_{x,y}=C\mathds{1}_{|x-y|_1,1}$ for $C>0$,
    \item (exponential decay / Yukawa potentials) $J_{x,y}=C\exp(-\mu |x-y|_1)$ for $\mu,C>0$,
    \item (algebraic decay) $J_{x,y}=C|x-y|_1^{-d-\alpha}$ for $\alpha,C>0$.
\end{enumerate}

Using Griffiths' inequalities \cite{GriffithsCorrelationsIsing1-1967}, one can obtain the associated infinite volume Gibbs measure by taking weak limits of $\langle \cdot \rangle_{\Lambda,J,h,\beta}$ as $\Lambda\nearrow \mathbb Z^d$. We denote the limit by $\langle \cdot \rangle_{J,h,\beta}$. For convenience, in what follows, we omit the mention of the interaction in the notation of the Gibbs measures.

In dimensions $d >1$, the model exhibits a phase transition for the vanishing of the \textit{spontaneous magnetisation}. That is, if  
\begin{equation}
    m^*(\beta):=\lim_{h\rightarrow 0^+}\langle\sigma_0\rangle_{h,\beta},
\end{equation}
then, $\beta_c:=\inf\lbrace \beta>0,\text{ } m^*(\beta)>0\rbrace\in (0,\infty)$. The above assumptions guarantee \cite{Fisher1967criticaltemp} that $\beta_c>0$ (in fact $\beta_c\geq |J|^{-1}$), while Peierls' argument \cite{PeierlsIsing1936} yields the bound $\beta_c<\infty$. In dimension $d=1$, the phase transition occurs \cite{Dyson1969} under the additional assumption that $J_{x,y}\asymp |x-y|^{-1-\alpha}$ with $\alpha\in(0,1]$. We now assume that $h=0$. Our results concern the nature of the scaling limits at\footnote{At a parameter $\beta<\beta_c$, finiteness of the susceptibility \cite{AizenmanBarskyFernandezSharpnessIsing1987} implies that the scaling limit is the (Gaussian) white noise distribution \cite{NewmanNormalFluctuationsFKGInequalities1980}.}, or near the critical parameter $\beta_c$.

To determine the nature of the scaling limit, we look at the joint distribution of the \textit{smeared observables}\footnote{Note that for $f=\mathds{1}_{[-1,1]^d}$,
\begin{equation}
    \langle T_{f,L,\beta}(\sigma)^2\rangle_\beta =1,
\end{equation}
and more generally for $f\neq 0$, one has $0<c_f\leq \langle T_{f,L,\beta}(\sigma)^2\rangle_\beta \leq C_f<\infty$,
which means that the following quantity is bounded away from $0$ and $\infty$ by constants that only depend on $f$. This indicates that this is the scaling that is the most likely to yield interesting limits.} given for $\beta>0$ and $L\geq 1$ by
\begin{equation}
    T_{f,L,\beta}(\sigma):=\dfrac{1}{\sqrt{\Sigma_L(\beta)}}\sum_{x\in \mathbb Z^d}f\left(\dfrac{x}{L}\right)\sigma_x,
\end{equation}
where $f$ ranges over the set $\mathcal{C}_0(\mathbb R^d)$ of continuous, real valued, and compactly supported functions, and where
\begin{equation}
    \Sigma_L(\beta):=\big\langle \big(\sum_{x\in \Lambda_L}\sigma_{x}\big)^2\big\rangle_\beta=\sum_{x,y\in \Lambda_L}\langle\sigma_x\sigma_y\rangle_\beta, \textup{ with }\Lambda_L:=[-L,L]^d\cap \mathbb Z^d.
\end{equation} 

\begin{Def}\label{def: cv fdd} A discrete system as above is said to converge in distribution to a scaling limit if the collection of random variables $(T_{f,L,\beta}(\sigma))_{f\in \mathcal{C}_0(\mathbb R^d)}$ converges in distribution (in the sense of finite dimensional distributions) as $L$ goes to infinity. Using Kolmogorov's extension theorem and the separability of $\mathcal{C}_0(\mathbb R^d)$, we can represent any scaling limit as a random field.
\end{Def}

Our first result concerns the study of models of effective dimension $d_{\textup{eff}}>4$. We postpone the precise definition of effective dimension to Section \ref{section: dim eff >4} and illustrate this concept using the example of algebraically decaying reflection positive interactions mentioned above, i.e. $J_{x,y}=C|x-y|_1^{-d-\alpha}$ for $\alpha,C>0$. In that case, we will see that $d_{\textup{eff}}\geq \frac{d}{1\wedge (\alpha/2)}$ so that the hypothesis $d_{\textup{eff}}>4$ corresponds to $d-2(\alpha\wedge 2)>0$.

\begin{Thm}\label{thm: intro deff 4 poly}Let $d\geq 1$. Let $J$ be the interaction defined for $x\neq y\in \mathbb Z^d$ by $J_{x,y}=C_0|x-y|_1^{-d-\alpha}$ where $C_0,\alpha>0$. We also assume that $d-2(\alpha\wedge 2)>0$. There exist $C=C(C_0,d),\gamma=\gamma(d)>0$ such that for all $\beta\leq \beta_c$, $L\geq 1$, $f\in \mathcal{C}_0(\mathbb R^d)$ and $z\in \mathbb R$, 
\begin{multline*}
    \left|\left\langle \exp\left(z T_{f,L,\beta}(\sigma)\right)\right\rangle_\beta-\exp\left(\frac{z^2}{2}\langle T_{f,L,\beta}(\sigma)^2\rangle_\beta\right)\right|\\\leq\exp\left(\frac{z^2}{2}\langle T_{|f|,L,\beta}(\sigma)^2\rangle_\beta\right)\dfrac{C(\beta^{-4}\vee \beta^{-2})\Vert f\Vert_\infty^4 r_f^{\gamma}z^4}{L^{d-2(\alpha\wedge 2)}},
\end{multline*}
where $\Vert f\Vert_\infty=\sup_{x\in \mathbb R^d}|f(x)|$ and $r_f=\left(\max\lbrace r\geq 0, \: \exists x\in \mathbb R^d, \: |x|=r, \: f(x)\neq 0\rbrace \vee 1\right).$ 

As a consequence, for $\beta\leq \beta_c$, every sub-sequential scaling limit (in the sense of Definition \textup{\ref{def: cv fdd}}) of the model is Gaussian.
\end{Thm}

We now move the focus to the case $d=4$.
As explained in the introduction, and discussed in Section \ref{section rcr}, the Gaussian behaviour of the model can be seen at the level of Ursell's four-point function \cite{NewmanInequalitiesUrsellIsing1975,AizenmanGeometricAnalysis1982} defined for all $x,y,z,t\in \mathbb Z^d$ by 
\begin{equation}\label{eq: def ursell}
    U^\beta_4(x,y,z,t):=\langle \sigma_x\sigma_y\sigma_z\sigma_t\rangle_{\beta}-\langle \sigma_x\sigma_y\rangle_\beta\langle \sigma_z\sigma_t\rangle_\beta-\langle \sigma_x\sigma_z\rangle_\beta\langle \sigma_y\sigma_t\rangle_\beta-\langle \sigma_x\sigma_t\rangle_\beta\langle \sigma_y\sigma_z\rangle_\beta.
\end{equation}
In dimensions $d>4$, Aizenman \cite{AizenmanGeometricAnalysis1982} obtained the triviality of the scaling limits (of the critical nearest-neighbour Ising model) using the tree diagram bound, which takes the following form:
\begin{equation}\label{eq: tree diagram bound intro}
    |U_4^\beta(x,y,z,t)|\leq 2\sum_{u\in \mathbb Z^d}\langle \sigma_x\sigma_u\rangle_\beta \langle \sigma_y\sigma_u\rangle_\beta \langle \sigma_z\sigma_u\rangle_\beta \langle \sigma_t\sigma_u\rangle_\beta, \qquad \forall x,y,z,t\in \mathbb Z^d,
\end{equation}
together with the crucial input of reflection positivity which provides two important tools: the Messager--Miracle-Solé inequalities \cite{MessagerMiracleSoleInequalityIsing}, and the \textit{infrared bound} \cite{FrohlichSimonSpencerIRBounds1976,FrohlichIsraelLiebSimon1978}. Combined together, these tools imply the existence of $C=C(d)>0$ such that for every $x,y\in \mathbb Z^d$,
\begin{equation}\label{eq: ir intro}
    \langle\sigma_x\sigma_y\rangle_{\beta_c}\leq \frac{C}{|x-y|^{d-2}},
\end{equation}
where $|\cdot|$ denotes the infinite norm on $\mathbb R^d$.
As noticed in \cite{AizenmanGeometricAnalysis1982}, the relevant question is to see whether $|\urs(x,y,z,t)|/\langle \sigma_x\sigma_y\sigma_z\sigma_t\rangle_\beta$ vanishes or not, as the \textit{mutual distance} $L(x,y,z,t):=\min_{u\neq v \in \lbrace x,y,z,t\rbrace}|u-v|$ between $x,y,z$ and $t$ goes to infinity but the distances between the pairs are all of the same order. The proof can be summed up by the following (incomplete) argument: assume that $\beta=\beta_c$ and that the bound \eqref{eq: ir intro} is sharp; for a set of points $x,y,z,t$ at mutual distance of order $L$, the sum of the right-hand side of \eqref{eq: tree diagram bound intro} is of order $O(L^{8-3d})$, and we expect the four-point function  $\langle \sigma_x\sigma_y\sigma_z\sigma_t\rangle_{\beta_c}$ to be of order the product of two two-point functions, hence of order at least $L^{4-2d}$. As a result, we have
\begin{equation}\label{eq: urs l 4-d intro}
    \frac{|U_4^{\beta_c}(x,y,z,t)|}{\langle \sigma_x\sigma_y\sigma_z\sigma_t\rangle_{\beta_c}}=O(L^{4-d}).
\end{equation}
The above bound is clearly inconclusive in the case $d=4$. However, in the case of nearest-neighbour interactions, \eqref{eq: tree diagram bound intro} was improved by a logarithmic factor to obtain Gaussianity. 

The case of long-range interactions is more subtle since we do not necessarily expect any improvement in the tree diagram bound in dimension $4$.
As it turns out, we do not need any such improvement when the decay of the interaction is sufficiently slow so that the decay of the model's two-point function is faster than \eqref{eq: ir intro}.

To determine whether this decay is fast enough or not, it is (almost) enough to look at whether the following quantity is finite or not:
\begin{equation}
    \mathfrak{m}_2(J):=\sum_{x\in \mathbb Z^d}|x|^2J_{0,x}.
\end{equation}
When $\mathfrak{m}_2(J)=\infty$, the decay of the interaction is slow enough to conclude using \ref{eq: tree diagram bound intro}.
\begin{Thm}\label{thm: main m2(J) infinite}Let $d=4$. Assume that $J$ satisfies $(\mathbf{A1})$--$(\mathbf{A5})$, and that $\mathfrak{m}_2(J)=\infty$. Then, for all $\beta\leq \beta_c$, $f\in \mathcal{C}_0(\mathbb R^d)$ and $z\in \mathbb R$, 
\begin{equation}
    \lim_{L\rightarrow\infty}\left|\left\langle \exp\left(z T_{f,L,\beta}(\sigma)\right)\right\rangle_\beta-\exp\left(\frac{z^2}{2}\langle T_{f,L,\beta}(\sigma)^2\rangle_\beta\right)\right|=0.
\end{equation}
As a consequence, for $\beta\leq \beta_c$, every sub-sequential scaling limit of the model is Gaussian.
\end{Thm}

\begin{Rem}
As we will see in Section \textup{\ref{section: dim eff >4}}, the rate of convergence to $0$ can be expressed in terms of 
\begin{equation}
    \sum_{|x|\leq k} |x|^2J_{0,x}.
\end{equation}
For instance, in the case of $J$ defined by $J_{x,y}=C|x-y|_1^{-d-2}$, one can check that $\mathfrak{m}_2(J)=\infty$. The rate of convergence to $0$ is then given by $C/\log L$.
\end{Rem}

We now discuss the case $\mathfrak{m}_2(J)<\infty$. In fact, we will have to restrict to interactions $J$ satisfying the following additional condition, which is slightly stronger:
\begin{enumerate}
    \item[$(\mathbf{A6})$] There exist $\mathbf{C},\varepsilon>0$ such that for all $x\in \mathbb Z^d$,
    \begin{equation}\label{eq: A5}
        J_{0,x}\leq \frac{\mathbf{C}}{|x|^{d+2+\varepsilon}}.
    \end{equation}
\end{enumerate}
As explained above, in this case we expect that the mechanism which leads to Gaussianity is the same as for the nearest-neighbour case. Hence, we first prove an improved tree diagram bound. The quantity $L(\beta)$ was briefly mentioned above and will be introduced in Section \ref{section lower bounds}.
\begin{Thm}[Improved tree diagram bound for $d=4$]\label{improved diagram bound}
Let $d=4$. Assume that $J$ satisfies $(\mathbf{A1})$--$(\mathbf{A6})$. There exist $c,C>0$ such that, for all $\beta\leq \beta_c$, for all $x,y,z,t\in \mathbb Z^4$ at mutual distance at least  $L$ of each other with $1\leq L\leq L(\beta)$, 
\begin{equation}
    |U_4^\beta(x,y,z,t)|\leq \frac{C}{B_L(\beta)^c}\sum_{u\in \mathbb Z^4}\langle \sigma_x\sigma_u\rangle_\beta \langle \sigma_y\sigma_u\rangle_\beta \langle \sigma_z\sigma_u\rangle_\beta \langle \sigma_t\sigma_u\rangle_\beta,
\end{equation}
where $B_L(\beta)$ is the bubble diagram truncated at distance $L$ defined by 
\begin{equation}
    B_L(\beta):=\sum_{x\in \Lambda_L}\langle \sigma_0\sigma_x\rangle_{\beta}^2.
\end{equation}
\end{Thm}
It is predicted, for an interaction $J$ satisfying $\mathbf{(A1)}$--$\mathbf{(A6)}$, that the bubble diagram diverges at criticality (see \cite{DuminilPanis2024newLB} for a proof of this result in the case of nearest-neighbour interactions). This improves the $O(1)$ of \eqref{eq: urs l 4-d intro} to a $O(B_L(\beta)^{-c})$.

\begin{Rem}\label{rem: finite bubble implies triviality} As noticed in \cite{AizenmanGeometricAnalysis1982,AizenmanFernandezCriticalBehaviorMagnetization1986}, the \emph{bubble condition}
\begin{equation}\label{eq: bubble condition}
    B(\beta_c)<\infty,
\end{equation}
implies that some of the model's critical exponents take their mean-field value. It is also possible to show that the bubble condition (together with some monotonicity properties of the two-point function) implies triviality of the scaling limits. We provide a proof of this fact in Appendix \textup{\ref{appendix: bubble finite}}.
\end{Rem}
\begin{Rem}\label{extension rem} The reason why we restrict to interactions satisfying $(\mathbf{A6})$ is technical and will become more transparent in Section \textup{\ref{section: properties of the current trivia}}. In this paper, the most interesting examples of reflection positive interactions are given by algebraically decaying interactions. These interactions always satisfy $\mathbf{(A6)}$ or $\mathfrak{m}_2(J)=\infty$ (when $d=4$). A more general setup is treated in \cite[Chapter~9]{PanThesis}. There, $\mathbf{(A6)}$ is replaced by a more ``averaged'' assumption:
\begin{enumerate}
    \item[$(\mathbf{A6}')$] There exist $\mathbf{C},\varepsilon>0$ such that for all $k\geq 1$,
    \begin{equation}\label{eq: A5'}
        \sum_{|x|=k}|x|^2 J_{0,x}\leq \frac{\mathbf{C}}{k^{1+\varepsilon}}.
    \end{equation}
\end{enumerate}
\end{Rem}

With the improved tree diagram bound, we can obtain a formulation of triviality similar to the one obtained in Theorem \ref{thm: main m2(J) infinite}.

\begin{Coro}\label{thm: main} Let $d=4$. Assume that $J$ satisfies $(\mathbf{A1})$--$(\mathbf{A6})$. There exist $C,c,\gamma>0$ such that, for all $\beta\leq \beta_c$, $1\leq L\leq L(\beta)$, $f\in \mathcal{C}_0(\mathbb R^d)$, and $z\in \mathbb R$,
\begin{equation}
    \left|\left\langle \exp\left(z T_{f,L,\beta}(\sigma)\right)\right\rangle_\beta-\exp\left(\frac{z^2}{2}\langle T_{f,L,\beta}(\sigma)^2\rangle_\beta\right)\right|\leq\exp\left(\frac{z^2}{2}\langle T_{|f|,L,\beta}(\sigma)^2\rangle_\beta\right)\dfrac{C\Vert f\Vert_\infty^4 r_f^{\gamma}z^4}{(\log L)^c}.
\end{equation}
As a consequence, for $\beta=\beta_c$, every sub-sequential scaling limit of the model is Gaussian.
\end{Coro}

We now turn to the case of low dimensional models of effective dimension equal to four. As above, we illustrate the result by focusing on the case of algebraically decaying reflection positive interactions. The situation of interest corresponds to choosing $\alpha=d/2$ for $1\leq d \leq 3$. More general versions of the following statements can be found in Section \ref{section : deff=4} (see Theorem \ref{improved diagram bound deff=4 but more general} and Corollary \ref{cor: improved tree diagramm deff=4 but more general}).

For technical reasons (see Remark \ref{rem: how to get better result}), the following statement is not as quantitative as the above ones, and for simplicity we restrict it to $\beta=\beta_c$.

\begin{Thm}[Improved tree diagram bound for $1\leq d \leq 3$]\label{improved diagram bound deff=4}
Let $1\leq d\leq 3$. Let $J$ be the interaction defined for $x\neq y\in \mathbb Z^d$ by $J_{x,y}=C_0|x-y|_1^{-3d/2}$ (i.e. $\alpha=d/2$) where $C_0>0$. There exist $C>0$ and a function $\psi: \mathbb R\rightarrow\mathbb R_{>0}$ which satisfies $\psi(t)\rightarrow \infty$ as $t\rightarrow \infty$, such that, for all $x,y,z,t\in \mathbb Z^d$ at mutual distance at least  $L$ of each other, 
\begin{equation}
    |U_4^{\beta_c}(x,y,z,t)|\leq \frac{C}{\psi(B_L(\beta_c))}\sum_{u\in \mathbb Z^d}\langle \sigma_x\sigma_u\rangle_{\beta_c} \langle \sigma_y\sigma_u\rangle_{\beta_c} \langle \sigma_z\sigma_u\rangle_{\beta_c} \langle \sigma_t\sigma_u\rangle_{\beta_c}.
\end{equation}
\end{Thm}
\begin{Rem} In fact, for $d=1$, the result is much stronger and we recover the improvement of order $O(B_L(\beta)^{-c})$ obtained when $d=4$. The precise statements will be given in Section \textup{\ref{section : deff=4}}.
\end{Rem}
We can still deduce a triviality statement from this improved tree diagram bound. It involves the so-called \emph{renormalised coupling constant}. We begin with a definition. For $\sigma>0$, we define the \emph{correlation length of order $\sigma$} by: for $\beta<\beta_c$,
\begin{equation}
    \xi_\sigma(\beta):=\left(\frac{\sum_{x\in\mathbb Z^d}|x|^\sigma\langle \sigma_0\sigma_x\rangle_\beta}{\chi(\beta)}\right)^{1/\sigma},
\end{equation}
where $\chi(\beta):=\sum_{x\in\mathbb Z^d}\langle \sigma_0\sigma_x\rangle_\beta$. As it turns out, the above quantity is well-defined when $J_{x,y}=C|x-y|^{-d-\alpha}$ as soon as $\sigma<\alpha$ (see for instance \cite{NewmanSpohn1998shiba,Aoun2021SharpAsymp,AounOttVelenik2023twopt}). Also, by the results of \cite{AizenmanBarskyFernandezSharpnessIsing1987}, one has $\xi_\sigma(\beta)\rightarrow \infty$ as $\beta\rightarrow \beta_c$.
For such a $\sigma>0$, we introduce another convenient measure of the interaction called the renormalised coupling constant of order $\sigma$ and defined for $\beta<\beta_c$ by:
\begin{equation}
    g_\sigma(\beta):=-\frac{1}{\chi(\beta)^2\xi_\sigma(\beta)^d}\sum_{x,y,z\in\mathbb Z^d}U_4^\beta(0,x,y,z).
\end{equation}
The vanishing of the above quantity is known to imply triviality of the scaling limits of the model (see \cite[Theorem~11]{NewmanInequalitiesUrsellIsing1975} or \cite{AizenmanGeometricAnalysis1982,Sokal1982Destructive,AizenmanGrahamRenormalizedCouplingSusceptibility4d1983}). 

\begin{Coro}\label{cor: d=2,3} We keep the assumptions of Theorem \textup{\ref{improved diagram bound deff=4}}. Then, for $\sigma\in(0,d/2)$,
\begin{equation}
    \lim_{\beta\nearrow\beta_c}g_\sigma(\beta)=0.
\end{equation}
As a consequence, for $\beta=\beta_c$, every sub-sequential scaling limit of the model is Gaussian.
\end{Coro}

\subsubsection{Results for the $\varphi^4$ model} We now extend the above results to the $\varphi^4$ model. These results also extend to models in the Griffiths--Simon class of measures, whose definition is postponed to the next section. We refer to Section \ref{section: gs class} for the general statement of triviality for these models.

We start with a proper definition of the $\varphi^4$ model. Let $\rho$ be given by \eqref{eq: phi4 measure}. As for the Ising model, the ferromagnetic $\varphi^4$ model on $\Lambda$ is defined by the finite volume Gibbs equilibrium state: for $F:\mathbb R^\Lambda\rightarrow \mathbb R$,
\begin{equation}
    \langle F(\varphi)\rangle_{\Lambda,\rho,\beta}=\frac{1}{Z(\Lambda,\rho,\beta)}\int F(\varphi)\exp\left(-\beta H_{\Lambda,J}(\varphi)\right)\prod_{x\in \Lambda}\textup{d}\rho(\varphi_x),
\end{equation}
where $Z(\Lambda,\rho,\beta)$ is the partition function and
\begin{equation}
    H_{\Lambda,J}(\varphi):=-\sum_{\lbrace x,y\rbrace \subset \Lambda}J_{x,y}\varphi_x\varphi_y.
\end{equation}
We call $\langle \cdot\rangle_{\rho,\beta}$ the model's infinite volume Gibbs measure. It is also possible to introduce a critical parameter $\beta_c(\rho)$, together with a sharp length $L(\rho,\beta)$, Ursell's four-point function $U^{\rho,\beta}_4$, and a renormalised coupling constant $g_\sigma(\rho,\beta)$. The extension of the results concerning models of effective dimension $d_{\textup{eff}}>4$ is quite straightforward and will be discussed in Section \ref{section: dim eff >4} (see Remark \ref{rem: extension to gs of deff>4}). We focus on the results for $1\leq d \leq 4$ with $d_{\textup{eff}}=4$.
\begin{Thm}\label{thm: main m2(J) infinite phi4}Let $d=4$. Assume that $J$ satisfies $(\mathbf{A1})$--$(\mathbf{A5})$, and that $\mathfrak{m}_2(J)=\infty$. Then, for all $\beta\leq \beta_c(\rho)$, $f\in \mathcal{C}_0(\mathbb R^d)$ and $z\in \mathbb R$, 
\begin{equation}
    \lim_{L\rightarrow \infty}\left|\left\langle \exp\left(z T_{f,L,\beta}(\varphi)\right)\right\rangle_{\rho,\beta}-\exp\left(\frac{z^2}{2}\langle T_{f,L,\beta}(\varphi)^2\rangle_{\rho,\beta}\right)\right|=0.
\end{equation}
As a consequence, for $\beta=\beta_c$, every sub-sequential scaling limit of the model is Gaussian.
\end{Thm}
For the $\varphi^4$ model, the tree diagram bound takes a slightly different form.
\begin{Thm}\label{thm: main gs} Let $d=4$. Assume that $J$ satisfies $\mathbf{(A1)}$--$\mathbf{(A6)}$. There exist $c,C>0$ such that, for all $\beta\leq \beta_c(\rho)$, for all $x,y,z,t\in \mathbb Z^4$ at mutual distance at least $L$ of each other with $1\leq L\leq L(\rho,\beta)$,
\begin{multline*}
    |U_4^{\rho,\beta}(x,y,z,t)|\\\leq C\left(\frac{B_0(\rho,\beta)}{B_L(\rho,\beta)}\right)^c\sum_{u\in \mathbb Z^4}\sum_{\substack{u',u''\in \mathbb Z^4}}\langle \tau_x\tau_u\rangle_{\rho,\beta}\beta J_{u,u'} \langle \tau_{u'}\tau_y\rangle_{\rho,\beta} \langle \tau_z\tau_u\rangle_{\rho,\beta}\beta J_{u,u''} \langle \tau_{u''}\tau_t\rangle_{\rho,\beta}.
\end{multline*}
\end{Thm}
\begin{Coro}\label{cor: main gs} Let $d=4$. Assume that $J$ satisfies $(\mathbf{A1})$--$(\mathbf{A6})$. Consider a $\varphi^4$ model on $\mathbb Z^4$ with coupling constants $J$. There exist $C,c,\gamma>0$ such that, for all $\beta\leq \beta_c(\rho)$, $1\leq L\leq L(\rho,\beta)$, $f\in \mathcal{C}_0(\mathbb R^d)$, and $z\in \mathbb R$,
\begin{multline*}
    \left|\left\langle \exp\left(z T_{f,L,\beta}(\varphi)\right)\right\rangle_{\rho,\beta}-\exp\left(\frac{z^2}{2}\langle T_{f,L,\beta}(\varphi)^2\rangle_{\rho,\beta}\right)\right|\\\leq\exp\left(\frac{z^2}{2}\langle T_{|f|,L,\beta}(\varphi)^2\rangle_{\rho,\beta}\right)\dfrac{C\Vert f\Vert_\infty^4 r_f^{\gamma}z^4}{(\log L)^c}.
\end{multline*}
As a consequence, for $\beta=\beta_c$, every sub-sequential scaling limit of the model is Gaussian.
\end{Coro}
We may also extend the results obtained for the Ising model for $1\leq d\leq 3$ and $d_{\textup{eff}}=4$. The corresponding modifications of Theorem \ref{improved diagram bound deff=4} and Corollary \ref{cor: d=2,3} will be stated in Section \ref{section: gs class}. Their main consequence for the $\varphi^4$ model is stated below.
\begin{Thm} Let $1\leq d\leq 3$. Let $J$ be the interaction defined for $x\neq y\in \mathbb Z^d$ by $J_{x,y}=C|x-y|_1^{-3d/2}$ (i.e. $\alpha=d/2$) where $C>0$. Then, for $\sigma\in(0,d/2)$,
\begin{equation}
    \lim_{\beta\nearrow\beta_c(\rho)}g_\sigma(\rho,\beta)=0.
\end{equation}
As a consequence, for $\beta=\beta_c(\rho)$, every sub-sequential scaling limit of the model is Gaussian.
\end{Thm}
\paragraph{Organisation of the paper.} In Section \ref{section: gs class def}, we define the Griffiths--Simon class of single-site measures to which our results apply. In Section \ref{section: reflection positivity}, we recall the definition of reflection positivity and present the main properties it implies for the models under consideration (monotonicity of the two-point function, infrared bound, etc). The main result of this section is the derivation of the existence of regular scales in Propositions \ref{prop: existence regular scales} and \ref{existence regular scales bis}. In Section \ref{section rcr}, we provide the basic knowledge on the random current representation of the Ising model and explain the heuristics it provides on Gaussianity of the scaling limits. We introduce the notion of effective dimension and prove a generalisation of Theorem \ref{thm: intro deff 4 poly} (see Theorem \ref{d>1}) to models with effective dimension $d_{\textup{eff}}>4$ in Section \ref{section: dim eff >4}. Then, in Section \ref{section d=4}, we prove Theorems \ref{thm: main m2(J) infinite} and \ref{improved diagram bound}, together with Corollary \ref{thm: main}. The handling of long-range interactions is performed in Section \ref{section: properties of the current trivia}. In Section \ref{section : deff=4}, we prove more general versions of Theorem \ref{improved diagram bound deff=4} and Corollary \ref{cor: d=2,3} (see Theorem \ref{improved diagram bound deff=4 but more general} and Corollary \ref{cor: improved tree diagramm deff=4 but more general}). The main modifications in comparison to the $d=4$ case are treated in Section \ref{section: prop currents deff nice}. In Section \ref{section: gs class}, we extend the results to all the models introduced in Section \ref{section: gs class def}. 

In Appendix \ref{appendix spectral representation}, we provide the proofs of the main spectral tools we will use for reflection positive models in the Griffiths--Simon class (see Theorem \ref{Spectral representation} and Proposition \ref{mono 2}). In Appendix \ref{appendix: bounds gs}, we recall some useful bounds for the probability of connectivity events for the random current representation of Ising-type models in the Griffiths--Simon class. Finally, in Appendix \ref{appendix: bubble finite}, we provide an alternative proof of some of our results in the case where the bubble condition \eqref{eq: bubble condition} is satisfied.

\paragraph{Ackowledgements.} We warmly thank Hugo Duminil-Copin for suggesting the problem, for stimulating discussions, and for constant support. We thank Sébastien Ott for pointing us references concerning the spectral representation of the Ising model. We thank Lucas D'Alimonte, Piet Lammers, Trishen Gunaratnam, Christoforos Panagiotis, and an anonymous referee for numerous valuable comments and suggestions on a first version of the paper.

\paragraph{Notations.} We write a point $x\in \mathbb R^d$ as $x=(x_1,\ldots,x_d)$ and denote by $\mathbf{e}_j$ the unit vector with $x_j=1$. We will use the following notations for the standard norms on $\mathbb R^d$: for $x\in \mathbb R^d$, $|x|_1:=|x_1|+\ldots+|x_d|$,  $\Vert x\Vert_2^2:=x_1^2+\ldots+x_d^2$, and $|x|:=\max_{1\leq i\leq d}|x_i|$. Finally, for $k\geq 1$, denote $\Lambda_k:=[-k,k]^d\cap \mathbb Z^d$. 

If $(a_n)_{n\geq 0},(b_n)_{n\geq 0}\in (\mathbb R_+^*)^{\mathbb N}$, we will write $a_n\gtrsim b_n$ (resp. $a_n \asymp b_n$) if there exists $C_1=C_1(d)$ (resp. $C_1=C_1(d),C_2=C_2(d)>0$) such that for all $n\geq 1$, $  b_n\leq C_1 a_n$ (resp. $C_1 a_n\leq b_n \leq C_2 b_n$). We will also use Landau's formalism and write $a_n=O(b_n)$ (resp. $a_n=o(b_n)$) if there exists $C=C(d)>0$ such that for all $n\geq 1$, $a_n\leq C b_n$ (resp. $\lim_{n\rightarrow \infty} a_n/b_n=0$).

\section{The Griffiths--Simon class of measures}\label{section: gs class def}
In this section, we define the proper class of single-site measures to which our results apply.
\begin{Def}[The GS class of measures]\label{def: gs class} A Borel measure $\rho$ on $\mathbb R$ is said to belong to the Griffiths--Simon \textup{(GS)} class of measures if it satisfies one of the following conditions:
\begin{enumerate}
    \item[$(i)$] there exist an integer $N\geq 1$, a renormalisation constant $Z>0$, and sequences $(K_{i,j})_{1\leq i,j\leq N}\in (\mathbb R^+)^{N^2}$ and $(Q_n)_{1\leq n\leq N}\in (\mathbb R^+)^{N}$ such that for every $F:\mathbb R\rightarrow\mathbb R^+$ bounded and measurable,
    \begin{equation}
        \int_{\mathbb R} F(\tau)\textup{d}\rho(\tau)=\frac{1}{Z}\sum_{\sigma\in \lbrace \pm 1\rbrace^N}F\left(\sum_{n=1}^N Q_n\sigma_n\right)\exp\left(\sum_{i,j=1}^N K_{i,j}\sigma_i\sigma_j\right),
    \end{equation}
    \item[$(ii)$] the measure $\rho$ can be presented as a weak limit of probability measures of the above type, and it is of sub-Gaussian growth: for some $\alpha>2$,
    \begin{equation}
        \int_{\mathbb R}e^{|\tau|^\alpha}\textup{d}\rho(\tau)<\infty.
    \end{equation}
\end{enumerate}
Measures that satisfy $(i)$ are said to be of the ``Ising type''.
\end{Def}
The following result was proved in \cite{Griffiths1969rigorous,GriffithsSimon} (see also \cite{KrachunPanagiotisPanisScalinglimit2023}) to extend the Lee-Yang theorem, together with Griffiths' correlation inequalities, to the $\varphi^4$ model on $\mathbb Z^d$. We sketch its proof for sake of completeness.
\begin{Prop}[The $\varphi^4$ measure belongs to the GS class, \cite{GriffithsSimon}]\label{prop: phi4 is GS} Let $g>0$ and $a\in \mathbb R$. The probability measure $\rho_{g,a}$ on $\mathbb R$ given by 
\begin{equation}
    \textup{d}\rho_{g,a}(\varphi)=\frac{1}{z_{g,a}}e^{-g\varphi^4-a\varphi^2}\textup{d}\varphi,
\end{equation}
where $z_{g,a}$ is a renormalisation constant, belongs to the GS class.
\end{Prop}
\begin{proof} Let $N\geq 1$. Let $\widetilde g = (12g)^{-1/4}$ and $\widetilde a = 2a {\widetilde g}^{2}$. Define the coupling constants 
\begin{equation}
    c_N:=\widetilde{g}N^{-3/4},
\qquad d_N:=\frac{1}{N}\left(1-\frac{\widetilde{a}}{\sqrt{N}}\right).
\end{equation} 
Define the Ising Gibbs measure $\mu_N$ on the complete graph $K_N$, with Hamiltonian given by
\begin{equation}
    \mathbf{H}_{N}(\sigma)
        :=
        -d_N\sum_{\lbrace i,j\rbrace\subset K_N}\sigma_{i}\sigma_{j}.
\end{equation}
Let $\rho_N$ be the law of the random variable $\Phi_N:=c_N\sum_{i=1}^N\sigma_i$. Then, $\rho_N$ converges weakly to $\rho_{g,a}$.
\end{proof}

\section{Reflection positivity}\label{section: reflection positivity}
In this section, we define reflection positivity and gather all the properties it implies in our setup (monotonicity, infrared bound, gradient estimates, etc).  We refer to the review \cite{BiskupReflectionPositivity2009} or the original papers \cite{FrohlichSimonSpencerIRBounds1976,FrohlichIsraelLiebSimon1978} for more information. 

The end-goal is to derive the existence of \emph{regular scales} for general reflection positive interactions (see Propositions \ref{prop: existence regular scales} and \ref{existence regular scales bis}). Most of the results are classical and their proofs in the case of nearest-neighbour ferromagnetic (n.n.f) models in the GS class were already derived in \cite{AizenmanDuminilTriviality2021} using the spectral representation of the Ising model through the lens of transfer matrices. This approach is not optimal in the most general setup. Our viewpoint will be that of self-adjoint operators in infinite dimensional spaces, which allows us to import general results from \cite{HallQuantumTheory2013} (see Appendix \ref{appendix spectral representation}).

The following statements apply to both Ising and $\varphi^4$ systems. To unify the notations, we refer to the spin or field variables by the symbol $\tau$, with an a-priori spin distribution $\textup{d}\rho(\tau)$ in the GS class which is supported on a set $\mathcal{S}\subset \mathbb R$. The expectation value functional with respect to the Gibbs measure, or functional integral, for a system in a domain $\Lambda$, is denoted $\langle\cdot\rangle_{\Lambda,\rho,\beta}$. We denote by $\langle \cdot \rangle_{\rho,\beta}$ the state's natural infinite volume limit. We also denote by $\beta_c(\rho)$ the critical inverse temperature, and $\xi(\rho,\beta)$ the correlation length. We sometimes omit $\rho$ in the notations when it is clear from context.

We will use the following notation for the model's two-point function,
\begin{equation}
    S_{\rho,\beta}(x):=\langle \tau_0\tau_x\rangle_{\rho,\beta}.
\end{equation}
Also, introduce \textit{the finite-volume susceptibility}, for $L\geq 1$,
\begin{equation}
    \chi_L(\rho,\beta):=\sum_{x\in \Lambda_L}S_{\rho,\beta}(x),
\end{equation}
and define the \textit{susceptibility} to be $\chi(\rho,\beta):=\lim_{L\rightarrow \infty}\chi_L(\rho,\beta)$.
Finally, recall that  $|J|=\sum_{x\in \mathbb Z^d} J_{0,x}$.

\subsection{Definition of reflection positivity}\label{section: definition ref pos}
Let $d\geq 1$. Consider the torus $\mathbb T_L:=(\mathbb Z/L\mathbb Z)^d$ with $L \geq 2$ an even integer. The torus is endowed with a natural reflection symmetry along hyperplanes $\mathcal{H}$ which are orthogonal to one of the lattice's directions. The hyperplane $\mathcal{H}$ either passes through sites of $\mathbb T_L$ or through mid-edges, and $\mathcal{H}$ divides the torus into two pieces $\mathbb T_L^+$ and $\mathbb T_L^-$. The two pieces are disjoint for mid-edges reflections and satisfy $\mathbb T_L^+\cap \mathbb T_L^-=\mathcal{H}$ for site reflections. Denote by $\mathcal{A}^{\pm}$ the algebra of all real valued functions $f$ that depend only on the spins in $\mathbb T_L^{\pm}$. Denote by $\Theta$ the reflection map associated with $\mathcal{H}$; it naturally acts on $\mathcal{A}^{\pm}$: for all $f\in \mathcal{A}^\pm$,
\begin{equation}
    \Theta(f)(\tau):=f(\Theta(\tau)), \qquad \forall \tau\in \mathcal{S}^{\mathbb T_L}.
\end{equation}
If $J=(J_{x,y})_{x,y\in \mathbb Z^d}\in(\mathbb R^+)^{\mathbb Z^d\times \mathbb Z^d}$, we can view it as an interaction $J^{(L)}$ on $\mathbb T_L$ by setting 
\begin{equation}
    J_{x,y}^{(L)}:=\sum_{z\in \mathbb Z^d}J_{x,y+Lz}.
\end{equation}
\begin{Def}[Reflection positivity] Let $J=(J_{x,y})_{x,y\in \mathbb Z^d}\in(\mathbb R^+)^{\mathbb Z^d\times \mathbb Z^d}$ be an interaction. The measure $\langle \cdot \rangle_{\mathbb T_L,\rho,\beta}=\langle \cdot \rangle_{\mathbb T_L,\rho,J^{(L)},\beta}$ is called reflection positive \textup{(RP)} with respect to $\Theta$, if for all $f,g\in \mathcal{A}^+$,
\begin{equation}
    \langle f\cdot\Theta(g)\rangle_{\mathbb T_L,\rho,\beta}=\langle \Theta(f)\cdot g\rangle_{\mathbb T_L,\rho,\beta},
\end{equation}
and,
\begin{equation}
    \langle f\cdot \Theta(f)\rangle_{\mathbb T_L,\rho,\beta}\geq 0.
\end{equation}
We say that $J$ is reflection positive if for all $L\geq 2$ even, the associated measure $\langle \cdot \rangle_{\mathbb T_L,\rho,\beta}$ is reflection positive with respect to $\Theta$ for all such reflections $\Theta$.
\end{Def}
Before discussing the interest of studying such interactions let us briefly mention some examples (for more details see \cite{FrohlichSimonSpencerIRBounds1976,FrohlichIsraelLiebSimon1978,AizenmanFernandezLongrange,BiskupReflectionPositivity2009}):
\begin{enumerate}
    \item[$(i)$] (nearest-neighbour interactions) $J_{x,y}=\mathds{1}_{|x-y|_1=1},$
    \item[$(ii)$] (exponential decay / Yukawa potentials) $J_{x,y}=C\exp(-\mu |x-y|_1)$ for $\mu,C>0$,
    \item[$(iii)$] (power law decay) $J_{x,y}=C|x-y|_1^{-d-\alpha}$ for $\alpha,C>0$,
    \item[$(iv)$] $J_{x,y}=1/\prod_{i=1}^d(|x_i-y_i|+a_i^2)^{\tau_i}$, for $\tau_i\geq 0$ and $a_i \in \mathbb R$.
\end{enumerate}
The last example above can be found in \cite{FrohlichIsraelLiebSimon1978,AizenmanFernandezLongrange} and is an example of a $d$-dimensional RP interaction constructed as the product of $1$-dimensional RP interactions. Furthermore, models of the GS class whose
couplings are linear combinations with positive coefficients of the couplings
mentioned above are also reflection-positive. Note that all finite range reflection positive interaction satisfy $J_{0,x}=0$ whenever $|x|>1$ (see p.12 of \cite{FrohlichIsraelLiebSimon1978} for more information).

We now turn to the main consequences of reflection positivity. In what follows, we consider reflection positive models with $\rho$ in the GS class. We fix $J$ satisfying $\mathbf{(A1)}$--$\mathbf{(A5)}$.
\subsection{The Messager--Miracle-Solé inequalities}
The Messager--Miracle-Solé inequality provides monotonicity properties for $S_{\rho,\beta}$ in the case of reflection positive interactions (see Figure \ref{figure: mms} for an illustration of this result).

\begin{Prop}[MMS inequalities, \cite{HegerCorrelationInequalitiesIsing1977,MessagerMiracleSoleInequalityIsing,SchraderCorrelation1977}]\label{prop: MMS inequalities general statement} Let $d\geq 1$.
Let $\Lambda\subset \mathbb Z^d$ be a region endowed with reflection symmetry with respect to a plane $\mathcal{P}$. Let $A,B$ be two sets of points on the same side of the reflection plane. If $\Theta$ is the reflection with respect to $\mathcal{P}$,
\begin{equation}
    \left\langle \prod_{x\in A}\tau_x\prod_{x\in B}\tau_x\right\rangle_{\Lambda,\rho,\beta}\geq \left\langle \prod_{x\in A}\tau_x\prod_{x\in \Theta(B)}\tau_x\right\rangle_{\Lambda,\rho,\beta}.
\end{equation}
Moreover, in the infinite volume limit $\Lambda \nearrow\mathbb Z^d$ this result can be extended to reflections with respect to hyperplanes passing through sites or mid-edges; and to reflections with respect to diagonal hyperplanes or more precisely, reflections changing only two coordinates $x_i$ and $x_j$ which are sent to $x_i\pm L$ and $x_j\mp L$ respectively, for some $L \in \mathbb Z$.
\end{Prop}
As a result, we get the following monotonicity property for reflection positive models in the GS class.
\begin{Coro}[Monotonicity of the two-point function]\label{corro mms 1} Let $d\geq 1$ and $\beta>0$. Then,
\begin{enumerate}
    \item[$(i)$] for every $1\leq j \leq d$, the sequence $(S_{\rho,\beta}(k\mathbf{e}_j))_{k\geq 0}$ is decreasing,
    \item[$(ii)$] for every $x\in \mathbb Z^d$, 
\begin{equation}\label{eq: MMS1}\tag{\textbf{MMS1}}
    S_{\rho,\beta}((|x |,0_\bot))\geq S_{\rho,\beta}(x)\geq S_{\rho,\beta}((|x|_1,0_\bot)),
\end{equation}
where $0_\bot\in \mathbb Z^{d-1}$ is the null vector. 
In particular, for every $x,y\in \mathbb Z^d$ with $d|x|\leq |y|$,
\begin{equation}\label{eq: consequence mms }\tag{\textbf{MMS2}}
    S_{\rho,\beta}(x)\geq S_{\rho,\beta}(y).
\end{equation}
\end{enumerate}
\end{Coro}
\begin{figure}[htb]
\begin{center}
\includegraphics[scale=1]{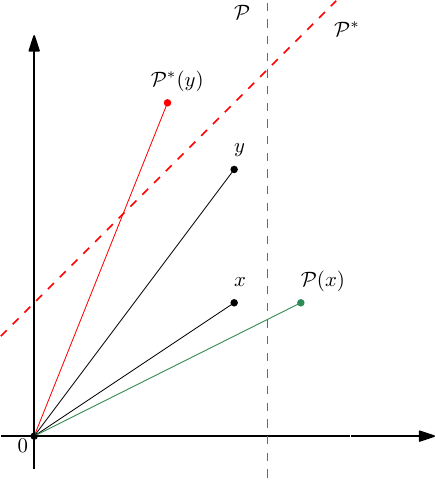}
\end{center}

\caption{An illustration of Proposition \ref{prop: MMS inequalities general statement}. The two reflection planes are represented by the green and red dashed lines. The MMS inequalities state that $\langle \tau_0\tau_x\rangle_{\rho,\beta}\geq \langle \tau_0\tau_{\mathcal{P}(x)}\rangle_{\rho,\beta}$ and $\langle \tau_0\tau_y\rangle_{\rho,\beta}\geq \langle \tau_0\tau_{\mathcal{P}^*(y)}\rangle_{\rho,\beta}$.}
\label{figure: mms}
\end{figure}

\subsection{The infrared bound} In this subsection, we show how to bound the critical two-point function of the model in terms of the Green function of the random walk associated\footnote{Indeed, $J$ generates a step distribution $p_J$ defined as follows: for every $x,y\in \mathbb Z^d$, $p_J(x,y):=\tfrac{J_{x,y}}{|J|}$.} with the interaction $J$. Unlike \cite{AizenmanDuminilTriviality2021}, we do not restrict our analysis to the nearest-neighbour interaction and we need a bound which is more general than \eqref{eq: ir intro}. It is derived in Proposition \ref{IR bound}.


We still work on the $d$-dimensional torus $\mathbb T_L$ (with $L$ even) with $d\geq 1$. In view of the model's translation invariance, it is natural to introduce the Fourier transform of $S_{\rho,\beta}$,
\begin{equation}
    \widehat{S}^{(L)}_{\rho,\beta}(p):=\sum_{x\in \mathbb T_L}\textup{e}^{ip\cdot x}\langle \tau_x\tau_y\rangle_{\mathbb T_L,\rho,\beta}
\end{equation}
where $p$ ranges overs $\mathbb T_L^*:=\left(\frac{2\pi}{L}\mathbb Z\right)^d\cap (-\pi,\pi]^d.$
The Fourier transform $\widehat{S}^{(L)}_{\rho,\beta}(p)$ can be expressed in terms of the Fourier \textit{spin-wave modes}, defined as
\begin{equation}
    \widehat{\tau}_\beta(p):=\frac{1}{\sqrt{(2L)^d}}\sum_{x\in \mathbb T_L}e^{ip\cdot x}\tau_x.
\end{equation}
Indeed, one has for $p\in \mathbb T_L^*$,
\begin{equation}\label{eq: thermal averages equal fourier transform}
    \widehat{S}^{(L)}_{\rho,\beta}(p)=\langle |\widehat{\tau}_\beta(p)|^2\rangle_{\mathbb T_L,\rho,\beta},
\end{equation}
so that in particular $\widehat{S}^{(L)}_{\rho,\beta}(p)\geq 0$.

The following fundamental result was first proved and used in \cite{FrohlichSimonSpencerIRBounds1976,FrohlichIsraelLiebSimon1978}. 
\begin{Prop}[Infrared bound]\label{gaussian domi} Let $d\geq 1$.
For every $p\in \mathbb T_L^*\setminus \lbrace 0 \rbrace$, 
\begin{equation}
    \widehat{S}^{(L)}_{\rho,\beta}(p)\leq\dfrac{1}{2\beta |J|(1-\widehat{J}(p))},
\end{equation}
where $\widehat{J}(p):= \sum_{x\in \mathbb Z^d}\textup{e}^{ip\cdot x}\frac{J_{0,x}}{|J|}$.
\end{Prop}
Introduce for $\beta<\beta_c(\rho)$, and $p\in (-\pi,\pi]^d$,
\begin{equation}
    \widehat{S}_{\rho,\beta}(p):=\sum_{x\in \mathbb Z^d}e^{ip\cdot x}S_{\rho,\beta}(x).
\end{equation}
Note that this quantity is well defined since $\sum_{x\in \mathbb Z^d}S_{\rho,\beta}(x)<\infty$ for $\beta<\beta_c(\rho)$ as proved in \cite{AizenmanBarskyFernandezSharpnessIsing1987}. The next result will use the celebrated Simon--Lieb inequality which we now recall.
\begin{Lem}[Simon--Lieb inequality, \cite{SimonInequalityIsing1980,LiebImprovementSimonInequality}]\label{lem: simon lieb} Let $d\geq 1$. For every ferromagnetic model in the GS class on $\mathbb Z^d$ with translation invariant coupling constants, every $\beta>0$, every finite subset $\Lambda$ of $\mathbb Z^d$ containing $0$, and every $x\notin \Lambda$,
\begin{equation}
    S_{\rho,\beta}(x)\leq \sum_{\substack{u\in \Lambda\\ v\notin \Lambda}}S_{\rho,\beta}(u)\beta J_{u,v} S_{\rho,\beta}(x-v).
\end{equation}
\end{Lem}
The following result extends the infrared bound to $\widehat{S}_{\rho,\beta}$.
\begin{Prop}\label{lem: limit gaussian bound} Let $d\geq 1$, $\beta<\beta_c(\rho)$, and $p\in (-\pi,\pi]^d$. Then,
\begin{equation}
    \widehat{S}_{\rho,\beta}(p)\leq \dfrac{1}{2\beta |J|(1-\widehat{J}(p))}.
\end{equation}
\end{Prop}
\begin{proof} If $\beta<\beta_c(\rho)$, it is classical that there is only one infinite volume equilibrium state that we denote $\langle \cdot \rangle_{\rho,\beta}$. Moreover, for all $x \in \mathbb Z^d$,
\begin{equation}\label{eq: conv two point function}
    \langle \tau_0\tau_x\rangle_{\mathbb T_L,\rho,\beta}\underset{L\rightarrow \infty}{\longrightarrow}\langle \tau_0\tau_x \rangle_{\rho,\beta}.
\end{equation}
Fix $L$ even and take limits along sequences of the form $(L^k)_{k\geq 1}$ so that $\mathbb T^*_L\subset \mathbb T_{L^k}^*$ for all $k\geq 1$. For $p \in \mathbb T_L^*$, notice that by Fatou's lemma and \eqref{eq: conv two point function},
\begin{eqnarray*}
    \widehat{S}_{\rho,\beta}(p)+\chi(\rho,\beta)&=&\sum_{x\in \mathbb Z^d}(1+\cos(p\cdot x))S_{\rho,\beta}(x)\\&\leq& \liminf \left(\widehat{S}^{(L^k)}_{\rho,\beta}(p)+\chi^{(L^k)}(\rho,\beta)\right)\\&\leq& \frac{1}{2\beta |J|(1-\widehat{J}(p))}+\liminf \chi^{(L^k)}(\rho,\beta),
\end{eqnarray*}
where $\chi^{(L^k)}(\rho,\beta):=\widehat{S}^{(L^k)}_{\rho,\beta}(0)$ and $\chi(\rho,\beta):=\widehat{S}_{\rho,\beta}(0)$.
It then suffices to show that $\chi^{(L^k)}(\rho,\beta)$ goes to $\chi(\rho,\beta)$ as $k$ goes to infinity. Using the Simon--Lieb inequality (as in \cite{DuminilTassionNewProofSharpness2016}), we get that for any $K\geq 0$,
\begin{equation}\label{eq: chi varphi sharpness}
    \chi^{(L^k)}(\rho,\beta)\leq \widetilde{\varphi}_{\rho,\beta}(\Lambda_K)\chi^{(L^k)}(\rho,\beta)+\chi_{K}(\rho,\beta),
\end{equation}
where $\widetilde{\varphi}_{\rho,\beta}(\Lambda_K):=\beta \sum_{\substack{x\in \Lambda_K\\ y\notin \Lambda_K}}J_{x,y}\langle \tau_0\tau_x\rangle_{\rho,\beta}$. Since $\chi(\rho,\beta)<\infty$, one has $\limsup_K \widetilde{\varphi}_{\rho,\beta}(\Lambda_K)=0$. In particular, for every $\varepsilon>0$,
\begin{equation}
    \limsup\chi^{(L^k)}(\rho,\beta)\leq \frac{1}{1-\varepsilon}\chi(\rho,\beta).
\end{equation}
This gives the result.
\end{proof}
Below, we say that $J$ is \emph{transient} if the associated random walk is transient, or equivalently if $(1-\widehat{J}(p))^{-1}$ is integrable near $0$.

As a first consequence of the above result, we see that if $J$ is transient (which is always the case in dimensions $d\geq 3$), there exists $C=C(d)>0$ such that for $\beta<\beta_c(\rho)$,
\begin{equation}
    \langle \tau_0^2\rangle_{\rho,\beta}=\int_{(-\pi,\pi]^d}\widehat{S}_{\rho,\beta}(p)\mathrm{d}p\leq \frac{C}{\beta |J|}.
\end{equation}

Note that this bound extends to $\beta_c(\rho)$ by continuity. Since $\beta \mapsto \langle \tau_0^2\rangle_{\rho,\beta}$ is increasing\footnote{This is a classical consequence of Griffiths' inequalities.}, we also get that for all $\beta\leq \beta_c(\rho)$,
\begin{equation}\label{eq: bound tau_0}
    \langle \tau_0^2\rangle_{\rho,\beta}\leq \frac{C}{\beta_c(\rho)|J|}.
\end{equation}
Proposition \ref{lem: limit gaussian bound} together with the MMS inequalities also yield the following result.
\begin{Prop}[Infrared bound]\label{IR bound} Let $d\geq 1$. There exists $C=C(d)>0$ such that for every $\beta\leq \beta_c(\rho)$,  and every $x \in \mathbb Z^d\setminus \lbrace 0\rbrace$,
\begin{equation}
    S_{\rho,\beta}(x)\leq \frac{C}{\beta|J||x|^d}\int_{\big(-\pi|x|,\pi|x|\big]^d}\frac{e^{-\Vert p\Vert_2^2}}{1-\widehat{J}(|p|/|x|)}\textup{d}p.
\end{equation}
In particular, if $d\geq 3$, for all $x \in \mathbb Z^{d}\setminus \lbrace 0\rbrace$,
\begin{equation}\label{eq: INFRARED BOUND WE USE}\tag{\textbf{IRB}}
    S_{\rho,\beta}(x)\leq S_{\rho,\beta_c(\rho)}(x)\leq\frac{C}{\beta_c(\rho)|J||x|^{d-2}}.
\end{equation}
\end{Prop}
\begin{proof} We prove the result for $\beta<\beta_c(\rho)$ and extend it to $\beta_c(\rho)$ with a continuity argument.

Using \eqref{eq: consequence mms }, we get that for some $C_1=C_1(d)>0$,
\begin{equation}\label{eq: s bound by chi}
    S_{\rho,\beta}(x)\leq \frac{C_1}{|x|^d}\sum_{y\in \textup{Ann}(|x|(2d)^{-1},|x|d^{-1})}S_{\rho,\beta}(y)\leq \frac{C_1}{|x|^d}\chi_{|x|}(\rho,\beta).
\end{equation}
We now observe that Proposition \ref{lem: limit gaussian bound} provides a control on the finite volume susceptibility $\chi_L(\rho,\beta)$. Let 
\begin{equation}
    \widetilde{\chi}_L(\rho,\beta):=\sum_{x\in \mathbb Z^d}e^{-(\Vert x\Vert _2/L)^2}S_{\rho,\beta}(x).
\end{equation}
There exists $C_2=C_2(d)>0$ such that, $\chi_L(\beta)\leq C_2\widetilde{\chi}_L(\rho,\beta)$. Using classical Fourier identities, we get $C_3=C_3(d)>0$ such that,
\begin{equation}
    \chi_{L}(\rho,\beta)\leq C_2\widetilde{\chi}_L(\rho,\beta)\leq C_3L^d \int_{(-\pi,\pi]^d}e^{-L^2\Vert p \Vert_2^2}\widehat{S}_{\rho,\beta}(p)\textup{d}p.
\end{equation}
With the change of variable $u=pL$ and Proposition \ref{lem: limit gaussian bound},
\begin{eqnarray*}
    \int_{(-\pi,\pi]^d}e^{-L^2\Vert p \Vert_2^2}\widehat{S}_{\rho,\beta}(p)\textup{d}p&\leq& \frac{1}{L^d}\int_{(-L\pi,L\pi]^d}e^{-\Vert u \Vert_2^2}\widehat{S}_{\rho,\beta}(u/L)\textup{d}u\\&\leq& \frac{C_3}{\beta |J|L^d}\int_{(-L\pi,L\pi]^d}\frac{e^{-\Vert u \Vert_2^2}}{1-\widehat{J}(u/L)}\textup{d}u.
\end{eqnarray*}
The second part of the statement then follows from monotonicity in $\beta$ of $S_\beta(x)$ together with the observation that $1-\widehat{J}(k)\gtrsim \Vert k\Vert_2^2$ as $k\rightarrow 0$.
\end{proof}
The preceding result essentially gives that the decay of the two-point function is governed by the behaviour of $1-\widehat{J}(p)$ as $p$ goes to $0$. We have the following estimates for the examples of RP interactions given above in dimensions $d\geq 3$:
\begin{enumerate}
    \item[$\bullet$] Nearest-neighbour interactions or Yukawa potentials: as $p\rightarrow 0$,
    \begin{equation}
        1-\widehat{J}(p)\asymp |p|^2.
    \end{equation}
    \item[$\bullet$] Power law decay interactions: as $p\rightarrow 0$,
    \begin{equation}
        1-\widehat{J}(p)\asymp \left\{
    \begin{array}{ll}
        |p|^2 & \mbox{if } \alpha>2, \\
        |p|^2\log \frac{1}{|p|} & \mbox{if }\alpha=2, \\
        |p|^{\alpha} & \mbox{if }\alpha\in (0,2).
    \end{array}
\right.
    \end{equation}
\end{enumerate}

Together with Proposition \ref{IR bound}, we get that for interactions with algebraic decay,
\begin{equation}\label{eq: IRbounds with algebraic interactions}
        \langle \tau_0\tau_x\rangle_{\rho,\beta}\leq \frac{C}{\beta |J|}\left\{
    \begin{array}{ll}
        |x|^{-(d-2)} & \mbox{if } \alpha>2, \\
        |x|^{-(d-2)}(\log|x|)^{-1} & \mbox{if }\alpha=2, \\
        |x|^{-(d-\alpha)} & \mbox{if }\alpha\in (0,2).
    \end{array}
\right.
    \end{equation}
Moreover, \eqref{eq: IRbounds with algebraic interactions} is also valid for $d=2$ in the case $\alpha\in (0,2)$ and for $d=1$ with $\alpha\in(0,1)$, since in both cases $|p|^{-\alpha}$ is locally integrable (or equivalently, the random walk associated with $J$ is transient).

Finally, Proposition \ref{IR bound} also yields the following improvement on the bound of the model's two-point function when the interaction $J$ has a slow decay.
\begin{Coro}\label{cor: infinite second moment improvement ir bound} Let $d=4$. Assume that $\mathfrak{m}_2(J)=\infty$. Then, as $|x|\rightarrow \infty$,
\begin{equation}
    \langle \tau_0\tau_x\rangle_{\rho,\beta_c(\rho)}=o\left(\frac{1}{|x|^2}\right).
\end{equation}
\end{Coro}

\subsection{Spectral representation of reflection positive models and applications} 
The goal of this subsection is to derive the so-called \emph{sliding-scale infrared bound}. This bound was first obtained in \cite{AizenmanDuminilTriviality2021} in the setup of nearest-neighbour interactions.

We fix a reflection positive model on $\mathbb Z^d$ ($d\geq 1$) with an interaction $J$ satisfying (\textbf{A1})--(\textbf{A5}). The following statement contains a minor technicality\footnote{For full disclosure, the requirement $L(\rho,\beta)\geq 3$ ensures, by the results of Section \ref{section lower bounds}, that $\chi_\ell(\rho,\beta)\geq \chi_1(\rho,\beta)\geq \frac{c}{\beta}$ for some $c=c(d)>0$. Plugging this estimate in \cite[Equation~(5.35)]{AizenmanDuminilTriviality2021} allows to derive a sliding-scale infrared bound with $C$ instead of $C/\beta$. This small technicality will be helpful in Section \ref{section: gs class}, but can be ignored in Sections \ref{section d=4} and \ref{section : deff=4}.} as it involves the sharp length $L(\rho,\beta)$ defined in Section \ref{section lower bounds} below.
\begin{Thm}[Sliding-scale infrared bound, {\cite[Theorem~5.6]{AizenmanDuminilTriviality2021}}]\label{sliding scale ir bound} Let $d\geq 1$.
There exists $C=C(d)>0$ such that for every $\beta\leq \beta_c(\rho)$ such that $L(\rho,\beta)\geq 3$, and for every $1\leq \ell\leq L$,
\begin{equation}
     \dfrac{\chi_L(\rho,\beta)}{L^2}\leq C\dfrac{\chi_\ell(\rho,\beta)}{\ell^2}.
\end{equation}
\end{Thm}
The proof of Theorem \ref{sliding scale ir bound} follows the same strategy as in \cite{AizenmanDuminilTriviality2021}. The first step is to generalise the spectral representation used in \cite{AizenmanDuminilTriviality2021}. The following result is classical for n.n.f models in the GS class (see \cite{GlimmJaffe1973PHI43d,Sokal1982Destructive} or \cite[Proposition~5.3]{AizenmanDuminilTriviality2021}). Its proof in our setup requires some work and we present it in Appendix \ref{appendix spectral representation}.
\begin{Thm}[Spectral representation]\label{Spectral representation} Let $d\geq 1$. For every $\beta\leq \beta_c(\rho)$ and every function $v:\mathbb Z^{d-1}\rightarrow \mathbb C$ in $\ell^2(\mathbb Z^{d-1})$, there exists a positive measure $\mu_{v,\beta}$ of finite mass
\begin{equation}
    \int_0^\infty\textup{d}\mu_{v,\beta}(a)=\sum_{\xb,\yb\in \mathbb Z^{d-1}}v_{\xb}\overline{v_{\yb}}S_{\rho,\beta}((0,\xb-\yb))\leq \Vert v\Vert_2^2\langle \tau_0^2\rangle_{\rho,\beta},
\end{equation}
such that for every $n\in \mathbb Z$,
\begin{equation}\label{spec rep eq}
    \sum_{\xb,\yb\in \mathbb Z^{d-1}}v_{\xb}\overline{v_{\yb}}S_{\rho,\beta}((n,\xb-\yb))=\int_0^\infty e^{-a|n|}\textup{d}\mu_{v,\beta}(a).
\end{equation}
\end{Thm}
Theorem \ref{Spectral representation} provides a very useful representation of the Fourier transform of $S_{\rho,\beta}$. We refer to \cite{AizenmanDuminilTriviality2021} (see also \cite[Chapter~9]{PanThesis}) for a proof.
\begin{Coro}\label{rep transf fourier}
Let $\beta<\beta_c(\rho)$. Let $p=(p_1,p_\bot)\in (-\pi,\pi]^d$. There exists a measure $\mu_{p_\bot,\beta}$ such that
\begin{equation}
    \widehat{S}_{\rho,\beta}(p)=\sum_{x\in \mathbb Z^d}e^{ip\cdot x}S_{\rho,\beta}(x)=\int_0^{\infty}\dfrac{e^a-e^{-a}}{\mathcal{E}_1(p_1)+\left(e^{a/2}-e^{-a/2}\right)^2}\textup{d}\mu_{p_\bot,\beta}(a),
\end{equation}
where $\mathcal{E}_1(k):=2(1-\cos k)=4\sin^2(k/2)$. Moreover, the result is still true under any permutation of the indices.
\end{Coro}
\begin{Coro}\label{mono 1}
Let $d\geq 1$ and $\beta<\beta_c(\rho)$. Then,
\begin{enumerate}
    \item[$(i)$] $\widehat{S}_{\rho,\beta}(p_1,\ldots,p_d)$ is monotone decreasing in each $|p_j|$ over $[-\pi,\pi]$.
    \item[$(ii)$] $\mathcal{E}_1(p_1)\widehat{S}_{\rho,\beta}(p)$ and $|p_1|^2\widehat{S}_{\rho,\beta}(p)$ are monotone increasing in $|p_1|$.
\end{enumerate}
\end{Coro}
\begin{proof}
The result is a direct consequence of Corollary \ref{rep transf fourier} and of the monotonicity of the following functions:
\begin{equation}
    u\in [0,\pi]\mapsto \mathcal{E}_1(u), \qquad u>0 \mapsto \dfrac{u^2}{\mathcal{E}_1(u)},
\end{equation}
and for all $a\geq 0$, 
\begin{equation}
    u>0 \mapsto\dfrac{\mathcal{E}_1(u)}{\mathcal{E}_1(u)+(e^{a/2}-e^{-a/2})^2}.
\end{equation}
\end{proof}
The next result, obtained in \cite{AizenmanDuminilTriviality2021} in the case of nearest-neighbour interactions, will be useful to derive Theorem \ref{sliding scale ir bound}. Again, in our more general setup, we cannot apply directly the argument of \cite{AizenmanDuminilTriviality2021}. We derive it using the same methods used to obtain Theorem \ref{Spectral representation}, and postpone the proof to Appendix \ref{appendix spectral representation}.
\begin{Prop}\label{mono 2}
Let $d\geq 2$ and $\beta<\beta_c(\rho)$. Introduce for $p\in \mathbb R^d$,
\begin{equation}
    \widehat{S}_{\rho,\beta}^{\textup{(mod)}}(p):=\widehat{S}_{\rho,\beta}(p)+\widehat{S}_{\rho,\beta}(p+\pi (1,1,0,\ldots,0)).
\end{equation}
Then $\widehat{S}_{\rho,\beta}^{\textup{(mod)}}$ is monotone decreasing in $|p_1-p_2|$ with $p_1+p_2$ and $(p_3,\ldots,p_d)$ constant.
\end{Prop}

\begin{Coro}\label{mono transf fourier}
Let $d\geq 1$ and $\beta<\beta_c(\rho)$. There exists $C=C(d)>0$ such that for all $p\in [-\pi/2,\pi/2]^d$, 
\begin{equation}
    \widehat{S}_{\rho,\beta}((| p |,0_\bot))\geq \widehat{S}_{\rho,\beta}(p)\geq \widehat{S}_{\rho,\beta}((| p|_1,0_\bot))-\dfrac{C}{\beta}.
\end{equation}
\end{Coro}
\begin{proof} Note that the case $d=1$ is immediate. We thus assume that $d\geq 2$.
The first inequality is a direct consequence of the first item of Corollary \ref{mono 1}. For the second inequality, notice that (an iteration of) Proposition \ref{mono 2} yields
\begin{equation}
    \widehat{S}_{\rho,\beta}^{\textup{(mod)}}(p)\geq \widehat{S}_{\rho,\beta}^{\textup{(mod)}}((| p|_1, 0_\bot)).
\end{equation}
Now, recall that by Corollary \ref{rep transf fourier}, one has $\widehat{S}_{\rho,\beta} \geq 0$, and that by the infrared bound of Proposition \ref{lem: limit gaussian bound}, there exists $C>0$ such that for $p\in [-\pi/2,\pi/2]^d$,
\begin{equation}
    |\widehat{S}_{\rho,\beta}(p+\pi(1,1,0,\ldots,0))|\leq \dfrac{C}{\beta}.
\end{equation}
\end{proof}

\begin{proof}[Proof of Theorem \textup{\ref{sliding scale ir bound}}] Now that Corollary \ref{mono transf fourier} is obtained in the general setup, the proof follows the exact same lines as in \cite{AizenmanDuminilTriviality2021}.
\end{proof}

\subsection{Gradient estimates}
The following is a consequence of Theorem \ref{Spectral representation}. It plays a crucial role in the proof of existence of regular scales that follows. 
\begin{Prop}[Gradient estimate, {\cite[Proposition~5.9]{AizenmanDuminilTriviality2021}}]\label{prop: gradient estimate} Let $d\geq 1$. There exists $C=C(d)>0$ such that for every $\beta\leq\beta_c(\rho)$, every $x\in \mathbb Z^d$ and every $1\leq i \leq d$,
\begin{equation}
    |S_{\rho,\beta}(x\pm\mathbf{e}_i)-S_{\rho,\beta}(x)|\leq \frac{F(|x|)}{|x|}S_{\rho,\beta}(x),
\end{equation}
where $F(n):=C\frac{S_{\rho,\beta}(dn\mathbf{e}_1)}{S_{\rho,\beta}(n\mathbf{e}_1)}\log\left(\frac{2S_{\rho,\beta}(\frac{n}{2}\mathbf{e}_1)}{S_{\rho,\beta}(n\mathbf{e}_1)}\right)$.
\end{Prop}
\begin{Rem}\label{rem: below gradient}
The above estimate becomes particularly interesting whenever there exists $c_0>0$ such that,
\begin{equation}
    S_{\rho,\beta}(2dn\mathbf{e}_1)\geq c_0S_{\rho,\beta}\left(\frac{n}{2}\mathbf{e}_1\right).
\end{equation}
Indeed, in that case, one can find $C_0=C_0(c_0,d)>0$ such that for all $x\in \partial \Lambda_n$ and $1\leq i\leq d$,
\begin{equation}
    |S_{\rho,\beta}(x\pm \mathbf{e}_i)-S_{\rho,\beta}(x)|\leq \frac{C_0}{|x|}S_{\rho,\beta}(x).
\end{equation}
\end{Rem}
\subsection{The sharp length and a lower bound on the two-point function}\label{section lower bounds}
Since we work with infinite range interactions, it is possible that $\xi(\rho,\beta)=\infty$ throughout the subcritical phase $\beta<\beta_c(\rho)$ (this is for instance the case for algebraically decaying RP interactions \cite{NewmanSpohn1998shiba,Aoun2021SharpAsymp}). This forces us to revisit the notion of ``typical length'' in these setups. As suggested by the work \cite{DuminilTassionNewProofSharpness2016}, the quantity defined below is a good candidate for the typical size of a box in which the model has a critical behaviour.
\begin{Def}[Sharp length] Let $\beta>0$. Let $S$ be a finite subset of $\mathbb Z^d$ containing $0$. Let
\begin{equation}
    \varphi_{\rho,\beta}(S):=\beta\sum_{\substack{x\in S\\ y\notin S}}J_{x,y}\langle \tau_0\tau_x\rangle_{S,\rho,\beta}.
\end{equation}
Define the sharp length of parameter $\alpha\in (0,1)$ by
\begin{equation}
    L^{(\alpha)}(\rho,\beta):=\inf\left\lbrace k\geq 1: \:\exists S\subset \mathbb Z^d \textup{ with }0\in S, \: \textup{rad}(S)\leq 2k, \: \varphi_{\rho,\beta}(S)< \alpha\right\rbrace,
\end{equation}
where $\textup{rad}(S):=\max \lbrace |x-y|: \: x,y\in S\rbrace$, and with the convention that $\inf \emptyset=\infty$. We will set $L(\rho,\beta):=L^{(1/2)}(\rho,\beta)$. 
\end{Def}

\begin{Rem}\label{rem: L infinite at criticality} Using the work of \cite{AizenmanBarskyFernandezSharpnessIsing1987}, together with the strategy implemented in \cite{DuminilTassionNewProofSharpness2016}, we see that for any $\alpha\in (0,1)$,
\begin{equation}
    L^{(\alpha)}(\rho,\beta_c(\rho))=\infty.
\end{equation}
Indeed, using the Simon--Lieb inequality as in \eqref{eq: chi varphi sharpness}, one can show that if $S$ is a finite subset of $\mathbb Z^d$ containing $0$ and satisfying $\varphi_{\rho,\beta_c(\rho)}(S)<1$, then
\begin{equation}
    \chi(\rho,\beta_c(\rho))\leq \frac{|S|\langle \tau_0^2\rangle_{\rho,\beta_c(\rho)}}{1-\varphi_{\rho,\beta_c(\rho)}(S)}.
\end{equation}
This is in contradiction with the infiniteness of the susceptibility at criticality. A similar argument gives that $L^{(\alpha)}(\rho,\beta)$ increases to infinity as $\beta$ tends to $\beta_c(\rho).$
\end{Rem}
Below $L(\rho,\beta)$, the two-point function can be lower bounded by an algebraically decaying function. We start by stating this result in the special case where $\mathfrak{m}_2(J)<\infty$.
\begin{Prop}[Lower bound on the two-point function]\label{prop: lower bound 2 pt function}
Let $d\geq 3$. Assume that $J$ satisfies $(\mathbf{A1})$--$(\mathbf{A5})$ and that $\mathfrak{m}_2(J)<\infty$. There exists $c=c(d,J)>0$ such that for all $\beta\leq \beta_c(\rho)$, and for all $x\in \mathbb Z^d$ satisfying $1\leq |x|\leq cL(\rho,\beta)$,
\begin{equation}
    \langle \tau_0\tau_{x}\rangle_{\rho,\beta}\geq \frac{c}{\beta|x|^{d-1}}.
\end{equation}
\end{Prop}
\begin{proof}
Let $n<L(\rho,\beta)$. By definition of $L(\rho,\beta)$, one has $\varphi_{\rho,\beta}(\Lambda_n)\geq 1/2$.
Using \eqref{eq: INFRARED BOUND WE USE} and the assumption $\mathfrak{m}_2(J)<\infty$, one has for some $C_1>0$,
\begin{equation}
    \beta\sum_{\substack{x\in \Lambda_{n/2}\\y\notin \Lambda_n}}J_{x,y}\langle \tau_0\tau_x\rangle_{\rho,\beta}\leq C_1n^2\sum_{|x|\geq n/2}J_{0,x}\leq 4C_1\sum_{|x|\geq n/2}|x|^2J_{0,x}\leq\frac{1}{4},
\end{equation}
provided that $n\geq N_0$ (where $N_0$ only depends on $J$). Hence, we now additionally assume that $L(\rho,\beta)> N_0$ . We obtained that,
\begin{equation}
    \beta\sum_{\substack{x\in \Lambda_n\setminus \Lambda_{n/2}\\y\notin \Lambda_n}}J_{x,y}\langle \tau_0\tau_x\rangle_{\Lambda_n,\rho,\beta}\geq \frac{1}{4}.
\end{equation}
Then, using \eqref{eq: MMS1}, for some $C_2,C_3>0$,
\begin{equation}
    \frac{1}{4}\leq C_2\beta n^{d-1}\langle \tau_0\tau_{(n/2)\mathbf{e}_1}\rangle_{\rho,\beta}\sum_{k=0}^n\sum_{|y|\geq k}J_{0,y}\leq C_3\beta \mathfrak{m}_2(J) n^{d-1}\langle \tau_0\tau_{(n/2)\mathbf{e}_1}\rangle_{\rho,\beta},
\end{equation}
so that for some $c_1>0$,
\begin{equation}
    \langle \tau_0\tau_{(n/2) \mathbf{e}_1}\rangle_{\rho,\beta}\geq \frac{c_1}{\beta n^{d-1}}.
\end{equation}
Hence, using this time \eqref{eq: consequence mms }, there exists $c_2>0$ such that, if $(2d)^{-1}N_0\leq k <(2d)^{-1} L(\rho,\beta)$ and $x\in \partial \Lambda_k$, 
\begin{equation}
    \langle \tau_0\tau_x\rangle_{\rho,\beta}\geq \langle \tau_0\tau_{\frac{2d|x|}{2}\mathbf{e}_1}\rangle_{\rho,\beta} \geq\frac{c_2}{\beta |x|^{d-1}}.
\end{equation}
We now handle the smaller values of $k=|x|$ by noticing that for $1\leq k\leq (2d)^{-1}N_0\wedge (2d)^{-1}L(\rho,\beta)$, the hypothesis that $\varphi_{\rho,\beta}(\Lambda_k)\geq \frac{1}{2}$, together with \eqref{eq: MMS1}, yield
\begin{equation}
    \langle\tau_0\tau_{k\mathbf{e}_1}\rangle_{\rho,\beta}\geq \frac{c_3}{\beta},
\end{equation}
for some $c_3=c_3(d,J)$. This concludes the proof.
\end{proof}
It is possible to extend this result provided we make the following assumption for $d\geq 1$: there exist $c_0,C_0,\alpha>0$ such that $\alpha \in (0,d)$ when $d\in \{1,2\}$, and
\begin{equation}\label{assumption: decay algebraic on slices}
    \frac{c_0}{|x|^{d+\alpha}}\leq J_{0,x}\leq \frac{C_0}{|x|^{d+\alpha}}, \qquad\forall x\in \mathbb Z^d\setminus\{0\}. 
\end{equation}
The restriction on the values of $\alpha$ when $d=1,2$ allows to use Proposition \ref{IR bound}. More precisely, using \eqref{eq: IRbounds with algebraic interactions}, we get that reflection positive interactions satisfying the above assumption also satisfy: there exists $C=C(d)>0$ such that for all $\beta\leq \beta_c(\rho)$, for all $x\in \mathbb Z^d\setminus\lbrace 0\rbrace$,
\begin{equation}\label{eq: consequence decay algebraic on slices}
    \langle \tau_0\tau_x\rangle_{\rho,\beta}\leq \frac{C}{\beta_c(\rho)|x|^{d-\alpha\wedge 2}(\log |x|)^{\delta_{\alpha,2}}}.
\end{equation}
The prototypical example of interactions satisfying $(\mathbf{A1})$--$(\mathbf{A5})$ and \eqref{assumption: decay algebraic on slices} is given by algebraically decaying RP interactions.

The next proposition will be useful in the study of models with $d\in\lbrace 1,2,3\rbrace$ and $d_{\textup{eff}}=4$, which do not satisfy $\mathfrak{m}_2(J)<\infty$.

\begin{Prop}\label{prop: general lower bound} Let $d\geq 1$. Assume that $J$ satisfies $(\mathbf{A1})$--$(\mathbf{A5})$ and \eqref{assumption: decay algebraic on slices}. There exists $c=c(d,J)>0$ such that, for all $\beta\leq \beta_c(\rho)$, and for all $1\leq |x|\leq c L(\rho,\beta)$,
\begin{equation}
    \langle \tau_0\tau_x\rangle_{\rho,\beta}\geq \frac{c}{\beta|x|^{d-1}}\times \left\{
    \begin{array}{ll}
        1 & \mbox{if } \alpha>1 \\
        (\log |x|)^{-1} & \mbox{if }\alpha=1 \\
        |x|^{\alpha-1} &\mbox{if }\alpha\in (0,1).
    \end{array}
\right.
\end{equation}
\end{Prop}
\begin{Rem} The lower bound matches \eqref{eq: consequence decay algebraic on slices} for $\alpha\in (0,1)$. 
\end{Rem}
\begin{proof} We proceed like in the proof of Proposition \ref{prop: lower bound 2 pt function}. Notice that if $\varepsilon>0$ is sufficiently small and $n\geq 1$,
\begin{equation}
    \beta\sum_{\substack{x\in \Lambda_{\varepsilon n}\\y\notin \Lambda_n}}J_{x,y}\langle \tau_0\tau_x\rangle_{\rho,\beta}\leq C_1 \Big(\sum_{x\in \Lambda_{\varepsilon n}}\langle \tau_0\tau_x\rangle_{\rho,\beta}\Big)\sum_{|u|\geq n/2}J_{0,u}\leq C_2\varepsilon^{\alpha\wedge 2}<\frac{1}{4},
\end{equation}
where $C_1,C_2>0$, and where we used \eqref{assumption: decay algebraic on slices} together with \eqref{eq: consequence decay algebraic on slices} on the second inequality. For such a choice of $\varepsilon$, we have that for $1\leq n\leq L(\rho,\beta)$,
\begin{equation}
    \beta \sum_{\substack{x\in \Lambda_n\setminus\Lambda_{\varepsilon n}\\y\notin \Lambda_n}}J_{x,y}\langle \tau_0\tau_x\rangle_{\rho,\beta}\geq \frac{1}{4}.
\end{equation}
Using \eqref{eq: MMS1}, we find that
\begin{equation}
    \frac{1}{4\beta}\leq C_3\langle \sigma_0\sigma_{\varepsilon n\mathbf{e}_1}\rangle_\beta n^{d-1}\sum_{k=1}^n\sum_{|u|\geq k}J_{0,u}\leq C_4 \langle \sigma_0\sigma_{\varepsilon n\mathbf{e}_1}\rangle_\beta n^{d-1}\sum_{k=1}^n \frac{1}{k^\alpha},
\end{equation}
from which we obtain the desired result.
\end{proof}
\subsection{Existence of regular scales}
In this subsection, we introduce the notion of \textit{regular} scales. These scales will be defined in such a way that, on them, the two-point function behaves ``nicely'' i.e.\ as if we knew that $\langle \tau_0\tau_x\rangle_{\rho,\beta}$ decayed algebraically fast (for $1\leq |x|\leq L(\rho,\beta)$).
\begin{Def}[Regular scales] Fix $c,C>0$. An annular region \textup{Ann}$(n/2,8n)$ is said to be $(c,C)$-regular if the following properties hold :
\begin{enumerate}
    \item[$(\mathbf{P1})$] For every $x,y\in \textup{Ann}(n/2,8n)$, $S_{\rho,\beta}(y)\leq CS_{\rho,\beta}(x)$,
    \item[$(\mathbf{P2})$] for every $x,y\in \textup{Ann}(n/2,8n)$, $|S_{\rho,\beta}(x)-S_{\rho,\beta}(y)|\leq \frac{C|x-y|}{|x|}S_{\rho,\beta}(x)$,
    \item[$(\mathbf{P3})$] $\chi_{2n}(\rho,\beta)\geq (1+c)\chi_n(\rho,\beta)$,
    \item[$(\mathbf{P4})$] for every $x\in \Lambda_n$ and $y\notin \Lambda_{Cn}$, $S_{\rho,\beta}(y)\leq \frac{1}{2}S_{\rho,\beta}(x)$.
\end{enumerate}
A scale $k$ is said to be \emph{regular} if $n=2^k$ is such that $\textup{Ann}(n/2,8n)$ is $(c,C)$-regular, a vertex $x \in \mathbb Z^d$ will be said to be \emph{in a regular scale} if it belongs to an annulus $\textup{Ann}(n,2n)$ with $n=2^k$ and $k$ a regular scale.

\end{Def}
We can now state the main result of this subsection. 
\begin{Prop}[Existence of regular scales]\label{prop: existence regular scales}
Let $d\geq 3$. Let $J$ satisfy $\mathbf{(A1)}$--$\mathbf{(A5)}$ and $\mathfrak{m}_2(J)<\infty$. Let $\gamma>2$. There exist $c,c_0,c_1,C_0>0$ such that for every $\rho$ in the GS class, every $\beta\leq\beta_c(\rho)$, and every $1\leq n^\gamma\leq N\leq cL(\rho,\beta)$, there are at least $c_1\log_2\left(\frac{N}{n}\right)$ $(c_0,C_0)$-regular scales between $n$ and $N$.
\end{Prop}
The proof of this result can be found in \cite{AizenmanDuminilTriviality2021}. However, since it is a crucial tool for what follows, we decide to include it.
\begin{proof} Using the lower bound of Proposition \ref{prop: lower bound 2 pt function}, together with \eqref{eq: INFRARED BOUND WE USE}, we get the existence of $c_1,c_2>0$ such that
\begin{equation}
    \chi_N(\rho,\beta)\geq \frac{c_1}{\beta_c(\rho)}N\geq \frac{c_1}{\beta_c(\rho)}\left(\frac{N}{n}\right)^{\frac{\gamma-2}{\gamma-1}}n^2\geq c_2\left(\frac{N}{n}\right)^{\frac{\gamma-2}{\gamma-1}} \chi_n(\rho,\beta).
\end{equation}
Using Theorem \ref{sliding scale ir bound}, we find $r,c_3>0$ and independent of $n,N$, such that there are at least $c_3\log_2(N/n)$ scales $m=2^k$ between $n$ and $N$ such that
\begin{equation}\label{eq: proof reg scales 1}
    \chi_{rm}(\rho,\beta)\geq \chi_{16dm}(\rho,\beta)+\chi_m(\rho,\beta).
\end{equation}
We prove that such an $m$ is a $(c_0,C_0)$-regular scale for a good choice of $c_0,C_0$. Indeed, to show it satisfies $(\mathbf{P1})$ it is enough\footnote{This comes from the fact that any $x\in \textup{Ann}(m/2,8m)$ satisfies $$ S_{\rho,\beta}(16dm\mathbf{e}_1)\leq S_{\rho,\beta}(x)\leq S_{\rho,\beta}\left(\frac{m\mathbf{e}_1}{2}\right).$$} to show that $S_{\rho,\beta}(\frac{1}{2}m\mathbf{e}_1)\leq C_4 S_{\rho,\beta}(16dm\mathbf{e}_1)$ for some constant $C_4=C_4(d)>0$. However, one has
\begin{multline}\label{eq: proof reg scales 2}
    |\textup{Ann}(16dm,rm)|S_{\rho,\beta}(16dm\mathbf{e}_1)\geq \chi_{rm}(\rho,\beta)-\chi_{16dm}(\rho,\beta)\\\geq \chi_m(\rho,\beta)\geq |\Lambda_{m/(2d)}|S_{\rho,\beta}(m\mathbf{e}_1/2)
\end{multline}
where in the first inequality we used \eqref{eq: MMS1} to get that for all $x\in \textup{Ann}(16dm,rm)$ one has 
$S_{\rho,\beta}(x)\leq S_{\rho,\beta}(16dm\mathbf{e}_1)$, in the second inequality we used \eqref{eq: proof reg scales 1}, and 
in the third one we used \eqref{eq: MMS1} again to argue that for all $x\in \Lambda_{m/(2d)}$ one has $S_{\rho,\beta}
(x)\geq S_{\rho,\beta}(\frac{1}{2}m\mathbf{e}_1)$. This gives $(\mathbf{P1})$. Note that $(\mathbf{P2})$ follows from the 
remark below Proposition \ref{prop: gradient estimate}. Now, using again \eqref{eq: proof reg scales 2} and \eqref{eq: consequence mms }, we get that for every 
$x\in \textup{Ann}(m,2m)$ one has 
\begin{equation}\label{eq: proof reg scales 3}
S_{\rho,\beta}(x)\geq S_{\rho,\beta}(16dm\mathbf{e}_1)\geq \frac{c_5}{m^d}\chi_m(\rho,\beta), 
\end{equation}
which implies $(\mathbf{P3})$. Finally, we obtain\footnote{This is the only place where the hypothesis $d\geq 3$ plays a role.} 
$(\mathbf{P4})$ by observing that for every $R$, if $y\notin \Lambda_{dRm}$ and $x\in \Lambda_m$,
\begin{equation}\label{eq: proof reg scales 4}
    |\Lambda_{Rm}|S_{\rho,\beta}(y)\leq \chi_{Rm}(\rho,\beta)\leq C_6 R^2\chi_m(\rho,\beta)\leq C_7 R^2m^dS_{\rho,\beta}(x), 
\end{equation}
where we used \eqref{eq: MMS1} in the first inequality, Theorem \ref{sliding scale ir bound} in the second, and \eqref{eq: proof reg scales 3} in the last one. We obtain the result by choosing $C_0$ sufficiently large, and $c_0$ sufficiently small.
\end{proof}
Using Proposition \ref{prop: general lower bound}, we can extend the above result to interactions $J$ satisfying \eqref{assumption: decay algebraic on slices}.
\begin{Prop}\label{existence regular scales bis}
Let $d\geq 1$. Let $J$ satisfy $\mathbf{(A1)}$--$\mathbf{(A5)}$ and \eqref{assumption: decay algebraic on slices} with $\alpha>0$ if $d\geq 3$ and $\alpha\in (0,1)$ if $d\in \lbrace 1,2\rbrace$. Let $\gamma>2$. There exist $c,c_0,c_1,C_0>0$ such that for every $\rho$ in the GS class, every $\beta\leq\beta_c(\rho)$, and every $1\leq n^\gamma\leq N\leq cL(\rho,\beta)$, there are at least $c_1\log_2\left(\frac{N}{n}\right)$ $(c_0,C_0)$-regular scales between $n$ and $N$.
\end{Prop}
\begin{proof} We only need to take care of the case $d\in\lbrace 1,2\rbrace$ and $\alpha\in (0,1)$. As noticed above, in this case $S_{\rho,\beta}(x)\asymp |x|^{d-\alpha}$ below $L(\rho,\beta)$. The existence of regular scales in that case is then a direct consequence of Remark \ref{rem: below gradient}.
\end{proof}

\section{Random current representation}\label{section rcr}
Let $d\geq 1$. Let $J$ be an interaction on $\mathbb Z^d$ satisfying (\textbf{A1})--(\textbf{A5}) and let $\Lambda$ be a finite subset of $\mathbb Z^d$.
\subsection{Definitions and the switching lemma}
\begin{Def} A \textit{current} $\n$ on $\Lambda$ is a function defined on the set $\mathcal{P}_2(\Lambda):=\lbrace \lbrace x,y\rbrace, \: x,y \in \Lambda \rbrace$ and taking its values in $\mathbb N=\lbrace 0,1,\ldots\rbrace$. We denote by $\Omega_\Lambda$ the set of currents on $\Lambda$. The set of \emph{sources} of $\n$, denoted by $\sn$, is defined as
\begin{equation}
\sn:=\left\lbrace x \in \Lambda \: : \: \sum_{y\in \Lambda}\n_{x,y}\emph{ is odd}\right\rbrace.
\end{equation}
We also set $w_\beta(\n):=\prod_{\lbrace x,y\rbrace\subset \Lambda}\dfrac{(\beta J_{x,y})^{\n_{x,y}}}{\n_{x,y}!}$.
\end{Def}
There is a way to expand the correlation functions of the Ising model to relate them to currents. Indeed, if we use, for $\sigma\in \lbrace \pm 1\rbrace ^\Lambda$, the expansion 
\begin{equation}
    \exp(\beta J_{x,y}\sigma_x\sigma_y)=\sum_{\n_{x,y}\geq 0}\dfrac{(\beta J_{x,y}\sigma_x\sigma_y)^{\n_{x,y}}}{\n_{x,y}!},
\end{equation}
we obtain that 
\begin{equation}
    Z(\Lambda,\beta)=2^{|\Lambda|}\sum_{\sn=\emptyset}w_\beta(\n).
\end{equation}
More generally, the correlation functions are given by: for
$A\subset \Lambda$, 

\begin{equation}\label{equation correlation rcr}
\left\langle \sigma_A\right\rangle_{\Lambda,\beta}=\dfrac{\sum_{\sn=A}w_\beta(\n)}{\sum_{\sn=\emptyset}w_\beta(\n)},
\end{equation}
where $\sigma_A:=\prod_{x\in A}\sigma_x$. 

A current configuration $\n$ with $\sn=\emptyset$ can be seen as the edge count of a multigraph obtained as a union of loops. Adding sources to a current configuration comes down to adding a collection of paths connecting pairwise the sources. For instance, a current configuration with sources $\sn=\lbrace x,y\rbrace$ can be seen as the edge count of a multigraph consisting of a family of loops together with a path from $x$ to $y$. 
As we are about to see, connectivity properties of the multigraph induced by a current will play a crucial role in the analysis of the underlying Ising model, this motivates the following definition.
\begin{Def} Let $\n\in \Omega_{\Lambda}$ and $x,y\in \Lambda$.
\begin{enumerate}
    \item[$(i)$] We say that $x$ is connected to $y$ in $\n$ and write $x\connect{\n\:}y$, if there exists a sequence of points $x_0=x,x_1,\ldots, x_m=y$ such that $\n_{x_i,x_{i+1}}>0$ for $0\leq i \leq m-1$.
    \item[$(ii)$] The cluster of $x$, denoted by $\mathbf{C}_\n(x)$, is the set of points connected to $x$ in $\n$.
\end{enumerate}
\end{Def}
The main interest of the above expansion lies in the following result that allows one to switch the sources of two currents. This combinatorial result first appeared in \cite{GriffithsHurstShermanConcavity1970} to prove the concavity of the magnetisation of an Ising model with positive external field, but the probabilistic picture attached to it was popularised in \cite{AizenmanGeometricAnalysis1982}.
\begin{Lem}[Switching lemma]\label{switching lemma}
For any $A,B\subset \Lambda$ and any function $F$ from the set of currents into $\mathbb R$,
\begin{multline}\tag{\textbf{SL}}\label{eq: switching lemma}
\sum_{\substack{\n_1\in \Omega_\Lambda : \:\sn_1=A\\ \n_2\in \Omega_\Lambda :\: \sn_2=B}}F(\n_1+\n_2)w_\beta(\n_1)w_\beta(\n_2)\\=\sum_{\substack{\n_1\in \Omega_\Lambda: \:\sn_1=A\Delta B\\ \n_2\in \Omega_\Lambda :\: \sn_2=\emptyset}}F(\n_1+\n_2)w_\beta(\n_1)w_\beta(\n_2)\mathds{1}_{(\n_1+\n_2)\in \mathcal{F}_B},
\end{multline}
where $A\Delta B=(A\cup B)\setminus (A\cap B)$ is the symmetric difference of sets and $\mathcal{F}_B$ is given by
\begin{equation}
\mathcal{F}_B=\lbrace \n \in \Omega_\Lambda \: :\: \exists \mathbf{m}\leq \n \:, \: \partial \mathbf{m}=B\rbrace.
\end{equation}
\end{Lem}
We will also need a slightly different version of the switching lemma, called the \textit{switching principle}, whose proof (which is almost identical to the proof of the switching lemma) can be found in \cite[Lemma~2.1]{AizenmanDuminilTassionWarzelEmergentPlanarity2019}. We use the representation of a current $\n\in \Omega_\Lambda$ into a multigraph $\mathcal{N}$ in which the vertex set is $\Lambda$ and where there are exactly $\n_{x,y}$ edges between $x$ and $y$. We will also use the notation $\partial \mathcal{N}=\sn$.
\begin{Lem}[Switching principle]\label{switching principle}
For any multigraph $\mathcal{M}$ with vertex set $\Lambda$, any $A\subset \Lambda$, and any function $f$ of a current,
\begin{equation}\tag{\textbf{SP}}\label{eq: switching principle}
    \sum_{\substack{\mathcal{N}\subset \mathcal{M}\\\partial \mathcal{N}=A}}f(\mathcal{N})=\mathds{1}_{\exists \mathcal{K}\subset \mathcal{M},\: \partial\mathcal{K}=A} \sum_{\substack{\mathcal{N}\subset \mathcal{M}\\\partial \mathcal{N}=\emptyset}}f(\mathcal{N}\Delta \mathcal{K}).
\end{equation}
\end{Lem}
The switching lemma provides probabilistic interpretations of several quantities of interest like differences or ratios of correlation functions. The natural probability measures are defined as follows. If $A\subset \Lambda$, define a probability measure $\mathbf{P}_{\Lambda,\beta}^A$ on $\Omega_\Lambda$ by: for $\n\in \Omega_\Lambda$,
\begin{equation}
    \mathbf{P}_{\Lambda,\beta}^{A}[\n]:=\mathds{1}_{\sn=A}\frac{w_\beta(\n)}{\sum_{\sm=A}w_\beta(\m)},
\end{equation}
and for $A_1,\ldots,A_k\subset \Lambda$, define
\begin{equation}
\mathbf{P}^{A_1,\ldots,A_k}_{\Lambda,\beta}:=\mathbf{P}_{\Lambda,\beta}^{A_1}\otimes\ldots\otimes \mathbf{P}_{\Lambda,\beta}^{A_k}.
\end{equation}
When $A=\lbrace x,y\rbrace$, we will write $xy$ instead of $\lbrace x,y\rbrace$ in the above notation.

The first consequence of the switching lemma is the following expression of a ratio of correlation functions in terms of the probability of the occurrence of a certain connectivity event in a system of random currents. One has,
\begin{equation}\label{first application of switching lemma}
    \frac{\langle \sigma_A\rangle_{\Lambda,\beta}\langle\sigma_B\rangle_{\Lambda,\beta}}{\langle \sigma_A\sigma_B\rangle_{\Lambda,\beta}}=\mathbf{P}_{\Lambda,\beta}^{A\Delta B,\emptyset}\left[\n_1+\n_2\in \mathcal{F}_B\right],
\end{equation}
where $\mathcal{F}_B$ was defined in Lemma \ref{switching lemma}. In particular, for $0,x,u\in \Lambda$,
\begin{equation}\label{second application of switching lemma}
    \frac{\langle \sigma_0\sigma_u\rangle_{\Lambda,\beta} \langle \sigma_u\sigma_x\rangle_{\Lambda,\beta}}{\langle \sigma_0\sigma_x\rangle_{\Lambda,\beta}}=\mathbf{P}_{\Lambda,\beta}^{0x,\emptyset}[0\connect{\n_1+\n_2\:}u].
\end{equation}
Later, it will also be interesting to control two-point connectivity probabilities. One can prove (see \cite[Proposition~A.3]{AizenmanDuminilTriviality2021}) the following result: for every $x,u,v\in \Lambda$,
\begin{multline}\label{multi connectivity inequality}
    \mathbf{P}_{\Lambda,\beta}^{0x,\emptyset}[u,v\connect{\n_1+\n_2\:} 0]\leq \frac{\langle \sigma_0\sigma_u\rangle_{\Lambda,\beta} \langle \sigma_u\sigma_v\rangle_{\Lambda,\beta} \langle \sigma_v\sigma_x\rangle_{\Lambda,\beta}}{\langle \sigma_0\sigma_x\rangle_{\Lambda,\beta}}+\frac{\langle \sigma_0\sigma_v\rangle_{\Lambda,\beta} \langle \sigma_v\sigma_u\rangle_{\Lambda,\beta} \langle \sigma_u\sigma_x\rangle_{\Lambda,\beta}}{\langle \sigma_0\sigma_x\rangle_{\Lambda,\beta}}.
\end{multline}
As proved in \cite{AizenmanDuminilSidoraviciusContinuityIsing2015}, the probability measure $\mathbf{P}_{\Lambda,\beta}^A$, for $A$ a finite (even) subset of $\mathbb Z^d$, admits a weak limit as $\Lambda\nearrow \mathbb Z^d$ that we denote by $\mathbf{P}_\beta^A$. This yields infinite volume versions of the above results. 

\paragraph{The backbone representation of the Ising model.} We end this subsection with the introduction of another representation of the Ising model--- which closely related to the random current representation--- called the \emph{backbone representation}.

 This representation was first introduced in \cite{AizenmanGeometricAnalysis1982}, and later used to capture fine properties of the Ising model \cite{AizenmanFernandezCriticalBehaviorMagnetization1986,AizenmanBarskyFernandezSharpnessIsing1987,DuminilTassionNewProofSharpness2016,AizenmanDuminilTassionWarzelEmergentPlanarity2019,AizenmanDuminilTriviality2021}. We refer to these papers and to \cite[Chapter~2]{PanThesis} for more details about this object. We fix a finite subset $\Lambda$ of $\mathbb Z^d$ and fix any ordering of $(\lbrace u,v\rbrace)_{u,v\in \Lambda}$ that we denote $\prec$.
\begin{Def}[Backbone exploration] Let $\n\in \Omega_\Lambda$. Assume that $\sn=\lbrace x,y\rbrace$. The backbone of $\n$, denoted $\Gamma(\n)$, is the unique oriented and edge self-avoiding path from $x$ to $y$ supported on pairs $\lbrace u,v\rbrace$ with $\n_{u,v}$ odd which is minimal for $\prec$. 

The backbone $\Gamma(\n)$ can be obtained via the following exploration process: 
\begin{enumerate}
    \item[$(1)$] Let $x_0=x$. The first edge $\lbrace x,x_1\rbrace$ of $\Gamma(\n)$ is the earliest one of all the edges emerging from $x$ with $\n_{x,x_1}$ odd.
    \item[$(2)$] Each edge $\lbrace x_i,x_{i+1}\rbrace$ is the earliest of all edges emerging from $x_i$ that have not been cancelled previously, and for which the flux number is odd.
    \item[$(3)$] The path stops when it reaches a site from which there are no more non-cancelled edges with odd flux number available. This always happen at a source of $\n$ (in that case $y$).
\end{enumerate}
We let $\overline{\Gamma(\n)}$ be the set of explored edges (this set is made of the $\lbrace x_i,x_{i+1}\rbrace$ together with all cancelled edges).

A path $\gamma : x \rightarrow y$ (viewed as a sequence of steps) is said to be \emph{consistent} if no step of the sequence uses an edge cancelled by a previous step.
\end{Def}
One can write
\begin{equation}\label{eq: backbone exp}
    \langle \sigma_x\sigma_y\rangle_{\Lambda,\beta}=\sum_{\gamma :x \rightarrow y \textup{ consistent}}\rho_\Lambda(\gamma),
\end{equation}
where for a consistent path $\gamma:x \rightarrow y$,
\begin{equation}
    \rho_\Lambda(\gamma):=\frac{\sum_{\sn=\partial\gamma}w_\beta(\n)\mathds{1}_{\Gamma(\n)=\gamma}}{\sum_{\sn=\emptyset}w_\beta(\n)}.
\end{equation}
The backbone representation has the following useful properties (see \cite[Chapter~2]{PanThesis} for the proofs):
\begin{enumerate}
	\item If $\gamma$ is a consistent path and $E$ is a subset of edges of $\Lambda$ such that $\overline{\gamma}\cap E^c=\emptyset$, then
	\begin{equation}\label{eq: prop backbone mono}
	\rho_\Lambda(\gamma)\leq \rho_E(\gamma).
	\end{equation}
	\item  If a consistent path $\gamma$ is the concatenation of $\gamma_1$ and $\gamma_2$ (which we denote by $\gamma=\gamma_1\circ \gamma_2$),
\begin{equation}\label{eq: prop backbone}
    \rho_{\Lambda}(\gamma)=\rho_{\Lambda}(\gamma_1)\rho_{\Lambda\setminus \overline{\gamma_1}}(\gamma_2).
\end{equation}
\end{enumerate}
This last property has the following useful consequence. 
\begin{Prop}[Chain rule for the backbone]\label{prop: chain rule for the backbone}  Let $x,y,u,v\in \Lambda$. Then,
\begin{equation}
    \mathbf{P}^{xy}_{\Lambda,\beta}[\Gamma(\n) \textup{ passes through }u \textup{ first and then through }v]\leq \frac{\langle \sigma_x\sigma_u\rangle_{\Lambda,\beta}\langle \sigma_u\sigma_v\rangle_{\Lambda,\beta}\langle \sigma_v\sigma_y\rangle_{\Lambda,\beta}}{\langle \sigma_x\sigma_y\rangle_{\Lambda,\beta}}.
\end{equation}
\end{Prop}

\subsection{Ursell's four-point function}
Recall that Ursell's four-point function was defined in \eqref{eq: def ursell}.
Newman \cite{NewmanInequalitiesUrsellIsing1975} proved that the triviality of the scaling limits of the Ising model is equivalent to the vanishing of the scaling limit of Ursell's four-point function. This result was later quantified by Aizenman \cite[Proposition~12.1]{AizenmanGeometricAnalysis1982} who obtained the following bound.
\begin{Prop}[Deviation from Wick's law]\label{4 ursell is enough to win}
Let $d\geq 1$ and $n\geq 2$. For every $x_1,\ldots,x_{2n} \in \mathbb Z^d$, 
\begin{multline*}
    \Big|\langle \sigma_{x_1}\ldots \sigma_{x_{2n}}\rangle_\beta-\sum_{\substack{\pi \textup{ pairing of}\\ \lbrace 1,\ldots,2n\rbrace}}\prod_{j=1}^n \langle \sigma_{x_{\pi(2j-1)}},\sigma_{x_{\pi(2j)}}\rangle_\beta\Big|\\\leq \dfrac{3}{2}\sum_{1\leq i<j<k<\ell\leq 2n}|U_4^\beta(x_i,x_j,x_k,x_\ell)|\sum_{\substack{\pi \textup{ pairing of}\\ \lbrace 1,\ldots,2n\rbrace\setminus \lbrace i,j,k,\ell\rbrace}}\prod_{j=1}^{n-2} \langle \sigma_{x_{\pi(2j-1)}},\sigma_{x_{\pi(2j)}}\rangle_\beta.
\end{multline*}
\end{Prop}
The switching lemma provides a probabilistic interpretation of Ursell's four-point function. This is a key step in the proof of triviality in \cite{AizenmanGeometricAnalysis1982,AizenmanDuminilTriviality2021}. Although it was first stated in the case of nearest-neighbour interactions, the proof is valid on any graph and thus remains valid in the case of general interactions.
\begin{Prop}[Representation of Ursell's four-point function]\label{prop repr ursell} Let $d\geq 1$. For every $x,y,z,t\in \mathbb Z^d$,
\begin{equation}\label{representation de ursell 4}
\urs(x,y,z,t)=-2\langle \sigma_{x}\sigma_{y}\rangle_\beta \langle \sigma_{z} \sigma_{t}\rangle_\beta \mathbf P_\beta^{xy,zt}[\mathbf{C}_{\n_1+\n_2}(x)\cap \mathbf{C}_{\n_1+\n_2}(z)\neq \emptyset].
\end{equation}
\end{Prop}
This identity might seem tricky to analyse due to the lack of independence between $\mathbf{C}_{\n_1+\n_2}(x)$ and $\mathbf{C}_{\n_1+\n_2}(z)$ but it is possible to show \cite{AizenmanGeometricAnalysis1982} that
\begin{equation}
    \mathbf P_\beta^{xy,zt}[\mathbf{C}_{\n_1+\n_2}(x)\cap \mathbf{C}_{\n_1+\n_2}(z)\neq \emptyset]\leq \mathbf P_\beta^{xy,zt,\emptyset,\emptyset}[\mathbf{C}_{\n_1+\n_3}(x)\cap \mathbf{C}_{\n_2+\n_4}(z)\neq \emptyset].
\end{equation}
In particular, if $\mathcal{I}:=\mathbf{C}_{\n_1+\n_3}(x)\cap \mathbf{C}_{\n_2+\n_4}(z)$, this leads to the following bound,
\begin{equation}\label{eq: bound urs}
    |\urs(x,y,z,t)|\leq 2\langle \sigma_{x}\sigma_{y}\rangle_\beta \langle \sigma_{z} \sigma_{t}\rangle_\beta \mathbf P_\beta^{xy,zt,\emptyset,\emptyset}[|\mathcal{I}|>0].
\end{equation}
The random current representation allows us to obtain an expression of 
$U_4^\beta$ in terms of the probability of intersection of two 
independent random currents of prescribed sources. As explained in 
\cite{AizenmanGeometricAnalysis1982}, the relevant question is then to see whether, in the limit $L(x,y,z,t)\rightarrow\infty$, the ratio
$|\urs(x,y,z,t)|/\langle \sigma_x\sigma_y\sigma_z\sigma_t\rangle_\beta$ 
vanishes or not.
\paragraph{\textbf{Heuristic for triviality.}} We work at $\beta\leq\beta_c$ and assume some regularity on the two-point function in the sense that it takes comparable values for pairs of points at comparable distances smaller than $L(\beta)$. 

Let us first consider the case of the nearest-neighbour Ising model. If we expect the intersection properties of two independent random current clusters to behave essentially like the ones of two independent random walks in $\mathbb Z^d$ conditioned to start and end at $x,y$ and $z,t$ respectively, we expect the probability on the right-hand side of \eqref{eq: bound urs} to be very small in dimension $d>4$. Following this analogy, Aizenman \cite{AizenmanGeometricAnalysis1982} argued the case $d>4$ by using a first moment method on $|\mathcal{I}|$ which yields the so-called \textit{tree diagram bound},
\begin{equation}\label{eq tree diagram bound classic}
    |\urs(x,y,z,t)|\leq 2\sum_{u\in \mathbb Z^d}\langle \sigma_{x}\sigma_u\rangle_{\beta} \langle \sigma_{y}\sigma_u\rangle_{\beta} \langle \sigma_{z}\sigma_u\rangle_{\beta} \langle \sigma_{t}\sigma_u\rangle_{\beta}.
\end{equation}
As discussed in the introduction, \eqref{eq tree diagram bound classic} together with \eqref{eq: INFRARED BOUND WE USE}, imply 
\begin{equation}
    \frac{|U_4^{\beta}(x,y,z,t)|}{\langle \sigma_x\sigma_y\sigma_z\sigma_t\rangle_{\beta}}=O(L^{4-d}),
\end{equation}
where $L\leq L(\beta)$ is the mutual distance between $x,y,z$ and $t$.

In the case of the ``marginal'' dimension $d=4$ the above bound yields no interesting result and we need to go one step further in the analysis of \eqref{eq: bound urs}. Going back to the analogy with random walks, it is a well-known result that ``four'' is the critical dimension in terms of intersection for the simple random walk, meaning that two independent (simple) random walks with starting and ending points at mutual distance $L$, will intersect with probability $O(1/\log L)$ while the expected number of points in the intersection will typically be of order $\Omega(1)$ (see \cite[Chapter~10]{LawlerLimicRandomWalks2010}). This shows that when two independent (simple) random walks in dimension four intersect, they do so a logarithmic number of times. Transposing this idea in the realm of random currents suggests that the probability that two independent random currents with sources at mutual distance at least $L$ intersect, but not so many times, should decay as $O(1/(\log L)^c)$ for some $c>0$. This is indeed the result that Aizenman and Duminil-Copin obtained to improve by a logarithmic factor the tree diagram bound. 

For general long-range interactions it is possible to extend these ideas. It is well known that long-range step distributions can virtually ``increase'' the effective dimension of a random walk to the point that some low-dimensional random walks start manifesting the above properties, only observed in dimension $d\geq 4$ for the simple random walk. This observation is made more explicit by the following computation (see Section \ref{section: dim eff >4}).
\begin{equation}\label{eq: urs decay long range}
	\frac{|\urs(x,y,z,t)|}{\langle \sigma_x\sigma_y\sigma_z\sigma_t\rangle_{\beta}}=O(L^{4(d/d_{\textup{eff}})-d}).
\end{equation}
As a result, models with effective dimension strictly above four are easily shown to be trivial, this is the content of Section \ref{section: dim eff >4}. The main contribution of this chapter is to treat the case $d_{\textup{eff}}=4$, and to show that we can still improve the tree diagram bound there.

\section{Reflection positive Ising models satisfying $d_{\textup{eff}}>4$}\label{section: dim eff >4}
In this section, we study models of effective dimension $d_{\textup{eff}}>4$ and prove a more general version of Theorem \ref{thm: intro deff 4 poly}. As discussed in the introduction, choosing sufficiently slowly decaying interactions might have the effect of increasing the dimension of the model. As a result, we expect to find models in low dimensions which admit trivial scaling limits. These results had already been obtained in \cite{AizenmanGeometricAnalysis1982,AizenmanFernandezLongrange}, although not under this slightly stronger form. We begin with a definition.

\begin{Def}[Effective dimension] Let $d\geq 1$. Assume that $J$ satisfies $\mathbf{(A1)}$--$\mathbf{(A4)}$. The effective dimension $d_{\textup{eff}}$ of the model is related to the critical exponent $\eta$ of the two-point function. 
Assume that there exists $\eta\geq 0$ such that, as $|x|\rightarrow \infty$,
\begin{equation}
    \langle \sigma_0\sigma_x\rangle_{\beta_c}\asymp\frac{1}{|x|^{d-2+\eta+o(1)}}.
\end{equation}
The effective dimension $d_{\textup{eff}}$ is then given by
\begin{equation}
    d_{\textup{eff}}:=\frac{d}{1-(\eta\wedge 2) /2}.
\end{equation}
\end{Def}
\begin{Rem} The above formula can be justified using Fourier transform considerations, see \cite{AizenmanFernandezLongrange}. We saw in \eqref{eq: thermal averages equal fourier transform} that the spin-wave mode squared averages $\langle |\widehat{\tau}_\beta(p)|^2\rangle_{\mathbb T_L,\rho,\beta}$--- or \emph{thermal} averages--- coincide with the two-point function's Fourier transform. As it turns out, the relevant quantity to look at is often the density of these spin-wave modes--- i.e.\ $\mathrm{d}p$--- expressed as a function of the \emph{excitation level} which is measured by $\widehat{S}_\beta$. For the Gaussian case, one has $\widehat{S}(p)\asymp p^{-2}$ which leads to a density of levels $\mathrm{d}\widehat{S}^{-d/2}$. For the case of the Ising model, if we assume that $\widehat{S}_{\beta_c}(p)\asymp p^{-(2-\eta)}$ for some critical exponent $\eta$, we end up with a density $\mathrm{d}\widehat{S}_{\beta_c}^{-d_{\textup{eff}}/2}$ where $d_{\textup{eff}}=d/(1-\eta/2)$.
\end{Rem}
Note that for the above definition to make sense, we need the existence of the critical exponent $\eta$, which is expected to hold (but only known to exist in particular cases). However, we can still bound the effective dimension. Glimm and Jaffe \cite{GlimmJaffe1977critical} proved that for reflection positive interactions $\eta<2$, which justifies that $d_{\mathrm{eff}}<\infty$ in our setup.
For reflection positive models, \eqref{eq: INFRARED BOUND WE USE} yields that $\eta\geq 0$, so that
\begin{equation}
    d_{\textup{eff}}\geq d.
\end{equation}
\begin{Rem}[Algebraically decaying RP interactions]\label{rem: dim eff algebraically decaying int}
In the case of reflection positive models with coupling constants of algebraic decay, i.e.\ $J_{x,y}=C|x-y|_1^{-d-\alpha}$ for $\alpha,C>0$, Proposition \textup{\ref{IR bound}} implies that $\eta\geq |2-\alpha|_+$, which yields,
\begin{equation}
    d_{\textup{eff}}\geq \frac{d}{1\wedge (\alpha/2)}.
\end{equation}
As a consequence, if $d-2(\alpha\wedge 2)>0$, one has $d_{\textup{eff}}>4$.
\end{Rem}
More generally, with the definition given above, we see that $d_{\textup{eff}}>4$ if 
\begin{equation}
    d>4(1-\eta/2).
\end{equation}
In what follows, we assume that $d_{\textup{eff}}>4$, that is, there exists $\mathbf{C}>0$ and $\eta\in [0,2)$ such that for all $x\in \mathbb Z^d\setminus\lbrace 0\rbrace$,
\begin{equation}\label{eq: condition dim eff>4}
    \langle \sigma_0\sigma_x\rangle_{\beta_c}\leq \frac{\mathbf{C}}{|x|^{d-2+\eta}},
\end{equation}
where $\eta\geq 0$ is such that $d+2\eta>4$. Note that the above assumption is automatically satisfied when the interaction $J$ satisfies $(\mathbf{A1})$--$(\mathbf{A5})$ and \eqref{assumption: decay algebraic on slices} with $d-2(\alpha\wedge 2)>0$.

\begin{Thm}\label{d>1} Let $d\geq 1$. Assume that $J$ satisfies $\mathbf{(A1)}$--$\mathbf{(A5)}$ and \eqref{eq: condition dim eff>4}. There exist $C=C(\mathbf{C},d),\gamma=\gamma(d)>0$ such that for all $\beta\leq \beta_c$, $L\geq 1$, $f\in \mathcal{C}_0(\mathbb R^d)$ and $z\in \mathbb R$, 
\begin{multline*}
    \left|\left\langle \exp\left(z T_{f,L,\beta}(\sigma)\right)\right\rangle_\beta-\exp\left(\frac{z^2}{2}\langle T_{f,L,\beta}(\sigma)^2\rangle_\beta\right)\right|\\\leq\exp\left(\frac{z^2}{2}\langle T_{|f|,L,\beta}(\sigma)^2\rangle_\beta\right)\dfrac{C(\beta^{-4}\vee \beta^{-2})\Vert f\Vert_\infty^4 r_f^{\gamma}z^4}{L^{d+2\eta-4}}.
\end{multline*}
\end{Thm}
\begin{Rem} If we assume that the critical exponent $\eta$ exists, we see that $d+2\eta-4=d[(4/d_{\textup{eff}})-1]$, which is consistent with the decay of \eqref{eq: urs decay long range}.
\end{Rem}
\begin{proof} Let $f \in \mathcal{C}_0(\mathbb R^d)$, $\beta\leq \beta_c$ and $z \in \mathbb C$. Note that $r_f\geq 1$. Using Proposition \ref{4 ursell is enough to win} one gets for $n\geq 2$ (the inequality being trivial for $n=0,1$), 
\begin{multline}\label{equ triv d=2,3}
    \left|\left\langle T_{f,L,\beta}(\sigma)^{2n}\right\rangle_\beta-\dfrac{(2n)!}{2^n n!}\left\langle T_{f,L,\beta}(\sigma)^2\right\rangle_\beta^n\right|\\\leq \dfrac{3}{2}(2n)^4\left\langle T_{|f|,L,\beta}(\sigma)^{2n-4}\right\rangle_\beta\Vert f\Vert_\infty^4 S(\beta,L,f),
\end{multline}
where
\begin{equation}
    S(\beta,L,f):=\sum_{x_1,x_2,x_3,x_4\in \Lambda_{r_f L}}\dfrac{\left|U_4^\beta(x_1,x_2,x_3,x_4)\right|}{\Sigma_L(\beta)^2}.
\end{equation}
Multiplying by $\frac{z^{2n}}{(2n)!}$ and summing (\ref{equ triv d=2,3}) over $n$, one gets,
\begin{multline*}
    \left|\left\langle \exp\left(z T_{f,L,\beta}(\sigma)\right)\right\rangle_\beta-\exp\left(\frac{z^2}{2}\langle T_{f,L,\beta}(\sigma)^2\rangle_\beta\right)\right|\\\leq C_1 z^4\exp\left(\frac{z^2}{2}\langle T_{|f|,L,\beta}(\sigma)^2\rangle_\beta\right)\Vert f\Vert_\infty^4S(\beta,L,f).
\end{multline*}
Applying the tree diagram bound \eqref{eq tree diagram bound classic}, we obtain
\begin{equation}
    S(\beta,L,f)\leq 2\sum_{\substack{x\in \mathbb Z^d \\ x_1,x_2,x_3,x_4\in \Lambda_{r_f L}}}\dfrac{\langle \sigma_x\sigma_{x_1}\rangle_\beta \langle \sigma_x\sigma_{x_2}\rangle_\beta \langle \sigma_x\sigma_{x_3}\rangle_\beta \langle \sigma_x\sigma_{x_4}\rangle_\beta}{\Sigma_L(\beta)^2}.
\end{equation}
Splitting the sum above,
\begin{equation}
    S(\beta,L,f)/2\leq \underbrace{\sum_{\substack{x \in \Lambda_{dr_{f}L} \\ x_1,x_2,x_3,x_4\in \Lambda_{r_{f}L}}}(\ldots)}_{(1)}+\underbrace{\sum_{\substack{x \notin \Lambda_{dr_{f}L} \\ x_1,x_2,x_3,x_4\in \Lambda_{r_{f}L}}}(\ldots)}_{(2)}.
\end{equation}
\paragraph{Bound on (1).} The first term can be written
\begin{equation}
    \sum_{\substack{x \in \Lambda_{dr_{f}L} \\ x_1,x_2,x_3,x_4\in \Lambda_{r_{f}L}}}\dfrac{\langle \sigma_x\sigma_{x_1}\rangle_\beta \langle \sigma_x\sigma_{x_2}\rangle_\beta \langle \sigma_x\sigma_{x_3}\rangle_\beta \langle \sigma_x\sigma_{x_4}\rangle_\beta}{\Sigma_L(\beta)^2}=\sum_{x\in \Lambda_{dr_f L}}\dfrac{\left(\sum_{y\in \Lambda_{r_f L}}\langle \sigma_x\sigma_y\rangle_\beta\right)^4}{\Sigma_L(\beta)^2}.
\end{equation}
Noticing that for $x \in \Lambda_{dr_f L}$, $\sum_{y\in \Lambda_{r_f L}}\langle \sigma_x\sigma_y\rangle_\beta\leq \chi_{2dr_f L}(\beta)$, using Theorem \ref{sliding scale ir bound} to bound $\chi_{2dr_f L}(\beta)$ in terms of $\chi_L(\beta)$, and using that $\chi_L(\beta)\leq C_2 L^{-d}\Sigma_L(\beta)$, we get
\begin{equation}
    (1)\leq C_3\beta^{-4}r_f^{8+d}\dfrac{\chi_L(\beta)^2}{L^d}.
\end{equation}
We can then use \eqref{eq: condition dim eff>4} to get the bound  $\chi_L(\beta)\leq C_4 L^{2-\eta}$, so that
\begin{equation}
    (1)\leq C_4 \beta^{-4}r_f^{8+d}L^{-(d+2\eta-4)}.
\end{equation}
\paragraph{Bound on (2).} Combining \eqref{eq: consequence mms } and the sliding-scale infrared bound of Theorem \ref{sliding scale ir bound}, we get that for $i=1,\ldots,4$,
\begin{equation}\label{bound (2)}
    \langle\sigma_{x}\sigma_{x_i}\rangle_\beta \leq \dfrac{C_{5}}{| x|^d}\chi_{|x|/d}(\beta)\leq \dfrac{C_{6}}{\beta L^2| x|^{d-2}}\chi_L(\beta).
\end{equation}
Bounding the terms indexed by $x_1$ and $x_2$ in the sum using \eqref{bound (2)} and the other two using \eqref{eq: condition dim eff>4}, we get 
\begin{eqnarray*}
    (2)&\leq & C_{7}\beta^{-2}\sum_{x\notin \Lambda_{dr_f L}}\sum_{x_1,\ldots,x_4\in \Lambda_{r_f L}}\dfrac{\chi_L(\beta)^2}{\Sigma_L(\beta)^2| x |^{2d-4} L^4}\dfrac{1}{| x-x_3|^{d-2+\eta}}\dfrac{1}{| x-x_4|^{d-2+\eta}}
    \\ &\leq &
     C_{8}\beta^{-2}r_f^{2d+4}L^{2d}\dfrac{\chi_L(\beta)^2}{\Sigma_L(\beta)^2}\sum_{x\notin \Lambda_{dr_f L}}\dfrac{1}{| x|^{2d-4+2\eta}}
     \\ &\leq & 
     C_{9}\beta^{-2}r_f^{8+d-2\eta}L^{-(d+2\eta-4)},
\end{eqnarray*}
where we used the inequality $| x-x_i|\geq (d-1)r_f L$ in the second line, and again $\chi_L(\beta)\leq C_2 L^{-d}\Sigma_L(\beta)$ in the third line.
\end{proof}
\begin{Rem}\label{rem: extension to gs of deff>4}
It is possible--- see \cite{AizenmanGeometricAnalysis1982} or Section \textup{\ref{section: gs class}}--- to obtain a version of the tree diagram bound for models in the GS class. In the general setup, it rewrites: for all $\beta>0$, for all $x,y,z,t\in \mathbb Z^d$,
\begin{equation}\label{eq: tree diagram pour gs general}
    |U^{\rho,\beta}_{4}(x,y,z,t)|\leq 2\sum_{u,u',u''\in \mathbb Z^d}\langle \tau_x\tau_u\rangle_{\rho,\beta}\beta J_{u,u'}\langle \tau_{u'}\tau_y\rangle_{\rho,\beta}\langle \tau_z\tau_u\rangle_{\rho,\beta}\beta J_{u,u''}\langle \tau_{u''}\tau_t\rangle_{\rho,\beta}.
\end{equation}
Using the results of Section \textup{\ref{section: reflection positivity}}, together with \eqref{eq: tree diagram pour gs general}, we can easily extend the result of this section to the case of models in the GS class for which the two-point function $\langle \tau_0\tau_x\rangle_{\rho,\beta_c(\rho)}$ satisfies condition \eqref{eq: condition dim eff>4}.
\end{Rem}

\section{Reflection positive Ising models in dimension $d=4$}\label{section d=4}
The goal of this section is to prove Theorems  \ref{thm: main m2(J) infinite} and \ref{improved diagram bound}, and Corollary \ref{thm: main}. The proof of the first result is a direct consequence of the methods developed in the preceding section.

\begin{proof}[Proof of Theorem \textup{\ref{thm: main m2(J) infinite}}] Let $d=4$. Assume that the interaction $J$ satisfies $\mathbf{(A1)}$--$\mathbf{(A5)}$ and $\mathfrak{m}_2(J)=\infty$. Let $\beta\leq \beta_c.$
Using Corollary \ref{cor: infinite second moment improvement ir bound} and adapting the proof of Section \ref{section: dim eff >4} to the case $\langle \sigma_0\sigma_x\rangle_{\beta_c}=o(|x|^{2-d})$, we get (using the above notations) that $S(\beta,L,f)=o(1)$ as $L\rightarrow \infty$. The explicit rate of convergence to $0$ is obtained by estimating (see Proposition \ref{IR bound})
\begin{equation}
    \int_{(-\pi L,\pi L]^d}\frac{e^{-\Vert p\Vert_2^2}}{1-\widehat{J}(|p|/L)}\text{d}p,
\end{equation}
as $L \rightarrow \infty$. For instance, in the case of algebraically decaying RP interactions of the form $J_{x,y}=C|x-y|_1^{-d-2}$ ($\alpha=2$), we get a decay of speed $O(1/\log L)$.
\end{proof}

We now turn to the proof of Theorem \ref{improved diagram bound}. In the rest of the section, we fix $d=4$ and an interaction $J$ satisfying $(\mathbf{A1})$--$(\mathbf{A6})$. Hence, there exist $\mathbf{C},\varepsilon>0$ such that for all $x\in \mathbb Z^d\setminus \{0\}$,
\begin{equation}\label{assumption on J}
    J_{0,x}\leq \frac{\mathbf{C}}{|x|^{d+2+\varepsilon}}.
\end{equation}
We expect this case to be more subtle than the $\mathfrak{m}_2(J)=\infty$ case since we do not get any improvement on the effective dimension or on the decay of the two-point function. We will use the argument of \cite{AizenmanDuminilTriviality2021} and adapt it to the long range setup. The proof follows essentially the same steps: we make up for the lack of precise knowledge on the behaviour of the two-point function at criticality by proving the existence of \textit{regular scales} in which the two-point function behave nicely (see Proposition \ref{prop: existence regular scales}); then, we introduce a nice local intersection event which occurs with positive probability; and finally, we obtain a \textit{mixing} statement (which is of independent interest) which allows us to argue that intersections at different scales are roughly independent events. As a result, we are able to prove that as soon as two independent currents intersect, they intersect a large number of time (see Proposition \ref{Clustering bound}). 
One difficulty occurs in the process: since the model is long-range, the clusters of the sources might make big jumps and avoid scales which may drastically reduce the probability of the intersection event. Furthermore, the infinite range interactions may be problematic to obtain the mixing statement as they create more correlation between pieces of the current at different scales. The solution to get rid of these difficulties is to prove that these scale jumps occur with sufficiently small probability: this is the main technical point of this section, and the proof will be enabled by the Assumption \eqref{assumption on J}. As a consequence, we will be able to argue that the sources' clusters have a similar geometry as the one obtained in the nearest-neighbour case. This will be enough to adapt the multi-scale analysis of \cite{AizenmanDuminilTriviality2021}.

\subsection{Proof of Theorem \ref{improved diagram bound} conditionally on the clustering bound}
We will need the following deterministic lemma which relates the number of points in a set $\mathcal{S}\subset \mathbb Z^d$ to the number of concentric annuli of the form $u+\text{Ann}(u_k,u_{k+1})$ with $u \in \mathcal{S}$ it takes to cover $\mathcal{S}$. For any (possibly finite) increasing sequence $\mathcal{U}=(u_k)_{k\geq 0}$, any $u \in \mathbb Z^d$, and any $K\geq 0$, define,
\begin{equation}
    \mathbf{M}_u(\mathcal{S};\mathcal{U},K):=\left|\lbrace 0\leq k\leq K:\: \mathcal{S}\cap [u+\text{Ann}(u_k,u_{k+1})]\neq \emptyset\rbrace\right|.
\end{equation}
\begin{Lem}[Covering Lemma, {\cite[Lemma~4.2]{AizenmanDuminilTriviality2021}}]\label{lem: covering lemma} With the above notations, for any sequence $\mathcal{U}=(u_k)_{k\geq 1}$ with $u_1\geq 1$ and $u_{k+1}\geq 2u_k$,
\begin{equation}
    |\mathcal{S}|\geq 2^{\min_{u\in \mathcal{S}}\mathbf{M}_u(\mathcal{S};\mathcal{U},K)/5}.
\end{equation}
\end{Lem}
Recall that for $k\geq 0$, $B_k(\beta)=\sum_{x\in \Lambda_k}\langle \sigma_0\sigma_x\rangle_{\beta}^2$.
Fix $D$ large enough. Define recursively a (possibly finite) sequence $\mathcal{L}$ of integers $\ell_k=\ell_k(\beta,D)$ by the formula: $\ell_0=0$ and
\begin{equation}
    \ell_{k+1}=\inf\lbrace \ell\geq \ell_k: \: B_\ell(\beta)\geq DB_{\ell_k}(\beta)\rbrace.
\end{equation}
Note that by \eqref{eq: INFRARED BOUND WE USE}, $B_L-B_\ell\leq C_0\log(L/\ell)$. From this remark and the definition of $\ell_k$ one can deduce\footnote{Indeed, the lower bound is immediate and for the upper bound, write for $k\geq 1$,
\begin{eqnarray*}
    B_{\ell_k-1}(\beta)&\leq& DB_{\ell_{k-1}}(\beta)\leq DB_{\ell_{k-1}-1}(\beta)-CD\log\left(1-\frac{1}{\ell_{k-1}}\right)\\&\leq& D^{k-1}B_{\ell_1-1}(\beta)+C\sum_{i=1}^{k-1}\frac{D^i}{\ell_{k-i}}\leq C_1D^k
\end{eqnarray*}
for $C_1$ large enough (independent of $D$ and $k$). Use \eqref{eq: INFRARED BOUND WE USE} once again to write,
\begin{equation*}
    B_{\ell_k}(\beta)\leq B_{\ell_{k}-1}+C\log 2\leq C_1D^k+C\log 2\leq C_2D^k.
\end{equation*}} that
\begin{equation}\label{eq: bound b lk}
    D^k\leq B_{\ell_k}(\beta)\leq CD^k,
\end{equation}
for every $k$ and some large constant $C$ independent of $\beta,k$, and $D$. 

Theorem \ref{improved diagram bound} will be a consequence of the following result. Recall that 
\begin{equation}
\mathcal{I}=\mathbf{C}_{\n_1+\n_3}(x)\cap \mathbf{C}_{\n_2+\n_4}(z).
\end{equation}
\begin{Prop}[Clustering bound]\label{Clustering bound}
For $D$ large enough, there exists $\delta=\delta(D)>0$ such that for all $\beta\leq \beta_c$, for all $K>3$ with $\ell_{K+1}\leq L(\beta)$, and for all $u,x,y,z,t\in \mathbb Z^4$ with mutual distance between $x,y,z,t$ larger than $2\ell_{K}$,
\begin{equation}
    \mathbf{P}_\beta^{ux,uz,uy,ut}[\mathbf{M}_u(\mathcal{I};\mathcal{L},K)<\delta K]\leq 2^{-\delta K}.
\end{equation}
\end{Prop}
Let us see why this bound implies the improved tree diagram bound.
\begin{proof}[Proof of Theorem \textup{\ref{improved diagram bound}}]
Choose $D$ large enough so that Proposition \ref{Clustering bound} holds. Fix $x,y,z,t$ at mutual distance larger than $2\ell_K$. Using Lemma \ref{lem: covering lemma} together with the switching lemma \eqref{eq: switching lemma} we get,
\begin{multline*}
    \mathbf{P}_\beta^{xy,zt,\emptyset,\emptyset}[0<|\mathcal{I}|<2^{\delta K/5}]
    \leq \sum_{u\in \mathbb Z^4}\mathbf{P}_\beta^{xy,zt,\emptyset,\emptyset}[u\in \mathcal{I}, \: \mathbf{M}_u(\mathcal{I};\mathcal{L},K)<\delta K]
    \\=\sum_{u\in \mathbb Z^4}\frac{\langle \sigma_x\sigma_u\rangle_\beta \langle \sigma_y\sigma_u\rangle_\beta \langle \sigma_z\sigma_u\rangle_\beta \langle \sigma_t\sigma_u\rangle_\beta}{\langle \sigma_x\sigma_y\rangle_\beta \langle \sigma_z\sigma_t\rangle_\beta}\mathbf{P}_\beta^{ux,uz,uy,ut}[\mathbf{M}_u(\mathcal{I};\mathcal{L},K)<\delta K],
\end{multline*}
so that, using Proposition \ref{Clustering bound},
\begin{equation}
    \mathbf{P}_\beta^{xy,zt,\emptyset,\emptyset}[0<|\mathcal{I}|<2^{\delta K/5}] 
    \leq 
    2^{-\delta K}\sum_{u\in \mathbb Z^4}\frac{\langle \sigma_x\sigma_u\rangle_\beta \langle \sigma_y\sigma_u\rangle_\beta \langle \sigma_z\sigma_u\rangle_\beta \langle \sigma_t\sigma_u\rangle_\beta}{\langle \sigma_x\sigma_y\rangle_\beta \langle \sigma_z\sigma_t\rangle_\beta}.
\end{equation}
Moreover,
\begin{equation}
     \mathbf{P}_\beta^{xy,zt,\emptyset,\emptyset}[|\mathcal{I}|\geq 2^{\delta K/5}]
     \leq
     2^{-\delta K/5}\sum_{u\in \mathbb Z^4}\frac{\langle \sigma_x\sigma_u\rangle_\beta \langle \sigma_y\sigma_u\rangle_\beta \langle \sigma_z\sigma_u\rangle_\beta \langle \sigma_t\sigma_u\rangle_\beta}{\langle \sigma_x\sigma_y\rangle_\beta \langle \sigma_z\sigma_t\rangle_\beta},
\end{equation}
which implies that 
\begin{equation*}
    |U_4^\beta(x,y,z,t)|\leq \frac{2}{2^{\delta K/5}}\sum_{u\in \mathbb Z^4}\langle \sigma_x\sigma_u\rangle_\beta \langle \sigma_y\sigma_u\rangle_\beta \langle \sigma_z\sigma_u\rangle_\beta \langle \sigma_t\sigma_u\rangle_\beta.
\end{equation*}
Now, if $L:=2\ell_K$, observe that $\ell_{K+1}\geq L$ so that by \eqref{eq: bound b lk}, $B_L(\beta)\leq B_{\ell_{K+1}}(\beta)\leq CD^{K+1}$. Hence, we may find $c>0$ sufficiently small (independent of $L$ and $\beta$), such that $K\geq c\log B_L(\beta)$. This gives the result.
\end{proof}
We now turn to the proof of Proposition \ref{Clustering bound}. We start by showing that thanks to the hypothesis $(\mathbf{A6})$, the interaction decays sufficiently fast so that the ``jumps'' made by the current are not so problematic. As a byproduct of these estimates, we are able to show that the clusters do not perform ``back and forth'' between different scales. This property was already a key step in the proof of \cite{AizenmanDuminilTriviality2021}.

\subsection{Properties of the current}\label{section: properties of the current trivia}

We will use a first moment method and argue that the expected number of long edges that have a non zero-weight under the current measure decays quickly in a certain sense. The following results rely on the existence of regular scales and thus require reflection positivity.

First, we prove a bound on the probability that an edge is open in the percolation configuration deduced from the current measure. Note that this bound is in fact valid on any graph with any interactions.
\begin{Lem}[Bound on open edge probability]\label{big edge} Let $d\geq 1$. Let $\beta>0$. For $x,y,u,v\in \mathbb Z^d$, one has
\begin{multline*}
     \mathbf{P}_\beta^{xy}[\n_{u,v}\geq 1]\leq \mathbf{P}_\beta^{xy,\emptyset}[(\n_1+\n_2)_{u,v}\geq 1]\\\leq \beta J_{u,v}\left(2\langle \sigma_u\sigma_v\rangle_\beta +\frac{\langle \sigma_x\sigma_u\rangle_\beta\langle \sigma_v\sigma_y\rangle_\beta}{\langle \sigma_x\sigma_y\rangle_\beta}+\frac{\langle \sigma_x\sigma_v\rangle_\beta\langle \sigma_u\sigma_y\rangle_\beta}{\langle \sigma_x\sigma_y\rangle_\beta}\right).
\end{multline*}
\end{Lem}
\begin{proof} The first inequality follows from noticing that $\{(\n_1)_{u,v}\geq 1\}\subset \{(\n_1+\n_2)_{u,v}\geq 1\}$. For the second inequality, we write,

\begin{equation}
    \mathbf{P}_\beta^{xy,\emptyset}[\n_{u,v}\geq 1]\leq \mathbf{P}_\beta^{xy}[\n_{u,v}\geq 1]+\mathbf{P}_\beta^{\emptyset}[\n_{u,v}\geq 1].
\end{equation}
Then, we observe that $\mathbf{P}_\beta^{\emptyset}[\n_{u,v}\geq 1]\leq \beta J_{u,v}\langle \sigma_u\sigma_v\rangle_\beta$, which leads to
\begin{multline*}
    \mathbf{P}_\beta^{xy}[\n_{u,v}\geq 1]
    \leq 
    \beta J_{u,v}\frac{\langle \sigma_x\sigma_y\sigma_u\sigma_v\rangle_\beta}{\langle \sigma_x\sigma_y\rangle_\beta}
    \\ \leq
    \beta J_{u,v}\left(\langle\sigma_u\sigma_v\rangle_\beta+\frac{\langle \sigma_x\sigma_u\rangle_\beta\langle \sigma_v\sigma_y\rangle_\beta}{\langle \sigma_x\sigma_y\rangle_\beta}+\frac{\langle \sigma_x\sigma_v\rangle_\beta\langle \sigma_u\sigma_y\rangle_\beta}{\langle \sigma_x\sigma_y\rangle_\beta}\right),
\end{multline*}
where we used Lebowitz' inequality \cite{Lebowitz1974Inequ} to get $\urs(x,y,u,v)\leq 0$ (see also Proposition \ref{prop repr ursell}).
\end{proof}
We now introduce the event we will be interested in. It is illustrated in Figure \ref{figure: jump event} below.
\begin{Def}[Jump event]\label{def: jump event} Let $1\leq k \leq m$. We define $\mathsf{Jump}(k,m)$ to be the event that there exist $u \in \Lambda_k$ and $v\notin \Lambda_{m}$ such that $\n_{u,v}\geq 1$.
\end{Def}

\begin{figure}[htb]
\begin{center}
\includegraphics[scale=1.2]{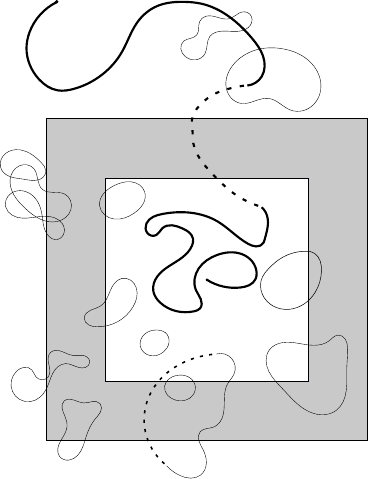}
\put(05,208){$\Lambda_{m}$}
\put(-28,173){$\Lambda_{k}$}
\put(-175,278){$y$}
\put(-213,220){$\Gamma(\n)$}
\put(-90,116){$0$}
\end{center}

 \caption{A realisation of the event $\mathsf{Jump}(k,m)$ for a current $\n$ with source set $\sn=\lbrace 0,y\rbrace$. The bold black path represents the backbone $\Gamma(\n)$. The dashed curves represent long open edges that jump over the annulus $\textup{Ann}(k,m)$.}
\label{figure: jump event}
\end{figure}

We now prove that if we consider a current with two sources, and an annulus located between them, but ``far away'' from each of them, then with high probability the current does not ``jump over it''. For convenience, we fix one of these sources to be the origin. Recall that $\varepsilon>0$ is given by \eqref{assumption on J}.
\begin{Lem}[Jumping a scale is unlikely]\label{no jump 1} Let $d=4$. Assume that $J$ satisfies $(\mathbf{A1})$--$(\mathbf{A6})$. Let $\nu\in (0,1)$ be such that $\nu>\frac{d}{d+\varepsilon}$. There exist $c,C,\eta>0$ such that for all $\beta\leq \beta_c$, for all $y\in \mathbb Z^d$ in a regular scale with $1\leq |y|\leq cL(\beta)$, and for all $k\geq 1$ such that $k^{2}\leq |y|$,
\begin{equation}
    \mathbf{P}_\beta^{0y,\emptyset}\left[\mathsf{Jump}(k,k+k^\nu)\right]\leq \frac{C}{k^\eta}.
\end{equation}
\end{Lem}
\begin{Rem} If one takes $|y|\geq \ell_{K+1}$, this lemma ensures that the current visits the annuli $\textup{Ann}(\ell_k,\ell_{k+1})$ for $1\leq k \leq K$ with high probability. This property will be crucial in the proof of Proposition \textup{\ref{Clustering bound}}.
\end{Rem}
\begin{proof}
In what follows $C=C(d)>0$ may change from line to line. It is sufficient to bound,
\begin{equation}
    \sum_{u\in \Lambda_k, \: v \notin \Lambda_{k+k^\nu}} \mathbf{P}_\beta^{0y,\emptyset}[\n_{u,v}\geq 1].
\end{equation}
Using Lemma \ref{big edge}, we have (see Figure \ref{figure: diagramme no long edge}),
\begin{multline}\label{eq: bound a1 a2 a3}
    \sum_{u\in \Lambda_k, \: v \notin \Lambda_{k+k^\nu}} \mathbf{P}_\beta^{0y,\emptyset}[\n_{u,v}\geq 1]\leq 2\beta\sum_{u\in \Lambda_k, \: v \notin \Lambda_{k+k^\nu}} J_{u,v}\Bigg(\langle \sigma_u\sigma_v\rangle_\beta 
    \\
    +
    \frac{\langle \sigma_0\sigma_u\rangle_\beta\langle \sigma_v\sigma_y\rangle_\beta}{\langle \sigma_0\sigma_y\rangle_\beta}+
    \frac{\langle \sigma_0\sigma_v\rangle_\beta\langle \sigma_u\sigma_y\rangle_\beta}{\langle \sigma_0\sigma_y\rangle_\beta}\Bigg)
    =:A_1+A_2+A_3.
\end{multline}
\paragraph{Bound on $A_1$.} One has,
\begin{eqnarray*}
    \beta\sum_{u\in \Lambda_k, \: v \notin \Lambda_{k+k^\nu}} J_{u,v}\langle \sigma_u\sigma_v\rangle_\beta&\leq& C\sum_{u\in \Lambda_k, \: v \notin \Lambda_{k+k^\nu}}\frac{1}{|u-v|^{d-2+d+2+\varepsilon}}\\&\leq& Ck^d\sum_{v\notin \Lambda_{k^\nu}}\frac{1}{|v|^{2d+\varepsilon}}\\&\leq& Ck^{d(1-\nu)-\nu\varepsilon},
\end{eqnarray*}
where we used \eqref{eq: INFRARED BOUND WE USE} and \eqref{assumption on J} on the first line.
\paragraph{Bound on $A_2$.} We split $A_2$ into two contributions: one coming from $v\in \Lambda_{|y|/2}(y)$ and the other from $v\notin (\Lambda_{k+k^\nu}\cup \Lambda_{|y|/2}(y))$. We begin with the former. By the lower bound on the two-point function of Proposition \ref{prop: lower bound 2 pt function} (together with the assumption that $1\leq |y|\leq cL(\beta)$) one has $\langle \sigma_0\sigma_y\rangle_\beta^{-1}\leq C\beta|y|^{d-1}$. Using \eqref{eq: INFRARED BOUND WE USE} and \eqref{assumption on J}, 
\begin{eqnarray*}
    \beta\sum_{\substack{u\in \Lambda_k,\: v \in \Lambda_{|y|/2}(y)}} J_{u,v}
    \frac{\langle \sigma_0\sigma_u\rangle_\beta\langle \sigma_v\sigma_y\rangle_\beta}{\langle \sigma_0\sigma_y\rangle_\beta}
    &\leq&
    C\beta^2|y|^{d-1}\sum_{\substack{u\in \Lambda_k\\ v \in \Lambda_{|y|/2}(y)}}J_{0,v-u}\langle \sigma_0\sigma_u\rangle_\beta\langle \sigma_v\sigma_y\rangle_\beta
    \\ &\leq& 
    Ck^2|y|^{d-1}|y|^{-(d+2+\varepsilon)}\sum_{v\in \Lambda_{2|y|/3}(y)}\frac{1}{(|v-y|+1)^{d-2}}
    \\ &\leq& 
    Ck^2|y|^{-(1+\varepsilon)}.
\end{eqnarray*}
Using the assumption that $|y|\geq k^{2}$, 
\begin{equation}
    \sum_{u\in \Lambda_k, \: v \in \Lambda_{|y|/2}(y)} J_{u,v}
    \frac{\langle \sigma_0\sigma_u\rangle_\beta\langle \sigma_v\sigma_y\rangle_\beta}{\langle \sigma_0\sigma_y\rangle_\beta}\leq\frac{C}{k^{2\varepsilon}}.
\end{equation}
Finally, in the case $v\notin \Lambda_{|y|/2}(y)\cup \Lambda_{k+k^\nu}$, we may use $(\mathbf{P1})$ (for $v\in \Lambda_{4|y|}(y)\setminus (\Lambda_{|y|/2}(y)\cup \Lambda_{k+k^\nu})$), as well as \eqref{eq: consequence mms } (for $v\notin \Lambda_{4|y|}(y)$), to show that $\langle \sigma_v\sigma_y\rangle_\beta\leq C_0\langle \sigma_0\sigma_y\rangle_\beta$, so that,
\begin{eqnarray*}
\beta\sum_{u\in \Lambda_k, \: v \notin \Lambda_{|y|/2}(y)\cup \Lambda_{k+k^\nu}} J_{u,v}
    \frac{\langle \sigma_0\sigma_u\rangle_\beta\langle \sigma_v\sigma_y\rangle_\beta}{\langle \sigma_0\sigma_y\rangle_\beta}
    &\leq& C_0\beta \sum_{u\in \Lambda_k, \: v \notin \Lambda_{|y|/2}(y)\cup \Lambda_{k+k^\nu}} J_{u,v}
    \langle \sigma_0\sigma_u\rangle_\beta \\&\leq&C k^{2}\sum_{v\notin \Lambda_{k^\nu}}\frac{1}{|v|^{d+2+\varepsilon}}\leq Ck^{2(1-\nu)-\nu\varepsilon}.
\end{eqnarray*}
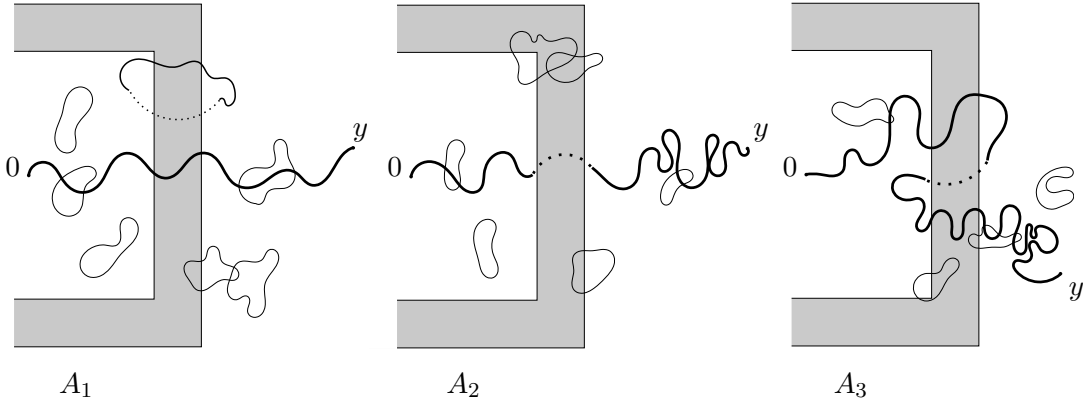
\begin{figure}[htb]
\begin{center}
\includegraphics[scale=1]{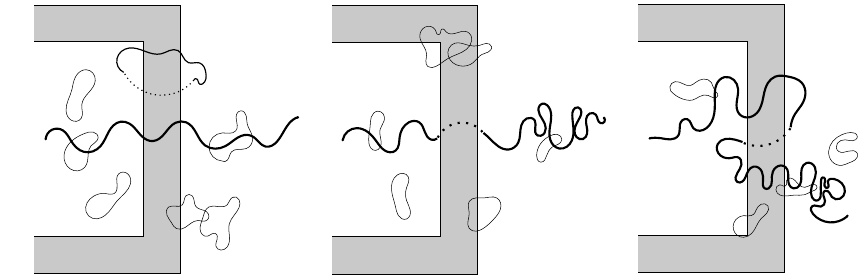}
\put(-2,22){$y$}
\put(-120,82){$y$}
\put(-270,82){$y$}
\put(-400,67){$0$}
\put(-257,67){$0$}
\put(-109,67){$0$}
\put(-380,-15){$A_1$}
\put(-235,-15){$A_2$}
\put(-90,-15){$A_3$}
\end{center}
\caption{A graphical representation of the bound \eqref{eq: bound a1 a2 a3}. We represented each potential contribution to $A_1,A_2,A_3$ (from left to right). The backbone is the bold path joining $0$ and $y$. Long open edges are the dashed curves. The largest contribution should come from $A_2$ in which long edges are induced by the backbone.}
\label{figure: diagramme no long edge}
\end{figure}
\paragraph{Bound on $A_3$.} Since $|y|\geq k^{2}$, one has for $u\in \Lambda_k$, $\langle\sigma_u\sigma_y\rangle_\beta\leq C_0\langle \sigma_0\sigma_y\rangle_\beta$ by the property $(\mathbf{P1})$ of regular scales. Using this remark, together with \eqref{eq: INFRARED BOUND WE USE} and \eqref{assumption on J}, yields,
\begin{eqnarray*}
    A_3&\leq& 2\beta C_0\sum_{u\in \Lambda_k, \: v \notin \Lambda_{k+k^\nu}} J_{u,v}\langle \sigma_0\sigma_v\rangle_\beta 
    \\&\leq& C \sum_{u\in \Lambda_k, \: v \notin \Lambda_{k+k^\nu}}\frac{1}{|v-u|^{d+2+\varepsilon}}\frac{1}{|v|^{d-2}}
    \\&\leq& Ck^{2(1-\nu)-\nu\varepsilon}.
\end{eqnarray*}
The proof follows from choosing $\eta$ sufficiently small, and $C$ large enough.
\end{proof}

As a corollary of the above result, we can show that the probability that a backbone does a \textit{zigzag} between two ``distanced'' scales is very small. This will be very useful later to argue that intersection events are (essentially) local.
\begin{Def}[Zigzag event]
For $1\leq k \leq \ell\leq M$ and $u,v \in \mathbb Z^d$, let $\mathsf{ZZ}(u,v;k,\ell,M)$ be the event that the backbone of $\n$ (with $\sn=\lbrace u,v\rbrace$) goes from $u$ to a point in $\textup{Ann}(\ell,M)$, then to a point in $\Lambda_k$, before finally hitting $v$. We let $\mathsf{ZZ}(u,v;k,\ell,\infty)$ be the union of all $\mathsf{ZZ}(u,v;k,\ell,M)$ for $M\geq \ell$.
\end{Def}
\begin{Coro}[No zigzag for the backbone]\label{coro: no jump 1} Let $d=4$. Assume that $J$ satisfies $(\mathbf{A1})$--$(\mathbf{A6})$. Let $\nu\in (0,1)$ be such that $\nu>\frac{d}{d+\varepsilon}$. There exist $C,\eta>0$ such that for all $\beta\leq \beta_c$, for all $k,\ell\geq 1$ and $y\in \mathbb Z^d$ in a regular scale with $k^{8/(1-\nu)}\leq \ell$ and $\ell^{2}\leq |y|$,
\begin{equation}
    \mathbf{P}_\beta^{0y}[\mathsf{ZZ}(0,y;k,\ell,\infty)]\leq \frac{C}{\ell^{\eta}}.
\end{equation} 
\end{Coro}
\begin{proof} Notice that,
\begin{equation}
    \mathsf{ZZ}(0,y;k,\ell,\infty) \subset \mathsf{ZZ}(0,y;k,\ell,\ell+\ell^{\nu})\cup \mathsf{Jump}(\ell,\ell+\ell^{\nu}).
\end{equation}
The chain rule for backbones (see Proposition \ref{prop: chain rule for the backbone}), the assumption of regularity made on $y$, as well as \eqref{eq: INFRARED BOUND WE USE}, yield
\begin{eqnarray*}
    \mathbf{P}_\beta^{0y}[\mathsf{ZZ}(0,y;k,\ell,\ell+\ell^{\nu})]
    &\leq& 
    \sum_{\substack{v\in \textup{Ann}(\ell,\ell+\ell^\nu)\\w\in \Lambda_k}}\frac{\langle \sigma_0\sigma_v\rangle_\beta\langle \sigma_v\sigma_w\rangle_\beta\langle\sigma_w\sigma_y\rangle_\beta}{\langle \sigma_0\sigma_y\rangle_\beta}
    \\ &\leq& 
    C\frac{k^4 \ell^{3+\nu}}{\ell^{4}}
    \leq 
    \frac{C}{\ell^{(1-\nu)/2}}.
\end{eqnarray*}
We conclude using Lemma \ref{no jump 1} and the fact that $\ell^2\leq |y|$.
\end{proof}
\begin{Rem}\label{rem: zigzag difficult for deff=4} Note that in the above result we heavily relied on the that fact the current cannot jump over an annulus of dimension strictly smaller than four. We will see that Lemma \textup{\ref{no jump 1}} does not hold anymore in the case $d_{\textup{eff}}=4$ and $1\leq d \leq 3$ which makes the study of the zigzag event more complicated.
\end{Rem}
This second technical result is a small modification of Lemma \ref{no jump 1} but will be crucial in the proof of the mixing. It is a little easier since the sources are now located close to the origin: there is no ``long backbone'' which might help to jump scales. Note that this lemma does not rely on the existence of regular scales.

\begin{Lem}\label{no jump 1 bis}
Let $d=4$. Assume that $J$ satisfies $(\mathbf{A1})$--$(\mathbf{A6})$. Let $\nu\in (0,1)$ be such that $\nu>\frac{d}{d+\varepsilon}$. There exist $c,C,\eta>0$ such that for all $\beta\leq \beta_c$, for all $n<  m\leq M\leq k$ with $1\leq M^{3/2}\leq k\leq cL(\beta)$, for all $x\in \Lambda_n$ and all $u \in \textup{Ann}(m,M)$,
\begin{equation}
    \mathbf{P}_\beta^{xu,\emptyset}[\mathsf{Jump}(k,k+k^\nu)]\leq \frac{C}{k^\eta}.
\end{equation}
\end{Lem}
\begin{proof}
We follow the same steps as in the proof of Lemma \ref{no jump 1}. Using Lemma \ref{big edge},
\begin{multline*}
    \sum_{w\in \Lambda_k, \: v \notin \Lambda_{k+k^\nu}} \mathbf{P}_\beta^{xu,\emptyset}[\n_{w,v}\geq 1]\\\leq 2\beta\sum_{w\in \Lambda_k, \: v \notin \Lambda_{k+k^\nu}} J_{w,v}\left(\langle \sigma_w\sigma_v\rangle_\beta +
    \frac{\langle \sigma_x\sigma_w\rangle_\beta\langle \sigma_v\sigma_u\rangle_\beta}{\langle \sigma_x\sigma_u\rangle_\beta}+
    \frac{\langle \sigma_x\sigma_v\rangle_\beta\langle \sigma_w\sigma_u\rangle_\beta}{\langle \sigma_x\sigma_u\rangle_\beta}\right).
\end{multline*}
As for the bound of $A_1$ above,
\begin{equation}
    \beta\sum_{w\in \Lambda_k, \: v \notin \Lambda_{k+k^\nu}} J_{w,v}\langle \sigma_w\sigma_v\rangle_\beta\leq C_1k^{d(1-\nu)-\nu\varepsilon}.
\end{equation}
Using the lower bound of Proposition \ref{prop: lower bound 2 pt function} (which is allowed because $1\leq |x-u|\leq cL(\beta)$) together with \eqref{eq: INFRARED BOUND WE USE} and \eqref{assumption on J}, we get
\begin{eqnarray*}
   \beta \sum_{w\in \Lambda_k, \: v \notin \Lambda_{k+k^\nu}} J_{w,v}\frac{\langle \sigma_x\sigma_v\rangle_\beta\langle \sigma_w\sigma_u\rangle_\beta}{\langle \sigma_x\sigma_u\rangle_\beta}
   &\leq&
   \beta^2 C_2M^{d-1}\sum_{w\in \Lambda_k, \: v \notin \Lambda_{k+k^\nu}} J_{w,v}\langle \sigma_x\sigma_v\rangle_\beta\langle \sigma_w\sigma_u\rangle_\beta
   \\&\leq& 
   C_3M^{d-1}k^2 k^{-(d-2)}\sum_{v\notin \Lambda_{k^\nu}}J_{0,v}
    \\ &\leq & 
    C_4M^{d-1}k^{-\nu(2+\varepsilon)}.
\end{eqnarray*}
Finally, with the same reasoning, we also get
\begin{equation}
    \sum_{w\in \Lambda_k, \: v \notin \Lambda_{k+k^\nu}} J_{w,v}\frac{\langle \sigma_x\sigma_w\rangle_\beta\langle \sigma_v\sigma_u\rangle_\beta}{\langle \sigma_x\sigma_u\rangle_\beta}\leq C_5M^{d-1}k^{-\nu(2+\varepsilon)}.
\end{equation}
The assumption made on $\nu$ and the inequality $M^3\leq k^{2}$ yield the result.
\end{proof}

Similarly, we can rule out the zigzag of the backbone in this setup.
\begin{Coro}\label{coro: no jump 1 bis} Let $d=4$. Assume that $J$ satisfies $(\mathbf{A1})$--$(\mathbf{A6})$. Let $\nu\in(0,1)$ be such that $\nu>\frac{d}{d+\varepsilon}$. There exist $c,C,\eta>0$ such that for all $\beta\leq \beta_c$, for all $n<  m\leq M\leq k$ with $1\leq M^{6/(1-\nu)}\leq k\leq cL(\beta)$, for all $x\in \Lambda_n$ and all $u \in \textup{Ann}(m,M)$,
\begin{equation}
    \mathbf{P}_\beta^{xu}[\mathsf{ZZ}(x,u;M,k,\infty)]\leq \frac{C}{k^\eta}.
\end{equation}
\end{Coro}
\begin{proof} The argument is exactly the same as the one used to prove Corollary \ref{coro: no jump 1} except that we replace Lemma \ref{no jump 1} by Lemma \ref{no jump 1 bis}.
%
    
\end{proof}
This final technical lemma will be useful to argue that for a current $\n$ with $\sn=\lbrace x,y\rbrace$, the restriction of $\n$ to $(\overline{\Gamma(\n)})^c$ (that we denote by $\n\setminus \overline{\Gamma(\n)}$ below), where $\Gamma(\n)$ is the backbone of $\n$, and $\overline{\Gamma(\n)}$ is the set of edges revealed during the exploration of $\Gamma(\n)$, essentially behaves like a sourceless current on a smaller graph. In that case, jump events should become even more unlikely since there is no backbone to create long connections anymore. 

If $\n$ is a current, and $E$ is a set of edges, we let $\n_E$ be the restriction of $\n$ to the edges in $E$. In particular, if $\gamma$ is a consistent path in the sense of Section \ref{section rcr}, then $\n_{\overline{\gamma}}$ is the restriction of $\n$ to the edges of $\overline{\gamma}$.

\begin{Lem}\label{no jump 2}
Let $d=4$. Assume that $J$ satisfies $(\mathbf{A1})$--$(\mathbf{A6})$. Let $\nu\in (0,1)$ be such that $\nu>\frac{d}{d+\varepsilon}$. There exist $C,\eta>0$ such that for all $\beta\leq \beta_c$, for all $k\geq 1$, for all $x,y\in \mathbb Z^d$, 
\begin{equation}
    \mathbf{P}_\beta^{xy}[\n\setminus\overline{\Gamma(\n)}\in\mathsf{Jump}(k,k+k^\nu)]\leq \mathbf{P}_\beta^{xy,\emptyset}[(\n_1+\n_2)\setminus \overline{\Gamma(\n_1)}\in \mathsf{Jump}(k,k+k^\nu)]\leq \frac{C}{k^\eta}.
\end{equation}
\end{Lem}
\begin{proof}  The first inequality follows from the observation that 
\begin{equation}
\{\n_1\setminus \overline{\Gamma(\n_1)}\in \mathsf{Jump}(k,k+k^\nu)\}\subset \{(\n_1+\n_2)\setminus \overline{\Gamma(\n_1)}\in \mathsf{Jump}(k,k+k^\nu)\}.
\end{equation}
Write $\mathcal A:=\lbrace (\n_1+\n_2)\setminus \overline{\Gamma(\n_1)}\in \mathsf{Jump}(k,k+k^\nu)\rbrace$ and for a consistent path $\gamma :x \rightarrow y$, $\mathcal{A}_\gamma:=\lbrace (\n_1+\n_2)_{\overline{\gamma}^c}\in \mathsf{Jump}(k,k+k^\nu)\rbrace$. The idea is to condition on the backbone of $\n_1$. Going to partition functions\footnote{One would need to restrict to a finite subset $\Lambda$ of $\mathbb Z^d$ first, and then take the limit $\Lambda\rightarrow \mathbb Z^d$. We omit this detail here.}, one has $\mathbf{P}_\beta^{xy,\emptyset}[\mathcal{A}]=Z_\beta^{xy,\emptyset}[\mathcal{A}]/Z^{xy,\emptyset}_\beta$ where,
\begin{eqnarray*}
    Z^{xy,\emptyset}_\beta[\mathcal A]&:=&\sum_{\gamma: x \rightarrow y \textup{ consistent}}\sum_{\substack{\sn_1=\lbrace x,y\rbrace\\\sn_2=\emptyset}}w_\beta(\n_1)w_\beta(\n_2)\mathds{1}_{\Gamma(\n_1)=\gamma}\mathds{1}_{\mathcal A}
    \\&=&
    \sum_{\gamma}\sum_{\substack{\partial(\n_1)_{\overline{\gamma}}=\lbrace x,y\rbrace\\ \partial(\n_1)_{\overline{\gamma}^c}=\emptyset\\ \sn_{2}=\emptyset}}w_\beta((\n_1)_{\overline{\gamma}})w_\beta((\n_1)_{\overline{\gamma}^c})w_\beta(\n_2)\mathds{1}_{\Gamma((\n_1)_{\overline{\gamma}})=\gamma}\mathds{1}_{\mathcal A_\gamma}
    \\&=&
    \sum_{\gamma} Z^{xy}_{\overline{\gamma},\beta} [\Gamma(\n_{\overline{\gamma}})=\gamma]Z^{\emptyset,\emptyset}_{\overline{\gamma}^c,\mathbb Z^d,\beta}[\mathcal A_\gamma]
    \\&=&\sum_{\gamma } Z^{xy,\emptyset}_\beta[\Gamma(\n_1)=\gamma]\mathbf{P}_{\overline{\gamma}^c,\mathbb Z^d,\beta}^{\emptyset,\emptyset}[\mathcal A_\gamma],
\end{eqnarray*}
where in the last equality, we used that $Z^{xy}_{\overline{\gamma},\beta}Z^\emptyset_{\overline{\gamma}^c,\beta}=Z^{xy}_{\beta}[\Gamma(\n_1)=\gamma]$.
Using Lemma \ref{big edge} as well as  Griffiths' inequality, for any $\gamma$ as above,
\begin{eqnarray*}
    \mathbf{P}_{\overline{\gamma}^c,\mathbb Z^d,\beta}^{\emptyset,\emptyset}[\mathcal A_\gamma]&\leq& \sum_{\substack{u\in \Lambda_k\\ v\notin \Lambda_{k+k^\nu}}}2\beta J_{u,v}\langle \sigma_u\sigma_v\rangle_{\beta}\leq Ck^{d(1-\nu)-\nu\varepsilon},
\end{eqnarray*}
where the last inequality was obtained in the proof of Lemma \ref{no jump 1}. Hence,
\begin{equation}
    Z^{xy,\emptyset}_\beta[\mathcal{A}]\leq  Ck^{d(1-\nu)-\nu\varepsilon}Z^{xy,\emptyset}_\beta,
\end{equation}
which yields the result since $\nu>\frac{d}{d+\varepsilon}$.
\end{proof}
Lemma \ref{no jump 2} is stating that a current cannot jump over a $3+\nu$-dimensional (if $\nu\in (0,1)$ is sufficiently close to $1$) annulus in the complement of its backbone. This has consequences on the geometry of the clusters of the current. We begin with a definition.
\begin{Def}[Crossing event]\label{def: crossing event}
For $1\leq k\leq \ell$ and $\n$ a current, we say that $\n$ realises the event $\mathsf{Cross}(k,L)$ if $\n$ ``crosses'' $\textup{Ann}(k,\ell)$, in the sense that there exists a cluster of $\n$ containing both a point in $\Lambda_k$ and in $\Lambda_\ell^c$.
\end{Def}
\begin{Coro}\label{coro: no jump 2} Let $d=4$. Assume that $J$ satisfies $(\mathbf{A1})$--$(\mathbf{A6})$.  Let $\nu\in (0,1)$ be such that $\nu>\frac{d}{d+\varepsilon}$. There exist $C,\eta>0$ such that for all $\beta\leq \beta_c$, for all $k,\ell\geq 1$ with $k^{8/(1-\nu)} \leq \ell$, for all $x,u\in \mathbb Z^d$,
\begin{equation}
    \mathbf{P}_\beta^{xu}[\n\setminus \overline{\Gamma(\n)}\in \mathsf{Cross}(k,\ell)]\leq \mathbf{P}_\beta^{xu,\emptyset}[(\n_1+\n_2)\setminus \overline{ \Gamma(\n_1)}\in \mathsf{Cross}(k,\ell)]\leq \frac{C}{\ell^\eta}.
\end{equation}
\end{Coro}
\begin{proof} The first inequality follows from  the inclusion $\{\n_1\setminus \overline{\Gamma(\n_1)}\in \mathsf{Cross}(k,\ell)\}\subset \{(\n_1+\n_2)\setminus \overline{\Gamma(\n_1)}\in \mathsf{Cross}(k,\ell)\}.$
Notice that,
\begin{multline*}
    \left\lbrace (\n_1+\n_2)\setminus \overline{ \Gamma(\n_1)}\in \mathsf{Cross}(k,\ell)\right\rbrace\subset \bigcup_{v\in \Lambda_k, \: w\in \textup{Ann}(\ell, \ell+\ell^\nu)}\left\lbrace v \longleftrightarrow w \textup{ in }(\n_1+\n_2)\setminus \overline{\Gamma(\n_1)}\right\rbrace 
    \\
    \cup \left\lbrace (\n_1+\n_2)\setminus \overline{\Gamma(\n_1)}\in \mathsf{Jump}(\ell,\ell+\ell^\nu)\right\rbrace.
\end{multline*}
The second event on the right-hand side above is handled using Lemma \ref{no jump 2}.

To handle the first event, we use the fact that the probability $v$ and $w$ are connected in $(\n_1+\n_2)\setminus \overline{\Gamma(\n_1)}$ can be bounded by $\langle \sigma_v\sigma_w\rangle_\beta^2$. Indeed, this result follows from an generalisation of the switching lemma that can be found in \cite[Lemma~2.2]{AizenmanDuminilSidoraviciusContinuityIsing2015}. Proceeding as in the proof of Lemma \ref{no jump 2}, we get
\begin{multline*}
    \mathbf{P}_\beta^{xu,\emptyset}[u \longleftrightarrow v \textup{ in }(\n_1+\n_2)\setminus\overline{\Gamma(\n_1)}]
    \\\leq \sum_{\gamma: x\rightarrow u \textup{ consistent}}\mathbf{P}_\beta^{xu}[\Gamma(\n)=\gamma]\mathbf{P}_{\overline{\gamma}^c,\mathbb Z^d,\beta}^{\emptyset,\emptyset}[v \longleftrightarrow w \textup{ in }(\m_1+\m_2)\setminus\overline{\gamma}].
\end{multline*}
The above-mentioned generalisation of the switching lemma, together with Griffiths' inequality, yield
\begin{equation}
    \mathbf{P}_{\overline{\gamma}^c,\mathbb Z^d,\beta}^{\emptyset,\emptyset}[v \longleftrightarrow w \textup{ in }(\m_1+\m_2)\setminus\overline{\gamma}]= \langle \sigma_v\sigma_w\rangle_{\overline{\gamma}^c,\beta}\langle \sigma_v\sigma_w\rangle_{\beta}\leq \langle \sigma_v\sigma_w\rangle_{\beta}^2.
\end{equation}
As a result,
\begin{eqnarray*}
    \mathbf{P}_\beta^{xu,\emptyset}\left[\bigcup_{v\in \Lambda_k, \: w\in \textup{Ann}(\ell, \ell+\ell^\nu)}\left\lbrace v \longleftrightarrow w \textup{ in }(\n_1+\n_2)\setminus \overline{\Gamma(\n_1)}\right\rbrace\right]&\leq& \sum_{\substack{v\in \Lambda_k\\w\in \textup{Ann}(\ell,\ell+\ell^\nu)}}\langle \sigma_v\sigma_w\rangle_\beta^2
    \\&\leq & C_1\frac{k^4}{\ell^{1-\nu}}\leq \frac{C_1}{\ell^{(1-\nu)/2}}.
\end{eqnarray*}
This concludes the proof.
    
\end{proof}

\subsection{Proof of the intersection property} Let $d=4$. Assume that $J$ satisfies $(\mathbf{A1})$--$(\mathbf{A6})$. Let $\beta\leq \beta_c$.
Recall the definition of $\mathcal{L}=\mathcal{L}(\beta,D)$ given at the beginning of the section: $\ell_0=0$ and
\begin{equation}\label{eq: definition sequence L}
    \ell_{k+1}=\inf\lbrace \ell\geq \ell_k: \: B_\ell(\beta)\geq DB_{\ell_k}(\beta)\rbrace.
\end{equation}
The existence of regular scales and the sliding-scale infrared bound have the following interesting consequence on how fast the bubble diagram grows from one scale to the other, it can be see as an improvement over the bound $B_L(\beta)-B_\ell(\beta)\leq C_0\log(L/\ell)$.
\begin{Lem}[Scale to scale comparison of the bubble diagram, {\cite[Lemma~6.3]{AizenmanDuminilTriviality2021}}]\label{bubble growth}
Let $d=4$. There exists $C=C(d)>0$ such that for every $\beta\leq \beta_c$, and for every $1\leq\ell\leq L\leq L(\beta)$,
\begin{equation}
    B_L(\beta)\leq\left(1+C\frac{\log_2(L/\ell)}{\log_2(\ell)}\right)B_\ell(\beta).
\end{equation}
\end{Lem}
\begin{proof} If $N\geq n$ and $n$ is a regular scale, we may write,
\begin{align}
    B_{2N}(\beta)-B_N(\beta)&\leq C_1 N^{-4}\chi_{N/d}(\beta)^2\notag
    \\&\leq C_2 n^{-4} \chi_n(\beta)^2\notag
        \\&\leq C_3 n^{-4} \left(\chi_{2n}(\beta)-\chi_n(\beta)\right)^2\notag
    \\&\leq C_4 \left(B_{2n}(\beta)-B_n(\beta)\right),
\end{align}
where we successively used \eqref{eq: consequence mms }, the sliding-scale infrared bound, the property $(\mathbf{P3})$ of regular scales, and the Cauchy--Schwarz inequality.
There are $\log_2(L/\ell)$ scales between $\ell$ and $L$, and at least $c\log_2(\ell)$ regular scales between $1$ and $\ell$. Using the above computation,
\begin{align}
    B_L(\beta)-B_\ell(\beta)&=\sum_{N \textup{ scale between }\ell \textup{ and}L/2}B_{2N}(\beta)-B_N(\beta)
    \notag\\&\leq \frac{\log_2(L/\ell)}{c\log_2(\ell)}\sum_{\substack{n \textup{ regular scale}\\ \textup{between }1 \textup{ and }\ell}}B_{2n}(\beta)-B_n(\beta)
    \notag\\&\leq \frac{\log_2(L/\ell)}{c\log_2(\ell)} B_\ell(\beta).
\end{align}
\end{proof}
The above property has the following important consequence which ensures that the scales explode sufficiently fast. This will be used later to make sure there is ``enough room'' in the annuli $\textup{Ann}(\ell_k,\ell_{k+1})$.
\begin{Rem}[Growth of $\mathcal{L}$]\label{rem: growth l}
Using Lemma \textup{\ref{bubble growth}} we get that,
\begin{equation}
    \log_2(\ell_k)\leq C\log_2(\ell_{k+1}/\ell_k)\frac{B_{\ell_k}(\beta)}{B_{\ell_{k+1}}(\beta)-B_{\ell_k}(\beta)}\leq \log_2(\ell_{k+1}/\ell_k)\frac{C}{D-1},
\end{equation}
so that,
\begin{equation}
    \ell_{k+1}\geq \ell_k^{(D-1)/C}.
\end{equation}
\end{Rem}

Recall that the event $\mathsf{Jump}(k,\ell)$--- which is made of currents containing an open edge which ``jumps'' above $\text{Ann}(k,\ell)$--- was defined in Definition \ref{def: intersection event}. In the remaining of this section, we fix $\nu\in (0,1)$ with $\nu>\frac{d}{d+\varepsilon}$ such that Lemmas \ref{no jump 1}, \ref{no jump 1 bis}, and \ref{no jump 2} hold. 
\begin{Def}[Intersection event]\label{def: intersection event}
Let $k\geq 1$ and $y\notin \Lambda_{\ell_{k+2}}$. A pair of currents $(\n,\m)$ with $(\sn,\sm)=(\lbrace 0,y\rbrace,\lbrace 0,y\rbrace)$ realises the event $I_k$ if the following properties are satisfied:
\begin{enumerate}
    \item[$(i)$] The restrictions of $\n$ and $\m$ to edges with both endpoints in \textup{Ann}$(\ell_k,\ell_{k+1}+\ell_{k+1}^\nu)$ contain a unique cluster ``strongly crossing'' \textup{Ann}$(\ell_k,\ell_{k+1})$, in the sense that it contains a vertex in \textup{Ann}$(\ell_{k},\ell_k+\ell_{k}^\nu)$ and a vertex in \textup{Ann}$(\ell_{k+1},\ell_{k+1}+\ell_{k+1}^\nu)$.
    \item[$(ii)$] The two clusters described in $(i)$ intersect.
\end{enumerate}

\end{Def}
Note that the event $I_k$ is measurable in term of edges with both endpoints in the annulus $\text{Ann}(\ell_k,\ell_{k+1}+\ell_{k+1}^\nu)$. 

The following lemma shows that intersections occur at every scale with a uniformly positive probability.
\begin{figure}[htb]
\begin{center}
\includegraphics[scale=1.0]{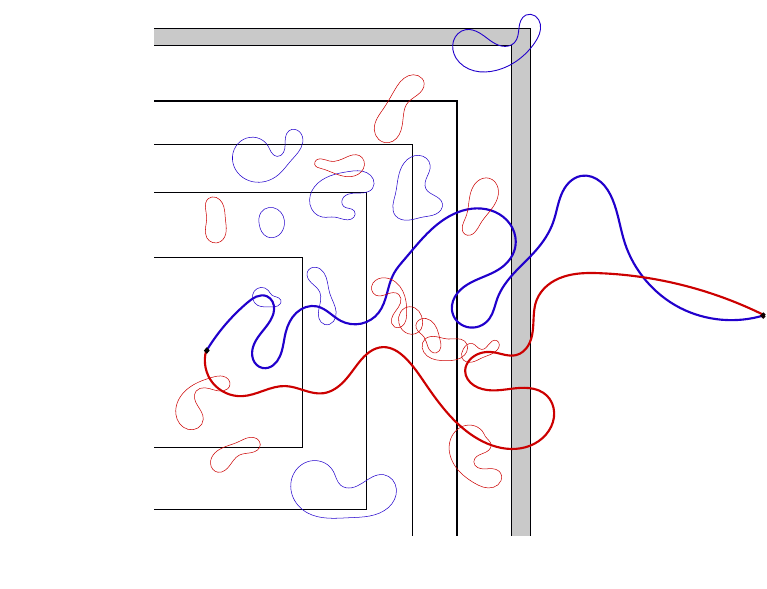}
\put(-305,250){$\Lambda_{\ell_{k+1}}$}
\put(-305,223){$\Lambda_{N}$}
\put(-315,198){$\mathrm{Ann}(m,M)$}
\put(-305,149){$\Lambda_{\ell_k}$}
\put(-277,115){$0$}
\put(-60,170){$\Gamma(\n_1)$}
\put(-96,90){$\Gamma(\n_2)$}
\put(-1,138){$y$}
\end{center}
\caption{An example of a configuration for which $\mathcal{M}\neq \emptyset$ but $I_k$ does not occur. $\n_1+\n_3$ (resp. $\n_2+\n_4$) is drawn in blue (resp. red), and the backbone of $\n_1$ (resp. $\n_2)$ is the blue (resp. red) bold curve joining $0$ and $y$. The outermost grey region represents the annulus $\mathrm{Ann}(\ell_{k+1},\ell_{k+1}+\ell_{k+1}^\nu)$. The clusters of the origin $\mathbf{C}_{\n_1+\n_3}(0)$ and $\mathbf{C}_{\n_2+\n_4}(0)$ intersect in $\mathrm{Ann}(m,M)$ thanks to a collection of red loops in $(\n_2+\n_4)\setminus \overline{\Gamma(\n_2)}$ which crosses $\mathrm{Ann}(M,N)$, which means that $\n_2+\n_4$ realises $\mathcal{F}_3$. As a result, this intersection is not measurable in terms of edges with both endpoints in $\mathrm{Ann}(\ell_k,\ell_{k+1}+\ell_{k+1}^\nu)$.}
\label{figure: non ik}
\end{figure}
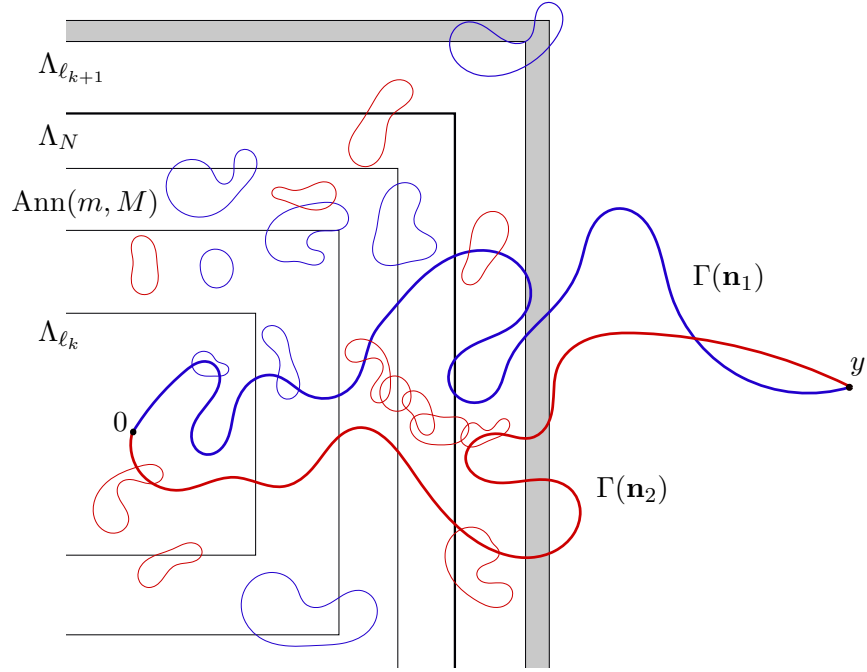
\begin{Lem}[Intersection property]\label{intersection prop}
Let $d=4$. For $D$ large enough, there exists $\kappa>0$ such that for every $\beta\leq \beta_c$, every $k\geq 2$, and every $y\notin \Lambda_{\ell_{k+2}}$ in a regular scale with $1\leq|y|\leq L(\beta)$,
\begin{equation}
    \mathbf{P}_\beta^{0y,0y,\emptyset,\emptyset}[(\n_1+\n_3,\n_2+\n_4)\in I_k]\geq \kappa.
\end{equation}
\end{Lem}
\begin{proof} We restrict ourselves to the case of $y$ in a regular scale to be able to use the properties $(\mathbf{P1})$ and $(\mathbf{P2})$. Introduce intermediate scales $\ell_k\leq n\leq m \leq M \leq N \leq \ell_{k+1}$ satisfying 
\begin{equation*}
    \ell_k^{\frac{8}{1-\nu}+1}\geq n\geq \ell_k^{\frac{8}{1-\nu}},\qquad n^{\frac{8}{1-\nu}+1}\geq m\geq n^{\frac{8}{1-\nu}},
\end{equation*} 
\begin{equation*}
    M^{\frac{8}{1-\nu}+1}\geq N\geq M^{\frac{8}{1-\nu}},\qquad N^{\frac{8}{1-\nu}+1}\geq \ell_{k+1}\geq N^{\frac{8}{1-\nu}},
\end{equation*}
which is possible provided $D$ is large enough by Remark \ref{rem: growth l}. Define
\begin{equation}
\mathcal{M}:=\mathbf{C}_{\n_1+\n_3}(0)\cap \mathbf{C}_{\n_2+\n_4}(0)\cap \text{Ann}(m,M).
\end{equation}
Using the Cauchy--Schwarz inequality,
\begin{equation}\label{eq: proof inter1}
    \mathbf{P}_\beta^{0y,0y,\emptyset,\emptyset}[|\mathcal{M}|>0]\geq \frac{\mathbf E_\beta^{0y,0y,\emptyset,\emptyset}[|\mathcal{M}|]^2}{\mathbf{E}_\beta^{0y,0y,\emptyset,\emptyset}[|\mathcal{M}|^2]}.
\end{equation}
One has for some $c_1>0$,
\begin{eqnarray}
    \mathbf E_\beta^{0y,0y,\emptyset,\emptyset}[|\mathcal{M}|]&=& \sum_{u\in \text{Ann}(m,M)}\mathbf P_\beta^{0y,\emptyset}[0\connect{\n_1+\n_2\:}u]^2
    \notag\\&=&\sum_{u\in \text{Ann}(m,M)}\left(\frac{\langle \sigma_0\sigma_u\rangle_\beta\langle \sigma_u\sigma_y\rangle_\beta}{\langle \sigma_0\sigma_y\rangle_\beta}\right)^2\notag \\&\geq&c_1(B_M(\beta)-B_{m-1}(\beta)),\label{eq: proof inter2}
\end{eqnarray}
where used \eqref{second application of switching lemma} in the second line, and regularity to compare $\langle \sigma_u\sigma_y\rangle_\beta$ with $\langle \sigma_0\sigma_y\rangle_\beta$ in the third line. Moreover, for some $c_2>0$,
\begin{align}
    \mathbf E_\beta^{0y,0y,\emptyset,\emptyset}[|\mathcal{M}|^2]&=\sum_{u,v\in \text{Ann}(m,M)}\mathbf P_\beta^{0y,\emptyset}[0\connect{\n_1+\n_2\:}u,v]^2
    \notag\\&\leq \sum_{u,v\in \text{Ann}(m,M)}\left(\frac{\langle \sigma_0\sigma_u\rangle_\beta\langle \sigma_u\sigma_v\rangle_\beta\langle \sigma_v\sigma_y\rangle_\beta}{\langle \sigma_0\sigma_y\rangle_\beta}+\frac{\langle \sigma_0\sigma_v\rangle_\beta\langle \sigma_v\sigma_u\rangle_\beta\langle \sigma_u\sigma_y\rangle_\beta}{\langle \sigma_0\sigma_y\rangle_\beta}\right)^2\notag\\&\leq c_2(B_M(\beta)-B_{m-1}(\beta))B_{2M}(\beta),\label{eq: proof inter3}
\end{align}
where we used (\ref{multi connectivity inequality}) in the second line, and then once again the regularity assumption to compare $\langle \sigma_u\sigma_y\rangle_\beta$ and $\langle \sigma_v\sigma_y\rangle_\beta$ with $\langle \sigma_0\sigma_y\rangle_\beta$. Using Lemma \ref{bubble growth}, we also get,
\begin{equation}
    B_M(\beta)\geq \left(1+C\frac{\log_2(\ell_{k+1}/M)}{\log_2(M)}\right)^{-1}B_{\ell_{k+1}}(\beta)\geq \frac{1}{1+C_1}B_{\ell_{k+1}}(\beta),\label{eq: proof inter4}
\end{equation}
where $C$ is the constant of Lemma \ref{bubble growth} and $C_1=C_1(\nu)=C\left[\left(\frac{8}{1-\nu}+1\right)^2-1\right]>0$. Similarly,
\begin{equation}
    B_{m-1}(\beta)\leq \left(1+C\frac{\log_2(m/\ell_k)}{\log_2(\ell_k)}\right)B_{\ell_k}(\beta)\leq (1+C_1)B_{\ell_k}(\beta)\leq \frac{1+C_1}{D}B_{\ell_{k+1}}(\beta),\label{eq: proof inter5}
\end{equation}
where we used that $B_{\ell_{k+1}}(\beta)\geq D B_{\ell_k}(\beta)$ in the last inequality.
As a result, combining \eqref{eq: proof inter1}--\eqref{eq: proof inter3} and using that $2M\leq \ell_{k+1}$, we obtain \begin{equation}\label{eq: proof inter6}
    \mathbf{P}_\beta^{0y,0y,\emptyset,\emptyset}[|\mathcal{M}|\neq 0]\geq \frac{c_1^2}{c_2}\frac{B_M(\beta)-B_{m-1}(\beta)}{B_{\ell_{k+1}}(\beta)}.
\end{equation}
Choosing $D$ large enough so that $\tfrac{1+C_1}{D}\leq \tfrac{1}{2}\tfrac{1}{1+C_1}$, and plugging \eqref{eq: proof inter4} and \eqref{eq: proof inter5} in \eqref{eq: proof inter6}, 
\begin{equation}
	\mathbf{P}_\beta^{0y,0y,\emptyset,\emptyset}[|\mathcal{M}|\neq 0]\geq \frac{c_1^2}{c_2}\frac{1}{2(1+C_1)}=:c_3.
\end{equation}
To conclude, we must prove uniqueness of the ``crossing clusters'' in $\n_1+\n_3$ and $\n_2+\n_4$. We reduce the argument to the one used in the nearest-neighbour case by not allowing the currents to make big jumps. More precisely, define the jump event $\mathsf{J}$ by
\begin{equation}
    \mathsf{J}:=\bigcup_{p\in \lbrace \ell_k,\ell_{k+1}\rbrace}\left\lbrace \n_1+\n_3\in\mathsf{Jump}(p,p+p^\nu)\right\rbrace\cup\left\lbrace \n_2+\n_4\in\mathsf{Jump}(p,p+p^\nu)\right\rbrace.
\end{equation}
Using Lemma \ref{no jump 1}, we get that for some $\eta>0$ and some constant $C_2>0$ (provided $\ell_{k+1}^{2}\leq \ell_{k+2}$, which might require to increase $D$),
\begin{equation}
    \mathbf{P}_\beta^{0y,0y,\emptyset,\emptyset}\left[\mathsf{J}\right]\leq \frac{C_2}{\ell_k^{\eta}}.
\end{equation}
Note that on the complement of the above event, the cluster connecting $0$ and $y$ in $\n_1+\n_3$ and $\n_2+\n_4$ must go through Ann$(p,p+p^\nu)$ for $p\in \lbrace \ell_k,\ell_{k+1}\rbrace$. In particular, it satisfies the ``strong crossing'' constraint in the definition of $I_k$. Take $D$ large enough so that there exists a constant $c_4>0$ such that,
\begin{equation}
    \mathbf{P}_\beta^{0y,0y,\emptyset,\emptyset}[|\mathcal{M}|\neq 0, \: \mathsf{J}^c]\geq c_4.
\end{equation}
If the event $\lbrace \mathcal{M}\neq \emptyset\rbrace\cap \mathsf{J}^c$ occurs but not $I_k$, then (see Figure \ref{figure: non ik}) one of the following events must occur for the pair $(\n_1,\n_3)$ (or $(\n_2,\n_4)$):
\begin{enumerate}
    \item[-] $\mathcal F_1:=$ the backbone $\Gamma(\n_1)$ of $\n_1$ does a zigzag between scales $\ell_k$ and $n$, i.e.\ it belongs to $\mathsf{ZZ}(0,y;\ell_k,n,\infty)$,
    \item[-] $\mathcal F_2:=$ $(\n_1+\n_3)\setminus \overline{\Gamma(\n_1)}$ belongs to $\mathsf{Cross}(n+n^\nu,m)$,
    \item[-] $\mathcal F_3:=$ the backbone $\Gamma(\n_1)$ of $\n_1$ does a zigzag between scales $N+N^\nu$ and $\ell_{k+1}$, i.e.\ it belongs to $\mathsf{ZZ}(0,y;N+N^\nu,\ell_{k+1},\infty)$,
    \item[-] $\mathcal F_4:=$ $(\n_1+\n_3)\setminus \overline{\Gamma(\n_1)}$ belongs to $\mathsf{Cross}(M,N)$,
    \item[-] $\mathcal{F}_5:= \mathsf{Jump}(n,n+n^\nu)\cup\mathsf{Jump}(N,N+N^\nu)$.
\end{enumerate}

Note that the event $\mathcal{F}_5$ takes into account situations in which (for instance) $(\n_1+\n_3)\setminus \overline{\Gamma(\n_1)}$ contains a cluster ``almost realising'' $\mathsf{Cross}(M,N)$ and $\Gamma(\n_1)$ ``almost'' performs $\mathsf{ZZ}(0,y;N+N^\nu,\ell_{k+1},\infty)$, and these two pieces are connected by a long open edge of $\overline{\Gamma(\n_1)}\setminus \Gamma(\n_1)$ which jumps over $\textup{Ann}(N,N+N^\nu)$.

Using Corollaries \ref{coro: no jump 1} and \ref{coro: no jump 2} we get the existence of $C,\eta>0$ such that,
\begin{equation}
    \mathbf{P}^{0y,\emptyset}_\beta[\mathcal{F}_1]\leq \frac{C}{n^\eta}, \qquad \mathbf{P}^{0y,\emptyset}_\beta[\mathcal{F}_2]\leq \frac{C}{m^\eta},
\end{equation}
\begin{equation}
    \mathbf{P}^{0y,\emptyset}_\beta[\mathcal{F}_3]\leq \frac{C}{\ell_{k+1}^\eta}, \qquad \mathbf{P}^{0y,\emptyset}_\beta[\mathcal{F}_4]\leq \frac{C}{N^\eta}.
\end{equation}
Moreover, as a consequence of Lemma \ref{no jump 1}, we also get that $\mathbf{P}_\beta^{0y,\emptyset}[\mathcal{F}_5]\leq C/n^\eta$.

As a result, taking $D$ large enough, we get that the sum of the probabilities of the five events for the pairs $(\n_1,\n_3)$ and $(\n_2,\n_4)$ does not exceed $c_4/2$ so that, setting $\kappa:=c_4/2$,
\begin{equation}
    \mathbf{P}_\beta^{0y,0y,\emptyset,\emptyset}[I_k]\geq \kappa.
\end{equation}
\end{proof}

\subsection{Proof of the mixing}\label{section: mixing d=4 proof}
We now turn to the proof of the mixing property. Recall that we assume that $J$ satisfies $(\mathbf{A1})$--$(\mathbf{A6})$.
\begin{Thm}[Mixing property]\label{mixing property}
Let $d= 4$ and $s\geq 1$. There exist $\gamma>0$ and $C=C(s)>0$, such that for every $1\leq t\leq s$, every $\beta\leq \beta_c$, every $1\leq n^\gamma\leq N\leq L(\beta)$, every $x_i\in \Lambda_n$ and $y_i\notin \Lambda_N$ $(i\leq t)$, and every events $E$ and $F$ depending on the restriction of $(\n_1,\ldots,\n_s)$ to edges with endpoints within $\Lambda_n$ and outside $\Lambda_N$ respectively,
\begin{multline}\label{eq 1 mixing}
    \left|\mathbf{P}_\beta^{x_1 y_1,\ldots, x_t y_t,\emptyset,\ldots,\emptyset}[E\cap F]-\mathbf{P}_\beta^{x_1 y_1,\ldots, x_t y_t,\emptyset,\ldots,\emptyset}[E]\mathbf{P}_\beta^{x_1 y_1,\ldots, x_t y_t,\emptyset,\ldots,\emptyset}[F]\right|\leq C\left(\log\frac{N}{n}\right)^{-1/2}.
\end{multline}
Furthermore, for every $x_1',\ldots, x_t'\in \Lambda_n$ and $y_1',\ldots,y'_t\notin \Lambda_N$, we have that
\begin{equation}\label{eq 2 mixing}
    \left|\mathbf{P}_\beta^{x_1 y_1,\ldots, x_t y_t,\emptyset,\ldots,\emptyset}[E]-\mathbf{P}_\beta^{x_1 y'_1,\ldots, x_t y'_t,\emptyset,\ldots,\emptyset}[E]\right|\leq C\left(\log\frac{N}{n}\right)^{-1/2},
\end{equation}
\begin{equation}\label{eq 3 mixing}
    \left|\mathbf{P}_\beta^{x_1 y_1,\ldots, x_t y_t,\emptyset,\ldots,\emptyset}[F]-\mathbf{P}_\beta^{x'_1 y_1,\ldots, x'_t y_t,\emptyset,\ldots,\emptyset}[F]\right|\leq C\left(\log\frac{N}{n}\right)^{-1/2}.
\end{equation}
\end{Thm}

Fix $\beta\leq \beta_c$. Fix two integers $t,s$ satisfying $1\leq t \leq s$. Introduce integers $m,M$ such that $n\leq m\leq M\leq N$, $m/n=(N/n)^{\mu/2}$, and $N/M=(N/n)^{1-\mu}$ for $\mu$ small to be fixed. For $\mathbf{x}=(x_1,\ldots,x_t)$ and $\mathbf{y}=(y_1,\ldots,y_t)$, write:
\begin{equation}
    \mathbf{P}^{\mathbf{xy}}_\beta:=\mathbf{P}_\beta^{x_1y_1,\ldots,x_ty_t,\emptyset,\ldots,\emptyset},\qquad \mathbf{P}_\beta^{\mathbf{xy},\emptyset}:=\mathbf{P}_\beta^{\mathbf{xy}}\otimes \mathbf{P}^{\emptyset,\ldots,\emptyset}_\beta,
\end{equation}
where $\mathbf{P}_\beta^{\emptyset,\ldots,\emptyset}$ is the law of a sum of $s$ independent sourceless currents that we denote by $(\n_1',\ldots,\n_s')$.

If $p\geq 1$, define for $y\notin \Lambda_{2dp}$,
\begin{equation}
    \mathbb A_y(p):=\left\lbrace u \in \text{Ann}(p,2p) \: : \: \forall x \in \Lambda_{p/d}, \: \langle \sigma_x\sigma_y\rangle_\beta\leq \left(1+C_0\frac{|x-u|}{|y|}\right)\langle \sigma_u\sigma_y\rangle_\beta\right\rbrace,
\end{equation}
where $C_0$ is the constant in the definition of regular scales. Note that if $y$ is in a regular scale, then $\mathbb A_y(p)=\text{Ann}(p,2p)$.

Let $\mathcal K$ be a set of $(c_0,C_0)$-regular scales $k$ between $m$ and $M/2$ constructed according to the following algorithm:
\begin{enumerate}
	\item[-] Let $k_1$ be the smallest $(c_0,C_0)$-regular scale in $\text{Ann}(m,M/2)$.
	\item[-] Let $i\geq 1$ and assume that $k_i$ is constructed. Choose the smallest regular scale $k>k_i$ such that $2^k\leq M/2$ and $2^{k}\geq C_02^{k_i}$ (this condition is useful to apply $(\mathbf{P4})$). If $k$ exists, set $k_{i+1}=k$. If not, stop the algorithm.
	\item[-] Set $\mathcal K:=\{k_i: i\geq 1\}$.
\end{enumerate}
By the existence of regular scales of Proposition \ref{prop: existence regular scales}, one has that $|\mathcal{K}|\geq c_1\log(N/n)$ for a sufficiently small $c_1=c_1(\mu)>0$.

Introduce $\mathbf{U}:=\prod_{i=1}^t\mathbf{U}_i$, where
\begin{equation}
    \mathbf{U}_i:=\frac{1}{|\mathcal{K}|}\sum_{k\in \mathcal{K}}\frac{1}{A_{x_i,y_i}(2^k)}\sum_{u\in \mathbb{A}_{y_i}(2^k)}\mathds{1}\{u\connect{\n_i+\n_i'\:}x_i\},
\end{equation}
and,
\begin{equation}
    a_{x,y}(u):=\frac{\langle \sigma_x\sigma_u\rangle_\beta\langle\sigma_u\sigma_y\rangle_\beta}{\langle \sigma_x\sigma_y\rangle_\beta},\qquad A_{x,y}(p):=\sum_{u\in \mathbb A_y(p)}a_{x,y}(u).
\end{equation}
Using the switching lemma \eqref{eq: switching lemma} and independence, we see that $\mathbf{E}_\beta^{\mathbf{xy},\emptyset}[\mathbf{U}]=1$.
We begin by importing a concentration inequality whose proof essentially relies on the definition of $\mathbb A_{y_i}(2^k)$ and the properties of regular scales\footnote{This is the only place where we need the property $\mathbf{(P4)}$ of regular scales. It also heavily relies on $(\mathbf{P3})$.} (see \cite[Proposition~6.6]{AizenmanDuminilTriviality2021}). 
\begin{Lem}[Concentration of $\mathbf{U}$]\label{lem: concentration of N}
For all $\gamma>2$, there exists $C=C(d,t,\gamma)>0$ such that for all $n$ sufficiently large satisfying $n^\gamma\leq N\leq L(\beta)$,
\begin{equation}
    \mathbf{E}_\beta^{\mathbf{xy},\emptyset}[(\mathbf{U}-1)^2]\leq \frac{C}{\log(N/n)}.
\end{equation}
\end{Lem}
We now fix $\gamma>2$ (it will be taken large enough later). Using the Cauchy--Schwarz inequality together with Lemma \ref{lem: concentration of N},  we find $C_1=C_1(d,t,\gamma)>0$ such that
\begin{multline}\label{eq preuve 1}
\left|\mathbf{P}_\beta^{\mathbf{xy}}[E\cap F]-\mathbf E_\beta^{\mathbf{xy},\emptyset}\big[\mathbf{U}\mathds{1}\{(\n_1,\ldots,\n_s)\in E\cap F\}\big]\right|\\\leq \sqrt{\mathbf{E}_\beta^{\mathbf{xy},\emptyset}[(\mathbf{U}-1)^2]}\leq\frac{C_1}{\sqrt{\log(N/n)}}. 
\end{multline}
At this stage of the proof, we need to analyse $\mathbf E_\beta^{\mathbf{xy},\emptyset}\big[\mathbf{U}\mathds{1}\{(\n_1,\ldots,\n_s)\in E\cap F\}\big]$. By definition of $\mathbf{U}$, this term can be rewritten as a weighted sum of terms of the form
\begin{equation}
\mathbf E_\beta^{\mathbf{x}\mathbf{y},\emptyset}\Big[\prod_{i=1}^t\mathds{1}\{u_i\connect{\n_i+\n_i'\:}y_i\}\mathds{1}\{(\n_1,\ldots,\n_s)\in E\cap F\}\Big]
\end{equation}
 where $u_i\in \mathbb A_{y_i}(2^{k_i})$ for some $k_i\in \mathcal{K}$. It would be very tempting to try to apply the switching lemma directly to turn the measure $\mathbf{P}_\beta^{\mathbf{xy},\emptyset}$ into a measure $\mathbf{P}_\beta^{\mathbf{xu},\mathbf{uy}}$ (up to a renormalisation weight). This would suggest that the occurrences of $E$ and $F$ are essentially due to $(\n_1,\ldots,\n_s)$ and $(\n_1',\ldots,\n_s')$ respectively under the measure $\mathbf{P}_\beta^{\mathbf{xu},\mathbf{uy}}$. However, the occurrence of $E\cap F$ is only a function of $(\n_1,\ldots,\n_s)$ which makes the use of the switching lemma impossible. Nevertheless, we can still use the switching principle. This motivates the introduction of the following event.
\begin{Def}\label{def: good event mixing}
Let $\mathbf{u}=(u_1,\ldots,u_t)$ with $u_i\in \textup{Ann}(m,M)$ for every $i$. Introduce the event $\mathcal G(u_1,\ldots,u_t)=\mathcal G(\mathbf{u})$ defined as follows: for every $i\leq s$, there exists $\mathbf{k}_i\leq \n_i+\n_i'$ such that $\mathbf{k}_i=0$ on $\Lambda_n$, $\mathbf{k}_i=\n_i+\n_i'$ outside $\Lambda_N$, $\partial \mathbf{k}_i=\lbrace u_i,y_i\rbrace$ for $i\leq t$, and $\partial \mathbf{k}_i=\emptyset$ for $t<i\leq s$.
\end{Def}
By the switching principle \eqref{eq: switching principle}, one has
\begin{multline*}
    \mathbf P_\beta^{\mathbf{x}\mathbf{y},\emptyset}\left[(\n_1,\ldots,\n_s)\in E\cap F, u_i\connect{\n_i+\n_i'\:}y_i ,\: \forall 1\leq i \leq t, \:\mathcal G(\mathbf{u})\right]\\=\left(\prod_{i=1}^ta_{x_i,y_i}(u_i)\right)\mathbf P_\beta^{\mathbf{xu},\mathbf{uy}}\left[(\n_1,\ldots,\n_s)\in E, (\n_1',\ldots,\n_s')\in F, \:\mathcal G(\mathbf{u})\right].
\end{multline*}
The trivial identity,
\begin{equation}
    \mathbf P_\beta^{\mathbf{xu},\mathbf{uy}}[(\n_1,\ldots,\n_s)\in E, (\n_1',\ldots,\n_s')\in F]=\mathbf P_\beta^{\mathbf{xu}}[(\n_1,\ldots,\n_s)\in E]\mathbf P_\beta^{\mathbf{uy}}[(\n_1',\ldots,\n_s')\in F]
\end{equation}
motivates us to prove that under $\mathbf{P}_\beta^{\mathbf{xu},\mathbf{uy}}$, the event $\mathcal G(\mathbf{u})$ occurs with high probability. This motivates the following result.
\begin{Lem}\label{technical lemma mixing}
Let $d=4$. There exist $C,\epsilon>0$, $\gamma=\gamma(\epsilon)>0$ large enough and $\mu=\mu(\epsilon)>0$ small enough such that for every $n^\gamma\leq N\leq L(\beta)$, and every $\mathbf{u}$ with $u_i\in \mathbb A_{y_i}(2^{k_i})$ with $k_i\in\mathcal{K}$ for $1\leq i \leq t$,
\begin{equation}
    \mathbf P_\beta^{\mathbf{x}\mathbf{y},\emptyset}\left[u_i\connect{\n_i+\n_i'\:}y_i ,\: \forall 1\leq i \leq t, \:\mathcal G(\mathbf{u})^c\right]\left(\prod_{i=1}^ta_{x_i,y_i}(u_i)\right)^{-1}=\mathbf P_\beta^{\mathbf{xu},\mathbf{uy}}[\mathcal G(\mathbf{u})^c]\leq C\left(\frac{n}{N}\right)^\epsilon.
\end{equation}
\end{Lem}
\begin{proof} Below, $C=C(d)>0$ may change from line to line\footnote{In fact, $C$ will also depend on $\beta_c$ (or more precisely on a lower bound on $\beta_c$).}. The equality follows from an application of the switching lemma \eqref{eq: switching lemma}. If we write $\mathcal G(\mathbf{u})=\cap_{1\leq i \leq s}G_i$ (where the definition of $G_i$ is implicit), then $H_i\cap F_i\subset G_i$ where,
\begin{equation}
    H_i:=\lbrace \text{Ann}(M,N) \text{ is not crossed by a cluster in }\n_i\rbrace=\lbrace \n_i\notin \mathsf{Cross}(M,N)\rbrace,
\end{equation}
and
\begin{equation}
    F_i:=\lbrace \text{Ann}(n,m) \text{ is not crossed by a cluster in }\n_i'\rbrace=\lbrace \n_i'\notin \mathsf{Cross}(n,m)\rbrace.
\end{equation}
Indeed, if $H_i\cap F_i$ occurs, we may define $\mathbf{k}_i$ as the sum of the restriction of $\n_i$ to the clusters intersecting $\Lambda_N^c$ and the restriction of $\n_i'$ to the clusters intersecting $\Lambda_m^c$. In the following, we assume that $1\leq i\leq t$. The argument for other values of $i$ is easier and follows from the first case.

Introduce intermediate scales $n\leq r \leq m \leq M\leq R \leq N$ with $r,R$ chosen below.

\paragraph{\textbf{$\bullet$ Bound on $H_i$}.} Following the ideas developed in the proof of Lemma \ref{intersection prop}, we define, 
\begin{equation}
    \mathsf{J}_i:=\lbrace \n_i\in \mathsf{Jump}(R-R^\nu,R+R^\nu)\rbrace.
\end{equation}  Notice that\footnote{We need to create a ``forbidden area'' $\textup{Ann}(R-R^\nu,R+R^\nu)$ between the two parts of the annulus $\textup{Ann}(M,N)$ to rule out the situation in which $\overline{\Gamma(\n_i)}\setminus\Gamma(\n_i)$ connects an ``almost successful'' excursion of $\Gamma(\n_i)$ (in the sense that it almost crossed $\textup{Ann}(M,R-R^\nu)$), with an ``almost successful'' excursion of $\n_i\setminus \overline{\Gamma(\n_i)}$. This is very similar to the argument used in the proof of Lemma \ref{intersection prop}.},
\begin{multline*}
    \mathbf P^{\mathbf{xu},\mathbf{uy}}_\beta[H_i^c] \leq \mathbf P^{\mathbf{xu}}_\beta[\Gamma(\n_i)\in\mathsf{ZZ}(x_i,u_i;M,R-R^\nu,\infty)]\\+\mathbf P^{\mathbf{xu}}_\beta[\n_i\setminus\overline{\Gamma(\n_i)} \in \mathsf{Cross}(R+R^\nu,N)]+\mathbf{P}_\beta^{\mathbf{xu}}[\mathsf{J}_i].
\end{multline*}
Assume $R=N^\iota$ where $\iota>\mu$ is chosen in such a way that $M^{6/(1-\nu)}\leq R$ and $R^{8/(1-\nu)}\leq N$ (note that this is possible if we choose $\mu$ sufficiently small and $\gamma$ sufficiently large since $M\leq N^{\mu+1/\gamma}$). This choice allows us to use Corollaries \ref{coro: no jump 1 bis} and \ref{coro: no jump 2} to get that for some $\eta>0$,
\begin{equation}
    \mathbf P^{\mathbf{xu}}_\beta[\n_i\in\mathsf{ZZ}(x_i,u_i;M,R-R^\nu,\infty)]\leq \frac{C}{R^\eta}, \qquad \mathbf P^{\mathbf{xu}}_\beta[\n_i\setminus\overline{\Gamma(\n_i)} \in \mathsf{Cross}(R+R^\nu,N)]\leq \frac{C}{N^\eta}.
\end{equation}
Moreover, at the cost of diminishing the value of $\iota$ (and hence also modifying $\mu$ and $\gamma$) to ensure that $R^{2}\leq N$, we may use Lemma \ref{no jump 1} to argue the existence of $\eta'>0$ such that
\begin{equation}
    \mathbf{P}_\beta^{\mathbf{xu}}[\mathsf{J}_i]\leq \frac{C}{R^{\eta'}}.
\end{equation}
Putting all the pieces together, we get for some $\epsilon>0$,
\begin{equation}
    \mathbf P^{\mathbf{xu},\mathbf{uy}}_\beta[H_i^c]\leq C\left(\frac{n}{N}\right)^\epsilon.
\end{equation}
\paragraph{\textbf{$\bullet$ Bound on $F_i$}.} We proceed similarly for $F_i$ although we now encounter an additional difficulty: we cannot rule out the possibility that $\n_i'$ will jump above the annulus $\textup{Ann}(r-r^\nu,r+r^\nu)$, we can only rule it out in the complement of $\overline{\Gamma(\n_i')}$. We set $r=(n^2m)^{1/3}$. Recall that $m\geq N^{\mu/2}$.  
We claim that
\begin{multline}\label{eq: proof mixing lemma}
    \mathbf P^{\mathbf{xu},\mathbf{uy}}_\beta[F_i^c]\leq \mathbf P^{\mathbf{uy}}_\beta[\Gamma(\n_i') \in \mathsf{ZZ}(u_i,y_i;r+r^\nu,m,\infty)]\\+\mathbf P^{\mathbf{uy}}_\beta[\n_i'\setminus\overline{\Gamma(\n_i')} \in \mathsf{Cross}(n,r-r^\nu)]+\mathbf P_\beta^{\mathbf{uy}}[K_i],
\end{multline}
where $K_i$ is the event that there exists $a\in \Lambda_{r-r^\nu}$ and $b\notin \Lambda_{r+r^\nu}$ such that $(\n_i')_{a,b}\geq 2$ and $\lbrace a,b\rbrace\in \overline{\Gamma(\n_i')}\setminus\Gamma(\n_i')$.

Indeed, if none of the events corresponding to the two first probabilities on the right-hand side of \eqref{eq: proof mixing lemma} occur, the only way we can find a cluster crossing $\textup{Ann}(n,m)$ is if $\overline{\Gamma(\n_i')}\setminus\Gamma(\n_i')$ has a long (even) open edge which jumps above $\textup{Ann}(r-r^\nu,r+r^\nu)$ (see Figure \ref{figure: bound Fi}). 

Using \eqref{eq: INFRARED BOUND WE USE} to get $\langle \sigma_{u_i}\sigma_v\rangle_\beta\leq Cm^{-2}$ and the assumption that $u_i\in \mathbb A_{y_i}(2^{k_i})$ to get that $\langle \sigma_v\sigma_{y_i}\rangle_\beta\leq C\langle \sigma_{u_i}\sigma_{y_i}\rangle_\beta$, we obtain
\begin{eqnarray*}
    \mathbf P^{\mathbf{uy}}_\beta[\Gamma(\n_i') \in \mathsf{ZZ}(u_i,y_i;r+r^\nu,m,\infty)]
    &\leq& 
    \sum_{v\in \Lambda_{r+r^\nu}}\frac{\langle \sigma_{u_i}\sigma_v\rangle_\beta \langle \sigma_{v}\sigma_{y_i}\rangle_\beta}{\langle \sigma_{u_i}\sigma_{y_i}\rangle_\beta}
    \\&\leq& 
    C\frac{r^4}{m^{2}}=C\frac{n^{8/3}}{m^{2/3}}\leq C\frac{N^{\frac{8}{3\gamma}}}{N^{\frac{\mu}{3}}}.
\end{eqnarray*}
Moreover, using Corollary \ref{coro: no jump 2} (which requires that $n^{8/(1-\nu)}\leq r$ and hence decreases the values of $\mu$ and $1/\gamma$), there exist $\zeta>0$ such that
\begin{equation}
    \mathbf P^{\mathbf{uy}}_\beta[\n_i'\setminus\overline{\Gamma(\n_i')} \in \mathsf{Cross}(n,r-r^\nu)]\leq \frac{C}{r^\zeta}.  
\end{equation}

We conclude the proof with the bound on $K_i$. If $\n_i'$ satisfies that $(\n_i')_{a,b}\geq 2$ where $a\in \Lambda_{r-r^\nu}$ and $b\notin \Lambda_{r+r^\nu}$ with $\lbrace a,b\rbrace\in \overline{\Gamma(\n_i')}\setminus\Gamma(\n_i')$ being the earliest such edge. We consider the map $\n_i'\mapsto \m_i'$ where $(\m_i')_{a,b}=(\n_i')_{a,b}-1$ and $\m_i'$ coincides with $\n_i'$ everywhere else. This maps $\n_i'$ to a current $\m_i'$ with sources $\lbrace u_i,y_i\rbrace\Delta\lbrace a,b\rbrace$ ($b$ might coincide with $u_i$ or $y_i$) such that the backbone $\Gamma(\m_i')$ always connects $u_i$ and $b$ (by definition of the exploration). Hence, using the chain rule for backbones,
\begin{multline*}
    \mathbf P_\beta^{\mathbf{uy}}[K_i]\leq \sum_{\substack{a\in \Lambda_{r-r^\nu}\\b\notin \Lambda_{r+r^\nu}}}\beta J_{a,b}\frac{\langle \sigma_{\lbrace u_i,y_i\rbrace\Delta\lbrace a,b\rbrace}\rangle_\beta}{\langle \sigma_{u_i}\sigma_{y_i}\rangle_\beta}\mathbf{P}_\beta^{\lbrace u_i,y_i\rbrace\Delta\lbrace a,b\rbrace}[\Gamma(\m_i')\textup{ connects }u_i \textup{ and }b]\\\leq \sum_{\substack{a\in \Lambda_{r-r^\nu}\\b\notin \Lambda_{r+r^\nu}}}\beta J_{a,b}\frac{\langle \sigma_{u_i}\sigma_b\rangle_\beta\langle \sigma_a\sigma_{y_i}\rangle_\beta}{\langle \sigma_{u_i}\sigma_{y_i}\rangle_\beta}.
\end{multline*}
Since $u_i\in \mathbb A_{y_i}(2^{k_i})$, $\langle \sigma_a\sigma_{y_i}\rangle_{\beta}\leq C\langle \sigma_{u_i}\sigma_{y_i}\rangle_\beta$. Hence,
\begin{equation}
    \beta\sum_{\substack{a\in \Lambda_{r-r^\nu}\\b\notin \Lambda_{r+r^\nu}}}J_{a,b}\frac{\langle \sigma_{u_i}\sigma_b\rangle_\beta\langle \sigma_a\sigma_{y_i}\rangle_\beta}{\langle \sigma_{u_i}\sigma_{y_i}\rangle_\beta}
    \leq 
    C\beta\sum_{\substack{a\in \Lambda_{r-r^\nu}\\b\notin \Lambda_{r+r^\nu}}}J_{a,b}\langle \sigma_{u_i}\sigma_b\rangle_\beta.
\end{equation}
We distinguish two-cases according to whether $b$ is close to $u_i$ or not. By \eqref{assumption on J} and \eqref{eq: INFRARED BOUND WE USE},
\begin{equation}
    \beta\sum_{\substack{a\in \Lambda_{r-r^\nu}\\b\in \Lambda_{|u_i|/2}(u_i)}}J_{a,b}\langle \sigma_{u_i}\sigma_b\rangle_\beta\leq Cr^4\sum_{|x|\geq m/4}J_{0,x}\leq \frac{Cr^4}{m^{2+\varepsilon}}.
\end{equation}
Again, using \eqref{assumption on J} and \eqref{eq: INFRARED BOUND WE USE},
\begin{equation}
    \sum_{\substack{a\in \Lambda_{r-r^\nu}\\b\notin \Lambda_{r+r^\nu}\cup \Lambda_{|u_i|/2}(u_i)}}J_{a,b}\langle \sigma_{u_i}\sigma_b\rangle_\beta \leq \frac{Cr^4}{m^2}\sum_{|x|\geq r^\nu}J_{0,x}\leq \frac{Cr^{4-\nu(2+\varepsilon)}}{m^2}.
\end{equation}
By definition of $r$ and $m/n$, we get the existence of $\zeta'>0$ such that
\begin{equation}
    \mathbf{P}_\beta^{\mathbf{uy}}[K_i]\leq \frac{C}{m^{\zeta'}}.
\end{equation}
As a result, if we choose $\mu$ and $1/\gamma$ sufficiently small, we get that for some $\epsilon'>0$,
\begin{equation}
    \mathbf P^{\mathbf{xu},\mathbf{uy}}_\beta[F_i^c]\leq C\left(\frac{n}{N}\right)^{\epsilon'}.
\end{equation}
\end{proof}

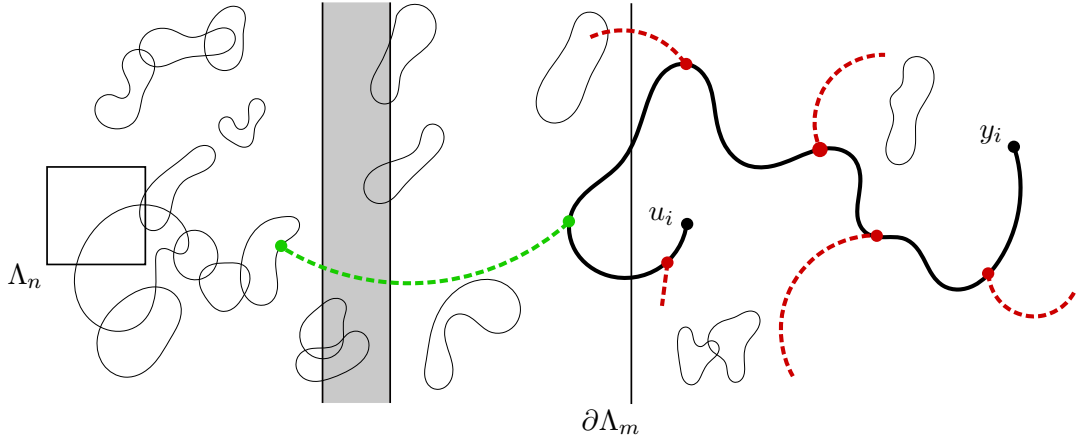
\begin{figure}[htb]
\begin{center}
\includegraphics[scale=1.5]{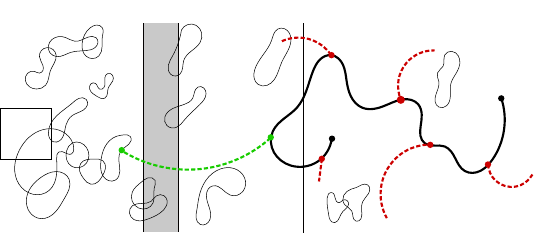}
\put(-400,50){$\Lambda_{n}$}
\put(-185,-5){$\partial\Lambda_m$}
\put(-160,74){$u_i$}
\put(-36,105){$y_i$}
\end{center}
\caption{An illustration of the event $K_i$. We represented the current $\n_i'$ in black, with a bold line representing its backbone $\Gamma(\n_i')$. The red/green dashed lines correspond to the open edges of $\overline{\Gamma(\n_i')}\setminus \Gamma(\n_i')$, they all carry an non-zero even weight. The grey region corresponds to the annulus $\textup{Ann}(r-r^\nu,r+r^\nu)$. In this picture, $\Gamma(\n_i')$ does not cross $\textup{Ann}(r+r^\nu,m)$, $\n_i'\setminus \overline{\Gamma(\n_i')}$ does not cross $\textup{Ann}(n,r-r^\nu)$, yet the long green open edge of $\overline{\Gamma(\n_i')}\setminus\Gamma(\n_i')$ creates a cluster of $\n_i'$ crossing $\textup{Ann}(n,m)$.}
\label{figure: bound Fi}
\end{figure}

We now turn to the proof of Theorem \ref{mixing property}.

\begin{proof}[Proof of Theorem \textup{\ref{mixing property}}]
Introduce the coefficients $\delta(\mathbf{u},\mathbf{x},\mathbf{y})$ defined by,
\begin{equation}
    \delta(\mathbf{u},\mathbf{x},\mathbf{y}):=\mathds{1}\{\exists (k_1,\ldots,k_t)\in \mathcal{K}^t,\: \mathbf{u}\in \mathbb A_{y_1}(2^{k_1})\times\ldots\times \mathbb A_{y_t}(2^{k_t})\}\prod_{i=1}^t\frac{a_{x_i,y_i}(u_i)}{|\mathcal{K}|A_{x_i,y_i}(2^{k_i})}.
\end{equation}
Note that
\begin{equation}
    \sum_{ \substack{(k_1,\ldots,k_t)\in \mathcal{K}^t\\\mathbf{u}\in \mathbb A_{y_1}(2^{k_1})\times\ldots\times \mathbb A_{y_t}(2^{k_t})}}\delta(\mathbf{u},\mathbf{x},\mathbf{y})=1.
\end{equation}
Equation (\ref{eq preuve 1}) together with Lemma \ref{technical lemma mixing} yield,
\begin{multline}\label{eq preuve mixing 1}
    \left|\mathbf{P}_\beta^{\mathbf{xy}}[E\cap F]-\sum_{\mathbf{u}}\delta(\mathbf{u},\mathbf{x},\mathbf{y})\mathbf{P}_\beta^{\mathbf{xu}}[E]\mathbf{P}_\beta^{\mathbf{uy}}[F]\right|\\\leq \frac{C_1}{\sqrt{\log(N/n)}}+C_2\left(\frac{n}{N}\right)^\epsilon\leq \frac{C_3}{\sqrt{\log(N/n)}},
\end{multline}
as long as $N\geq n^\gamma$ where $\gamma>2$ is given by Lemma \ref{technical lemma mixing}. We begin by proving (\ref{eq 2 mixing}) when $y_i,y_i'$ are in regular scales (but not necessarily the same ones). Applying the above inequality once for $\mathbf{y}$ and once for $\mathbf{y'}$ with the event $E$ and $F=\Omega_{\mathbb{Z}^d}$,
\begin{equation}\label{eq: preuve mixing }
    \left|\mathbf{P}_\beta^{\mathbf{xy}}[E]-\mathbf{P}_\beta^{\mathbf{xy'}}[E]\right|
    \leq 
    \left|\sum_{\mathbf{u}}(\delta(\mathbf{u},\mathbf{x},\mathbf{y})-\delta(\mathbf{u},\mathbf{x},\mathbf{y'}))\mathbf{P}_\beta^{\mathbf{xu}}[E]\right|+\frac{2C_3}{\sqrt{\log(N/n)}}.
\end{equation}
Since all the $y_i,y_i'$ are in regular scales, one has $\mathbb{A}_{y_i}(2^{k_i})=\mathbb{A}_{y_i'}(2^{k_i})=\text{Ann}(2^{k_i},2^{k_i+1}).$ Moreover, using Property $\mathbf{(P2)}$ of regular scales,
\begin{equation}
    \left|\delta(\mathbf{u},\mathbf{x},\mathbf{y})-\delta(\mathbf{u},\mathbf{x},\mathbf{y'})\right|
    \leq C_4  \left(\frac{M}{N}\right)\delta(\mathbf{u},\mathbf{x},\mathbf{y})
    \leq C_5  \left(\frac{n}{N}\right)^{1-\mu}\delta(\mathbf{u},\mathbf{x},\mathbf{y}),
\end{equation}
where $\mu$ is also given by Lemma \ref{technical lemma mixing}. Indeed, in that case $\delta(\mathbf{u},\mathbf{x},\mathbf{y})$ and $\delta(\mathbf{u},\mathbf{x},\mathbf{y}')$ are both close to 
\begin{equation}
    \prod_{i\leq t}\frac{\langle \sigma_{x_i}\sigma_{u_i}\rangle}{|\mathcal{K}|\sum_{v_i\in \textup{Ann}(2^{k_i},2^{k_i+1})}\langle \sigma_{x_i}\sigma_{v_i}\rangle}.
\end{equation}
This gives (\ref{eq 2 mixing}) in that case. 
Now, assume that $N\geq n^{2(\gamma/\mu)+1}$ (so that $m\geq n^\gamma$). Consider $\mathbf{z}=(z_1,\ldots,z_t)$ with $z_i\in \text{Ann}(m,M)$ in a regular scale. Also, pick $\mathbf{y}$ on which we do not assume anything. We have,
\begin{eqnarray*}
    \left|\mathbf{P}_\beta^{\mathbf{xy}}[E]-\mathbf{P}_\beta^{\mathbf{xz}}[E]\right|&=&\left|\mathbf{P}_\beta^{\mathbf{xy}}[E]-\sum_{\mathbf{u}}\delta(\mathbf{u},\mathbf{x},\mathbf{y})\mathbf{P}_\beta^{\mathbf{xz}}[E]\right|
    \\&\leq& 
    \left|\mathbf{P}_\beta^{\mathbf{xy}}[E]-\sum_{\mathbf{u}}\delta(\mathbf{u},\mathbf{x},\mathbf{y})\mathbf{P}_\beta^{\mathbf{xu}}[E]\right|+\frac{C_6}{\sqrt{\log(m/n)}}
    \\&\leq& \frac{C_7}{\sqrt{\log(N/n)}},
\end{eqnarray*}
where in the second line we used \eqref{eq: preuve mixing } with $(\mathbf{y},\mathbf{y}')=(\mathbf{u},\mathbf{z})$ together with the fact that $m/n=(N/n)^{\mu/2}\geq n^\gamma $, and in the third line we used (\ref{eq preuve mixing 1}) with $F=\Omega_{\mathbb Z^d}$. This gives (\ref{eq 2 mixing}). 

The same argument works for (\ref{eq 3 mixing}) for every $\mathbf{x},\mathbf{x'},\mathbf{y}$, noticing that for every  regular $\mathbf{u}$ for which $\delta(\mathbf{u},\mathbf{x},\mathbf{y})\neq 0$, using once again $\mathbf{(P2)}$,
\begin{equation}
    \left|\delta(\mathbf{u},\mathbf{x},\mathbf{y})-\delta(\mathbf{u},\mathbf{x'},\mathbf{y})\right|
    \leq C_8  \left(\frac{n}{m}\right)\delta(\mathbf{u},\mathbf{x},\mathbf{y})\leq C_9 \left(\frac{n}{N}\right)^{\mu/2} \delta(\mathbf{u},\mathbf{x},\mathbf{y}).
\end{equation}

To get (\ref{eq 1 mixing}) we repeat the same line of reasoning. We start by applying (\ref{eq preuve mixing 1}). Then, we replace each $\mathbf{P}_\beta^{\mathbf{xu}}[E]$ by $\mathbf{P}_\beta^{\mathbf{xy}}[E]$ using the above reasoning. Finally, replace $\mathbf P_\beta^{\mathbf{uy}}[F]$ by $\mathbf{P}_\beta^{\mathbf{xy}}[F]$ using the same methods. 

\end{proof}
\subsection{Proof of the clustering bound}
With the intersection property and the mixing statement, we are now in a position to conclude. We will use the following (natural) monotonicity property of the currents measures.
\begin{Prop}[Monotonicity in the number of sources, {\cite[Corollary~A.2]{AizenmanDuminilTriviality2021}}]\label{prop: monotonicity sources} For every $\beta>0$, for every $x,y,z,t\in \mathbb Z^d$ and every set $S\subset \mathbb Z^d$, one has,
\begin{multline*}
    \mathbf{P}^{0x,0z,0y,0t}_\beta[\mathbf{C}_{\n_1+\n_3}(0)\cap \mathbf{C}_{\n_2+\n_4}(0)\cap S=\emptyset]\\\leq \mathbf{P}^{0x,0z,\emptyset,\emptyset}_\beta[\mathbf{C}_{\n_1+\n_3}(0)\cap \mathbf{C}_{\n_2+\n_4}(0)\cap S=\emptyset].
\end{multline*}

\end{Prop}
\begin{proof}[Proof of Proposition \textup{\ref{Clustering bound}}] Fix $\gamma>2$ sufficiently large so that Theorem \ref{mixing property} holds. Remember by Remark \ref{rem: growth l} that we may choose $D=D(\gamma)$ sufficiently large in the definition of $\mathcal{L}$ such that $\ell_{k+1}\geq \ell_k^\gamma$. 

We may assume $u=0$. Since $x,y$ are at distance at least $2\ell_K$ of each other, one of them must be at distance at least $\ell_K$ of $u$. Without loss of generality we assume that this is the case of $x$ and make the same assumption about $z$. Let $\delta>0$ to be fixed below.

Let $\mathcal{S}_K^{(\delta)}$ denote the set of subsets of $\lbrace 2\delta K,\ldots, K-3\rbrace$ which contain only even integers. If $\lbrace \mathbf{M}_u(\mathcal{I};\mathcal{L},K)<\delta K\rbrace$ occurs, then there must be $S\in \mathcal{S}_K^{(\delta)}$ with $|S|\geq (1/2-2\delta)K$ such that $\mathfrak{B}_S$ occurs, where $ \mathfrak{B}_S$ is the event that the clusters of $0$ in $\n_1+\n_3$ and $\n_2+\n_4$ do not intersect in any of the annuli $\textup{Ann}(\ell_i,\ell_{i+1})$ for $i \in S$. Using the monotonicity property recalled above,
\begin{align}
	\mathbf{P}_\beta^{0x,0z,0y,0t}[\mathbf{M}_u(\mathcal{I};\mathcal{L},K)<\delta K]
    &\leq \sum_{\substack{S\in \mathcal{S}_K^{(\delta)}\\|S|\geq (1/2-2\delta)K}} \mathbf{P}_\beta^{0x,0z,0y,0t}[\mathfrak{B}_S] \label{eq: clustering bound eq 1}
    \\ &\leq \sum_{\substack{S\in \mathcal{S}_K^{(\delta)}\\|S|\geq (1/2-2\delta)K}} \mathbf{P}_\beta^{0x,0z,\emptyset,\emptyset}[\mathfrak{B}_S].\nonumber 
\end{align}

Let $\mathsf{J}$ be the event defined by
\begin{equation}
\mathsf{J}:=\bigcup_{k=2\delta K}^{K-3}\lbrace \n_1+\n_3\in \mathsf{Jump}(\ell_k,\ell_{k+k^\nu})\rbrace\cup \lbrace \n_2+\n_4\in \mathsf{Jump}(\ell_k,\ell_{k+k^\nu})\rbrace.
\end{equation}
Using Lemma \ref{no jump 1} and Remark \ref{rem: growth l}, if $D$ is large enough, there exist $C_0,C_1,\eta>0$ such that
\begin{equation}\label{eq: borne j complementaire}
    \mathbf{P}_\beta^{0x,0z,\emptyset,\emptyset}[\mathsf{J}]\leq \frac{C_0 K }{\ell_{\delta K}^\eta}\leq C_1 e^{-\eta 2^{\delta K}}.
\end{equation}
Fix some $S\in \mathcal{S}_K^{(\delta)}$. Let $\mathfrak{A}_S$ be the event that none of the events $I_k$ (defined in Definition \ref{def: intersection event}) occur for $k\in S$. 

We now make a crucial observation: if $k\in S$ and $\mathsf{J}^c$ occurs, the events $I_k$ and $\mathfrak{B}_S$ are incompatible. Indeed, the occurrence of $\mathsf{J}^c\cap I_k$ ensures that the only cluster of $\n_1+\n_3$ (resp. $\n_2+\n_4$) crossing $\mathrm{Ann}(\ell_{k},\ell_{k+1})$ is the cluster of $0$. Hence, $\big(\mathfrak{B}_S\cap \mathsf{J}^c\big)\subset \mathfrak{A}_S$. Using \eqref{eq: borne j complementaire},
\begin{equation}
    \mathbf{P}_\beta^{0x,0z,0y,0t}[\mathbf{M}_u(\mathcal{I};\mathcal{L},K)<\delta K]\leq \sum_{\substack{S\in \mathcal{S}_K^{(\delta)}\\|S|\geq (1/2-2\delta)K}} \mathbf{P}_\beta^{0x,0z,\emptyset,\emptyset}[\mathfrak{A}_S]+C_1e^{-\eta 2^{\delta K}}\binom{(1/2-2\delta)K}{2\delta K}. 
\end{equation}
Standard binomial estimates give that the second term on the right-hand side above is bounded by $C_2 2^{-\delta K}$. We are left with the study of $\mathbf{P}_\beta^{0x,0z,\emptyset,\emptyset}[\mathfrak{A}_S]$.

Since $\ell_K\leq L(\beta)$, we know by Proposition \ref{prop: existence regular scales} that there exists $w\in \textup{Ann}(\ell_{K-1},\ell_K)$ in a regular scale. The event $\mathfrak{A}_S$ depends on edges with both endpoints in $\Lambda_{\ell_{K-2}}$. Using the mixing property together with Remark \ref{rem: growth l}, we get that
\begin{equation}
    \mathbf{P}_\beta^{0x,0z,\emptyset,\emptyset}[\mathfrak{A}_S]\leq \mathbf{P}_\beta^{0w,0w,\emptyset,\emptyset}[\mathfrak{A}_S]+\frac{C_3}{(\log \ell_{K-1})^c}\leq \mathbf{P}_\beta^{0w,0w,\emptyset,\emptyset}[\mathfrak{A}_S]+C_4 \left(\frac{D}{\log 2}\right)^{-c(K-1)}.
\end{equation}
Let $s=\max S$. Using once again the mixing property between $n=\ell_{s-1}$ and $N=\ell_s$, we get,
\begin{multline*}
    \mathbf{P}_\beta^{0w,0w,\emptyset,\emptyset}[\mathfrak{A}_S]\leq \mathbf{P}_\beta^{0w,0w,\emptyset,\emptyset}[I_s^c]\mathbf{P}_\beta^{0w,0w,\emptyset,\emptyset}[\mathfrak{A}_{S\setminus \lbrace s\rbrace}]+\frac{C_3}{(\log \ell_s/\ell_{s-1})^c}\\\leq (1-\kappa)\mathbf{P}_\beta^{0w,0w,\emptyset,\emptyset}[\mathfrak{A}_{S\setminus \lbrace s\rbrace}]+C_5(1-\kappa)^{|S|-1},
\end{multline*}
where $\kappa$ is given by Lemma \ref{intersection prop} and where we chose $D$ large enough. Iterating the above yields,
\begin{equation}
    \mathbf{P}_\beta^{0w,0w,\emptyset,\emptyset}[\mathfrak{A}_S]\leq C_6(1-\kappa)^{|S|}.
\end{equation}
Going back to \eqref{eq: clustering bound eq 1}, we get that
\begin{equation}
    \mathbf{P}_\beta^{0x,0z,0y,0t}[\mathbf{M}_u(\mathcal{I};\mathcal{L},K)<\delta K]\leq C_7\binom{(1/2-2\delta)K}{2\delta K}(1-\kappa)^{(1/2-2\delta)K}+C_2 2^{-\delta K}\leq C_8 2^{-\delta K},
\end{equation}
for $\delta>0$ sufficiently small.
\end{proof}
\subsection{Proof of Corollary \ref{thm: main}}\label{section: proof coro main}
Now that we were able to obtain the improved the tree diagram bound, the proof of Corollary \ref{thm: main} follows the strategy used in Section \ref{section: dim eff >4}, although some additional technical difficulties appear due to the lack of knowledge on the growth of $B_L(\beta)$. We refer to \cite{AizenmanDuminilTriviality2021} for the details of the proof.
\begin{proof}[Proof of Corollary \textup{\ref{thm: main}}]
We import notations from the proof of Theorem \ref{d>1}. Applying the improved tree diagram bound we obtain
\begin{equation}\label{eq: proof coro quantitative trivia}
    S(\beta,L,f)\leq C\sum_{\substack{x\in \mathbb Z^d \\ x_1,x_2,x_3,x_4\in \Lambda_{r_f L}}}\dfrac{\langle \sigma_x\sigma_{x_1}\rangle_\beta \langle \sigma_x\sigma_{x_2}\rangle_\beta \langle \sigma_x\sigma_{x_3}\rangle_\beta \langle \sigma_x\sigma_{x_4}\rangle_\beta}{B_{L(x_1,x_2,x_3,x_4)}(\beta)^c\Sigma_L(\beta)^2},
\end{equation}
where $L(x_1,x_2,x_3,x_4)$ is the minimal distance between the $x_i$. We now fix $a\in(0,1)$. The strategy consists in splitting the right-hand side of \eqref{eq: proof coro quantitative trivia} according to the following four possibilities: $x\in \Lambda_{dr_f L}$ and $L(x_1,x_2,x_3,x_4)\leq L^a$, $x\notin \Lambda_{dr_f L}$ and $L(x_1,x_2,x_3,x_4)\leq L^a$, $x\in \Lambda_{dr_f L}$ and $L(x_1,x_2,x_3,x_4)> L^a$, $x\notin \Lambda_{dr_f L}$ and $L(x_1,x_2,x_3,x_4)> L^a$.

From there the conclusion builds on the same tools as the ones used in Section \ref{section: dim eff >4}, and also on Lemma \ref{bubble growth}.
\end{proof}
\begin{Rem} In \cite{AizenmanDuminilTriviality2021}, the proof of this result crucially relies on a \emph{sharp} sliding-scale infrared bound which leads to a sharp bound in Lemma \textup{\ref{bubble growth}}. This shows how remarkable this bound is since it yields a quantitative decay even though we do not even know that $B_L(\beta_c)$ explodes. This argument breaks down in the next section (for $d\in\lbrace 2,3\rbrace$) due to the lack of such a sharp result.
\end{Rem}
\section[Ising models in dimension $1\leq d \leq 3$ with $d_{\rm eff}=4$]{Reflection positive Ising models in dimension $1\leq d \leq 3$ satisfying $d_{\textup{eff}}=4$}\label{section : deff=4} We now explain how to adapt the above strategy to the remaining ``marginal case'': $d_{\textup{eff}}=4>d$. As seen in Section \ref{section: dim eff >4} (and more precisely in Remark \ref{rem: dim eff algebraically decaying int}), for algebraically decaying RP interactions $J_{x,y}=C|x-y|_1^{-d-\alpha}$, this corresponds to choosing $d-2(\alpha\wedge 2)=0$.

We will assume that $J$ satisfies $(\mathbf{A1})$--$(\mathbf{A5})$ with the complementary assumption\footnote{This assumption is a little more restrictive than before since we now require some regularity property on the interaction. This restriction is essentially technical and we believe that the proof which follows should hold under more general assumptions.} that there exist $\mathbf{c},\mathbf{C}>0$ such that, for all $x\in \mathbb Z^d\setminus \lbrace 0\rbrace$, 
\begin{equation}\label{assumption d=2,3}
    \frac{\mathbf{c}}{|x|^{d+\boldsymbol{\alpha}(d)}}\leq J_{0,x}\leq \frac{\mathbf{C}}{|x|^{d+\boldsymbol{\alpha}(d)}},
\end{equation}
where $\boldsymbol{\alpha}(d)=d/2$. Our goal is to prove the following result, which is a slightly stronger version of Theorem \ref{improved diagram bound deff=4}.
\begin{Thm}[Improved tree diagram bound]\label{improved diagram bound deff=4 but more general}
Let $1\leq d \leq 3$. Assume that $J$ satisfies $(\mathbf{A1})$--$(\mathbf{A5})$ together with \eqref{assumption d=2,3}. There exists $C>0$ such that the following holds: for all $\beta\leq \beta_c$, there exists an increasing function $\phi_\beta: \mathbb R\rightarrow \mathbb R_{>}$  such that for all $x,y,z,t\in \mathbb Z^d$ at mutual distance at least  $L$ of each other with $1\leq L\leq L(\beta)$, 
\begin{equation}
    |U_4^\beta(x,y,z,t)|\leq \frac{C}{\phi_\beta(B_L(\beta))}\sum_{u\in \mathbb Z^d}\langle \sigma_x\sigma_u\rangle_\beta \langle \sigma_y\sigma_u\rangle_\beta \langle \sigma_z\sigma_u\rangle_\beta \langle \sigma_t\sigma_u\rangle_\beta.
\end{equation}
If $B(\beta_c)=\infty$, one has $\phi_{\beta_c}(t)\rightarrow \infty$ as $t\rightarrow \infty$.
\end{Thm}
\begin{Rem}\label{rem: how to get better result} We could replace $\phi_\beta(B_L(\beta))$ by $B_L(\beta)^c$ (as in Theorem \textup{\ref{thm: main}}) provided we could prove the following (under the assumptions of Theorem \textup{\ref{improved diagram bound deff=4 but more general}}): there exists $C>0$ such that, if $1\leq \ell\leq L\leq L(\beta)$,
\begin{equation}
    \frac{\chi_N(\beta)}{N^{\gamma(d)}}\leq C\frac{\chi_n(\beta)}{n^{\gamma(d)}}, 
\end{equation}
where $\gamma(2)=1$ and $\gamma(3)=3/2$. This will become more transparent below.

When $d=1$, such a sharp result is obtained thanks to the exact knowledge on the decay of the two-point function.
\end{Rem}
From this result, we will obtain 

\begin{Coro}\label{cor: improved tree diagramm deff=4 but more general} We keep the assumptions of Theorem \textup{\ref{improved diagram bound deff=4 but more general}}. Then, for $\sigma\in(0,d/2)$,
\begin{equation}
    \lim_{\beta\nearrow\beta_c}g_\sigma(\beta)=0.
\end{equation}
As a consequence, for $\beta=\beta_c$, every sub-sequential scaling limit of the model is Gaussian.
\end{Coro}

To prove the improved tree diagram bound, we will extend the strategy of Section \ref{section d=4} and prove a mixing statement together with the intersection property. The proofs heavily relies on a finer analysis of the geometry of the clusters compared to what was done in Section \ref{section: properties of the current trivia} (see Remark \ref{rem: zigzag difficult for deff=4}), that is in fact valid in a wider generality (i.e.\ also when $d_{\textup{eff}}>4$). Hence, we begin by proving the mixing statement in its most general form.

In Sections \ref{section: weak regular scales}, \ref{section: prop currents deff nice}, and \ref{section: mixing for deff nice} we consider an interaction $J$ on $\mathbb Z^d$ ($d\geq 1$) satisfying $(\mathbf{A1})$--$(\mathbf{A5})$ together with the following assumption: there exist $\mathbf{c}_1,\mathbf{C}_1>0$ such that, for all $x\in \mathbb Z^d\setminus \lbrace 0\rbrace$,
\begin{equation}\label{eq: assumption deff at least 4}\tag{$\mathbf{Assumption}_\alpha$}
    \frac{\mathbf{c}_1}{|x|^{d+\alpha}}\leq J_{0,x}\leq \frac{\mathbf{C}_1}{|x|^{d+\alpha}},
\end{equation}
where $\alpha>0$ will be specified below. By the results of Section \ref{section: reflection positivity}, as a consequence, there exists $\mathbf{C}_2>0$ such that: for all $\beta\leq \beta_c$, for all $x\in \mathbb Z^d\setminus\lbrace 0\rbrace$,
\begin{equation}\label{eq: IRB alpha}\tag{$\mathbf{IRB}_\alpha$}
    \langle \sigma_0\sigma_x\rangle_{\beta}\leq \frac{\mathbf{C}_2}{\beta_c|x|^{d-\alpha\wedge 2}(\log |x|)^{\delta_{2,\alpha}}}.
\end{equation}
Moreover, using Proposition \ref{prop: general lower bound}, we find that there exists $c,\mathbf{C}_3>0$ such that: for all $\beta\leq \beta_c$, for all $x\in \mathbb Z^d\setminus\lbrace 0\rbrace$ with $1\leq |x|\leq cL(\beta)$,
\begin{equation}\tag{$\mathbf{LB}_\alpha$}\label{eq: LB alpha}
	\langle \sigma_0\sigma_x\rangle_\beta\geq \frac{\mathbf{C}_3}{\beta |x|^{d-1}f_\alpha(|x|+1)},
\end{equation}
where if $t>1$ $f_\alpha(t):=1$ if $\alpha>1$, $f_1(t):=\log t$, and $f_\alpha(t):=t^{1-\alpha}$ for $\alpha\in(0,1)$.

\subsection{Existence of weak regular scales}\label{section: weak regular scales}

In the case $d\geq3$, the existence of regular scales was proved in Section \ref{section: reflection positivity}. For $d=2$, the proof of $(\mathbf{P4})$ failed. The reason behind this is purely technical: the sliding-scale infrared bound is not optimal in dimension $2$ since we expect the growth of $\chi_n(\beta_c)$ to be smaller than $n^2$. We can circumvent this technical difficulty by allowing ourselves a ``weaker'' property $(\mathbf{P4}')$ in the definition of a regular scale.
\begin{Def}[Weak regular scales] Fix $c,C>0$. An annular region \textup{Ann}$(n/2,8n)$ is said to be $(c,C)$-weak regular if it satisfies the properties $(\mathbf{P1})$--$(\mathbf{P3})$ and
\begin{enumerate}
    \item[$(\mathbf{P4}')$] For every $x\in \Lambda_n$ and $y\notin \Lambda_{C(\log n)^2n}$, $S_{\beta}(y)\leq \frac{1}{2}S_{\beta}(x)$.
\end{enumerate}
A scale $k$ is said to be \emph{weak regular} if $n=2^k$ is such that $\textup{Ann}(n/2,8n)$ is $(c,C)$- weak regular, a vertex $x \in \mathbb Z^d$ will be said to be in a weak regular scale if it belongs to an annulus $\textup{Ann}(n,2n)$ with $n=2^k$ and $k$ a weak regular scale.

\end{Def}

\begin{Prop}[Existence of weak regular scales]\label{prop: existence of weak regular scales} Let $d\geq 1$. Assume that $J$ satisfies $(\mathbf{A1})$--$(\mathbf{A5})$ and \eqref{eq: assumption deff at least 4} where $\alpha>0$ if $d\geq 3$, $\alpha\in (0,1]$ if $d=2$, and $\alpha\in(0,1)$ if $d=1$. Let $\gamma>2$. There exist $c_0,c_1,C_0>0$ such that for every $\beta\leq\beta_c$, and every $1\leq n^\gamma\leq N\leq L(\beta)$, there are at least $c_1\log_2\left(\frac{N}{n}\right)$ $(c_0,C_0)$-weak regular scales between $n$ and $N$.
\end{Prop}
\begin{proof} The cases $d\geq 3$, and $d\in \lbrace 1,2\rbrace$ with $\alpha\in (0,1)$ were already settled in Propositions \ref{prop: existence regular scales} and \ref{existence regular scales bis} since $(\mathbf{P4}')$ is weaker than $(\mathbf{P4})$. We only need to take care of the case $d=2$ and $\alpha=1$. Using \eqref{eq: LB alpha} together with \eqref{eq: IRB alpha}, we get that for $x\in \mathbb Z^d$ with $2\leq |x|\leq L(\beta)$,
\begin{equation}\label{eq: eq weak regular scales 1}
    \frac{c_1}{\beta|x|(\log |x|)}\leq \langle \sigma_0\sigma_x\rangle_\beta\leq \frac{C_1}{\beta|x|}.
\end{equation}
Using  \eqref{eq: eq weak regular scales 1} and the assumption\footnote{In fact $n\leq N/\log N$ is enough.} $1\leq n^\gamma\leq N$, we get the existence of $c_2,c_3,c_4>0$ such that,
\begin{equation}
    \chi_N(\beta)\geq \frac{c_2}{\beta}\frac{N}{\log N} =\frac{c_2}{\beta}\left(\frac{N}{n}\right)\frac{n}{\log N} \geq c_3\left(\frac{N}{n}\right)\frac{1}{\log N}\chi_n(\beta)\geq c_4\left(\frac{N}{n}\right)^{1/2} \chi_n(\beta).
\end{equation}
Using Theorem \ref{sliding scale ir bound}, we find $r,c_5>0$ and independent of $n,N$, such that there are at least $c_5\log_2(N/n)$ scales $m=2^k$ between $n$ and $N$ such that
\begin{equation}\label{eq: eq weak regular scales 2}
    \chi_{rm}(\beta)\geq \chi_{16dm}(\beta)+\chi_m(\beta).
\end{equation}
We prove that such an $m$ is a $(c_0,C_0)$-weak regular scale for a good choice of $c_0,C_0$. The proof of $(\mathbf{P1})$--$(\mathbf{P3})$ follows the same lines as the proof of Proposition \ref{prop: existence regular scales}. Now, using \eqref{eq: eq weak regular scales 1}, we get that for $1\leq \ell \leq L$ with $\ell\leq L(\beta)$,
\begin{equation}\label{eq: eq proof low dim reg scales}
    \frac{\chi_L(\beta)}{L}\leq C_2\frac{\log \ell}{\ell}\chi_\ell(\beta).
\end{equation}
Let $R\geq 1$. Using the same strategy we used to obtain \eqref{eq: proof reg scales 4} and replacing the sliding-scale infrared bound by \eqref{eq: eq proof low dim reg scales}, we get for $y\notin \Lambda_{dR(\log m)^2m}$ and $x\in \Lambda_m$,
\begin{equation}
    |\Lambda_{R(\log m)^2m}|S_{\beta}(y)\leq \chi_{R(\log m)^2m}(\beta)\leq C_3 R(\log m)^3\chi_m(\beta)\leq C_4 R(\log m)^3 m^2S_{\beta}(x),
\end{equation}
which implies that
\begin{equation}
    S_\beta(y)\leq \frac{C_5}{R \log m}S_\beta(x)\leq \frac{1}{2}S_\beta(x),
\end{equation}
if $R$ is large enough. This concludes the proof.
\end{proof}

\subsection{Properties of the current}\label{section: prop currents deff nice} In this subsection, we let $d\geq 1$ and assume that $J$ satisfies $(\mathbf{A1})$--$(\mathbf{A5})$ together with \eqref{eq: assumption deff at least 4} with $d-2(\alpha\wedge 2)\geq 0$. As explained in Section \ref{section: dim eff >4}, this choice corresponds to $d_{\textup{eff}}\geq 4$.

We import the notations from Section \ref{section: properties of the current trivia}. The main difficulty below will come from the fact that $\mathsf{Jump}(k,k+k^\nu)$ (for $\nu<1$) now occurs with high probability (if $\alpha\in(0,2]$). However, at the cost of considering thicker annuli we can keep a similar statement. 

Many of the computations done here are very similar to what was done in Section \ref{section: properties of the current trivia} so we only present the main changes and omit the trivial modifications.

\begin{Lem}\label{no jump 1 deff=4} Let $d\geq 1$. Assume that $J$ satisfies $(\mathbf{A1})$--$(\mathbf{A5})$ and \eqref{eq: assumption deff at least 4} with $\alpha>0$. Let $\epsilon>0$. There exist $c,C,\eta>0$ such that for all $\beta\leq \beta_c$, $y \in \mathbb Z^d$ in a weak regular scale,  with $1\leq |y|\leq cL(\beta)$, and for all $k\geq 1$ such that $k^{4}\leq |y|$,
\begin{equation}
    \mathbf{P}^{0y,\emptyset}_{\beta}[\mathsf{Jump}(k,k^{1+\epsilon})]\leq \frac{C}{k^{\eta}}.
\end{equation}
\end{Lem}
\begin{proof} We repeat the strategy of proof of Lemma \ref{no jump 1}. Lemma \ref{big edge} yields
\begin{multline*}
    \mathbf{P}^{0y,\emptyset}_{\beta}[\mathsf{Jump}(k,k^{1+\epsilon})]\leq 2\beta\sum_{u\in \Lambda_k, \: v \notin \Lambda_{k^{1+\epsilon}}} J_{u,v}\Bigg(\langle \sigma_u\sigma_v\rangle_\beta 
    \\
    +
    \frac{\langle \sigma_0\sigma_u\rangle_\beta\langle \sigma_v\sigma_y\rangle_\beta}{\langle \sigma_0\sigma_y\rangle_\beta}+
    \frac{\langle \sigma_0\sigma_v\rangle_\beta\langle \sigma_u\sigma_y\rangle_\beta}{\langle \sigma_0\sigma_y\rangle_\beta}\Bigg)
    =:A_1+A_2+A_3.
\end{multline*}
Using \eqref{eq: IRB alpha} and \eqref{eq: assumption deff at least 4},
\begin{equation}
    A_1\leq C_1\sum_{u\in \Lambda_k,\: v\notin\Lambda_{k^{1+\epsilon}}}\frac{1}{|u-v|^{d+\alpha+d-\alpha\wedge 2}}\leq C_2\frac{k^d}{k^{(1+\epsilon)(d+\alpha-\alpha\wedge 2)}}.
\end{equation}
Using $(\mathbf{P1})$ of weak regular scales, together with \eqref{eq: IRB alpha} and \eqref{eq: assumption deff at least 4}, we similarly obtain
\begin{equation}
	A_2\leq C_3\frac{k^d}{k^{(1+\epsilon)(d+\alpha-\alpha\wedge 2)}}.
\end{equation}
Finally, proceeding as in the proof of Lemma \ref{no jump 1}, using additionally \eqref{eq: LB alpha},
\begin{equation}
	\beta\sum_{u\in \Lambda_k,\: v\in \Lambda_{|y|/2}(y)}J_{u,v} \frac{\langle \sigma_0\sigma_u\rangle_\beta\langle \sigma_v\sigma_y\rangle_\beta}{\langle \sigma_0\sigma_y\rangle_\beta}\leq C_5\frac{k^{\alpha\wedge 2}|y|^{d-1}f_\alpha(|y|)|y|^{\alpha\wedge 2}}{|y|^{d+\alpha}},
\end{equation}
where the right-hand side is bounded by $C_5k^\alpha/|y|^\alpha$ for $\alpha\in(0,1)$, by $C_5k\log(|y|)/|y|$ for $\alpha=1$, and by $C_5k^{\alpha\wedge 2}/|y|$ for $\alpha>1$. Hence, if $k^{4}\leq |y|$, we always have that it is bounded by $C_4/|y|^{\delta}$ for some $\delta=\delta(\alpha)>0$. Using again $(\mathbf{P1})$ as in Lemma \ref{no jump 1},
\begin{equation}
	\beta\sum_{u\in \Lambda_k,\: v\notin \Lambda_{k^{1+\epsilon}}\cup\Lambda_{|y|/2}(y)}J_{u,v} \frac{\langle \sigma_0\sigma_u\rangle_\beta\langle \sigma_v\sigma_y\rangle_\beta}{\langle \sigma_0\sigma_y\rangle_\beta}\leq C_6\frac{k^{\alpha\wedge 2}}{k^{(1+\epsilon)\alpha}}. 
\end{equation}
This concludes the proof.

\end{proof}
Despite being equipped with a very weak version of Lemma \ref{no jump 1}, we can still obtain a version of Corollary \ref{coro: no jump 1} in our context.
\begin{Coro}[No zigzag for the backbone]\label{coro: no jump 1 deff=4} Let $d\geq 1$. Assume that $J$ satisfies $(\mathbf{A1})$--$(\mathbf{A5})$ and \eqref{eq: assumption deff at least 4} with $d-2(\alpha\wedge 2)\geq 0$. Fix $\nu\in(0,1)$ and $\epsilon>0$. There exist $C,\eta>0$ such that, for all $\beta\leq \beta_c$, for all $k,\ell\geq 1$ and $y\in \mathbb Z^d$ in a weak regular scale with $k^{[2d/(1-\nu)]\vee [2d/(\nu\alpha)]}\leq \ell$ and $\ell^{4}\leq |y|$,
\begin{equation}
    \mathbf{P}_\beta^{0y}[\Gamma(\n_1)\in \mathsf{ZZ}(0,y;k,\ell,\infty)]\leq \frac{C}{\ell^{\eta}}.
\end{equation} 
\end{Coro}
\begin{proof} Notice that
\begin{equation}
    \mathsf{ZZ}(0,y;k,\ell,\infty)\subset \mathsf{ZZ}(0,y;k,\ell,\ell^{1+\epsilon})\cup \mathsf{Jump}(\ell,\ell^{1+\epsilon}).
\end{equation}
Using Lemma \ref{no jump 1 deff=4} we find $C_1,\eta>0$ such that
\begin{equation}
	\mathbf{P}_\beta^{0y}[\mathsf{Jump}(\ell,\ell^{1+\epsilon})]\leq \frac{C_1}{\ell^\eta}.
\end{equation}
If $\mathsf{ZZ}(0,y;k,\ell,\ell^{1+\epsilon})$ occurs, there are two possibilities: either the backbone actually visits $\textup{Ann}(\ell,\ell+\ell^{\nu})$ before hitting $\Lambda_k$, an event we denote by $\mathsf{B}_1$; or it does not in which case there must be an open edge which jumps from $\Lambda_\ell$ to $\textup{Ann}(\ell+\ell^\nu,\ell^{1+\epsilon})$, an event we denote by $\mathsf{B}_2$. By the chain rule for the backbone, we find that,
\begin{equation}
    \mathbf{P}_\beta^{0y}[\mathsf{B}_1]\leq 
    \sum_{\substack{u\in \textup{Ann}(\ell,\ell+\ell^{\nu})\\v\in \Lambda_k}}\frac{\langle \sigma_0\sigma_u\rangle_\beta\langle\sigma_u\sigma_v\rangle_\beta\langle\sigma_v\sigma_y\rangle_\beta}{\langle \sigma_0\sigma_y\rangle_\beta}
    \leq C_2\frac{k^d\ell^{d-1+\nu}}{\ell^{2d-2(\alpha\wedge 2)}}\leq \frac{C_2}{\ell^{(1-\nu)/2}},
\end{equation}
where we used \eqref{eq: IRB alpha}, the property $(\mathbf{P2})$ of weak regular scales to compare $\langle \sigma_v\sigma_y\rangle_\beta$ and $\langle \sigma_0\sigma_y\rangle_\beta$, and the hypothesis $d-2(\alpha\wedge 2)\geq 0$.
Using \eqref{eq: prop backbone mono}, we see that for a one-step walk $\gamma: a \rightarrow b$, one has $\rho(\gamma)\leq \rho_{\lbrace a,b\rbrace}(\gamma)=\tanh(\beta J_{a,b})$. Combining this observation with the chain rule,
\begin{equation}
    \mathbf{P}_\beta^{0y}[\mathsf{B}_2]\leq \sum_{\substack{a\in \Lambda_\ell\\b\in \textup{Ann}(\ell+\ell^{\nu},\ell^{1+\epsilon})\\c\in \Lambda_k}}\frac{\langle \sigma_0\sigma_a\rangle_\beta\tanh(\beta J_{a,b})\langle\sigma_b\sigma_c\rangle_\beta\langle\sigma_c\sigma_y\rangle_\beta}{\langle \sigma_0\sigma_y\rangle_\beta}\leq \frac{C_3k^d\ell^{\alpha\wedge 2}}{\ell^{d-\alpha\wedge 2}\ell^{\nu\alpha}},
\end{equation}
where we used $(\mathbf{P2})$ to compare $\langle \sigma_c\sigma_y\rangle_\beta$ and $\langle \sigma_0\sigma_y\rangle_\beta$, \eqref{eq: IRB alpha} to argue that
\begin{equation}
\sum_{c\in \Lambda_k}\langle \sigma_b\sigma_c\rangle_\beta\leq C_4\frac{k^d}{\ell^{d-\alpha\wedge 2}},
\end{equation}
and \eqref{eq: IRB alpha} once again with \eqref{eq: assumption deff at least 4} to get
\begin{equation}
\sum_{\substack{a\in \Lambda_\ell\\b\notin \Lambda_{\ell+\ell^\nu}}}\langle \sigma_0\sigma_a\rangle_\beta \tanh(\beta J_{a,b})\leq C_5\frac{\ell^{\alpha\wedge 2}}{\ell^{\alpha \nu}}.
\end{equation}
This concludes the proof.
\end{proof}
We can also obtain the corresponding modification of Lemma \ref{no jump 1 bis}.
\begin{Lem}\label{no jump 1 bis deff=4}
Let $d\geq 1$. Assume that $J$ satisfies $(\mathbf{A1})$--$(\mathbf{A5})$ and \eqref{eq: assumption deff at least 4} with $d-2(\alpha\wedge 2)\geq0$. Let $\epsilon>0$. There exist $C,\eta>0$ such that, for all $\beta\leq \beta_c$, for $n<  m\leq M\leq k$ with $1\leq M^{2/\epsilon}\leq k\leq L(\beta)$, for all $x\in \Lambda_n$, and all $u \in \textup{Ann}(m,M)$,
\begin{equation}
    \mathbf{P}_\beta^{xu,\emptyset}[\mathsf{Jump}(k,k^{1+\epsilon})]\leq \frac{C}{k^\eta}.
\end{equation}
\end{Lem}
\begin{proof} We repeat the proof of Lemma \ref{no jump 1 bis}. Using Lemma \ref{big edge},
\begin{multline*}
    \sum_{w\in \Lambda_k, \: v \notin \Lambda_{k^{1+\epsilon}}} \mathbf{P}_\beta^{xu,\emptyset}[\n_{w,v}\geq 1]\\\leq 2\beta\sum_{w\in \Lambda_k, \: v \notin \Lambda_{k^{1+\epsilon}}} J_{w,v}\left(\langle \sigma_w\sigma_v\rangle_\beta +
    \frac{\langle \sigma_x\sigma_w\rangle_\beta\langle \sigma_v\sigma_u\rangle_\beta}{\langle \sigma_x\sigma_u\rangle_\beta}+
    \frac{\langle \sigma_x\sigma_v\rangle_\beta\langle \sigma_w\sigma_u\rangle_\beta}{\langle \sigma_x\sigma_u\rangle_\beta}\right).
\end{multline*}
Using \eqref{eq: assumption deff at least 4} and \eqref{eq: IRB alpha},
\begin{equation}\label{eq: proof no jump 1 bis deff=4}
    \beta\sum_{w\in \Lambda_k, \: v \notin \Lambda_{k^{1+\epsilon}}} J_{w,v}\langle \sigma_w\sigma_v\rangle_\beta\leq C_1 k^{d}\sum_{p\geq k^{1+\epsilon}}\frac{p^{d-1}}{p^{d+\alpha+d-(\alpha\wedge 2)}}\leq \frac{C_2}{k^{d\epsilon}}.
\end{equation}
Then, using \eqref{eq: LB alpha} (which is licit since $1\leq |x-u|\leq L(\beta)$) together with \eqref{eq: assumption deff at least 4} and \eqref{eq: IRB alpha}, we get 
\begin{eqnarray*}
   \beta \sum_{w\in \Lambda_k, \: v \notin \Lambda_{k^{1+\epsilon}}} J_{w,v}\frac{\langle \sigma_x\sigma_v\rangle_\beta\langle \sigma_w\sigma_u\rangle_\beta}{\langle \sigma_x\sigma_u\rangle_\beta}&\leq&
   \beta^2 C_3M^{d-1}f_\alpha(M)\sum_{\substack{w\in \Lambda_k\\ v \notin \Lambda_{k^{1+\epsilon}}}} J_{w,v}\langle \sigma_x\sigma_v\rangle_\beta\langle \sigma_w\sigma_u\rangle_\beta
    \\&\leq& 
    C_4M^{d-1}f_{\alpha}(M)k^{\alpha\wedge 2}\sum_{v\notin \Lambda_{k^{1+\epsilon}}}|v|^{-(d-\alpha\wedge 2)}J_{0,v}
    \\ &\leq & 
    C_5M^{d}k^{\alpha\wedge 2-d(1+\epsilon)}.
\end{eqnarray*}
Finally, with the same reasoning we also get
\begin{equation}
    \sum_{w\in \Lambda_k, \: v \notin \Lambda_{k^{1+\epsilon}}} J_{w,v}\frac{\langle \sigma_x\sigma_w\rangle_\beta\langle \sigma_v\sigma_u\rangle_\beta}{\langle \sigma_x\sigma_u\rangle_\beta}\leq C_6M^{d}k^{\alpha \wedge 2-d(1+\epsilon)}.
\end{equation}
Now, clearly, if $M^{2/\epsilon}\leq k$, using that $d=2(\alpha\wedge 2)\geq 0$, we can choose $\eta=d\epsilon/2$ and $C$ a sufficiently large constant.
\end{proof}
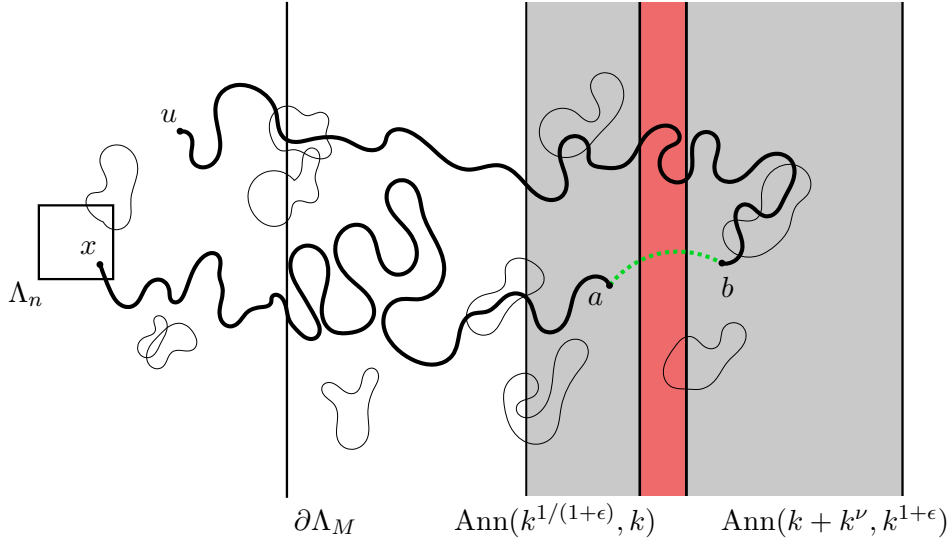
\begin{figure}[htb]
\begin{center}
\includegraphics[scale=1.5]{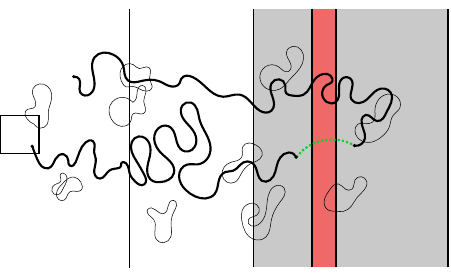}
\put(-230,-6){$\partial \Lambda_M$}
\put(-337,80){$\Lambda_n$}
\put(-280,148){$u$}
\put(-310,98){$x$}
\put(-170,-6){$\textup{Ann}(k^{1/(1+\epsilon)},k)$}
\put(-70,-6){$\textup{Ann}(k+k^{\nu},k^{1+\epsilon})$}
\put(-120,79){$a$}
\put(-70,82){$b$}
\end{center}
\caption{An illustration of the occurrence of the event $\mathsf{B}_2$ defined in the proof of Corollary \ref{coro: no jump 1 deff=4 bis}.\ The ``exclusion zone'' $\textup{Ann}(k,k+k^\nu)$ is represented in red.\ The green dashed line illustrates the long open edge which jumps above it.}
\label{figure: bound Fi deff=4}
\end{figure}
\begin{Coro}\label{coro: no jump 1 deff=4 bis} Let $d\geq 1$. Assume that $J$ satisfies $(\mathbf{A1})$--$(\mathbf{A5})$ and \eqref{eq: assumption deff at least 4} with $d-2(\alpha\wedge 2)\geq 0$. There exist $\nu \in(0,1)$ and $C,\epsilon,\eta>0$ such that for all $\beta\leq \beta_c$, for all $n<  m\leq M\leq k$ with $1\leq M^{[2(1+\epsilon)/\epsilon]\vee [2d/(1-\nu)]}\leq k\leq L(\beta)$, for all $x\in \Lambda_n$ and all $u \in \textup{Ann}(m,M)$,
\begin{equation}
    \mathbf{P}_\beta^{xu}[\mathsf{ZZ}(x,u;M,k,\infty)]\leq \frac{C}{k^\eta}.
\end{equation}
\end{Coro}
\begin{proof} We repeat the strategy used to get Corollary \ref{coro: no jump 1 deff=4}. Let $\nu\in(0,1)$ and $\epsilon>0$ to be fixed below. Notice that,
\begin{multline*}
    \mathsf{ZZ}(x,u;M,k,\infty)\subset \mathsf{ZZ}(x,u;M,k,k^{1+\epsilon})\cap (\mathsf{Jump}(k^{1/(1+\epsilon)},k))^c \\\cup \mathsf{Jump}(k^{1/(1+\epsilon)},k)\cup \mathsf{Jump}(k,k^{1+\epsilon}).
\end{multline*}
Using Lemma \ref{no jump 1 bis deff=4} together with the hypothesis on $M$ and $k$, we find $C_1,\eta>0$ such that
\begin{equation}
    \mathbf{P}_\beta^{xu}[\mathsf{Jump}(k^{1/(1+\epsilon)},k)\cup \mathsf{Jump}(k,k^{1+\epsilon})]\leq \frac{C_1}{k^\eta}.
\end{equation}
We handle $\mathsf{ZZ}(x,u;M,k,k^{1+\epsilon})\cap (\mathsf{Jump}(k^{1/(1+\epsilon)},k))^c$ as we did in the proof of Corollary \ref{coro: no jump 1 deff=4} by splitting it into two events $\mathsf{B}_1$ and $\mathsf{B}_2$ according to whether or not the backbone reaches $\textup{Ann}(k,k+k^\nu)$. Using the chain rule together with \eqref{eq: LB alpha} and \eqref{eq: IRB alpha},
\begin{eqnarray*}
    \mathbf P^{xu}_\beta[\mathsf{B}_1]
    &\leq&
    \sum_{v\in \textup{Ann}(k,k+k^{\nu})}\frac{\langle \sigma_{x}\sigma_{v}\rangle_\beta\langle \sigma_v\sigma_{u}\rangle_\beta}{\langle \sigma_{x}\sigma_{u}\rangle_\beta}
    \\&\leq& \frac{C_2M^{d-1}f_{\alpha}(M)k^{d-1+\nu}}{k^{2d-2(\alpha\wedge 2)}}
    \\&\leq & \frac{C_2M^{d}}{k^{1-\nu}}\leq \frac{C_2}{k^{(1-\nu)/2}},
\end{eqnarray*}
where we used the assumption that $d\geq 2(\alpha\wedge 2)$. 

It remains to analyse $\mathsf{B}_2$. In that case, the backbone has to jump above $\textup{Ann}(k,k+k^{\nu})$: it goes from $x$ to a point in $\textup{Ann}(k^{1/(1+\epsilon)},k)$ (recall that we excluded jumps above $\textup{Ann}(k^{1/(1+\epsilon)},k)$), then jumps in $\textup{Ann}(k+k^{\nu},k^{1+\epsilon})$, before finally hitting $u$ (see Figure \ref{figure: bound Fi deff=4}). 

Using the chain rule for the backbone as we did in the proof of Corollary \ref{coro: no jump 1 deff=4} together with \eqref{eq: assumption deff at least 4}, \eqref{eq: LB alpha}, and \eqref{eq: IRB alpha}, we get
\begin{eqnarray*}
    \mathbf P^{xu}_\beta[\mathsf{B}_2]
    &\leq& 
    \sum_{\substack{a\in \textup{Ann}(k^{1/(1+\epsilon)},k)\\b\in \textup{Ann}(k+k^{\nu},k^{1+\epsilon})}}\frac{\langle \sigma_{x}\sigma_a\rangle_\beta \tanh(\beta J_{a,b})\langle \sigma_b\sigma_{u}\rangle_\beta}{\langle \sigma_{x}\sigma_{u}\rangle_\beta}
    \\&\leq& \frac{C_3M^{d-1}f_\alpha(M) k^d k^{d(1+\epsilon)}}{k^{\frac{d-\alpha\wedge 2}{1+\epsilon}}k^{\nu(d+\alpha)}k^{d-\alpha\wedge 2}}.
\end{eqnarray*}
Now, recall that $M^d\leq k^{d\epsilon}$, so that we may find $\zeta=\zeta(\epsilon,\nu)>0$ such that
\begin{equation}
    \mathbf{P}_\beta^{xu}[\mathsf{B}_2]\leq \frac{C_3k^{2d+3d\epsilon}}{k^{\frac{d-\alpha\wedge 2}{1+\epsilon}}k^{\nu(d+\alpha)}k^{d-\alpha\wedge 2}}\leq \frac{C_3}{k^{\zeta}},
\end{equation}
provided $\epsilon>0$ is small enough, and $\nu$ is sufficiently close to $1$. This concludes the proof.
\end{proof}
Finally, as we did in Section \ref{section: properties of the current trivia}, we conclude this subsection with some properties concerning the current $\n\setminus \overline{\Gamma(\n)}$.
\begin{Lem}\label{no jump 2 deff=4}
Let $d\geq 1$. Assume that $J$ satisfies $(\mathbf{A1})$--$(\mathbf{A5})$ and \eqref{eq: assumption deff at least 4} with $\alpha>0$. Let $\epsilon>0$. There exist $C,\eta>0$ such that for all $\beta\leq \beta_c$, for all $k\geq 1$, for all $x,y\in \mathbb Z^d$, 
\begin{equation}
    \mathbf{P}_\beta^{xy}[\n\setminus\overline{\Gamma(\n)}\in\mathsf{Jump}(k,k^{1+\epsilon})]\leq \mathbf{P}_\beta^{xy,\emptyset}[(\n_1+\n_2)\setminus \overline{\Gamma(\n_1)}\in \mathsf{Jump}(k,k^{1+\epsilon})]\leq \frac{C}{k^\eta}.
\end{equation}
\end{Lem}
\begin{proof} Proceeding exactly as in the proof of Corollary \ref{coro: no jump 2} we find that
\begin{equation}
    \mathbf{P}_\beta^{xy,\emptyset}[(\n_1+\n_2)\setminus \overline{\Gamma(\n_1)}\in \mathsf{Jump}(k,k^{1+\epsilon})]\leq 2\sum_{u\in \Lambda_k,\: v\notin \Lambda_{k^{1+\epsilon}}}\beta J_{u,v}\langle \sigma_u\sigma_v\rangle_\beta.
\end{equation}
This last sum is then smaller than $Ck^{-d\epsilon}$ as shown in \eqref{eq: proof no jump 1 bis deff=4}.
\end{proof}
Recall that the event $\mathsf{Cross}$ was defined in Definition \ref{def: crossing event}.
\begin{Coro}\label{coro: no jump 2 deff=4} Let $d\geq 1$. Assume that $J$ satisfies $(\mathbf{A1})$--$(\mathbf{A5})$ and \eqref{eq: assumption deff at least 4} with $d-2(\alpha\wedge 2)\geq 0$. There exist $\nu\in (0,1)$ and $C,\epsilon,\eta>0$ such that for all $\beta\leq \beta_c$, for all $k,\ell\geq 1$ with $k^{2d/(1-\nu)} \leq \ell$, for all $x,u\in \mathbb Z^d$,
\begin{equation}
    \mathbf{P}_\beta^{xu}[\n\setminus \overline{\Gamma(\n)}\in \mathsf{Cross}(k,\ell)]\leq \mathbf{P}_\beta^{xu,\emptyset}[(\n_1+\n_2)\setminus \overline{ \Gamma(\n_1)}\in \mathsf{Cross}(k,\ell)]\leq \frac{C}{\ell^\eta}.
\end{equation}
\end{Coro}
\begin{proof} We use the ideas developed in the proofs of Lemma \ref{no jump 2} and Corollary \ref{coro: no jump 1 deff=4 bis}. Let $\nu\in(0,1)$ and $\epsilon>0$ to be fixed below. Start by writing,
\begin{equation}
    \lbrace(\n_1+\n_2)\setminus \overline{ \Gamma(\n_1)}\in \mathsf{Cross}(k,\ell)\rbrace= \left[(\mathsf{B}_1\cup \mathsf{B}_2)\cap \mathsf{J}^c\right]\cup \mathsf{J},
\end{equation}
where $\mathsf{B}_1$ is the event that $(\n_1+\n_2)\setminus \overline{ \Gamma(\n_1)}$ crosses $\textup{Ann}(k,\ell)$ by passing through $\textup{Ann}(\ell,\ell+\ell^{\nu})$, $\mathsf{B}_2$ is the complement of $\mathsf{B}_1$ in $\lbrace(\n_1+\n_2)\setminus \overline{ \Gamma(\n_1)}\in \mathsf{Cross}(k,\ell)\rbrace$, and $\mathsf{J}=\lbrace (\n_1+\n_2)\setminus \overline{ \Gamma(\n_1)}\in\mathsf{Jump}(\ell^{1/(1+\epsilon)},\ell)\rbrace\cup\lbrace (\n_1+\n_2)\setminus \overline{ \Gamma(\n_1)}\in\mathsf{Jump}(\ell,\ell^{1+\epsilon})\rbrace $.

Using Lemma \ref{no jump 2 deff=4}, we get the existence of $C_1,\eta_1>0$ such that
\begin{equation}
    \mathbf{P}_\beta^{xu,\emptyset}[\mathsf{J}]\leq \frac{C_1}{\ell^{\eta_1}}.
\end{equation}
Using \eqref{eq: IRB alpha} together with what was done in the proof of Corollary \ref{coro: no jump 2},
\begin{equation}
    \mathbf{P}_\beta^{xu,\emptyset}[\mathsf{B}_1, \: \mathsf{J}^c]\leq \sum_{\substack{v\in \textup{Ann}(\ell,\ell+\ell^{\nu})\\w\in \Lambda_{k}}}\langle \sigma_w\sigma_v\rangle_\beta^2\leq \frac{C_2\ell^{d-1+\nu}k^{d}}{\ell^{2d-2(\alpha\wedge 2)}}\leq \frac{C_2}{\ell^{(1-\nu)/2}},
\end{equation}
where we used that $d-2(\alpha\wedge 2)\geq 0$.
Finally, using a similar strategy as in the proof of Lemma \ref{no jump 2},
\begin{equation}
    \mathbf{P}_\beta^{xu,\emptyset}[\mathsf{B}_2, \: \mathsf{J}^c]
    \leq 
    \sum_{\substack{\gamma:x\rightarrow u\\ \textup{consistent}}}\sum_{\substack{a\in \Lambda_{k}\\b\in \textup{Ann}(\ell^{1/(1+\epsilon)},\ell)\\c\in \textup{Ann}(\ell+\ell^{\nu},\ell^{1+\epsilon})}}\mathbf{P}_\beta^{xu}[\Gamma(\n)=\gamma]\mathbf{P}_{\overline{\gamma}^c,\mathbb Z^d,\beta}^{\emptyset,\emptyset}[a \leftrightarrow b \textup{ in } \overline{\gamma}^c,\: (\n_1+\n_2)_{b,c}\geq 1 ].
\end{equation}
Using the generalisation of the switching lemma mentioned in the proof of Corollary \ref{coro: no jump 2} together with Griffith's inequality,
\begin{eqnarray*}
    \mathbf{P}_{\overline{\gamma}^c,\mathbb Z^d,\beta}^{\emptyset,\emptyset}[a \leftrightarrow b \textup{ in } \overline{\gamma}^c,\: (\n_1+\n_2)_{b,c}\geq 1 ] 
    &=&
    \langle \sigma_a\sigma_b\rangle_{\overline{\gamma}^c,\beta}\langle \sigma_a\sigma_b\rangle_{\beta}
    \mathbf{P}_{\overline{\gamma}^c,\mathbb Z^d,\beta}^{ab,ab}[(\n_1+\n_2)_{b,c}\geq 1]
    \\
    &\leq& 
    \beta J_{b,c}\langle \sigma_a\sigma_c\rangle_{\overline{\gamma}^c,\beta}\langle \sigma_a\sigma_b\rangle_\beta+\beta J_{b,c}\langle \sigma_a\sigma_b\rangle_{\overline{\gamma}^c,\beta}\langle \sigma_a\sigma_c\rangle_\beta
    \\&\leq& 2\langle \sigma_a\sigma_b\rangle_\beta\beta J_{b,c}\langle \sigma_c\sigma_a\rangle_\beta.
\end{eqnarray*}
We obtained,
\begin{equation}
    \mathbf{P}_\beta^{xu,\emptyset}[\mathsf{B}_2, \: \mathsf{J}^c]\leq 2\sum_{\substack{a\in \Lambda_{k}\\b\in \textup{Ann}(\ell^{1/(1+\epsilon)},\ell)\\c\in \textup{Ann}(\ell+\ell^{\nu},\ell^{1+\epsilon})}}\langle \sigma_a\sigma_b\rangle_\beta\beta J_{b,c}\langle \sigma_c\sigma_a\rangle_\beta.
\end{equation}
Now, using \eqref{eq: assumption deff at least 4} and \eqref{eq: IRB alpha},
\begin{equation}
    \mathbf{P}_\beta^{xu,\emptyset}[\mathsf{B}_2, \: \mathsf{J}^c]\leq 
    \frac{C_3k^d\ell^d\ell^{d(1+\epsilon)}}{\ell^{(d-\alpha\wedge 2)}\ell^{\frac{d-\alpha\wedge 2}{1+\epsilon}}\ell^{\nu(d+\alpha)}}.
\end{equation}
We can now use that $k^d\leq \ell^{(1-\nu)}$ and choose $\epsilon>0$ sufficiently small, and $\nu$ sufficiently close to $1$ to conclude.
\end{proof}

\subsection{Mixing property for $d_{\textup{eff}}\geq 4$}\label{section: mixing for deff nice}

The goal of this subsection is to prove the following result.
\begin{Thm}[Mixing property for $d_{\textup{eff}}\geq 4$]\label{thm: mixing general deff=4}
Let $d\geq 1$ and $s\geq 1$. Assume that $J$ satisfies $(\mathbf{A1})$--$\mathbf{(A5)}$ and \eqref{eq: assumption deff at least 4} with $d-2(\alpha\wedge 2)\geq 0$. There exist $\gamma,C>0$, such that for every $1\leq t\leq s$, every $\beta\leq \beta_c$, every $1\leq n^\gamma\leq N\leq L(\beta)$, every $x_i\in \Lambda_n$ and $y_i\notin \Lambda_N$ $(i\leq t)$, and every events $E$ and $F$ depending on the restriction of $(\n_1,\ldots,\n_s)$ to edges with endpoints within $\Lambda_n$ and outside $\Lambda_N$ respectively,
\begin{multline}\label{eq 1 mixing deff=4}
    \left|\mathbf{P}_\beta^{x_1 y_1,\ldots, x_t y_t,\emptyset,\ldots,\emptyset}[E\cap F]-\mathbf{P}_\beta^{x_1 y_1,\ldots, x_t y_t,\emptyset,\ldots,\emptyset}[E]\mathbf{P}_\beta^{x_1 y_1,\ldots, x_t y_t,\emptyset,\ldots,\emptyset}[F]\right|\\\leq C\left(\frac{\log (N/n)}{\log \log (N/n)}\right)^{-1/2}.
\end{multline}
Furthermore, for every $x_1',\ldots, x_t'\in \Lambda_n$ and $y_1',\ldots,y'_t\notin \Lambda_N$, we have that
\begin{equation}\label{eq 2 mixing deff=4}
    \left|\mathbf{P}_\beta^{x_1 y_1,\ldots, x_t y_t,\emptyset,\ldots,\emptyset}[E]-\mathbf{P}_\beta^{x_1 y'_1,\ldots, x_t y'_t,\emptyset,\ldots,\emptyset}[E]\right|\leq C\left(\frac{\log (N/n)}{\log \log (N/n)}\right)^{-1/2},
\end{equation}
\begin{equation}\label{eq 3 mixing deff=4}
    \left|\mathbf{P}_\beta^{x_1 y_1,\ldots, x_t y_t,\emptyset,\ldots,\emptyset}[F]-\mathbf{P}_\beta^{x'_1 y_1,\ldots, x'_t y_t,\emptyset,\ldots,\emptyset}[F]\right|\leq C\left(\frac{\log (N/n)}{\log \log (N/n)}\right)^{-1/2}.
\end{equation}
\end{Thm}

We follow the strategy employed above and import all the notations from Section \ref{section: mixing d=4 proof}. Fix $\beta\leq \beta_c$. Fix two integers $t,s$ satisfying $1\leq t \leq s$. Introduce integers $m,M$ such that $n\leq m\leq M\leq N$, $m/n=(N/n)^{\mu/2}$, and $N/M=(N/n)^{1-\mu}$ for $\mu$ small to be fixed. 

We fix $\nu\in(0,1)$ and $\epsilon>0$ such that Corollaries \ref{coro: no jump 1 deff=4 bis} and \ref{coro: no jump 2 deff=4} hold.

Introduce the set $\mathcal{K}$ of $(c_0,C_0)$-weak regular scales $k$ between $m$ and $M/2$ with every $2^k$ for $k\in \mathcal{K}$ differing by a  multiplicative factor at least $2C_0\log(N/n)^2$. By Proposition \ref{prop: existence of weak regular scales}, we may assume that $|\mathcal{K}|\geq c_1\frac{\log(N/n)}{\log \log (N/n)}$ for a sufficiently small $c_1=c_1(\mu)>0$. Recall that $\mathbf{U}$ was defined in Section \ref{section: mixing d=4 proof}.

The property $(\mathbf{P4}')$ of weak regular scales allows us to prove,
\begin{Lem}[Concentration of $\mathbf{U}$]\label{lem: concentration of N deff=4}
For all $\gamma>2$, there exists $C_1=C_1(d,t,\gamma)>0$ such that for all $n$ sufficiently large satisfying $n^\gamma\leq N\leq L(\beta)$,
\begin{equation}
    \mathbf{E}_\beta^{\mathbf{xy},\emptyset}[(\mathbf{U}-1)^2]\leq C_1\frac{\log \log (N/n)}{\log(N/n)}.
\end{equation}
\end{Lem}

The only place where the argument needs to be adapted is located in the proof of Lemma \ref{technical lemma mixing}, which bounds the occurrence of $\mathcal{G}(\mathbf{u})^c$ (where $\mathcal{G}(\mathbf{u})$ was defined in Definition \ref{def: good event mixing}). The remainder of the subsection concerns the extension of this lemma to our setup.

\begin{Lem}\label{technical lemma mixing deff=4}
We keep the assumptions of Theorem \textup{\ref{thm: mixing general deff=4}}. There exist $C,\delta>0$, $\gamma=\gamma(\delta)>0$ large enough and $\mu=\mu(\delta)>0$ small enough such that for every $n^\gamma\leq N\leq L(\beta)$, and every $\mathbf{u}$ with $u_i\in \mathbb A_{y_i}(2^{k_i})$ with $m\leq 2^{k_i}\leq M/2$ for every $1\leq i \leq t$,
\begin{equation}
    \mathbf P_\beta^{\mathbf{xu},\mathbf{uy}}[\mathcal G(\mathbf{u})^c]\leq C\left(\frac{n}{N}\right)^\delta.
\end{equation}
\end{Lem}
\begin{proof} Recall that we have fixed the values of $\epsilon$ and $\nu$. We follow the notations used in the proof of Lemma \ref{technical lemma mixing}.

Recall that $\mathcal G(\mathbf{u})=\cap_{1\leq i \leq s}G_i$, and $H_i\cap F_i\subset G_i$ where
\begin{equation}
    H_i=\lbrace \n_i\notin \mathsf{Cross}(M,N)\rbrace, \qquad  F_i=\lbrace \n_i'\notin \mathsf{Cross}(n,m)\rbrace.
\end{equation}
Introduce intermediate scales $n\leq r \leq m \leq M\leq R \leq N$ with $r,R$ chosen below.

\paragraph{\textbf{Bound on $H_i$}.} Define, 
\begin{equation}
    \mathsf{J}_i:=\bigcup_{p\in \left\lbrace R,N^{1/(1+\epsilon)},N\right\rbrace}\lbrace \n_i\in \mathsf{Jump}(p,p^{1+\epsilon})\rbrace.
\end{equation}  Notice that,
\begin{multline*}
    \mathbf P^{\mathbf{xu},\mathbf{uy}}_\beta[H_i^c] 
    \leq 
    \mathbf P^{\mathbf{xu}}_\beta[\Gamma(\n_i)\in \mathsf{ZZ}(x_i,u_i;M,R,\infty)]\\+\mathbf P^{\mathbf{xu}}_\beta[\n_i\setminus\overline{\Gamma(\n_i)}\in \mathsf{Cross}(R^{1+\epsilon},N)]+\mathbf{P}_\beta^{\mathbf{xu}}[\mathsf{J}_i].
\end{multline*}

Assume $R=N^\iota$ where $\iota>\mu$ will be fixed below, and recall that $M\leq N^{\mu+1/\gamma}$. We might decrease $\mu$ and increase $\gamma$ to ensure that $(2/\epsilon)(\mu+1/\gamma)\leq \iota$. As a result, we may use Lemma \ref{no jump 1 bis deff=4} to obtain $C_1,\eta_1>0$ such that
\begin{equation}
    \mathbf{P}_\beta^{\mathbf{xu}}[\mathsf{J}_i]\leq \frac{C_1}{R^{\eta_1}}.
\end{equation}
Moreover, thanks to Corollaries \ref{coro: no jump 1 deff=4 bis} and \ref{coro: no jump 2 deff=4}, if we additionally require that\footnote{This might decrease $\mu$ and increase $\gamma$.}:
\begin{equation}
    M^{[2(1+\epsilon)/\epsilon]\vee [2d/(1-\nu)]}\leq R,\qquad R^{2d(1+\epsilon)/(1-\nu)}\leq N,
\end{equation} 
we find $C_2,\eta_2>0$ such that,
\begin{equation}
    \mathbf P^{\mathbf{xu}}_\beta[\Gamma(\n_i)\in \mathsf{ZZ}(x_i,u_i;M,R,\infty)]\leq \frac{C_2}{R^{\eta_2}},\qquad \mathbf P^{\mathbf{xu}}_\beta[\n_i\setminus\overline{\Gamma(\n_i)}\in \mathsf{Cross}(R^{1+\epsilon},N)]\leq \frac{C_2}{R^{\eta_2}}.
\end{equation}

\paragraph{\textbf{Bound on $F_i$}.} We follow the exact same strategy as in the proof of Lemma \ref{technical lemma mixing}. The modifications are similar to what was done for the bound on $H_i$. Again, we replace Corollary \ref{coro: no jump 2} by Corollary \ref{coro: no jump 2 deff=4} and choose accordingly the values of $\mu$ and $\gamma$.

We set $r=m^\kappa$ with $2\kappa d(1+\epsilon)\leq \alpha\wedge 2$. Recall that $m\geq N^{\mu/2}$. We find that
\begin{multline*}
    \mathbf P^{\mathbf{xu},\mathbf{uy}}_\beta[F_i^c]\leq \mathbf P^{\mathbf{uy}}_\beta[\Gamma(\n_i') \in \mathsf{ZZ}(u_i,y_i;r^{1+\epsilon},m,\infty)]\\+\mathbf P^{\mathbf{uy}}_\beta[\n_i'\setminus\overline{\Gamma(\n_i')} \in \mathsf{Cross}(n,r)]+\mathbf P_\beta^{\mathbf{uy}}[\widetilde{K}_i],
\end{multline*}
where $\widetilde{K}_i$ is the event that there exists $a\in \Lambda_{r}$ and $b\notin \Lambda_{r^{1+\epsilon}}$ such that $(\n_i')_{a,b}\geq 2$ and $\lbrace a,b\rbrace\in \overline{\Gamma(\n_i')}\setminus\Gamma(\n_i')$.

Using \eqref{eq: IRB alpha} and the assumption that $u_i\in \mathbb A_{y_i}(2^{k_i})$ to get that $\langle \sigma_v\sigma_{y_i}\rangle_\beta\leq C_3\langle \sigma_{u_i}\sigma_{y_i}\rangle_\beta$, we obtain
\begin{eqnarray*}
    \mathbf P^{\mathbf{uy}}_\beta[\Gamma(\n_i') \in \mathsf{ZZ}(u_i,y_i;r^{1+\epsilon},m,\infty)]
    &\leq& 
    \sum_{v\in \Lambda_{r^{1+\epsilon}}}\frac{\langle \sigma_{u_i}\sigma_v\rangle_\beta \langle \sigma_{v}\sigma_{y_i}\rangle_\beta}{\langle \sigma_{u_i}\sigma_{y_i}\rangle_\beta}
    \\&\leq& 
    C_4\frac{r^{d(1+\epsilon)}}{m^{d-\alpha\wedge 2}}\leq\frac{C_4}{m^{(\alpha\wedge 2)/2}},
\end{eqnarray*}
where we used that $d-\alpha\wedge 2\geq \alpha\wedge 2$.
Moreover, using Corollary \ref{coro: no jump 2 deff=4} (which requires that $n^{2d/(1-\nu)}\leq r$ and hence decreases the values of $\mu$ and $1/\gamma$), there exist $\zeta>0$ such that
\begin{equation}
    \mathbf P^{\mathbf{uy}}_\beta[\n_i'\setminus\overline{\Gamma(\n_i')} \in \mathsf{Cross}(n,r)]\leq \frac{C_5}{r^\zeta}.  
\end{equation}

We conclude the proof with the bound on $\widetilde{K}_i$. Proceeding as in the proof of Lemma \ref{technical lemma mixing},
\begin{equation}
    \mathbf P_\beta^{\mathbf{uy}}[\widetilde{K}_i]\leq C_6\sum_{\substack{a\in \Lambda_{r}\\b\notin \Lambda_{r^{1+\epsilon}}}}\beta J_{a,b}\langle \sigma_{u_i}\sigma_b\rangle_\beta.
\end{equation}
Using \eqref{eq: assumption deff at least 4} and \eqref{eq: IRB alpha} we obtain $\zeta'>0$ such that
\begin{equation}
    \mathbf{P}_\beta^{\mathbf{uy}}[K_i]\leq \frac{C_7}{m^{\zeta'}}.
\end{equation}
This concludes the proof
\end{proof}
We are now in a position to conclude.
\begin{proof}[Proof of Theorem \textup{\ref{thm: mixing general deff=4}}] The proof follows the exact same lines as the proof of Theorem \ref{mixing property} except that we replace Lemma \ref{technical lemma mixing} by Lemma \ref{technical lemma mixing deff=4}.
\end{proof}

\subsection{Proof of Theorem \ref{improved diagram bound deff=4 but more general}}

The lack of precision of the sliding-scale infrared bound in comparison to the situation in $d=4$ slightly weakens the result. The three cases of interest are a little different and thus are treated in different sections.

\subsubsection{The case $d=3$}\label{section: case d=3} We assume that $d=3$ and that $J$ satisfies $(\mathbf{A1})$--$(\mathbf{A5})$ and \eqref{eq: assumption deff at least 4} with $\alpha=3/2$. 

Let us first observe that the sliding-scale infrared bound of Theorem \ref{sliding scale ir bound} is not sharp in this setup. Indeed, we expect the finite-volume susceptibility to grow like $n^{3/2}$ (below $L(\beta)$). \eqref{eq: LB alpha} and \eqref{eq: IRB alpha} can (almost) make up for this lack of precision as they yield the existence of $C>0$ such that for $1\leq n \leq N \leq L(\beta)$,
\begin{equation}\label{eq: sliding scale 3d}
    \frac{\chi_N(\beta)}{N^{3/2}}\leq C\sqrt{n}\frac{\chi_n(\beta)}{n^{3/2}}.
\end{equation}
Recall that \eqref{eq: IRB alpha} still gives $B_L(\beta)-B_\ell(\beta)\leq C_0\log(L/\ell)$ in our setup. However, \eqref{eq: sliding scale 3d} not being sharp, we modify Lemma \ref{bubble growth} accordingly.
\begin{Lem}\label{bubble growth d=3} There exists $C>0$ such that for every $\beta\leq \beta_c$, and for every $1\leq\ell\leq L\leq L(\beta)$,
\begin{equation}
    B_L(\beta)\leq\left(1+C\frac{\log_2(L/\ell)}{\log_2(\ell)}\ell\right)B_\ell(\beta).
\end{equation}
\end{Lem}
\begin{proof} We repeat the proof of Lemma \ref{bubble growth} except that we replace the use of Theorem \ref{sliding scale ir bound} by \eqref{eq: sliding scale 3d}.
\end{proof}

We define a (possibly finite) sequence $\mathcal{L}_3=\mathcal{L}_3(\beta,D)$ by $\ell_0=0$ and
\begin{equation}
    \ell_{k+1}=\inf\left\lbrace \ell\geq \ell_k, \: B_\ell(\beta)\geq  D\cdot(\ell_k+1) \cdot B_{\ell_k}(\beta)\right\rbrace.
\end{equation}
We also define a sequence $\mathcal{U}_3=\mathcal{U}_3(\beta,D)$ by $u_k=\ell_{3k}$ for $k\geq 0$.

\begin{Rem} Note that the sequence $\mathcal{L}_3$ grows much faster than $\mathcal{L}$ introduced in Section \textup{\ref{section d=4}}. The reason why we need an additional sequence $\mathcal{U}_3$ is technical and will become transparent later.
\end{Rem}
We begin with a technical result on $\mathcal{L}_3$.
\begin{Lem}[Growth of $\mathcal{L}_3$]\label{lem: growth of bubble d=3} There exist $c,C>0$ such that, for all $k\geq 1$,
\begin{equation}
    \prod_{i=0}^{k-1}[D\cdot(\ell_i+1)]\leq B_{\ell_k}(\beta)\leq C \prod_{i=0}^{k-1}[D\cdot(\ell_i+1)] ,
\end{equation}
and, as long as $\ell_{k+1}\leq L(\beta)$,
\begin{equation}
    \ell_{k+1}\geq \ell_k^{cD}.
\end{equation}
\end{Lem}
\begin{proof} We repeat the argument used to study $\mathcal{L}$ in Section \ref{section d=4}. The lower bound is immediate and for the upper bound, using that $B_L(\beta)-B_\ell(\beta)\leq C_0\log(L/\ell)$, for $k\geq 1$,
\begin{eqnarray*}
    B_{\ell_k-1}(\beta)&\leq&
    D(\ell_{k-1}+1)B_{\ell_{k-1}}(\beta)
    \leq D(\ell_{k-1}+1) \left(B_{\ell_{k-1}-1}(\beta)-C_0\log\left(1-\frac{1}{\ell_{k-1}}\right)\right)
    \\&\leq& 
    \left(\prod_{i=1}^{k-1}[D\cdot(\ell_i+1)]\right)B_{\ell_1-1}(\beta)+C_0\sum_{i=1}^{k-1}\frac{\prod_{j=k-i}^{k-1}[D\cdot(\ell_j+1)]}{\ell_{k-i}}
    \\&\leq& C\left(\prod_{i=0}^{k-1}[D\cdot(\ell_i+1)]\right).
\end{eqnarray*}
for $C$ large enough (independent of $D$ and $k$). We conclude by noticing that
\begin{equation}
    B_{\ell_k}(\beta)\leq B_{\ell_k-1}(\beta)+C_0\log 2.
\end{equation}
As for the growth of $\mathcal{L}_3$, we use Lemma \ref{lem: growth of bubble d=3} and proceed as in Remark \ref{rem: growth l}
to get that
\begin{equation}
    \log_2(\ell_k)\leq C\ell_k\log_2(\ell_{k+1}/\ell_k)\frac{B_{\ell_k}(\beta)}{B_{\ell_{k+1}}(\beta)-B_{\ell_k}(\beta)}\leq C\ell_k\log_2(\ell_{k+1}/\ell_k)\frac{1}{(D\ell_k-1)},
\end{equation}
which yields, for some $c_1>0$, $\ell_{k+1}\geq \ell_k^{c_1D}$. This concludes the proof.
\end{proof}
\begin{Rem} The second part of the above statement is in some sense the most important one since it ensures that there is ``room'' between successive scales. This has been used before to apply the results of Section \textup{\ref{section: properties of the current trivia}} and Theorem \textup{\ref{mixing property}}. It will also be useful in our case, and it explains the introduction of the additional multiplicative factor $\ell_k$ in the definition of $\mathcal{L}_3$. This choice backfires when we try to estimate $B_{\ell_k}(\beta)$.
\end{Rem}
Our goal now is to prove,
\begin{Prop}[Clustering bound for $d=3$ and $\alpha=3/2$]\label{Clustering bound deff=4} 
For $D$ large enough, there exists $\delta=\delta(D)>0$ such that for all $\beta\leq \beta_c$, for all $K>3$ with $u_{K+1}\leq L(\beta)$, and for all $v,x,y,z,t\in \mathbb Z^3$ with mutual distance between $x,y,z,t$ larger than $2u_{K}$,
\begin{equation}
    \mathbf{P}_\beta^{vx,vz,vy,vt}[\mathbf{M}_u(\mathcal{I};\mathcal{U}_3,K)<\delta K]\leq 2^{-\delta K}.
\end{equation}
\end{Prop}
We first see how this result implies Theorem \ref{improved diagram bound deff=4 but more general}.
\begin{proof}[Proof of Theorem \textup{\ref{improved diagram bound deff=4 but more general}} for $d=3$] We follow the proof of Section \ref{section d=4}. The only change occurs in the connection between $L$ and $B_L(\beta)$ when $L=2u_{K}$. Using Lemma \ref{lem: growth of bubble d=3}, we find that
\begin{equation}
    B_L(\beta)\leq B_{\ell_{3K+1}}(\beta)\leq C\left(\prod_{i=0}^{3K}[D\cdot(\ell_i+1)]\right)=:\Pi_{\beta,D}(3K),
\end{equation}
so if $\Phi=\Phi_{\beta,D}$ is defined for $t\geq 1$ by:
\begin{equation}
    \Phi(t):=\inf\lbrace k\geq 0: \: \Pi_{\beta,D}(3k)\geq t\rbrace\wedge \left(\lfloor N_0(\mathcal{L}_3)/3\rfloor+1\right) ,
\end{equation}
where $N_0(\mathcal{L}_3)$ is the index of the last element of $\mathcal{L}_3$ (possibly equal to $\infty$), we find that $K\geq \Phi(B_L(\beta))$. This gives the result setting $\phi_\beta(t):=2^{-\delta \Phi_{\beta,D}(t)/5}$.
\end{proof}
As before, Proposition \ref{Clustering bound deff=4} will follow from a combination of Theorem \ref{thm: mixing general deff=4} and of an intersection property. We begin by modifying the definition of the intersection event. 

Below, we fix $\nu\in  (0,1)$ and $\epsilon>0$ such that the results of Section \ref{section: prop currents deff nice} hold.
\begin{Def}[Intersection event for $d=3$ and $\alpha=3/2$]\label{def: intersection event d=3} Let $k\geq 1$ and $y\notin \Lambda_{u_{k+2}}$. A pair of currents $(\n,\m)$ with $(\sn,\sm)=(\lbrace 0,y\rbrace,\lbrace 0,y\rbrace)$ realises the event $\widetilde{I}_k$ if the following properties are satisfied:
\begin{enumerate}
    \item[$(i)$] The restrictions of $\n$ and $\m$ to edges with both endpoints in \textup{Ann}$(\ell_{3k},\ell_{3k+3}^{1+\epsilon})$ contain a unique cluster ``strongly crossing'' \textup{Ann}$(\ell_{3k},\ell_{3k+3}^{1+\epsilon})$, in the sense that it contains a vertex in \textup{Ann}$(\ell_{3k},\ell_{3k}^{1+\epsilon})$ and a vertex in \textup{Ann}$(\ell_{3k+3},\ell_{3k+3)}^{1+\epsilon})$.
    \item[$(ii)$] The two clusters described in $(i)$ intersect.
\end{enumerate}

\end{Def}
\begin{Lem}[Intersection property for $d=3$ and $\alpha=3/2$]\label{intersection prop deff=4 d=3} For $D$ large enough, there exists $\kappa>0$ such that for every $\beta\leq \beta_c$, every $k\geq 2$, and every $y\notin \Lambda_{u_{k+2}}$ in a weak regular scale with $1\leq|y|\leq L(\beta)$,
\begin{equation}
    \mathbf{P}_\beta^{0y,0y,\emptyset,\emptyset}[(\n_1+\n_3,\n_2+\n_4)\in \widetilde{I}_k]\geq \kappa.
\end{equation}
\end{Lem}
\begin{proof} We repeat the two-step proof done in the preceding section. Introduce intermediate scales $u_k=\ell_{3k}\leq n\leq 
m\leq M \leq N \leq \ell_{3k+3}=u_{k+1}$ with $n=\sqrt{\ell_k \ell_{k+1}}$, $N=\sqrt{\ell_{k+2}\ell_{k+3}}$, $m=\ell_{3k+1}$, and $M=\ell_{3k+2}$. Keeping the same notations as in the proof of Lemma \ref{intersection prop}, one has for some $c_1>0$,
\begin{equation}
    \mathbf E_\beta^{0y,0y,\emptyset,\emptyset}[|\mathcal{M}|]\geq c_1(B_M(\beta)-B_{m-1}(\beta)),
\end{equation}
and for some $c_2>0$
\begin{equation}
    \mathbf E_\beta^{0y,0y,\emptyset,\emptyset}[|\mathcal{M}|^2]\leq c_2(B_M(\beta)-B_{m-1}(\beta))B_{2M}(\beta).
\end{equation}
Now, by definition of $\mathcal{L}_3$, one has $B_M(\beta)\geq D(\ell_{3k+1}+1)B_m(\beta)$ so that $B_M(\beta)-B_{m-1}(\beta)\geq \frac{B_M(\beta)}{2}$ for $D$ large enough. Moreover, by \eqref{eq: IRB alpha}, $B_{2M}(\beta)\leq B_M(\beta)+C\log 2$. As a result, we may find $c_3>0$ such that,
\begin{equation}
    \mathbf{P}_\beta^{0y,0y,\emptyset,\emptyset}[|\mathcal{M}|>0]\geq c_3.
\end{equation}
The conclusion of the proof follows the same lines as in Lemma \ref{intersection prop}: we replace Lemma \ref{no jump 1} by Lemma \ref{no jump 1 deff=4}, and Corollaries \ref{coro: no jump 1} and \ref{coro: no jump 2} by Corollaries \ref{coro: no jump 1 deff=4} and \ref{coro: no jump 2 deff=4}. The proof is enabled by Lemma \ref{lem: growth of bubble d=3} which ensures that the different scales are sufficiently ``distanced'' when $D$ is large enough.
\end{proof}
We are now equipped to prove Proposition \ref{Clustering bound deff=4}.
\begin{proof}[Proof of Proposition \textup{\ref{Clustering bound deff=4}}]
Fix $\gamma>2$ sufficiently large so that Theorem \ref{thm: mixing general deff=4} holds. Remember by Lemma \ref{lem: growth of bubble d=3} that we may choose $D=D(\gamma)$ sufficiently large in the definition of $\mathcal{L}_3$ such that $\ell_{k+1}\geq \ell_k^\gamma$. 
We may assume $v=0$. Since $x,y$ are at distance at least $2u_K$ of each other, one of them must be at distance at least $u_K$ of $u$. Without loss of generality we assume that this is the case of $x$ and make the same assumption about $z$. Let $\delta>0$ to be fixed below. Recall that $\mathcal{S}_K^{(\delta)}$ denote the set of subsets of $\lbrace 2\delta K,\ldots, K-3\rbrace$ which contain only even integers.

As in the proof of Proposition \ref{Clustering bound},
\begin{equation}
	\mathbf{P}_\beta^{0x,0z,0y,0t}[\mathbf{M}_u(\mathcal{I};\mathcal{U}_3,K)<\delta K]
    \leq \sum_{\substack{S\in \mathcal{S}_K^{(\delta)}\\|S|\geq (1/2-2\delta)K}} \mathbf{P}_\beta^{0x,0z,\emptyset,\emptyset}[\mathfrak{B}_S].
\end{equation}
Let $\mathsf{J}$ be the event defined by
\begin{equation}
\mathsf{J}:=\bigcup_{k=2\delta K}^{K-3}\lbrace \n_1+\n_3\in \mathsf{Jump}(u_k,u_{k}^{1+\epsilon})\rbrace\cup \lbrace \n_2+\n_4\in \mathsf{Jump}(u_{k},u_{k}^{1+\epsilon})\rbrace.
\end{equation}
Using Lemmas \ref{no jump 1 deff=4} and \ref{lem: growth of bubble d=3}, if $D$ is large enough, there exists $C_0,C_1,\eta>0$ such that
\begin{equation}\label{eq: borne j complementaire d=3}
    \mathbf{P}_\beta^{0x,0z,\emptyset,\emptyset}[\mathsf{J}]\leq \frac{C_0 K }{\ell_{3\delta K}^\eta}\leq C_1 e^{-\eta 2^{\delta K}}.
\end{equation}

Fix some $S\in \mathcal{S}_K^{(\delta)}$. Let $\widetilde{\mathfrak{A}}_S$ be the event that none of the events $\widetilde{I}_k$ (defined in Definition \ref{def: intersection event d=3}) occur for $k\in S$. As above, if $k\in S$ and $\mathsf{J}^c$ occurs, the events $\widetilde{I}_k$ and $\mathfrak{B}_S$ are incompatible. Using \eqref{eq: borne j complementaire d=3},
\begin{equation}
    \mathbf{P}_\beta^{0x,0z,0y,0t}[\mathbf{M}_u(\mathcal{I};\mathcal{U}_3,K)<\delta K]\leq \sum_{\substack{S\in \mathcal{S}_K^{(\delta)}\\|S|\geq (1/2-2\delta)K}} \mathbf{P}_\beta^{0x,0z,\emptyset,\emptyset}[\widetilde{\mathfrak{A}}_S]+C_1 e^{-\eta 2^{\delta K}}\binom{(1/2-2\delta)K}{2\delta K}. 
\end{equation}
From this point, the analysis follows the exact same lines as before and we refer to the proof of Proposition \ref{Clustering bound} for the rest of the argument.
\end{proof}
\subsubsection{The case $d=2$}
We now assume that $d=2$ and that $J$ satisfies $(\mathbf{A1})$--$(\mathbf{A5})$ and \eqref{eq: assumption deff at least 4} with $\alpha=1$. 

As before, the sliding-scale infrared bound of Theorem \ref{sliding scale ir bound} is not sharp since we expect the finite-volume susceptibility to grow like $n$ (below $L(\beta)$). Proposition \ref{prop: general lower bound} and \eqref{eq: IRB alpha} yield the existence of $C>0$ such that for $1\leq n \leq N \leq L(\beta)$,
\begin{equation}\label{eq: sliding scale 2d}
    \frac{\chi_N(\beta)}{N}\leq C\log n\frac{\chi_n(\beta)}{n}.
\end{equation}

\begin{Lem}\label{bubble growth d=2} There exists $C>0$ such that for every $\beta\leq \beta_c$, and for every $1\leq\ell\leq L\leq L(\beta)$,
\begin{equation}
    B_L(\beta)\leq\left(1+C\frac{\log_2(L/\ell)}{\log_2(\ell)}\log \ell\right)B_\ell(\beta).
\end{equation}
\end{Lem}
\begin{proof} We repeat the proof of Lemma \ref{bubble growth d=3} using this time \eqref{eq: sliding scale 2d}.
\end{proof}
We define a (possibly finite) sequence $\mathcal{L}_2=\mathcal{L}_2(\beta,D)$ by $\ell_0=1$ and
\begin{equation}
    \ell_{k+1}=\inf\left\lbrace \ell\geq \ell_k: \: B_\ell(\beta)\geq D\cdot(\log (\ell_k)+1)\cdot B_{\ell_k}(\beta)\right\rbrace.
\end{equation}
We also define a sequence $\mathcal{U}_2=\mathcal{U}_2(\beta,D)$ by $u_k=\ell_{3k}$ for $k\geq 0$.  Adapting the proof of Lemma \ref{lem: growth of bubble d=3} to our setup, we obtain,
\begin{Lem}[Growth of $\mathcal{L}_2$]\label{lem: growth of bubble d=2} There exists $c,C_1,C_2>0$ such that, for all $k\geq 1$,
\begin{equation}
    \prod_{i=0}^{k-1}[D\cdot(\log(\ell_i)+1)]\leq B_{\ell_k}(\beta)\leq C \prod_{i=0}^{k-1}[D\cdot(\log(\ell_i)+1)] ,
\end{equation}
and as long as $\ell_{k+1}\leq L(\beta)$,
\begin{equation}
    \ell_{k+1}\geq \ell_k^{cD}.
\end{equation}

\end{Lem}
The second part of Theorem \ref{improved diagram bound deff=4} will follow from the following proposition.
\begin{Prop}[Clustering bound for $d=2$ and $\alpha=1$]\label{Clustering bound deff=4 d=2} 
For $D$ large enough, there exists $\delta=\delta(D)>0$ such that for all $\beta\leq \beta_c$, for all $K>3$ with $u_{K+1}\leq L(\beta)$, and for all $v,x,y,z,t\in \mathbb Z^2$ with mutual distance between $x,y,z,t$ larger than $2u_{K}$,
\begin{equation}
    \mathbf{P}_\beta^{vx,vz,vy,vt}[\mathbf{M}_u(\mathcal{I};\mathcal{U}_2,K)<\delta K]\leq 2^{-\delta K}.
\end{equation}
\end{Prop}
\begin{proof} The proof follows the exact same lines as the proof of Proposition \ref{Clustering bound deff=4} (in particular, we keep the same intersection event $\widetilde{I}_k$).    
\end{proof}
As above, this result easily implies \ref{improved diagram bound deff=4 but more general} for $d=2$.
\begin{proof}[Proof of Theorem \textup{\ref{improved diagram bound deff=4 but more general}} for $d=2$] We follow the proof of Section \ref{section: case d=3}.  Using Lemma \ref{lem: growth of bubble d=2}, we find that
\begin{equation}
    B_L(\beta)\leq B_{\ell_{3K+1}}(\beta)\leq C \prod_{i=0}^{3K}[D\cdot(\log(\ell_i)+1)]=:\Pi'_{\beta,D}(3K),
\end{equation}
so that $K\geq \Phi'(B_L(\beta))$ where $\Phi'=\Phi'_{\beta,D}$ is defined for $t\geq 1$ by:
\begin{equation}
    \Phi'(t):=\inf\lbrace k\geq 0: \: \Pi'(3k)\geq t\rbrace\wedge \left(\lfloor N_0(\mathcal{L}_2)/3\rfloor+1\right).
\end{equation}
This concludes the proof.
\end{proof}
\subsubsection{The case $d=1$} Finally, we treat the case $d=1$. Assume that $J$ satisfies $(\mathbf{A1})$--$(\mathbf{A5})$ and \eqref{eq: assumption deff at least 4} with $\alpha=1/2$. The results of Section \ref{section: reflection positivity} give us the good rate of decay for $S_\beta$: there exist $c,C>0$ such that for all $x\in\mathbb Z^d$ with $1\leq |x|\leq L(\beta)$,
\begin{equation}\label{eq: exact asymp d=1}
    \frac{c}{|x|^{1/2}}\leq \langle \sigma_0\sigma_x\rangle_\beta\leq \frac{C}{|x|^{1/2}}.
\end{equation}
This observation greatly simplifies the proof\footnote{To avoid writing yet another proof of triviality we simply import the results of Section \ref{section d=4}. However, note that the knowledge of the critical exponent $\eta$ yields a shorter proof of the improved tree diagram bound, see \cite[Section~4]{AizenmanDuminilTriviality2021}.} and allows one to proceed like in Section \ref{section d=4}.

As before, define a (possibly finite) sequence $\mathcal{L}_1=\mathcal{L}_1(\beta,D)$ by $\ell_0=0$ and
\begin{equation}
    \ell_{k+1}=\inf\left\lbrace \ell\geq \ell_k: \: B_\ell(\beta)\geq DB_{\ell_k}(\beta)\right\rbrace.
\end{equation}
With the above work, it is easy to obtain,
\begin{Prop}[Clustering bound for $d=1$ and $\alpha=1/2$]\label{Clustering bound d=1}
For $D$ large enough, there exists $\delta=\delta(D)>0$ such that for all $\beta\leq \beta_c$, for all $K>3$ with $\ell_{K+1}\leq L(\beta)$, and for all $u,x,y,z,t\in \mathbb Z$ with mutual distance between $x,y,z,t$ larger than $2\ell_{K}$,
\begin{equation}
    \mathbf{P}_\beta^{ux,uz,uy,ut}[\mathbf{M}_u(\mathcal{I};\mathcal{L},K)<\delta K]\leq 2^{-\delta K}.
\end{equation}
\end{Prop}
From this result and \eqref{eq: exact asymp d=1}, we can obtain Theorem \ref{improved diagram bound deff=4 but more general} with $\phi_\beta(B_L(\beta))=(\log L)^c$ for some constant $c>0$.
\subsection{Proof Corollary \textup{\ref{cor: improved tree diagramm deff=4 but more general}}}
\begin{proof}[Proof of Corollary \textup{\ref{cor: improved tree diagramm deff=4 but more general}}] We use the same strategy as in Appendix \ref{appendix: bubble finite}. In particular, we will use Proposition \ref{Sokal bound}. Recall that $\alpha=d/2$. Let $\sigma\in(0,d/2)$ so that $\xi_{\sigma}(\beta)$ is well defined for all $\beta<\beta_c$. We begin by noticing that there exists $C>0$ such that, for $\beta<\beta_c$,
\begin{equation}\label{eq: proof deff=4 finale 1}
    \chi(\beta)\leq C\xi_{\sigma}(\beta)^{d/2}.
\end{equation}
Indeed, if $K>0$, \eqref{eq: IRB alpha} implies that for some $C_1=C_1(d)>0$,
\begin{equation}
    \chi_{K\xi_{\sigma}(\beta)}(\beta)\leq C_1(K\xi_{\sigma}(\beta))^{d/2}.
\end{equation}
Moreover, using Proposition \ref{Sokal bound},
\begin{equation}
    \chi(\beta)-\chi_{K\xi_\sigma(\beta)}(\beta)\leq C_2\frac{\chi(\beta)}{K^\sigma}.
\end{equation}
Combining the two last inequalities, and choosing $K$ large enough, we obtain \eqref{eq: proof deff=4 finale 1}. 

Now, assume that we are given $1\leq L\leq L(\beta)$. Write,
\begin{equation}
    0\leq g_\sigma(\beta)\leq A_1+A_2,
\end{equation}
where 
\begin{equation}
    A_1:= -\frac{1}{\chi(\beta)^2\xi_\sigma(\beta)^d}\sum_{\substack{x,y,z\in\mathbb Z^d\\L(0,x,y,z)\leq L}}U_4^\beta(0,x,y,z),
\quad
    A_2:= -\frac{1}{\chi(\beta)^2\xi_\sigma(\beta)^d}\sum_{\substack{x,y,z\in\mathbb Z^d\\L(0,x,y,z)> L}}U_4^\beta(0,x,y,z).
\end{equation}
Using the (standard) tree diagram bound \eqref{eq tree diagram bound classic}, we get that
\begin{equation}
    A_1\leq C_3 L^d \frac{\chi(\beta)}{\xi_\sigma(\beta)^d}\leq C_4L^d\xi_\sigma(\beta)^{-d/2}.
\end{equation}
Moreover, using this time Theorem \ref{improved diagram bound deff=4 but more general},
\begin{equation}
    A_2\leq \frac{C_5}{\phi_\beta(B_L(\beta))}\frac{\chi(\beta)^2}{\xi_\sigma(\beta)^d}\leq \frac{C_6}{\phi_\beta(B_L(\beta))}.
\end{equation}
As a result, for any $L\geq 1$,
\begin{equation}
    \limsup_{\beta\nearrow \beta_c}g_\sigma(\beta)\leq \frac{C_6}{\phi_{\beta_c}(B_L(\beta_c))}.
\end{equation}
Hence, if $B(\beta_c)=\infty$, one obtains the result taking $L$ to infinity. If $B(\beta_c)<\infty$, we may conclude using Theorem \ref{thm: bubble condition implies triviality}.
\end{proof}

\section[Extension of the results to models in the GS class]{Extension of the results to models in the Griffiths--Simon class}\label{section: gs class}
In this section, we extend the results of Sections \ref{section d=4} and \ref{section : deff=4} to the case of single-site measures in the GS class.
Let us mention that using Remark \ref{rem: extension to gs of deff>4} we can adapt the proof of Theorem \ref{thm: main m2(J) infinite} to the case of measures in the GS class.

We focus on the results of Section \ref{section d=4} and briefly explains in Section \ref{section: extension deff=4 gs} how similar considerations permit to extend the results of Section \ref{section : deff=4}. Let $J$ be an interaction satisfying $\mathbf{(A1)}$--$\mathbf{(A6)}$, and $\rho$ be a measure in the GS class.

Since the measure $\rho$ might be of unbounded support, we will have to be careful in the derivation of the diagrammatic bounds. More precisely, to be able to take weak limits in $\rho$, we will need to write them in a spin-dimension balanced way. Let us give a concrete example. The tree diagram bound \eqref{eq tree diagram bound classic} has four spins on the left and four pairs of spins on the right. As such, it is not spin balanced. We can obtain a balanced version of this inequality by ``site-splitting'' each term where an Ising spin is repeated by using Lemma \ref{lem: simon lieb}. The resulting bound is given in \eqref{eq: tree diagram pour gs general}. Though they are more complicated, the diagrammatic bounds obtained via this procedure have the advantage of being spin-balanced. An alternative route\footnote{The two methods are comparable through \eqref{eq: bound tau_0}.} would be to divide by $\langle \tau_0^2\rangle_{\rho,\beta}$. 

Below, we let $U^{\rho,\beta}_{4}$ be the corresponding Ursell's four-point function for the field variable $\tau$, and for $n\geq 1$,
\begin{equation}
    B_n(\rho,\beta):=\sum_{x\in \Lambda_n}\langle \tau_0\tau_x\rangle_{\rho,\beta}^2.
\end{equation}
Also, we will use the definitions of $\beta_c(\rho)$ and $L(\rho,\beta)$ that were introduced in Section \ref{section: reflection positivity}.

We will prove the following result, which in particular covers the case of the $\varphi^4$ lattice models by Proposition \ref{prop: phi4 is GS}.
\begin{Thm} Let $d=4$. Assume that $J$ satisfies $\mathbf{(A1)}$--$\mathbf{(A6)}$. Let $\kappa>0$. There exist $c,C>0$ such that, for all $\rho$ in the GS class satisfying $\beta_c(\rho)\geq \kappa$, for all $\beta\leq \beta_c(\rho)$, for all $x,y,z,t\in \mathbb Z^4$ at mutual distance at least $L$ with $1\leq L\leq L(\rho,\beta)$, 
\begin{multline}\label{eq: tree diagram gs}
    |U_4^{\rho,\beta}(x,y,z,t)|\\\leq C\left(\frac{B_0(\rho,\beta)}{B_L(\rho,\beta)}\right)^c\sum_{u\in \mathbb Z^4}\sum_{\substack{u',u''\in \mathbb Z^4}}\langle \tau_x\tau_u\rangle_{\rho,\beta}\beta J_{u,u'} \langle \tau_{u'}\tau_y\rangle_{\rho,\beta} \langle \tau_z\tau_u\rangle_{\rho,\beta}\beta J_{u,u''} \langle \tau_{u''}\tau_t\rangle_{\rho,\beta}.
\end{multline}
\end{Thm}
As for the Ising model, we can deduce from Theorem \ref{thm: main gs} and Proposition \ref{4 ursell is enough to win} the following triviality statement for measures in the GS class.
\begin{Coro} Let $d=4$. Assume that $J$ satisfies $\mathbf{(A1)}$--$\mathbf{(A6)}$. Let $\kappa>0$. There exist $C,c,\gamma>0$ such that, for any $\rho$ in the GS class satisfying $\beta_c(\rho)\geq \kappa$, for all $\beta\leq \beta_c(\rho)$, $1\leq L\leq L(\rho,\beta)$, $f\in \mathcal{C}_0(\mathbb R^d)$, and $z\in \mathbb R$,
\begin{multline*}
    \left|\left\langle \exp\left(z T_{f,L,\beta}(\tau)\right)\right\rangle_{\rho,\beta}-\exp\left(\frac{z^2}{2}\langle T_{f,L,\beta}(\tau)^2\rangle_{\rho,\beta}\right)\right|\leq\exp\left(\frac{z^2}{2}\langle T_{|f|,L,\beta}(\tau)^2\rangle_{\rho,\beta}\right)\dfrac{C\Vert f\Vert_\infty^4 r_f^{\gamma}z^4}{(\log L)^c}.
\end{multline*}
\end{Coro}
We will extend the results to the GS class using the following strategy.
\begin{enumerate}

    \item[\textbf{Step $1$}] Fix a measure $\rho_0$ of the Ising-type in the GS class, i.e.\ a measure that falls into $(i)$ of Definition \ref{def: gs class}. Prove that $\rho_0$ satisfies Theorem \ref{thm: main gs} with constants $c,C>0$ which only depend on $\beta_c(\rho_0)$. To do so, we prove an analogous version of the intersection clustering bound that was derived in Proposition \ref{Clustering bound}. We proceed as above by first proving that big jumps occur with small probability, and then by obtaining a version of Proposition \ref{intersection prop}, together with a mixing statement as in Theorem \ref{mixing property}.
    
    \item[\textbf{Step $2$}] Take any $\rho$ is the GS class that is obtained as a weak limit of measures $(\rho_k)_{k\geq 1}$ of the type $(i)$ in Definition \ref{def: gs class}. Prove that the statement available for each $k\geq 1$ passes to the limit $k\rightarrow \infty$. This requires a control of $(L(\rho_k,\beta))_{k\geq 1}$ and $(\beta_c(\rho_k))_{k\geq 1}$, together with ``infinite volume'' version of the GS approximation in the sense that: for all $\beta< \beta_c(\rho)$, for all $x,y\in \mathbb Z^d$,
    \begin{equation}
        \lim_{k\rightarrow \infty}\langle \tau_x\tau_y\rangle_{\rho_k,\beta}=\langle \tau_x\tau_y\rangle_{\rho,\beta}.
    \end{equation}

\end{enumerate}
In Section \ref{section: proof for ising type gs class cond on clustering bound}, we prove Theorem \ref{thm: main gs} for measures of the Ising type in the GS class modulo an intermediate result (Proposition \ref{prop: clustering gs})  that is similar to Proposition \ref{Clustering bound}. In Section \ref{section: extension to entire gs class}, we implement \textbf{Step $2$} of the above strategy and extend the result to all measures in the GS class. In Section \ref{section: proof of the clustering bound for gs}, we prove Proposition \ref{prop: clustering gs}. Finally, in Section \ref{section: extension deff=4 gs} we explain how this strategy can also be used to extend the results of Section \ref{section : deff=4}.
\subsection{Improved tree diagram bound for measures of the Ising type in the GS class}\label{section: proof for ising type gs class cond on clustering bound}
Fix $\rho$ in the GS class of the Ising-type, and $\beta<\beta_c(\rho)$. The measure $\langle \cdot \rangle_{\rho,\beta}$ can be represented as an Ising measure on $\mathbb Z^d\times K_N$ that we still denote by $\langle \cdot \rangle_{\rho,\beta}$. In that case, we can identify $\tau_x$ with averages of the form
\begin{equation}
    \sum_{i=1}^N Q_i\sigma_{(x,i)},
\end{equation}
where $Q_i\geq 0$ for $1\leq i \leq N$. For $x\in \mathbb Z^d$, we will denote $\mathcal{B}_x:=\lbrace (x,i),\: 1\leq i \leq N\rbrace$.
This point of view allows us to use the random current representation. We introduce a measure $\mathbb{P}^{xy}_{\Lambda,\rho,\beta}$ on $\Omega_{\Lambda\times K_N}$ which we define in the following two steps procedure:
\begin{enumerate}
    \item[-] first, we sample two integers $1\leq i,j\leq N$ with probability
    \begin{equation}
        \frac{Q_iQ_j\langle \sigma_{(x,i)}\sigma_{(y,j)}\rangle_{\rho,\beta}}{\langle \tau_x\tau_y\rangle_{\Lambda,\rho,\beta}},
    \end{equation}
    \item[-] then, sample a current according to the ``usual'' current measure $\mathbf{P}^{\lbrace (x,i),(y,j)\rbrace}_{\rho,\beta}$ introduced in Section \ref{section rcr}.
\end{enumerate}
It is also possible to define the infinite volume version of the above measure that we will denote $\mathbb P^{xy}_{\rho,\beta}$.
Samples of $\mathbb P^{xy}_{\rho,\beta}$ are random currents with random sources in $\mathcal{B}_x$ and $\mathcal{B}_y$. The interest of this measure lies in the fact that it is better suited for the derivation of bounds on connection probabilities in terms of the correlation functions of the field variable $\tau$. These bounds can be directly imported from \cite{AizenmanDuminilTriviality2021} and are recalled in Appendix \ref{appendix: bounds gs}.

Define a sequence $\mathcal{L}$ similarly as in \eqref{eq: definition sequence L}, using this time $B_\ell(\rho,\beta)$. Call $\mathcal{I}_u$ the set of vertices in $v\in\mathbb Z^d$ such that $\mathcal{B}_v$ is connected in $\n_1+\n_3$ and $\n_2+\n_4$ to $\mathcal{B}_{u}$. With this definition, we now consider \textit{coarse intersections} instead of proper intersections. This point of view is better for the analysis that follows. 

For models in the GS class of the Ising type, the clustering bound takes the following form.

\begin{Prop}[Clustering bound for models in the GS class]\label{prop: clustering gs}
Let $d=4$. Assume that $J$ satisfies $\mathbf{(A1)}$-$\mathbf{(A6)}$. Let $\kappa>0$. For $D$ large enough, there exists $\delta=\delta(D,\kappa)>0$ such that for every $\rho$ in the GS class of the Ising-type with $\beta_c(\rho)\geq \kappa$, for every $\beta< \beta_c(\rho)$, every $K>3$ with $\ell_{K+1}\leq L(\rho,\beta)$, every $x,y,z,t\in \mathbb Z^d$ with mutual distance between $x,y,z,t$ larger than $2\ell_{K}$, every $u,u',u''\in \mathbb Z^d$ with $u',u''$ $J$-neighbours\footnote{In the sense that $J_{u,u'},J_{u,u''}>0$.} of $u$ satisfying $|u-x|,|u-z|\geq \ell_K$, 
\begin{equation}
    \mathbb{P}_{\rho,\beta}^{ux,uz,u'y,u''t}[\mathbf{M}_u(\mathcal{I}'_u;\mathcal{L},K)<\delta K]\leq 2^{-\delta K},
\end{equation}
where $\mathcal I'_u$ is the set of vertices in $v\in\mathbb Z^d$ such that $\mathcal{B}_v$ is connected in\footnote{Here, for $u,v\in \mathbb Z^d$ such that $J_{u,v}>0$, $\delta_{(u,v)}$ denotes the current identically equal to $0$ except on the pair $\lbrace u,v\rbrace$ where it is equal to $1$.} $\n_1+\n_3+\delta_{(\partial\n_1\cap \mathcal{B}_u,\partial \n_3\cap \mathcal{B}_{u'})}$ and $\n_2+\n_4+\delta_{(\partial\n_2\cap \mathcal{B}_u,\partial \n_4\cap \mathcal{B}_{u''})}$ to $\mathcal{B}_{u}$.
\end{Prop}
\begin{Rem}
The reason why we have $\mathcal{I}'_u$ instead of $\mathcal{I}_u$ in the bound above is technical. This is a consequence of the form the switching lemma takes in that context, as seen in \cite[Lemma~A.7]{AizenmanDuminilTriviality2021}.
\end{Rem}
\begin{Rem}\label{rem: change of L}
In fact, the same result holds with $L^{(\alpha)}(\rho,\beta)$, $\alpha\in (0,1)$, instead of $L(\rho,\beta)$ (with a change of parameter\footnote{One way to see this is to observe that the constant $c$ in Proposition \ref{prop: lower bound 2 pt function} depends on $\alpha$.} $\delta$). This will be useful below.
\end{Rem}
We postpone the derivation of this bound to the next section and now explain how one can derive Theorem \ref{thm: main gs} for measures $\rho$ of the Ising type.
\begin{proof}[Proof of Theorem \textup{\ref{thm: main gs}} for a measure $\rho$ of the Ising type in the GS class.]
As for the case of the Ising model, one can show (by summing $\eqref{eq tree diagram bound classic}$ over points in $\mathcal{B}_x,\mathcal{B}_y,\mathcal{B}_z$ and $\mathcal{B}_t$) that
\begin{equation}
    |U^{\rho,\beta}_4(x,y,z,t)|\leq 2\langle \tau_x\tau_y\rangle_{\rho,\beta}\langle \tau_z\tau_t\rangle_{\rho,\beta}\mathbb P^{xy,zt,\emptyset,\emptyset}_{\rho,\beta}[\mathbf{C}_{\n_1+\n_3}(\partial\n_1)\cap \mathbf{C}_{\n_2+\n_4}(\partial\n_2)\neq \emptyset],
\end{equation}
where $\mathbf{C}_{\n_1+\n_3}(\partial\n_1)$ and $\mathbf{C}_{\n_2+\n_4}(\partial\n_2)$ refer to the clusters in $\n_1+\n_2$ and $\n_3+\n_4$ of the (random) sources $\partial\n_1$ and $\sn_2$ respectively. As above, we may find $c_0>0$ such that if $x,y,z,t$ are at mutual distance at least $L$, there exists $K=K(L)$ such that $K\geq c_0\log(B_L(\rho,\beta)/B_0(\rho,\beta))$ and $2\ell_K\leq L$. The rest of the proof is conceptually identical to what was done before, except that now we look at coarse intersections. Let $D$ be large enough so that Proposition \ref{prop: clustering gs} holds for some $\delta=\delta(D,\kappa)>0$.
Using Markov's inequality together with \eqref{eq: bound connectivity gs sources},
\begin{multline*}
    \langle \tau_x\tau_y\rangle_{\rho,\beta}\langle \tau_z\tau_t\rangle_{\rho,\beta}\mathbb P^{xy,zt,\emptyset,\emptyset}_{\rho,\beta}[|\mathbf{C}_{\n_1+\n_3}(\partial\n_1)\cap \mathbf{C}_{\n_2+\n_4}(\partial\n_2)|\geq 2^{\delta K/5}]\\\leq 2^{-\delta K/5}\sum_{u,u',u''\in \mathbb Z^d}\langle \tau_x\tau_u\rangle_{\rho,\beta}\beta J_{u,u'}\langle \tau_{u'}\tau_y\rangle_{\rho,\beta}\langle \tau_z\tau_u\rangle_{\rho,\beta}\beta J_{u,u''}\langle \tau_{u''}\tau_t\rangle_{\rho,\beta}.
\end{multline*}
Now, using Lemma \ref{lem: covering lemma}, write
\begin{multline*}
    \langle \tau_x\tau_y\rangle_{\rho,\beta}\langle \tau_z\tau_t\rangle_{\rho,\beta}\mathbb P^{xy,zt,\emptyset,\emptyset}_{\rho,\beta}[0<|\mathbf{C}_{\n_1+\n_3}(\partial\n_1)\cap \mathbf{C}_{\n_2+\n_4}(\partial\n_2)|< 2^{\delta K/5}]\\\leq \sum_{u\in \mathbb Z^d}\mathbb{P}_{\rho,\beta}^{xy,zt,\emptyset,\emptyset}[\sn_1\connect{\n_1+\n_3\:}\mathcal{B}_u,\: \sn_2\connect{\n_2+\n_4\:}\mathcal{B}_u, \: \mathbf{M}_u(\mathcal{I}_u;\mathcal{L},K)< \delta K].
\end{multline*}
Notice that by hypothesis $u$ must satisfy $|u-x|\vee|u-y|\geq \ell_K$ and $|u-z|\vee |u-t|\geq \ell_K$. Hence, the upper bound above can be rewritten as the sum of four terms which represent each case. We assume without loss of generality that $|u-x|\geq \ell_K$ and $|u-z|\geq \ell_K$.

Using a proper formulation of the switching lemma in that context \cite[Lemma~A.7]{AizenmanDuminilTriviality2021}, we get
\begin{multline*}
    \mathbb{P}_{\rho,\beta}^{xy,zt,\emptyset,\emptyset}[\sn_1\connect{\n_1+\n_3\:}\mathcal{B}_u,\: \sn_2\connect{\n_2+\n_4\:}\mathcal{B}_u, \:\mathbf{M}_u(\mathcal{I}_u;\mathcal{L},K)< \delta K]\\\leq \sum_{u',u''\neq u}\langle \tau_x\tau_u\rangle_{\rho,\beta}\beta J_{u,u'}\langle \tau_{u'}\tau_y\rangle_{\rho,\beta}\langle \tau_z\tau_u\rangle_{\rho,\beta}\beta J_{u,u''}\langle \tau_{u''}\tau_t\rangle_{\rho,\beta}\mathbb P^{ux,uz,u'y,u''t}_{\rho,\beta}[\mathbf{M}_u(\mathcal{I}'_u;\mathcal{L},K)<\delta K].
\end{multline*}

We then conclude using Proposition \ref{prop: clustering gs}. We obtained the existence of $c,C>0$ such that: for all $\rho$ in the GS class of the Ising type such that $\beta_c(\rho)\geq \kappa$, for all $\beta<\beta_c(\rho)$, for all $x,y,z,t\in \mathbb Z^4$ at mutual distance at least $L$ with $L\leq L(\rho,\beta)$, \eqref{eq: tree diagram gs} holds. We can then extend the result to $\beta_c(\rho)$ by a continuity argument\footnote{Here we use the left-continuity of the two-point and four-point correlation functions together with \eqref{eq: INFRARED BOUND WE USE} which uniformly bounds the two-point function for $\beta\leq\beta_c(\rho)$.} together with the observation that $L(\rho,\beta)\rightarrow \infty$ as $\beta\rightarrow \beta_c(\rho)$.

\end{proof}

 \begin{Rem}\label{rem: we can change L}
 Using Remark \textup{\ref{rem: change of L}} we see that we also obtained the same result replacing $L(\rho,\beta)$ by $L^{(\alpha)}(\rho,\beta)$, for some $\alpha\in (0,1)$. Note that this affects the constant $c,C$ in Theorem \textup{\ref{thm: main gs}}.
 \end{Rem}
\subsection{Tree diagram bound for weak limits of Ising type measures}\label{section: extension to entire gs class}
The goal of this section is to extend the improved tree diagram bound to the entire GS class of measures. The result will be a consequence of the following proposition.
\begin{Prop}\label{prop: cvu in gs} Let $\rho$ be a measure in the GS class. There exists a sequence of measures $(\rho_k)_{k\geq 1}$ of the Ising type in the GS class such that:
\begin{enumerate}
    \item[$(1)$] $(\rho_k)_{k\geq 1}$ converges weakly to $\rho$,
    \item[$(2)$] $\liminf \beta_c(\rho_k)\geq \beta_c(\rho)$,
    \item[$(3)$] for every $\beta< \beta_c(\rho)$, for every $x,y,z,t\in \mathbb Z^d$,
    \begin{equation}
        \lim_{k\rightarrow \infty}\langle \tau_x\tau_y\rangle_{\rho_k,\beta}=\langle \tau_x\tau_y\rangle_{\rho,\beta},\qquad \lim_{k\rightarrow \infty}\langle \tau_x\tau_y\tau_z\tau_t\rangle_{\rho_k,\beta}=\langle \tau_x\tau_y\tau_z\tau_t\rangle_{\rho,\beta},
    \end{equation}
    \item[$(4)$] for every $\beta>0$, $\liminf L^{(1/4)}(\rho_k,\beta)\geq L(\rho,\beta)$.
\end{enumerate}
\end{Prop}
Assuming this result, and knowing Theorem \ref{thm: main gs} for Ising type measures, we can easily extend Theorem \ref{thm: main gs} to all models in the GS class.
\begin{proof}[Proof of Theorem \textup{\ref{thm: main gs}}.] Fix $\rho$ in the GS class which is a weak limit of Ising type measure in the GS class of measures, and which satisfies $\beta_c(\rho)\geq \kappa$. Let also $(\rho_k)_{k\geq 1}$ be given by Proposition \ref{prop: cvu in gs}. By the property $(2)$ of the same proposition, the exists $k_0\geq 0$ such that for $k\geq k_0$, $\beta_c(\rho_k)\geq \kappa/2$. Since the tree diagram bound holds uniformly over Ising type measures $\rho'$ satisfying $\beta_c(\rho')\geq \kappa/2$, and using Remark \ref{rem: we can change L}, there exist $C,c>0$ such that for all $k\geq k_0$, for all $\beta\leq \beta_c(\rho_k)$, for all $x,y,z,t$ at mutual distance at least $L$ with $1\leq L\leq L^{(1/4)}(\rho_k,\beta)$,
\begin{multline}\label{eq: gs to the limit}
    |U^{\rho_k,\beta}_4(x,y,z,t)|\\\leq  C\left(\frac{B_0(\rho_k,\beta)}{B_L(\rho_k,\beta)}\right)^c\sum_{u,u',u''\in \mathbb Z^d}\langle \tau_x\tau_u\rangle_{\rho_k,\beta}\beta J_{u,u'}\langle \tau_{u'}\tau_y\rangle_{\rho_k,\beta}\langle \tau_z\tau_u\rangle_{\rho_k,\beta}\beta J_{u,u''}\langle \tau_{u''}\tau_t\rangle_{\rho_k,\beta}.
\end{multline}
Fix $\beta<\beta_c(\rho)$ and $1\leq L\leq L(\rho,\beta)$. By properties $(2)$ and $(4)$ of Proposition \ref{prop: cvu in gs}, there exists $k_1\geq k_0$ such that $\beta\leq \beta_c(\rho_k)$ and $L\leq L^{(1/4)}(\rho_k,\beta)$ for $k\geq k_1$. As a result, \eqref{eq: gs to the limit} holds with $\beta$ and $L$ for $k\geq k_1$. We now use $(3)$ to pass the inequality to the limit.
Using \eqref{eq: INFRARED BOUND WE USE}, we know that there exists $C=C(d)>0$ such that for all $u,v\in \mathbb Z^d$, for $k\geq k_1$ 
\begin{equation}
    \langle \tau_u\tau_v\rangle_{\rho_k,\beta}\leq \frac{C}{\beta_c(\rho)|J||u-v|^{d-2}}.
\end{equation}
This justifies passing to the limit in \eqref{eq: gs to the limit} and yields the result. We extend the result to $\beta=\beta_c(\rho)$ by a continuity argument as above.
\end{proof}
We now prove Proposition \ref{prop: cvu in gs}. We split the statement into lemmas.
\begin{Lem}\label{lem: borne beta} Assume that $(\rho_k)_{k\geq 1}$ converges weakly to $\rho$. Then,
\begin{equation}
    \liminf \beta_c(\rho_k)\geq \beta_c(\rho),
\end{equation}
\end{Lem}
\begin{proof} Assume $\beta> \liminf \beta_c(\rho_k)$. If $S\subset \mathbb Z^d$ is a finite set containing $0$,
\begin{equation}\label{eq: proof lemma 88}
    \lim_{k\rightarrow \infty}\varphi_{\rho_k,\beta}(S)=\varphi_{\rho,\beta}(S).
\end{equation}
Since $\beta> \liminf \beta_c(\rho_k)$, one has that for $k$ large enough $\beta\geq \beta_c(\rho_k)$ and hence\footnote{Otherwise the susceptibility would be finite at $\beta$.} $\varphi_{\rho_k,\beta}(S)\geq 1$ (using the same argument as in Remark \ref{rem: L infinite at criticality}) so that 
\begin{equation}
    \varphi_{\rho,\beta}(S)\geq 1.
\end{equation}
Since this holds for any finite set $S$ containing $0$, one has $\beta\geq \beta_c(\rho)$.
\end{proof}
\begin{Lem} Assume that $(\rho_k)_{k\geq 1}$ converges weakly to $\rho$. Then, for all $\beta>0$, for all $\alpha\in(0,1/2)$
\begin{equation}
    L(\rho,\beta)\leq \liminf L^{(\alpha)}(\rho_k,\beta).
\end{equation}
\end{Lem}
\begin{proof} If $\liminf L^{(\alpha)}(\rho_k,\beta)=\infty$ then the upper bound is trivial.

 Otherwise, fix $n> (2d)\liminf L^{(\alpha)}(\rho_k,\beta)$. There exists $S\subset \mathbb Z^d$ finite and containing $0$ with $\textup{rad}(S)\leq 2n$ such that for all $k$ sufficiently large,
\begin{equation}
    \varphi_{\rho_k,\beta}(S)<\alpha.
\end{equation}
Using \ref{eq: proof lemma 88}, we get that $\varphi_{\rho,\beta}(S)\leq \alpha<1/2$ so that $n\geq(2d)L(\rho,\beta)$.
\end{proof}
\begin{Lem}\label{lem: approx inf volume gs}Let $d=4$. Assume that $(\rho_k)_{k\geq 1}$ converges weakly to $\rho$. Let $\beta<\beta_c(\rho).$  For every $x,y,z,t\in \mathbb Z^4$,
    \begin{equation}
        \lim_{k\rightarrow \infty}\langle \tau_x\tau_y\rangle_{\rho_k,\beta}=\langle \tau_x\tau_y\rangle_{\rho,\beta},\qquad \lim_{k\rightarrow \infty}\langle \tau_x\tau_y\tau_z\tau_t\rangle_{\rho_k,\beta}=\langle \tau_x\tau_y\tau_z\tau_t\rangle_{\rho,\beta}.
    \end{equation}
\end{Lem}
\begin{proof} We only prove the first part of the statement, the second part follows by a similar argument. Let $x,y\in \mathbb Z^d$. It is sufficient to show that
\begin{equation}\label{eq: sufficient unif}
    \lim_{n\rightarrow \infty}\sup_{k\geq 1}\Big(\langle \tau_x\tau_y\rangle_{\rho_k,\beta}-\langle \tau_x\tau_y\rangle_{\Lambda_n,\rho_k,\beta}\Big)=0.
\end{equation}
Fix $n$ large enough such that $x,y\in \Lambda_n$. Recall that $\rho_k$ is defined by averages on $K_{N_k}$ for some $N_k\geq 1$. Using the switching lemma,
\begin{align*}
    \langle \tau_x\tau_y\rangle_{\rho_k,\beta}-\langle \tau_x\tau_y\rangle_{\Lambda_n,\rho_k,\beta}
    &=\langle \tau_x\tau_y\rangle_{\rho_k,\beta}\mathbb P^{xy,\emptyset}_{\mathbb Z^d,\Lambda_n,\rho_k,\beta}\Big[\sn_1\cap \mathcal{B}_x\overset{(\n_1+\n_2)_{|\Lambda_n\times K_{N_k}}}{\nleftrightarrow}\sn_1\cap \mathcal{B}_y\Big]\\&\leq \langle \tau_x\tau_y\rangle_{\rho,\beta}\mathbb P^{xy}_{\rho_k,\beta}\Big[\sn\cap \mathcal{B}_x\overset{\n_{|\Lambda_n\times K_{N_k}}}{\nleftrightarrow}\sn\cap \mathcal{B}_y\Big].
\end{align*}
Let $\ell:=|x|+|y|$ and introduce the event $\mathsf{ZZGS}_k(x,y;\ell,n,\infty)$ that the backbone of $\n$ goes from $\partial \n\cap \mathcal{B}_x$ to $\partial \n\cap \mathcal{B}_y$ by exiting $\Lambda_n\times K_{N_k}$.
As explained below, it is possible to extend the proof of Corollary \ref{coro: no jump 1 bis} to this setup to get that there exist $\eta,C_1>0$ such that for all $n$ large enough, for all $k\geq 1$,
\begin{equation}
    \mathbb P^{xy}_{\rho_k,\beta}[\mathsf{ZZGS}_k(x,y;\ell,n,\infty)]\leq \frac{C_1}{n^\eta}.
\end{equation}
The observation that $\Big\lbrace \sn\cap \mathcal{B}_x\overset{\n_{|\Lambda_n\times K_{N_k}}}{\nleftrightarrow}\sn\cap \mathcal{B}_y\Big\rbrace\subset \mathsf{ZZGS}_k(x,y;\ell,n,\infty)$ gives \eqref{eq: sufficient unif}.
\end{proof}
\begin{proof}[Proof of Proposition \textup{\ref{prop: cvu in gs}}] If $\rho$ falls into $(i)$ of Definition \ref{def: gs class} the statement is trivial. Otherwise, fix any sequence $(\rho_k)_{k\geq 1}$ of Ising type measures in the GS class that converges weakly to $\rho$. Using the three above lemmas we verify that $(\rho_k)_{k\geq 1}$ satisfies all the desired properties.

\end{proof}

\subsection{Intersection clustering bound for models in the GS class}\label{section: proof of the clustering bound for gs}

We now turn to the proof of Proposition \ref{prop: clustering gs}. The proof follows the exact same lines as for the Ising case and is reduced to the adaptation of the results of Section \ref{section: properties of the current trivia}, together with an extension of the intersection property of Lemma \ref{intersection prop}, and the mixing statement of Theorem \ref{mixing property}. We fix $d=4$ and an interaction $J$ which satisfies $\mathbf{(A1)}$--$\mathbf{(A6)}$.

We start by excluding the existence of ``big jumps'' in our context. As it turns out, the results of Section \ref{section: properties of the current trivia} directly follow from the following adaptation of Lemma \ref{big edge}.
\begin{Lem}\label{lem: big edge gs} Let $\beta>0$. Let $\rho$ be of the Ising type in the GS class. For $x,y,u,v\in \mathbb Z^d$,
\begin{multline*}
    \mathbb{P}^{xy,\emptyset}_{\rho,\beta}[\exists i,j, \: \n_{(u,i),(v,j)}\geq 1]\leq \beta J_{u,v}\left(2\langle \tau_u\tau_v\rangle_{\rho,\beta}+\frac{\langle \tau_x\tau_u\rangle_{\rho,\beta}\langle \tau_v\tau_y\rangle_{\rho,\beta}}{\langle \tau_x\tau_y\rangle_{\rho,\beta}}+\frac{\langle \tau_x\tau_v\rangle_{\rho,\beta}\langle \tau_u\tau_y\rangle_{\rho,\beta}}{\langle \tau_x\tau_y\rangle_{\rho,\beta}}\right).
\end{multline*}
\end{Lem}
At this stage of the proof, the arguments essentially build on what was done in the Ising case together with a proper adaptation of the proofs, as already explained in \cite{AizenmanDuminilTriviality2021}. Below we explain the main changes in the proofs and refer to \cite{AizenmanDuminilTriviality2021} for more details. We first state the intersection property for Ising-type models in the GS class.
\begin{Lem}[Intersection property for models in the GS class] Let $\kappa>0$. For $D=D(\kappa)>0$ large enough, there exists $\delta=\delta(\kappa)>0$ such that for every $\rho$ of the Ising type in the GS class satisfying $\beta_c(\rho)\geq \kappa$, every $\beta\leq \beta_c(\rho)$, every $k\geq 2$, and every $y\notin \Lambda_{\ell_{k+2}}$ in a regular scale with $1\leq |y|\leq L(\rho,\beta)$,
\begin{equation}
    \mathbb{P}_{\rho,\beta}^{0y,0y,\emptyset,\emptyset}[(\n_1+\n_3,\n_2+\n_4)\in I_k(0)]\geq \delta,
\end{equation} 
where $I_k(0)$ is defined similarly to the intersection event of Definition \textup{\ref{def: intersection event}}, except that we now ask that the clusters of $\n_1+\n_3$ and $\n_2+\n_4$ coarse-intersect in the sense that there exists $v\in \textup{Ann}(\ell_k,\ell_{k+1})$ such that $\mathcal{B}_v$ is connected to $\mathcal{B}_0$ in  $\n_1+\n_3$ and $\n_2+\n_4$.
\end{Lem}
\begin{proof} We keep the notations introduced in the proof of Lemma \ref{intersection prop}. Define,
\begin{equation}
    \mathcal{M}:=\sum_{v\in \textup{Ann}(m,M)}\sum_{i,i'}Q_i^2 \mathds{1}\{\partial\n_1\connect{\n_1+\n_3\:}(v,i)\}Q_{i'}^2 \mathds{1}\{\partial\n_2\connect{\n_2+\n_4\:}(v,i')\}.
\end{equation}
The extra $Q_i^2, Q_{i'}^2$ terms allow to rewrite moments of $\mathcal{M}$ in terms of correlation functions of the field variables $(\tau_z)_{z\in \mathbb Z^d}$.
Using a similar computation as for the case of the Ising model, together with the results of Proposition \ref{prop: bounds gs}, we get $c_1,C_1>0$ such that
\begin{equation}
    \mathbb E_{\rho,\beta}^{0y,0'y,\emptyset,\emptyset}[|\mathcal{M}|]\geq c_1(B_M(\rho,\beta)-B_{m-1}(\rho,\beta)),
\end{equation}
\begin{equation}
    \mathbb E_{\rho,\beta}^{0y,0'y,\emptyset,\emptyset}[|\mathcal{M}|^2]\leq C_1 B_{\ell_{k+1}}(\rho,\beta)^2.
\end{equation}
Similarly as above, we deduce, for some $c_2>0$,
\begin{equation}
    \mathbb P_{\rho,\beta}^{0y,0y,\emptyset,\emptyset}[\mathcal{M}\neq \emptyset]\geq c_2.
\end{equation}
The second part of the proof consists in making the intersection event local. We proceed exactly as we did for the Ising case by first excluding the possibility of jumping any of the intermediate scales, and by then repeating the analysis that lead to the bounds on the events $\mathcal{F}_1,\ldots, \mathcal{F}_5$. At this stage one needs to be careful in the use of the infrared bound and it is required to have bounds involving $\beta |J|$. This will ensure that the bound on the intersection probability we end up with does not depend on $\rho$.
\end{proof}
\begin{Thm}[Mixing property for models in the GS class] 
Let $d= 4$. Let $\kappa>0$ and $s\geq 1$ There exist $\gamma,C>0$, such that for every $\rho$ of the Ising-type in the GS class satisfying $\beta_c(\rho)\geq \kappa$, for every $1\leq t\leq s$, every $\beta\leq \beta_c(\rho)$, every $n^\gamma\leq N\leq L(\rho,\beta)$, every $x_i\in \Lambda_n$ and $y_i\notin \Lambda_N$ $(i\leq t)$, and every events $E$ and $F$ depending on the restriction of $(\n_1,\ldots,\n_s)$ to edges with endpoints within $\Lambda_n$ and outside $\Lambda_N$ respectively,
\begin{multline}\label{eq 1 mixing gs}
    \left|\mathbb{P}_{\rho,\beta}^{x_1 y_1,\ldots, x_t y_t,\emptyset,\ldots,\emptyset}[E\cap F]-\mathbb{P}_{\rho,\beta}^{x_1 y_1,\ldots, x_t y_t,\emptyset,\ldots,\emptyset}[E]\mathbb{P}_{\rho,\beta}^{x_1 y_1,\ldots, x_t y_t,\emptyset,\ldots,\emptyset}[F]\right|\leq C\left(\log\frac{N}{n}\right)^{-1/2}.
\end{multline}
Furthermore, for every $x_1',\ldots, x_t'\in \Lambda_n$ and $y_1',\ldots,y'_t\notin \Lambda_N$, we have that
\begin{equation}\label{eq 2 mixing gs}
    \left|\mathbb{P}_{\rho,\beta}^{x_1 y_1,\ldots, x_t y_t,\emptyset,\ldots,\emptyset}[E]-\mathbb{P}_{\rho,\beta}^{x_1 y'_1,\ldots, x_t y'_t,\emptyset,\ldots,\emptyset}[E]\right|\leq C\left(\log\frac{N}{n}\right)^{-1/2},
\end{equation}
\begin{equation}\label{eq 3 mixing gs}
    \left|\mathbb{P}_{\rho,\beta}^{x_1 y_1,\ldots, x_t y_t,\emptyset,\ldots,\emptyset}[F]-\mathbb{P}_{\rho,\beta}^{x'_1 y_1,\ldots, x'_t y_t,\emptyset,\ldots,\emptyset}[F]\right|\leq C\left(\log\frac{N}{n}\right)^{-1/2}.
\end{equation}
\end{Thm}
\begin{proof} The main modification in the proof comes in the definition of $\mathbf{U}_i$:
\begin{equation*}
    \mathbf{U}_i:=\frac{1}{|\mathcal{K}|}\sum_{k\in\mathcal{K}}\frac{1}{A_{x_i,y_i}(2^k)}\sum_{u\in \mathbb{A}_{y_i}(2^k)}\sum_{j=1}^N Q_j^2\mathds{1}\{(u,j)\connect{\n_i+\n_i'\:}\sn_i\},
\end{equation*}
where
\begin{equation*}
    a_{x,y}(u):=\frac{\langle \tau_x\tau_u\rangle_{\rho,\beta}\langle \tau_u\tau_y\rangle_{\rho,\beta}}{\langle \tau_x\tau_y\rangle_{\rho,\beta}}, \qquad A_{x,y}(k):=\sum_{u\in \mathbb A_{y_i}(2^k)}a_{x,y}(u).
\end{equation*}
Note that, as above, the extra term $Q_j^2$ allows one to express the moments in terms of the field variables $(\tau_z)_{z\in \mathbb Z^d}$. Also, in the derivation of an analogue of Lemma \ref{technical lemma mixing}, one will have to be careful to use infrared bounds involving $\beta |J|$.
\end{proof}
We are now in a position to prove Proposition \ref{prop: clustering gs}.
\begin{proof}[Proof of Proposition \textup{\ref{prop: clustering gs}}] The proof follows the exact same lines as for the Ising case, except that we need to slightly take care of the monotonicity property we want to use. We keep the notations introduced in the proof of Proposition \ref{Clustering bound}. Let $\delta>0$ to be fixed later. Let $S\in \mathcal{S}_K^{(\delta)}$. Let $\mathfrak{B}_S$ (resp. $\mathfrak{B}_S'$) be the event that the clusters of $\mathcal{B}_u$ in $\n_1+\n_3$ and $\n_2+\n_4$ (resp. $\n_1+\n_3+\delta_{(\partial\n_1\cap \mathcal{B}_u,\partial \n_3\cap \mathcal{B}_{u'})}$ and $\n_2+\n_4+\delta_{(\partial\n_2\cap \mathcal{B}_u,\partial \n_4\cap \mathcal{B}_{u''})}$), do not coarse intersect in any of the annuli $\textup{Ann}(\ell_i,\ell_{i+1})$ for $i \in S$. Then, using an adaptation of the monotonicity argument of Proposition \ref{prop: monotonicity sources} to our context (see Proposition \ref{prop: monotonicity sources gs}),
\begin{align*}
    \mathbb{P}_{\rho,\beta}^{ux,uz,u'y,u''t}[\mathbf{M}_u(\mathcal{I}_u';\mathcal{L},K)<\delta K]&\leq \sum_{\substack{S\in \mathcal{S}_K^{(\delta)}\\|S|\geq (1/2-2\delta)K}} \mathbb{P}_{\rho,\beta}^{ux,uz,u'y,u''t}[\mathfrak{B}_S']\\&\leq \sum_{\substack{S\in \mathcal{S}_K^{(\delta)}\\|S|\geq (1/2-2\delta)K}} \mathbb{P}_{\rho,\beta}^{ux,uz,\emptyset,\emptyset}[\mathfrak{B}_S].   
\end{align*}
The rest of the proof is identical to what was done in Section \ref{section d=4}.
\end{proof}
\subsection{Extension of the results of Section \ref{section : deff=4}}\label{section: extension deff=4 gs}
We now briefly explain how to extend to results of Section \ref{section : deff=4} to models in the GS class. The strategy is very similar to what was done above so we only present the main modifications in the proof. We begin by discussing the modifications involved in the proofs of the results obtained in Sections \ref{section: weak regular scales} and \ref{section: prop currents deff nice}

Let $d\geq 1$. We fix an interaction $J$ on $\mathbb Z^d$ satisfying $(\mathbf{A1})$--$(\mathbf{A5})$ and \eqref{eq: assumption deff at least 4} with $d-2(\alpha\wedge 2)\geq 0$. In that setup, we get that for any $\rho$ in the GS class: if $\beta\leq \beta_c(\rho)$ and $x\in \mathbb Z^d\setminus \lbrace 0\rbrace$,
\begin{equation}\label{eq: ir alpha gs}
    \langle \tau_0\tau_x\rangle_{\rho,\beta}\leq \frac{C}{\beta_c(\rho)|x|^{d-\alpha\wedge 2}(\log |x|)^{\delta_{2,\alpha}}}.
\end{equation}

The first important observation is to notice that, although stated for the Ising model, the results of Section \ref{section: weak regular scales} extend \textit{mutatis mutandis} to every single-site measure $\rho$ in the GS class thanks to Proposition \ref{prop: general lower bound}.

Similarly, we may extend the results of Section \ref{section: prop currents deff nice} to all measures $\rho$ of the Ising type in the GS class by using Lemma \ref{lem: big edge gs} and \eqref{eq: ir alpha gs}. 
With these tools, it is possible to extend the results of Section \ref{section : deff=4} to measures of the Ising-type in the GS class by using the same strategy as in Section \ref{section: proof of the clustering bound for gs}.

The extension to all measures in the GS class uses again the approximation step of Section \ref{section: extension to entire gs class}. The only non-trivial modification concerns Proposition \ref{prop: cvu in gs}, and more precisely Lemma \ref{lem: approx inf volume gs}. We will prove the following result.
\begin{Lem} Let $1\leq d\leq 3$. Assume that $J$ satisfies $(\mathbf{A1})$--$(\mathbf{A5})$ and \eqref{eq: assumption deff at least 4} with $d-2(\alpha\wedge 2)\geq 0$. Let $\rho$ be a measure in the GS class, and let $(\rho_k)_{k\geq 1}$ be a sequence of measures of the Ising type which converges weakly to $\rho$. Let $\beta<\beta_c(\rho)$. For every $x,y,z,t\in\mathbb Z^d$,
\begin{equation}
    \lim_{k\rightarrow \infty}\langle \tau_x\tau_y\rangle_{\rho_k,\beta}=\langle \tau_x\tau_y\rangle_{\rho,\beta},\qquad \lim_{k\rightarrow \infty}\langle \tau_x\tau_y\tau_z\tau_t\rangle_{\rho_k,\beta}=\langle \tau_x\tau_y\tau_z\tau_t\rangle_{\rho,\beta}.
\end{equation} 
\end{Lem}
\begin{proof} Again we only prove the first part of the statement. We follow the proof of Lemma \ref{lem: approx inf volume gs}. As before, if $\ell:=|x|+|y|$, the key observation is that $\left\lbrace \sn\cap \mathcal{B}_x\overset{\n_{|\Lambda_n\times K_{N_k}}}{\nleftrightarrow}\sn\cap \mathcal{B}_y\right\rbrace$ is included in the event $\mathsf{ZZGS}_k(x,y;\ell,n,\infty)$. However, as explained above, we can extend the results of Section \ref{section: prop currents deff nice}, and in particular Corollary \ref{coro: no jump 1 deff=4 bis}, to obtain a bound the probability of the latter event. This is enough to conclude.

\end{proof}
\appendix
\section[Spectral representation of RP models]{Spectral representation of reflection positive Ising models}\label{appendix spectral representation}
The aim of this appendix is to prove Theorem \ref{Spectral representation}. We use the notations of Section \ref{section: reflection positivity}. In what follows $\rho$ is a measure in the GS class. We assume that $J$ satisfies $\mathbf{(A1)}$--$\mathbf{(A5)}$. The following lines are inspired by \cite{BorgsChayesCovarianceMatrixPotts1996,Ott2019OZThesis}.

We will make good use of the spectral theorem (see \cite{HallQuantumTheory2013}) which will be applied to diagonalize the shift operator $T$ given by,
\begin{equation}
    T: x\in \mathbb Z^d\mapsto x+(1,0,\ldots,0).
\end{equation}
Before that, we introduce some notations and a proper Hilbert space. 

Let $\beta>0$. Let $\mathbf{e}_1=(1,0,\ldots,0)$. Let $\Sigma$ be the hyperplane orthogonal to $\mathbf{e}_1$ passing through $0$. Let $\Theta$ be the reflection through $\Sigma$. Notice that $\Sigma$ cuts $\mathbb Z^d$ in two half-planes $\Lambda_+$ and $\Lambda_-$ with $\Lambda_+\cap \Lambda_-=\Sigma$. Let $\mathcal{A}_+$ be the algebra generated by local functions with support in $\Lambda_+$. Reflection positivity with respect to $\Theta$ implies that for all $f\in \mathcal{A}_+$, 
\begin{equation}
    \langle\overline{\Theta(f)}f\rangle_{\rho,\beta}\geq 0.
\end{equation}
We define a positive semi-definite bi-linear form on $\mathcal{A}_+$ by : for all $f,g\in \mathcal{A}_+$,
\begin{equation}
    (f,g):=\langle \overline{\Theta(f)} g\rangle_{\rho,\beta}.
\end{equation}
Quotienting $\mathcal{A}_+$ by the kernel of $(\cdot,\cdot)$ and completing the resulting space one obtains a Hilbert space $(\mathcal{H},(\cdot,\cdot))$. We denote by $\Vert \cdot\Vert$ the norm on this Hilbert space and $\Vert \cdot\Vert^{\textup{op}}$ the associated operator norm. The shift $T$ in the $\mathbf{e}_1$ direction defines an operator on $\mathcal{H}$ whose properties are described in the next proposition whose proof can be found in \cite{BorgsChayesCovarianceMatrixPotts1996,Ott2019OZThesis}.
\begin{Prop}[Properties of $T$]\label{prop: properties of T} The shift operator $T:\mathcal{H}\rightarrow \mathcal{H}$ has the following properties,
\begin{enumerate}
    \item[$(i)$] $T$ is self-adjoint,
    \item[$(ii)$] $T$ is positive,
    \item[$(iii)$] $\Vert T\Vert^{\rm op}=1$.
\end{enumerate}
\end{Prop}
 \begin{proof}
 \begin{enumerate}
     \item[$(i)$] Since $\langle \cdot \rangle_{\rho,\beta}$ is invariant under the action of $T$ (by $(\mathbf{A2})$), for all $f,g\in \mathcal{H}$,
     \begin{equation} 
     (Tf,g)=\langle \overline{\Theta(Tf)}g\rangle_{\rho,\beta} = 
     \langle T^{-1}(\overline{\Theta(f)})g\rangle_{\rho,\beta}= 
     \langle \overline{\Theta(f)}Tg\rangle_{\rho,\beta}=
     (f,Tg),
     \end{equation}
     so that $T$ is self-adjoint.
     \item[$(ii)$] For all $f \in \mathcal{H}$,
     \begin{equation} 
     (f,Tf)=\langle \overline{\Theta(f)}Tf\rangle_{\rho,\beta}=\langle\overline{\Theta'(Tf)}Tf\rangle_{\rho,\beta} \geq 0,
     \end{equation}
     where $\Theta'$ is the reflection through the hyperplane orthogonal to $\mathbf{e}_1$ passing through $\mathbf{e}_1$. We used reflection positivity to obtain the last inequality.
     \item[$(iii)$] Iterating the Cauchy--Schwarz inequality as in \cite{Ott2019OZThesis}, we get, for all $f \in \mathcal{H}$,
     \begin{equation}
     |(Tf,f)|\leq (f,f),
     \end{equation}
     and thus, $\Vert T\Vert^{\textup{op}} \leq 1.$ To conclude, it suffices to notice that the constant function equal to one, that we denote by $\mathbf{1}$, satisfies $T\mathbf{1}=\mathbf{1}$.
 \end{enumerate}
 \end{proof}
In what follows we introduce many classical objects in the study of bounded self-adjoint operators in a Hilbert space. For all the definitions we refer to \cite{HallQuantumTheory2013}.
We are now in a position to apply the spectral theorem \cite[Theorem~7.12]{HallQuantumTheory2013}. 
\begin{Prop} There exists a unique projection valued measure $\mu^T$ such that
\begin{equation}
    T=\int_{\sigma(T)}\lambda \textup{d}\mu^T(\lambda).
\end{equation}
\end{Prop}
\begin{Rem}
One has that $\sigma(T)\subset [0,1]$.
\end{Rem}
We also state two propositions (which can be found in chapter 7 of \cite{HallQuantumTheory2013}) that will allow us to make good use of the preceding proposition.
\begin{Prop}
Let $f: \sigma(T)\rightarrow \mathbb C$ be a bounded measurable function. Then, 
\begin{equation}
    f(T)=\int_{\sigma(T)}f(\lambda)\textup{d}\mu^T(\lambda).
\end{equation}
\end{Prop}
\begin{Prop}\label{useful prop spectral measure}
If $f : \mathbb [0,1]\rightarrow \mathbb C$ is a bounded measurable function, and $\psi \in \mathcal{H}$, there exists a (positive) real-valued measure $\mu_\psi$ such that
\begin{equation}
    \left(\psi, \left(\int_{0}^1f(\lambda)\textup{d}\mu^T(\lambda)\right)\psi\right)=\int_0^1 f\textup{d}\mu_{\psi}.
\end{equation}
\end{Prop}
\begin{Rem}
The measure $\mu_\psi$ is given for $E \in \Omega$, by
\begin{equation}
    \mu_\psi(E)=(\psi,\mu(E)\psi).
\end{equation}
\end{Rem}
Recall that for $f,g \in \mathcal{H}$, the truncated correlation of $f$ and $g$ is given by
\begin{equation}
    \langle f;g\rangle_{\rho,\beta}:=\langle fg\rangle_{\rho,\beta}-\langle f\rangle_{\rho,\beta}\langle g\rangle_{\rho,\beta}.
\end{equation}
\begin{Prop}[Representation of truncated correlation functions]\label{prop: representation of truncated correlations}
For all $f \in \mathcal{H}$, and all $n\geq 0$, there exists $f_{\bot}\in \mathcal{H}$ such that, 
\begin{equation}
    \langle \overline{\Theta(f)};T^nf\rangle_{\rho,\beta}=(f_\bot,T^nf_\bot).
\end{equation}
\end{Prop}
\begin{proof}
Recall from above that $\mathbf{1}\in \mathcal{H}$ satisfies $T\mathbf{1}=\mathbf{1}$. By definition, for all $f \in \mathcal{H}$, $\langle f \rangle_{\rho,\beta}=(\mathbf{1},f)$. Thus,
\begin{equation}
    \langle \overline{\Theta(f)};T^n f\rangle_{\rho,\beta}=(f,T^nf)-(\mathbf{1},f)(\mathbf{1},f).
\end{equation}
Write $P_\bot$ the orthogonal projection on $\textup{Vect}(\mathbf{1})^\bot$. Letting
\begin{equation}
    f_\bot:=P_\bot f=f-(\mathbf{1},f)\mathbf{1},
\end{equation}
we find that 
\begin{equation}
    \langle \overline{\Theta(f)};T^nf\rangle_{\rho,\beta}=(f_\bot,T^nf_\bot).
\end{equation}
\end{proof}
We are now in a position to prove the main result of this section.
\begin{proof}[Proof of Theorem \textup{\ref{Spectral representation}}] Let $\beta\leq \beta_c(\rho)$.
Apply Proposition \ref{useful prop spectral measure} to $f:x\in [0,1]\mapsto x^n$ for $n\geq 0$, and $\psi=V_\bot$ where $V=\sum_{\xb\in \mathbb Z^{d-1}}v_{\xb}\tau_{(0,\xb)}\in \mathcal{H}$ to get
\begin{equation}
    (V_\bot,T^nV_\bot)=\int_0^1\lambda^n\textup{d}\mu_{V_\bot}(\lambda).
\end{equation}
Using Proposition \ref{prop: representation of truncated correlations}, we obtain
\begin{equation}
    \langle \overline{\Theta(V)}T^nV\rangle_{\rho,\beta}=\int_0^1\lambda^n\textup{d}\mu_{V_\bot}(\lambda).
\end{equation}
Now, notice that $\langle \overline{\Theta(V)}T^nV\rangle_{\rho,\beta}$ is exactly the left-hand side of \eqref{spec rep eq}. Moreover, considering the push-forward of $\mu_{V_\bot}$ under the map $a\in [0,1]\mapsto -\log a \in \mathbb R^+\cup\lbrace \infty \rbrace$, that we denote $\mu_{v,\beta}$, 
\begin{equation}
    \int_0^1\lambda^n\textup{d}\mu_{V_\bot}(\lambda)=\int_0^\infty e^{-an}\textup{d}\mu_{v,\beta}(a),
\end{equation}
and the result follows for all $n\in \mathbb Z$ using that $\langle \cdot \rangle_{\rho,\beta}$ is invariant under $T$.
\end{proof}
\begin{Rem}
Note that one may have
\begin{equation}
\mu_{v,\beta}(\lbrace \infty\rbrace)>0,
\end{equation}
which is exactly equivalent to the fact that $\xi(\rho,\beta)=\infty$.
\end{Rem}
We now present the proof of the monotonicity property of the two-point function's Fourier transform.
\begin{proof}[Proof of Proposition \textup{\ref{mono 2}}] First, notice that 
\begin{equation}
    \widehat{S}_{\rho,\beta}^{\textup{(mod)}}(p)=2\sum_{\substack{x\in \mathbb Z^d\\ x_1+x_2=0[2]}}e^{ip\cdot x}S_{\rho,\beta}(x).
\end{equation}
We follow the proof of Theorem \ref{Spectral representation} and keep the same notations. This time we introduce the operator $T':x \mapsto x+(1,1,0,\ldots).$ Let $R$ be the reflection with respect to the hyperplane $\Sigma'$ orthogonal to $\mathbf{e}_1+\mathbf{e}_2$ passing through $0$. $\Sigma'$ cuts $\mathbb Z^d$ in two half-planes $\Lambda'_+$ and $\Lambda'_-$ with $\Lambda'_+\cap \Lambda'_-=\Sigma'$. Let $\mathcal{A}'_+$ be the algebra generated by local functions with support in $\Lambda'_+$. Reflection positivity with respect to $R$ implies that for all $f\in \mathcal{A}'_+$, 
\begin{equation}
    \langle\overline{R(f)}f\rangle_{\rho,\beta}\geq 0.
\end{equation}
We define a positive semi-definite bilinear form on $\mathcal{A}'_+$ by : for all $f,g\in \mathcal{A}'_+$,
\begin{equation}
    (f,g):=\langle \overline{R(f)} g\rangle_\beta.
\end{equation}
Quotienting $\mathcal{A}'_+$ by the kernel of $(\cdot,\cdot)$ and completing the obtained space one obtains a Hilbert space $(\mathcal{H}',(\cdot,\cdot))$. Then, $T'$ can be seen as an operator of $\mathcal{H}'$. Using the same arguments as in Proposition \ref{prop: properties of T}, we also have that $T'$ is a self-adjoint, bounded and positive operator of $\mathcal{H}'$. Just as in Theorem \ref{Spectral representation}, we obtain that for all $v: \mathbb Z^{d-1}\rightarrow \mathbb C$ in $\ell^2(\mathbb Z^{d-1}),$ there exists a positive measure $\mu'_{v,\beta}$ such that, for all $n\in \mathbb Z$, 
\begin{equation}
    \sum_{(e,x_\flat),(e',y_\flat)\in \mathbb Z^{d-1}}v_{(e,x_\flat)}\overline{v_{(e',y_\flat)}}S_{\rho,\beta}(((e-e')+n,-(e-e')+n,x_\flat-y_\flat))=\int_0^\infty e^{-a|n|}\textup{d}\mu'_{v,\beta}(a).
\end{equation}
Fix $p_\flat=(p_3,\ldots,p_d)$. Let $q\in \mathbb R$ which will be fixed later. Considering the sequence of $\ell^2$ functions given for $L\geq 1$ by
\begin{equation}
    v^{(L)}_{(e,x_\flat)}=\dfrac{e^{iqe}e^{ip_\flat \cdot x_\flat}}{\sqrt{|\Lambda_L^{(d-1)}|}}\mathds{1}_{(e,x_\flat)\in \Lambda_L^{(d-1)}}, 
\end{equation}
we get, that there exists a positive measure $\mu'_{q,p_\flat,\beta}$  such that for $r \in \mathbb R$,
\begin{equation}
    \sum_{(n,e,z_\flat)\in \mathbb Z^d}e^{irn+iqe+ip_\flat \cdot z_\flat}S_{\rho,\beta}(e+n,-e+n,z_\flat)=\int_0^\infty \dfrac{e^a-e^{-a}}{\mathcal{E}_1(r)+\left(e^{a/2}-e^{-a/2}\right)^2}\text{d}\mu'_{q,p_\flat,\beta}(a).
\end{equation}
Taking $r=p_1+p_2$ and $q=p_1-p_2$, we get that
\begin{equation}
    \sum_{\substack{x\in \mathbb Z^d\\ x_1+x_2=0[2]}}e^{ip\cdot x}S_{\rho,\beta}(x)=\int_0^{\infty}\dfrac{e^a-e^{-a}}{\mathcal{E}_1(p_1+p_2)+\left(e^{a/2}-e^{-a/2}\right)^2}\textup{d}\mu'_{p_1-p_2,p_\flat,\beta}(a).
\end{equation}
Notice that in the formula above, one can use the symmetries of $\widehat{S}_{\rho,\beta}^{\textup{(mod)}}(p)$ to change $p_2$ into $-p_2$. As a result, we obtain that 
\begin{equation}
    \widehat{S}_{\rho,\beta}^{\textup{(mod)}}(p)=2\int_0^{\infty}\dfrac{e^a-e^{-a}}{\mathcal{E}_1(p_1-p_2)+\left(e^{a/2}-e^{-a/2}\right)^2}\textup{d}\mu'_{p_1+p_2,p_\flat,\beta}(a).
\end{equation}
This yields the result using the monotonicity of $u\in [0,\pi]\mapsto \mathcal{E}_1(u)$, as in the proof of Corollary \ref{mono 1}.
\end{proof}

\section[Properties of currents for Ising-type models]{Properties of currents for models of the Ising-type in the GS class}\label{appendix: bounds gs}
We recall a few classical bounds that can be found in \cite[Appendix~A.4]{AizenmanDuminilTriviality2021}. We keep the notations introduced in Section \ref{section: gs class}. Fix a measure $\rho$ of the Ising type in the GS class, and $\beta>0$.
\begin{Prop}\label{prop: bounds gs} For every distinct $x,y,u,v\in \mathbb Z^d$,
\begin{equation}\label{eq: bound connectivity gs sources}
    \mathbb P^{xy,\emptyset}_{\rho,\beta}[\partial\n_1\connect{\n_1+\n_2\:} \mathcal{B}_u]\leq \sum_{u'\in \mathbb Z^d}\frac{\langle \tau_x\tau_u\rangle_{\rho,\beta}(\beta J_{u,u'})\langle \tau_{u'}\tau_y\rangle_{\rho,\beta}}{\langle \tau_x\tau_y\rangle_{\rho,\beta}},
\end{equation}
and
\begin{equation}\label{eq: bound connectivity gs}
    \mathbb P^{\emptyset,\emptyset}_{\rho,\beta}[\mathcal{B}_x\connect{\n_1+\n_2\:} \mathcal{B}_y]\leq \sum_{x',y'\in \mathbb Z^d}\langle \tau_x\tau_y\rangle_{\rho,\beta}(\beta J_{y,y'})\langle \tau_{y'}\tau_{x'}\rangle_{\rho,\beta}\beta J_{x',x}.
\end{equation}
Moreover,
\begin{multline}\label{eq: 2 box connectivity gs}
    \mathbb P^{0x,\emptyset}_{\rho,\beta}[\partial\n_1\connect{\n_1+\n_2\:}\mathcal{B}_u,\mathcal{B}_v]
    \leq \sum_{u',v'\in \mathbb Z^d}\frac{\langle \tau_x\tau_u\rangle_{\rho,\beta}(\beta J_{u,u'})\langle \tau_{u'}\tau_v\rangle_{\rho,\beta}(\beta J_{v,v'})\langle \tau_{v'}\tau_y\rangle_{\rho,\beta}}{\langle \tau_x\tau_y\rangle_{\rho,\beta}}
    \\+
    \frac{\langle \tau_x\tau_v\rangle_{\rho,\beta}(\beta J_{v,v'})\langle \tau_{v'}\tau_u\rangle_{\rho,\beta}(\beta J_{u,u'})\langle \tau_{u'}\tau_y\rangle_{\rho,\beta}}{\langle \tau_x\tau_y\rangle_{\rho,\beta}}
\end{multline}
\end{Prop}
In the spirit of Proposition \ref{prop: monotonicity sources} we also have the following result.
\begin{Prop}[Monotonicity in the number of sources for the GS class]\label{prop: monotonicity sources gs} For every $x,y,z,t\in \mathbb Z^d$, every $u,u',u''$ with $u',u''$ $J$-neighbours of $u$, and every $S\subset \mathbb Z^d\times K_N$,
\begin{multline*}
    \mathbb{P}^{ux,uz,u'y,u''t}_{\rho,\beta}[\mathbf{C}_{\n_1+\n_3+\delta_{(\partial\n_1\cap \mathcal{B}_u,\partial \n_3\cap \mathcal{B}_{u'})}}(\partial\n_1)\cap \mathbf{C}_{\n_2+\n_4+\delta_{(\partial\n_2\cap \mathcal{B}_u,\partial \n_4\cap \mathcal{B}_{u''})}}(\partial\n_2)\cap S=\emptyset]\\\leq \mathbb{P}^{ux,uz,\emptyset,\emptyset}_{\rho,\beta}[\mathbf{C}_{\n_1+\n_3}(\partial\n_1)\cap \mathbf{C}_{\n_2+\n_4}(\partial\n_2)\cap S=\emptyset].
\end{multline*}
\end{Prop}
\section{Triviality and finiteness of the Bubble diagram}\label{appendix: bubble finite}
In this appendix, we prove that models in the GS class for which the bubble diagram is finite at criticality behave trivially. This provides an alternative proof to the results of Section \ref{section: dim eff >4} but it also captures more cases (for instance we may apply it to the case of algebraically decaying RP interactions for $d=4$ and $\alpha=2$).

Below we fix a measure $\rho$ in the GS class and, as in Section \ref{section: reflection positivity}, we denote the spin-field by $\tau$. The correlation length or order $\sigma>0$, mentioned in the introduction is given by
\begin{equation}
    \xi_{\sigma}(\rho,\beta):=\left(\frac{\sum_{x\in \mathbb Z^d}|x|^{\sigma}\langle \tau_0\tau_x\rangle_{\rho,\beta}}{\chi(\rho,\beta)}\right)^{1/\sigma}.
\end{equation}
We assume that we are given an interaction $J$ on $\mathbb Z^d$ satisfying $(\mathbf{A1})$--$(\mathbf{A5})$ and such that the above quantity can be defined for ${\sigma}$ small enough throughout the critical phase. 

The following result can be found in \cite{Sokal1982Destructive} and is a direct consequence of the Messager--Miracle-Solé inequality.
\begin{Prop}\label{Sokal bound}
Let $\beta<\beta_c(\rho)$. There exists a constant $C>0$ such that for all $x \in \mathbb Z^d$,
\begin{equation}
    \langle \tau_0\tau_x\rangle_{\rho,\beta}\leq C\frac{\chi(\rho,\beta)\xi_{\sigma}(\rho,\beta)^{\sigma}}{(1+|x|)^{d+{\sigma}}}.
\end{equation}
\end{Prop}
\begin{proof}
Using \eqref{eq: consequence mms }, 
\begin{equation}
    |\lbrace y\notin \Lambda_{|x|/(2d)}:\: S_{\rho,\beta}(y)\geq S_{\rho,\beta}(x)\rbrace|\geq C_1(1+|x|)^d,
\end{equation}
for some $C_1>0$. As a consequence, 
\begin{equation}
    \chi(\rho,\beta)\xi_{\sigma}(\rho,\beta)^{\sigma}=\sum_{y \in \mathbb Z^d}|y|^{\sigma} S_{\rho,\beta}(y)\geq C_2(1+|x|)^{d+{\sigma}}S_{\rho,\beta}(x),
\end{equation}
which concludes the proof.
\end{proof}
Recall that the \textit{renormalised coupling constant} of order ${\sigma}$ is defined by,
\begin{equation}
    g_{\sigma}(\rho,\beta):=-\frac{1}{\chi(\rho,\beta)^2\xi_{\sigma}(\rho,\beta)^d}\sum_{x,y,z\in \mathbb Z^d}U_4^{\rho,\beta}(0,x,y,z).
\end{equation}

\begin{Thm}[The bubble condition implies triviality]\label{thm: bubble condition implies triviality}
Let $d\geq 2$. For a reflection positive model in $\mathbb Z^d$ with an interaction $J$ satisfying the above conditions and such that
\begin{equation}
    B(\rho,\beta_c(\rho))=\sum_{x\in \mathbb Z^d}\langle\tau_0\tau_x\rangle_{\rho,\beta_c(\rho)}^2<\infty,
\end{equation}
one has,
\begin{equation}
    \lim_{\beta\nearrow \beta_c(\rho)}g_\sigma(\rho,\beta)=0.
\end{equation}
\end{Thm}
\begin{proof}
Using the tree diagram bound \eqref{eq tree diagram bound classic}, we get
\begin{equation}\label{borne g alpha}
    0\leq g_{\sigma}(\rho,\beta)\leq 2\frac{\chi(\rho,\beta)^2}{\xi_{\sigma}(\rho,\beta)^d}.
\end{equation}
Now, take $L\gg \varepsilon>0$ to be fixed later. Write
\begin{equation*}
    \chi(\rho,\beta)=\underbrace{\chi_{\varepsilon\xi_{\sigma}(\rho,\beta)}(\rho,\beta)}_{(1)}+\underbrace{\left(\chi_{L\xi_{\sigma}(\rho,\beta)}(\rho,\beta)-\chi_{\varepsilon\xi_{\sigma}(\rho,\beta)}(\rho,\beta)\right)}_{(2)}+\underbrace{\left(\chi(\rho,\beta)-\chi_{L\xi_{\sigma}(\rho,\beta)}(\rho,\beta)\right)}_{(3)}.
\end{equation*}
Using the Cauchy--Schwarz inequality one gets for $C_1>0$,
\begin{equation*}
    (1)\leq C_1 \varepsilon^{d/2}\xi_{\sigma}(\rho,\beta)^{d/2}\sqrt{B(\rho,\beta_c(\rho))},
\end{equation*}
and for $C_2=C_2(L,\varepsilon)>0$,
\begin{equation*}
    (2)\leq C_2\xi_{\sigma}(\rho,\beta)^{d/2}\sqrt{B_{L\xi_{\sigma}(\rho,\beta)}(\rho,\beta)-B_{\varepsilon\xi_{\sigma}(\rho,\beta)}(\rho,\beta)}.
\end{equation*}
Moreover, using Proposition \ref{Sokal bound}, we get that for $C_3>0$,
\begin{equation*}
    (3)\leq C_3\frac{\xi_{{\sigma}}(\rho,\beta)^{d/2}}{L^{\sigma}}\frac{\chi(\rho,\beta)}{\xi_{\sigma}(\rho,\beta)^{d/2}}.
\end{equation*}
Putting all the pieces together we get that,
\begin{equation*}
    \frac{\chi(\rho,\beta)}{\xi_{\sigma}(\rho,\beta)^{d/2}}\leq C_1\sqrt{B(\rho,\beta_c)}\varepsilon^{d/2}+C_2\sqrt{B_{L\xi_{\sigma}(\rho,\beta)}(\rho,\beta)-B_{\varepsilon\xi_{\sigma}(\rho,\beta)}(\rho,\beta)}+\frac{C_3}{L^{\sigma}}\frac{\chi(\rho,\beta)}{\xi_{\sigma}(\rho,\beta)^{d/2}}.
\end{equation*}
Fix $L>0$ large enough so that $\frac{C_3}{L^{\sigma}}<1$. Using the left-continuity of the two-point function, and the fact that $\xi_\sigma(\rho,\beta)\rightarrow \infty$ as $\beta\nearrow \beta_c(\rho)$ together with the monotone convergence theorem, we get that
\begin{equation*}
    B_{L\xi_{\sigma}(\rho,\beta)}(\rho,\beta),B_{\varepsilon\xi_{\sigma}(\rho,\beta)}(\rho,\beta)\underset{\beta\nearrow\beta_c(\rho)}{\longrightarrow}B(\rho,\beta_c(\rho)),
\end{equation*}
so that for all $\varepsilon>0$, 
\begin{equation*}
    \limsup_{\beta \rightarrow \beta_c(\rho)} \frac{\chi(\rho,\beta)}{\xi_{\sigma}(\rho,\beta)^{d/2}}\leq \frac{C_1\sqrt{B(\rho,\beta_c(\rho))}}{1-\frac{C_3}{L^{\sigma}}}\varepsilon^{d/2},
\end{equation*}
which yields the result using (\ref{borne g alpha}).
\end{proof}
\begin{Rem}
The above result could be extended to more general models in the GS class: if $J$ satisfies $\mathbf{(A1)}$-$\mathbf{(A4)}$, and if we consider an interaction $J$ for which both the bubble condition and the MMS inequalities hold (or more precisely \eqref{eq: consequence mms }), then the renormalised coupling constant vanishes at criticality.
Using the proof of the MMS inequality of \cite{AizenmanDuminilTassionWarzelEmergentPlanarity2019}, we may then extend our result to finite-range interactions\footnote{Recall that these interactions are not reflection positive in general.}.
\end{Rem}

\bibliographystyle{alpha}
\bibliography{biblio}
\end{document}